\documentclass[a4paper,10pt]{amsart}
\usepackage{fullpage}

\usepackage[utf8]{inputenc}
\usepackage[T2A,T1]{fontenc}
\usepackage{textcomp}
\usepackage{yfonts}

\usepackage{amsmath}
\usepackage{amsfonts}
\usepackage{amssymb}
\usepackage{amscd}
\usepackage{stmaryrd}
\usepackage{mathrsfs}
\usepackage{mathtext}
\usepackage{accents}
\usepackage{nicefrac}

\usepackage{comment}
\usepackage{enumerate}
\usepackage{fvextra}
\usepackage{xparse}

\usepackage{subfig}
\usepackage[dvipsnames]{xcolor}
\usepackage{tikz}
\usetikzlibrary{calc}
\usetikzlibrary{arrows.meta}
\usetikzlibrary{decorations.markings}

\tikzset{
  edge/.style={line width=.7pt},
  vertex/.style={circle,fill=black,inner sep=1.7pt},
  vlabel/.style={font=\small},
  gluearrow/.style={-{Stealth[length=2.8mm]},line width=.95pt}
}

\definecolor{myred}{HTML}{FF0000}
\definecolor{mygreen}{HTML}{00FF00}
\definecolor{myblue}{HTML}{0000FF}

\definecolor{myantired}{HTML}{00FFFF}
\definecolor{myantigreen}{HTML}{FF00FF}
\definecolor{myantiblue}{HTML}{FFFF00}

\definecolor{myyellow}{HTML}{FFFF00}
\definecolor{mymagenta}{HTML}{FF69B4}
\definecolor{mykhaki}{HTML}{BDB76B}
\definecolor{mycyan}{HTML}{55FFFF}
\definecolor{myorange}{HTML}{FFAA00}

\usepackage[all,tips]{xy}
\usepackage{soul}
\usepackage[normalem]{ulem}
\usepackage{tcolorbox}
\tcbuselibrary{breakable,skins}

\makeatletter
\@ifclassloaded{svjour3}{}{\RequirePackage{amsthm}}
\makeatother

\usepackage{hyperref}

\makeatletter
\@ifclassloaded{svjour3}{}{%
  \newtheorem{theorem}{Theorem}[section]
  \newtheorem{lemma}[theorem]{Lemma}

  \theoremstyle{definition}
  
  \newtheorem{example}[theorem]{Example}
  
  \newtheorem{proposition}[theorem]{Proposition}

  \theoremstyle{remark}
  \newtheorem{remark}[theorem]{Remark}

  \numberwithin{equation}{section}
}
\makeatother

\newcommand{\bbN}{{\mathbb N}}

\newcommand{\bbR}{{\mathbb R}}

\newcommand{\bbZ}{{\mathbb Z}}

\newcommand{\bfH}{{\mathbf H}}

\newcommand{\bfN}{{\mathbf N}}

\newcommand{\bfR}{{\mathbf R}}

\newcommand{\bfT}{{\mathbf T}}

\newcommand{\bfy}{{\mathbf y}}

\newcommand{\calA}{{\mathcal A}}
\newcommand{\calB}{{\mathcal B}}

\newcommand{\calD}{{\mathcal D}}
\newcommand{\calE}{{\mathcal E}}

\newcommand{\calG}{{\mathcal G}}

\newcommand{\calK}{{\mathcal K}}
\newcommand{\calL}{{\mathcal L}}
\newcommand{\calM}{{\mathcal M}}

\newcommand{\calP}{{\mathcal P}}

\newcommand{\calT}{{\mathcal T}}
\newcommand{\calU}{{\mathcal U}}

\newcommand{\calW}{{\mathcal W}}

\newcommand{\frake}{{\mathfrak e}}

\newcommand{\frakg}{{\mathfrak g}}

\newcommand{\frakB}{{\mathfrak B}}

\newcommand{\frakH}{{\mathfrak H}}

\newcommand{\frakT}{{\mathfrak T}}

\newcommand{\frakX}{{\mathfrak X}}

\newcommand{\frakZ}{{\mathfrak Z}}

\newcommand{\scrC}{{\mathscr C}}

\newcommand{\scrL}{{\mathscr L}}

\newcommand{\hata}{\hat{a}}
\newcommand{\hatb}{\hat{b}}

\newcommand{\hatu}{\hat{u}}
\newcommand{\hatv}{\hat{v}}

\newcommand{\hatx}{\hat{x}}

\newcommand{\hatz}{\hat{z}}

\newcommand{\hatF}{\hat{F}}

\newcommand{\hatT}{\hat{T}}

\newcommand{\hatV}{\hat{V}}

\makeatletter
\DeclareFontFamily{OMX}{MnSymbolE}{}
\DeclareSymbolFont{MnLargeSymbols}{OMX}{MnSymbolE}{m}{n}
\SetSymbolFont{MnLargeSymbols}{bold}{OMX}{MnSymbolE}{b}{n}
\DeclareFontShape{OMX}{MnSymbolE}{m}{n}{
    <-6>  MnSymbolE5
   <6-7>  MnSymbolE6
   <7-8>  MnSymbolE7
   <8-9>  MnSymbolE8
   <9-10> MnSymbolE9
  <10-12> MnSymbolE10
  <12->   MnSymbolE12
}{}
\DeclareFontShape{OMX}{MnSymbolE}{b}{n}{
    <-6>  MnSymbolE-Bold5
   <6-7>  MnSymbolE-Bold6
   <7-8>  MnSymbolE-Bold7
   <8-9>  MnSymbolE-Bold8
   <9-10> MnSymbolE-Bold9
  <10-12> MnSymbolE-Bold10
  <12->   MnSymbolE-Bold12
}{}

\let\llangle\@undefined
\let\rrangle\@undefined
\DeclareMathDelimiter{\llangle}{\mathopen}%
                     {MnLargeSymbols}{'164}{MnLargeSymbols}{'164}
\DeclareMathDelimiter{\rrangle}{\mathclose}%
                     {MnLargeSymbols}{'171}{MnLargeSymbols}{'171}
\makeatother

\DeclareSymbolFont{bbold}{U}{bbold}{m}{n}
\DeclareSymbolFontAlphabet{\mathbbold}{bbold}
\newcommand{\bboldone}{\mathbbold{1}}

\DeclareMathOperator*{\esssup}{esssup}

\newcommand{\cardinality}[1]{\# #1}
\DeclareMathOperator{\dom}{dom}

\newcommand{\inv}{{-1}}

\newcommand{\absolutevalue}[1]{|#1|}

\newcommand{\rng}{\operatorname{ran}}

\newcommand{\image}{\operatorname{ran}}

\newcommand{\Id}{\operatorname{Id}}

\newcommand{\diam}{{\operatorname{diam}}}

\newcommand{\dist}{\operatorname{dist}}
\newcommand{\volume}{\operatorname{vol}}
\newcommand{\density}{{\dif\mu}}

\newcommand{\trace}{\operatorname{tr}}

\newcommand{\determinant}{{\operatorname{det}}}

\newcommand{\grad}{\operatorname{grad}}
\newcommand{\curl}{\operatorname{curl}}

\newcommand{\divergence}{\operatorname{div}}

\newcommand{\cartan}{{\mathsf d}}
\newcommand{\cartanx}{{{\mathsf d}x}}
\newcommand{\cartanz}{{{\mathsf d}z}}

\newcommand{\Jacobian}{{\mathrm D}}
\newcommand{\dif}{{\mathrm d}}
\newcommand{\Lip}{\operatorname{Lip}}
\newcommand{\supp}{\operatorname{supp}}

\renewcommand{\restriction}{|}

\makeatletter
\newcommand\suchthat{%
 \@ifstar
  {\mathrel{}\middle|\mathrel{}}
  {\mid}%
}
\makeatother

\newcommand{\extrinsic}[1]{\overline{#1}}
\newcommand{\compute}[1]{\widetilde{#1}}

\newcommand{\Ball}[3]{%
  \if\relax\detokenize{#3}\relax
    \calB(#1,#2)%
  \else
    \calB(#1,#2,#3)%
  \fi
}
\newcommand{\Umfeld}[3]{%
  \if\relax\detokenize{#3}\relax
    \calD(#1,#2)%
  \else
    \calD(#1,#2,#3)%
  \fi
}

\newcommand{\closure}[1]{\operatorname{cl}{#1}}
\newcommand{\interior}[1]{\operatorname{int}{#1}}

\newcommand{\Cont}{C}
\newcommand{\Lebesgue}{L}
\newcommand{\Sobdiff}{\calL}
\newcommand{\Sobolev}{\calW}
\newcommand{\DegreeOfFreedom}{\scrC}
\newcommand{\ES}{\calE}

\newcommand{\Tangential}{\frakX}
\newcommand{\Tensor}{\frakT}
\newcommand{\Hom}{\operatorname{Hom}}

\newcommand{\loc}{\rm{loc}}
\newcommand{\dualityast}{\ast}

\newcommand{\Inside}{\omega}

\newcommand{\Domain}{{U}} %
\newcommand{\subdomain}{{V}} %
\newcommand{\subsubdomain}{{W}} %
\newcommand{\Manifold}{{\calM}} %
\newcommand{\otherManifold}{{\calM'}} %
\newcommand{\extManifold}{\extrinsic{\calM}} %

\newcommand{\metric}{\frakg}
\newcommand{\othermetric}{\frakg'}
\newcommand{\euclideanmetric}{\frake}

\newcommand{\Mesh}{{\calT}} %
\newcommand{\extMesh}{\extrinsic{\Mesh}} %
\newcommand{\referencesimplex}{\Delta}

\newcommand{\underlying}[1]{\big| #1 \big|}
\newcommand{\patchstar}{\star}

\newcommand{\betti}{\mathsf{b}}

\newcommand{\canonical}{\breve}
\newcommand{\widecanonical}{\breve}

\newcommand{\karte}{\tilde}
\newcommand{\widekarte}{\widetilde}

\newcommand{\homeo}{\Theta}
\newcommand{\incl}{\imath}
\newcommand{\patchincl}{\varphi}
\newcommand{\umleitung}{\vartheta}
\newcommand{\cpp}{a}
\newcommand{\mychart}{G}

\newcommand{\pullbackast}{\ast}

\newcommand{\Regu}{R}
\newcommand{\Interpolant}{I}
\newcommand{\smoothedinterpol}{Q}
\newcommand{\discreteinverse}{J}
\newcommand{\smoothedproj}{\pi}
\newcommand{\GramOperator}{\calG}
\newcommand{\Laplace}{\scrL}

\newcommand{\Ctrafosatz}[3]{C_{#1,#2,#3}} %
\newcommand{\Cinterpol}[3]{C_{#1,#2,#3}}
\newcommand{\CAri}[1]{C_{\calA,#1}}
\newcommand{\CinvAri}[1]{c_{\calA,#1}}

\newcommand{\Cfoo}[1]{C_{\circ}(#1)}

\newcommand{\Charmoniccomparison}{\beta}

\newcommand{\shapemeasurelipschitz}[1]{C_{\lambda}(#1)}
\newcommand{\shapemeasuredeterminant}[1]{C_{\det}(#1)}
\newcommand{\shapemeasure}[1]{C_{\mu}(#1)}
\newcommand{\highershapeconstant}[2]{C_{\mu,#1,#2}}

\newcommand{\comparisonconstant}[1]{C_{\rho}(#1)} %
\newcommand{\cardinalitybound}[1]{C_{\kappa}(#1)} %

\newcommand{\Cmeshsizevalue}{C_{h}}
\newcommand{\Cmeshsizelipschitz}{C_{\Lip,h}}

\newcommand{\Cpatchjacobian}[1]{C_{\mu,\star}(#1)}
\newcommand{\Cpatchjacobianhigher}[2]{C_{\mu,\star,#2}(#1)}
\newcommand{\Cpatchlipschitz}{C_{\mu}}

\newcommand{\Cumleitung}[1]{C_{\vartheta,\star}(#1)}
\newcommand{\Cumleitunghigher}[2]{C_{\vartheta,\star,#2}(#1)}

\newcommand{\Clebesguenormcomparison}[5]{C^{\calP}_{{#1},{#2},{#3},{#4},{#5}}}
\newcommand{\Clebesguenormlipschitz}[4]{C^{\calP,\nabla}_{{#1},{#2},{#3},{#4}}}

\newcommand{\neighborconstant}{\beta}

\newcommand{\NumberOfCharts}{{N_{\Psi}}}

\newcommand{\Catlasregularity}[1]{C_{\Psi,#1}}
\newcommand{\Catlaslipschitz}{C_{\Psi}}
\newcommand{\atlasradius}{\rho_{\Psi}}

\newcommand{\Cpoincareweber}[1]{C_{\PF,#1}}
\newcommand{\Cdiscretepoincareweber}[1]{C_{\PF,h,#1}}
\newcommand{\Ccomputablepoincareweber}[1]{\widetilde{C}_{\PF,h,#1}}

\newcommand{\CGramError}[1]{C_{G,#1}}
\newcommand{\CinvGramError}[1]{c_{G,#1}}

\newcommand{\molli}{\mu}
\newcommand{\skalierer}{\varrho}
\newcommand{\meshfunc}{\mathsf{h}}
\newcommand{\distancefunc}{\delta}

\newcommand{\Ned}{{\bfN\bf{d}}}
\newcommand{\RT}{{\bfR\bfT}}

\newcommand{\Dim}{d}
\newcommand{\PF}{\mathrm{PF}}

\newcommand{\coeff}{c}

\definecolor{OldColor}{HTML}{FF9900} %
\definecolor{NewColor}{HTML}{0000CC} %
\definecolor{MarkColor}{HTML}{FFF59D} %
\sethlcolor{MarkColor}

\newif\ifshowmarkup
\showmarkupfalse       %

\ifshowmarkup

\newenvironment{newstuff}
  {\begingroup\color{NewColor}\bfseries}
  {\endgroup}

\NewDocumentCommand{\todobox}{ O{yellow} m }{%
  \begin{tcolorbox}[
    enhanced,
    breakable,
    colback=#1!10,        %
    colframe=#1!70!black, %
    boxrule=0.8pt,
    arc=2pt,
    left=6pt,right=6pt,top=4pt,bottom=4pt
  ]
  \emph{TODO:}~#2
  \end{tcolorbox}
}

\NewDocumentCommand{\donebox}{ O{gray} m }{%
  \begin{tcolorbox}[
    enhanced,
    breakable,
    colback=#1!10,        %
    colframe=#1!70!black, %
    boxrule=0.8pt,
    arc=2pt,
    left=6pt,right=6pt,top=4pt,bottom=4pt
  ]
  \emph{DONE:}~#2
  \end{tcolorbox}
}

\NewDocumentCommand{\todoboxc}{ O{yellow!15} O{yellow!60!black} m }{%
  \begin{tcolorbox}[
    enhanced, breakable,
    colback=#1,
    colframe=#2,
    boxrule=0.8pt, arc=2pt
  ]
  \emph{TODO:}~#3
  \end{tcolorbox}
}

\else

  \excludecomment{oldstuff}

  \NewDocumentCommand{\todobox}{ O{yellow} m }{}
  \NewDocumentCommand{\donebox}{ O{gray} m }{}
  \NewDocumentCommand{\todoboxc}{ O{yellow!15} O{yellow!60!black} m }{}

\fi

\begin{document}

\title{Smoothed projections over manifolds\\in finite element exterior calculus}

\author{Martin Licht}
\address{Zellescher Weg 12-14, 01069 Dresden, Germany}
\email{martin\_werner.licht@tu-dresden.de}

\thanks{}


\date{\today}

\begin{abstract}
    We develop commuting finite element projections on compact Riemannian smooth manifolds.
    The commuting projections are constructed using localized smoothing operators, building on a classical construction by de~Rham.
    The projections map onto finite element spaces defined with respect to an intrinsic smooth triangulation of the manifold. 
    They are uniformly bounded on Lebesgue spaces of differential forms over shape-regular families of triangulations.
    These bounded cochain projections are a key component in the stability and convergence analysis of finite element methods for the Hodge--Laplace equation on manifolds.
    We derive Galerkin error estimates for the intrinsic finite element method and prove a piecewise Bramble--Hilbert lemma.
    Since many practical computations use extrinsic finite element methods on approximate computational manifolds, we also analyze the resulting geometric error.
    We bound the geometric variational crime in terms of the discrepancy between the exact Riemannian metric and its computational approximation. 
\end{abstract}

\maketitle

\section{Introduction}\label{section:intro}

Numerical methods for partial differential equations on surfaces and manifolds are an active area of research~\cite{holst2001adaptive,NeNiRuWh07,LeNeRu09,demlow2012adaptive,dziuk2013finite,olshanskii2009finite,nitschke2018nematic,bonito2020finite,bonito2019posteriori,bonito2020divergence,boon2021functional,brandner2022finite,hardering2023tangential}.
Finite element methods for scalar-valued partial differential equations on manifolds are extensively covered in the literature, but numerical methods for partial differential equations on manifolds involving vector fields remain less thoroughly understood.
This manuscript addresses the stability and convergence of finite element methods for the Hodge--Laplace equation on manifolds.
Its main contribution is the construction of smoothed projections from Lebesgue--de~Rham complexes onto finite element de~Rham complexes.

Partial differential equations for vector and tensor fields on manifolds naturally interact with concepts from differential geometry and algebraic topology. 
The de~Rham complex is pivotal to their analysis, and this is also reflected in numerical methods. 
The foundational role of \emph{finite element de~Rham complexes} in the design and analysis of mixed finite element methods is well established.
This line of thought has given rise to the formalism of \emph{finite element exterior calculus}~(FEEC), 
which recasts many vector-valued partial differential equations as variants of the Hodge--Laplace equation; see~\cite{hiptmair2001higher,hiptmair2002finite,AFW1,AFW2,christiansen2007stability,hiptmair2013sparse}.

\emph{Commuting projections} from the original de~Rham complex onto the finite element de~Rham complex are the mathematical link between the smooth and discrete settings: 
they facilitate the stability and quasi-optimal convergence of mixed finite element methods for the Hodge--Laplace equation.
However, constructing them is an intricate endeavor, even for polyhedral domains in Euclidean space.
Numerous projections have been constructed in the literature, applicable to different variants of the de~Rham complex, subject to different boundary conditions, and with different locality and boundedness properties; 
see, e.g.,~\cite{AFW1,christiansen2007stability,schoberl2008posteriori,christiansen2008smoothed,falk2012local,boffi2013mixed,ern2016mollification,christiansen2018eigenmode,licht2019smoothed,licht2019mixed,arnold2021local,ern2022equivalence}.
Existing constructions for partial differential equations on domains typically exploit the affine structure of Euclidean space and therefore do not admit a direct transfer to manifolds.
The present contribution extends \emph{smoothed projections} to manifolds.

We study the class of connected compact Riemannian smooth manifolds without boundary. 
We do not assume that the manifolds are orientable. 
The Lebesgue norms and inner products of differential forms are defined in terms of the Riemannian metric and the canonical density on a Riemannian manifold, irrespective of whether an orientation exists. 

We construct a smoothed projection on such manifolds equipped with a triangulation; 
we also discuss how our abstract assumptions align with computational practice.
The operator norm bounds of the smoothed projection depend only on geometric properties of the manifold, on the polynomial degree of the finite element spaces, and on the shape regularity of the triangulation. 
Here, the relevant shape regularity parameters also include second-order terms due to the non-Euclidean setting, and \emph{shape regularity} is to be understood with respect to this requirement: the control of the first- and second-order derivatives of reference transformations.
We ensure that all our estimates are explicit with regard to which constants they depend on.
The smoothed projections have uniformly bounded operator norms for the shape-regular families of meshes that are considered, e.g., in surface finite element methods. 

Besides the construction of the smoothed projection, a main contribution of the present work is in identifying the regularity conditions that are essential in the construction. 
While the construction of the commuting quasi-interpolants only uses the Lipschitz structure of the smooth manifold, the commuting projection additionally require quantitative control of the manifold's $C^{2}$ structure. 

In particular, we identify shape regularity constants of first and second order that control the stability of the smoothed projection and, a fortiori, the stability and convergence of finite element methods. 
The higher-order smoothness measures are a genuine feature of manifold triangulations, trivial for affine triangulations of polyhedral domains but consequential for approximation error estimates on triangulated manifolds.

\subsection{Outline of the construction}

We outline the construction of smoothed projections and their role in the theory of finite element methods.
For illustration, consider a connected compact $3$-dimensional Riemannian smooth manifold $\Manifold$ without boundary.
For simplicity, assume that $\Manifold$ is orientable.
In this geometric setting, exterior calculus can be expressed in the language of vector calculus through the gradient, curl, and divergence operators on scalar and vector fields~\cite{LeeSmooth}.
We begin by recalling the $\Lebesgue^{2}$ de~Rham complex:
\begin{gather*}
	\begin{CD}
		0 \to
		H^{1}(\Manifold)
		@>{\grad}>>
		\bfH(\curl,\Manifold)
		@>{\curl}>>
		\bfH(\divergence,\Manifold)
		@>{\divergence}>>
		\Lebesgue^{2}(\Manifold)
		\to
		0
		.
	\end{CD}
\end{gather*}
As a theoretical tool, we introduce an intrinsic smooth triangulation $\Mesh$ of the manifold $\Manifold$.
With respect to this triangulation, we may consider any \emph{finite element de~Rham complex}; one example is
\begin{gather*}
	\begin{CD}
		0 \to
		\calP_{r+1}(\Mesh)
		@>{\grad}>>
		\Ned_{r}(\Mesh)
		@>{\curl}>>
		\RT_{r}(\Mesh)
		@>{\divergence}>>
		\calP_{r,\rm{DC}}(\Mesh)
		\to
		0
		.
	\end{CD}
\end{gather*}
Here, broadly speaking, we define finite element spaces by transporting polynomial differential forms from (Euclidean) reference simplices to the intrinsic simplices of the manifold via inverse pullback.\footnote{There is, in general, no canonical intrinsic notion of polynomial differential forms on a manifold.}
For instance, this construction provides intrinsic notions of Lagrange elements, N\'ed\'elec elements, and Raviart--Thomas elements on triangulations of $3$-manifolds.
We seek to construct projections $\smoothedproj^{k}$, uniformly bounded for shape-regular families of triangulations, such that the following diagram is commutative:
\begin{gather*}
	\begin{CD}
		0 \to
		H^{1}(\Manifold)
		@>{\grad}>>
		\bfH(\curl,\Manifold)
		@>{\curl}>>
		\bfH(\divergence,\Manifold)
		@>{\divergence}>>
		\Lebesgue^{2}(\Manifold)
		\to
		0
		\\
		@V{\smoothedproj^{0}}VV
		@V{\smoothedproj^{1}}VV
		@V{\smoothedproj^{2}}VV
		@V{\smoothedproj^{3}}VV
		\\
		0 \to
		\calP_{r+1}(\Mesh)
		@>{\grad}>>
		\Ned_{r}(\Mesh)
		@>{\curl}>>
		\RT_{r}(\Mesh)
		@>{\divergence}>>
		\calP_{r,\rm{DC}}(\Mesh)
		\to
		0
		.
	\end{CD}
\end{gather*}
Once such smoothed projections are available on the manifold, the \emph{Galerkin theory of Hilbert complexes}~\cite{AFW2} implies the stability and quasi-optimality of numerous finite element methods for the Hodge--Laplace equation.

In this manuscript, we construct the smoothed projections in the following steps; see also~\cite{AFW1,christiansen2008smoothed,schoberl2008posteriori}. 
First, we construct smoothing operators on $\Manifold$.
These linear operators map $\Lebesgue^{2}$ differential forms to smooth differential forms.
They commute with the differential operators and satisfy local bounds.
This yields a commutative diagram:
\begin{gather*}
	\begin{CD}
		0 \to
		H^{1}(\Manifold)
		@>{\grad}>>
		\bfH(\curl,\Manifold)
		@>{\curl}>>
		\bfH(\divergence,\Manifold)
		@>{\divergence}>>
		\Lebesgue^{2}(\Manifold)
		\to
		0
		\\
		@V{R_{h}^{0}}VV
		@V{R_{h}^{1}}VV
		@V{R_{h}^{2}}VV
		@V{R_{h}^{3}}VV
		\\
		0 \to
		\Cont^{\infty}(\Manifold)
		@>{\grad}>>
		\Cont^{\infty}\Tangential(\Manifold)
		@>{\curl}>>
		\Cont^{\infty}\Tangential(\Manifold)
		@>{\divergence}>>
		\Cont^{\infty}(\Manifold)
		\to
		0
		.
	\end{CD}
\end{gather*}
If global coordinates were available, as they are for Euclidean domains, then we could model the smoothing operator after classical convolution with a bump function.
However, on manifolds without global coordinate systems, we use \emph{localized} smoothing operators in coordinate charts, extending ideas from de~Rham's seminal monograph~\cite{derham1984differentiable}.
The composition of these local smoothing operations has the desired properties.

The second component of the smoothed projection is a family of canonical finite element interpolants $\Interpolant_{h}^{k}$.
These interpolants also fit into a commutative diagram:
\begin{gather*}
	\begin{CD}
		0 \to
		\Cont^{\infty}(\Manifold)
		@>{\grad}>>
		\Cont^{\infty}\Tangential(\Manifold)
		@>{\curl}>>
		\Cont^{\infty}\Tangential(\Manifold)
		@>{\divergence}>>
		\Cont^{\infty}(\Manifold)
		\to
		0
		\\
		@V{\Interpolant_{h}^{0}}VV
		@V{\Interpolant_{h}^{1}}VV
		@V{\Interpolant_{h}^{2}}VV
		@V{\Interpolant_{h}^{3}}VV
		\\
		0 \to
		\calP_{r+1}(\Mesh)
		@>{\grad}>>
		\Ned_{r}(\Mesh)
		@>{\curl}>>
		\RT_{r}(\Mesh)
		@>{\divergence}>>
		\calP_{r,\rm{DC}}(\Mesh)
		\to
		0
		.
	\end{CD}
\end{gather*}
Composing canonical interpolants and smoothing operators yields commuting quasi-interpolants
\[ 
    \smoothedinterpol_{h}^{k} := \Interpolant_{h}^{k} \circ R_{h}^{k} 
    .
\]
These operators satisfy bounds that are uniform in the mesh size and commute with the differential operators, but they are generally not idempotent.

The smoothed quasi-interpolants do not act as the identity on the finite element spaces.
However, they act as perturbations of the identity on the finite element spaces. Choosing parameters to make these perturbations sufficiently small, the restrictions of the smoothed quasi-interpolants to these spaces have uniformly bounded commuting inverses $\discreteinverse_{h}^{k}$.
These restrictions' inverses $\discreteinverse_{h}^{k}$ satisfy uniform bounds and commute with the differential operators too.
This leads to another commutative diagram:
\begin{gather*}
	\begin{CD}
		0 \to
		\calP_{r+1}(\Mesh)
		@>{\grad}>>
		\Ned_{r}(\Mesh)
		@>{\curl}>>
		\RT_{r}(\Mesh)
		@>{\divergence}>>
		\calP_{r,\rm{DC}}(\Mesh)
		\to
		0
		\\
		@V{\discreteinverse_{h}^{0}}VV
		@V{\discreteinverse_{h}^{1}}VV
		@V{\discreteinverse_{h}^{2}}VV
		@V{\discreteinverse_{h}^{3}}VV
		\\
		0 \to
		\calP_{r+1}(\Mesh)
		@>{\grad}>>
		\Ned_{r}(\Mesh)
		@>{\curl}>>
		\RT_{r}(\Mesh)
		@>{\divergence}>>
		\calP_{r,\rm{DC}}(\Mesh)
		\to
		0
		.
	\end{CD}
\end{gather*}
This finite-dimensional (albeit non-local) inverse allows us to construct the desired smoothed projections:
\[
    \smoothedproj^{k} := \discreteinverse_{h}^{k} \circ \Interpolant_{h}^{k} \circ R_{h}^{k}.
\]

The preceding illustration was stated in terms of vector calculus on three-dimensional manifolds. 
More generally, any connected compact $n$-dimensional Riemannian smooth manifold without boundary admits the \emph{calculus of differential forms}, also known as \emph{exterior calculus}, which concerns alternating covariant tensor fields. It is a higher-dimensional generalization of vector calculus.
Consequently, we formulate our findings in the formalism of \emph{finite element exterior calculus}.
Furthermore, we do not restrict ourselves to Hilbert spaces but address the full range of Lebesgue exponents $1 \leq p \leq \infty$.

\subsection{Error analysis of intrinsic finite element methods}

Having outlined the main contribution, we now place this smoothed projection within the broader context of finite element methods on manifolds. 
This is to validate how our smoothed projection enters the error analysis of finite element methods on manifolds, analogous to its counterpart on polyhedral domains.

We first define the \emph{intrinsic triangulation} of the manifold and then establish the quasi-optimality and stability of the \emph{intrinsic finite element method}. 
The notion of smooth triangulations of smooth manifolds is central for this purpose~\cite{LeeSmooth}. 
Here, a smooth simplex is a smooth embedding of the reference simplex into the manifold, in the sense of smooth manifolds with corners~\cite{joyce2012manifolds}.

The primary purpose of the smoothed projection is to establish the stability and convergence of practically implementable numerical methods for the Hodge--Laplace equation on manifolds. 
The finite element spaces are defined via pullback from reference simplices; any of the finite element de~Rham sequences can be used~\cite{hiptmair2002finite,AFW1,AFW2}. 
The smoothed projection is central in the Galerkin theory of Hilbert complexes to establish the stability and quasi-optimal convergence of the intrinsic finite element method, in full analogy to the convergence analysis on polyhedral domains~\cite{AFW2,boffi2013mixed}. 

A critical but easily overlooked question is how to establish best-approximation error estimates in terms of the mesh size for the intrinsic finite element method.
The relevant regularity classes are higher-order Sobolev spaces.
On manifolds, higher-order Sobolev regularity is expressed through weak covariant derivatives, such as the Levi-Civita covariant derivative. 
Scaling arguments for a differential form ${u}$ along a diffeomorphism of Riemannian manifolds $\Phi : (\Manifold,\frakg) \to (\otherManifold,\frakg')$ rely on Fa\`a~di~Bruno formulas of the form 
\begin{align*}
    \| \nabla^{m} \Phi^{\pullbackast} {u} \|_{\Lebesgue^{p}(\Manifold,\metric)}
    \leq
    C
    \sum_{j=0}^{m}
    P_{m,j} \left( 
        \|              \Jacobian\Phi \|_{\infty}
        ,
        \| \nabla       \Jacobian\Phi \|_{\infty}
        ,
        \dots
        ,
        \| \nabla^{m-j} \Jacobian\Phi \|_{\infty}
    \right) 
    \| \det \Jacobian\Phi^{-1} \|^{\frac{1}{p}}_{\Lebesgue^{\infty}(\otherManifold,\othermetric)}
    \| \nabla^{j} {u} \|_{\Lebesgue^{p}(\otherManifold,\othermetric)}
    ,
\end{align*}
involving covariant derivatives of the Jacobians. 

If the solutions admit sufficiently many weak covariant derivatives, then Sobolev embedding theorems imply their continuity and the canonical interpolant satisfies optimal error estimates. 
More generally, such estimates can be obtained using a \emph{piecewise Bramble--Hilbert lemma}~\cite{veeser2016approximating,camacho2015L2,licht2021local,licht2023averaging}.
We prove a piecewise Bramble--Hilbert lemma for differential forms on manifolds. 
For example, if $u \in \bfH(\divergence,\Manifold)$ and $T \in \Mesh$ is a smooth tetrahedron, 
then the local error in the $\Lebesgue^{2}$ norm satisfies an inequality of the form 
\begin{align} 
    \inf_{ u_{h} \in \RT_{r}(\Mesh) } \| u - u_{h} \|_{\Lebesgue^{2}(T,\metric)}
    \leq 
    C
    \sum_{ \substack{ T' \in \Mesh \\ T \cap T' \neq \emptyset } }
    \left( h_{T}^{m}
    \| u \|_{\Sobolev^{m,2}(T')}
    +
    h_{T}^{\ell + 1}
    \| \divergence u \|_{\Sobolev^{\ell,2}(T')} \right)
\end{align}
whenever the respective elementwise Sobolev norms of orders $m$ and $\ell$ on the right-hand side are finite. 
The error analysis is generally analogous to finite element approximations on polyhedral domains, but on manifolds it also involves higher-order shape constants that measuring the smoothness of the simplices.

Smoothed projections enable the Galerkin theory of Hilbert complexes for finite element methods on a broad class of non-Euclidean geometries.
The intrinsic finite element method admits an error analysis fully analogous to what is known on triangulated Euclidean domains~\cite{hiptmair2002finite,AFW1,AFW2}.
However, such intrinsic triangulations are directly computable only in special settings.
We therefore regard the intrinsic finite element method as primarily a theoretical construction, though it is the centerpiece of our error analysis.
Therefore, a treatment of practical methods must take into account geometric variational crimes.

\subsection{Geometric variational crimes}

Since the intrinsic finite element method is only directly computable in special cases, many practical methods must resort to an approximate \emph{computational finite element method} on an \emph{extrinsic triangulation}. 
The literature describes several different approaches, all of which incur an additional \emph{geometric error} that we need to control to ensure optimal approximation rates. 
To ensure that our analysis aligns with computational practice, we study the geometric error in more detail. 

The analysis proceeds as follows.
We introduce an explicitly known \emph{computational manifold} $\extManifold$ equipped with an affine triangulation $\extMesh$.
We regard $\extManifold$ as an approximation of the true \emph{physical manifold} $\Manifold$.
The error analysis of the computational problem relies on a bi-Lipschitz piecewise smooth homeomorphism $\homeo \colon \extManifold \to \Manifold$ that maps the extrinsic computational triangulation onto the intrinsic one.
Such a homeomorphism $\homeo$ is explicitly known in some applications, such as for the cubed sphere or multicube domains~\cite{sadourny1972conservative,lindblom2013solving}. 
But often only its existence and regularity properties are known, as is the typical case with the closest-point projection onto hypersurfaces in the surface finite element literature~\cite{demlow2009higher}.
Consequently, the homeomorphism might generally be only of theoretical significance. 

The intrinsic triangulation is defined by transferring the extrinsic triangulation to the physical manifold via that homeomorphism.
Conversely, we translate the intrinsic finite element problem into an equivalent extrinsic finite element problem on the computational triangulation.
This extrinsic formulation may have no direct implementation because the exact metric on the computational manifold, which is the pullback of the Riemannian metric from the physical manifold, may be incomputable in practice.
Replacing this pullback Riemannian metric with a computationally feasible (approximate) Riemannian metric finally defines the implementable \emph{computational finite element problem}. 
How to specifically replace the exact Riemannian metric by an approximate Riemannian metric depends on the concrete implementation details.

This replacement introduces an additional \emph{geometric error}, which our formalism reduces to the comparison between the exact metric and the computational metric.
We take inspiration from Holst and Stern's abstract approximation theory~\cite{HS1}.
Among the examples that we address, the error analysis of surface finite element methods is the most challenging. 
The analysis of surface finite element methods assumes that the extrinsic surface is sufficiently close to the physical surface and then defines the extrinsic problem using an interpolated closest-point projection; the resulting geometric error can then be bounded using an estimate proved in~\cite{christiansen2002resolution}; see also~\cite{demlow2007adaptive,olshanskii2009finite,dziuk2013finite}. 
In concrete discretization schemes, the geometric error is controlled by mesh constants and the curvature properties of the physical manifold.
Intuitively, the geometric error is small if the homeomorphism $\homeo$ is close to being an isometry.

\subsection{Related work and outlook} 

We build on the theory of smoothed projections on polyhedral domains. 
Smoothed projections were first introduced for quasi-uniform meshes~\cite{AFW1}, and subsequently for shape-regular meshes and essential boundary conditions~\cite{christiansen2008smoothed}. There are extensions to weakly Lipschitz domains and mixed boundary conditions~\cite{licht2019smoothed,licht2019mixed}. The literature has also provided projections with enhanced locality~\cite{falk2012local,arnold2021local}. Commuting projections generalize the Fortin interpolant in the theory of mixed finite element methods and can be found in standard reference works~\cite{boffi2013mixed,ern2016mollification,ern2021finite}. 
We emphasize that the commuting projections, whilst stable for regular families of triangulations of a fixed manifold, are not local, nor do they exhibit any proven stability in the polynomial degree. 

Commuting cochain projections have certain parallels in geometric analysis. There, \emph{de~Rham regularizers} correspond to our global smoothing operators, being local cochain mappings from Sobolev de~Rham complexes onto smooth forms, and serve in establishing the equivalence of cohomology theories.
These eponymous regularizers were popularized in the work of de~Rham~\cite{derham1984differentiable} and have been applied to the cohomology of smooth and non-smooth de~Rham complexes~\cite{Dodziuk1981,Goldshtein1984,Goldshtein1988,Kopylov2007}. 

Our construction of smoothed projections on manifolds is complementary to the abstract analysis of geometric variational crimes on Hilbert complexes~\cite{HS1}, which provides an error analysis for finite element methods on manifolds and hypersurfaces under the working assumption that a commuting projection is available and that finite element best-approximation estimates hold; we contribute such a commuting projection. The discussion of the geometric error already points to the main areas where our smoothed projection can play a role. 

Surface finite element methods pose finite element problems on curved elements that approximate a given smooth hypersurface~\cite{demlow2007adaptive,demlow2009higher,olshanskii2009finite,dziuk2007finite,dziuk2013finite,nitschke2018nematic,bonito2019posteriori,bonito2020finite,bonito2020divergence}. 
Implicitly, these methods use an extrinsic affine triangulation that is equipped with a Riemannian metric pulled back from the curved elements.
Surface FEM therefore fall within the scope of our geometric error analysis. 
Besides the analysis of surface finite element methods, the commuting projection is applicable to multicube discretizations~\cite{lindblom2013solving} and cubed-sphere discretizations~\cite{sadourny1972conservative,ronchi1996cubed} in any dimension. 
Representative applications include mixed Darcy flow on surfaces and Maxwell-type problems on curved three-dimensional Riemannian manifolds.
The latter may arise, for instance, on spatial slices used in relativistic formulations of electromagnetism.
We also consider this a prelude to the construction of smoothed projections for differential complexes in the analysis of mixed-dimensional PDEs~\cite{boon2021functional,boon2023mixed}, where PDEs on three-dimensional volumes are coupled with PDEs on two-dimensional interfaces.
We point out that the canonical interpolant, which may serve as a Fortin interpolant on two-dimensional surfaces in high-regularity situations, is not defined on rougher function spaces. 
By contrast, our smoothed projection is still defined under low regularity, such as when the physical system is subject to rough coefficients.

Among the multitude of possible directions for future research, we point out evolving surfaces. 
Geometric evolution problems couple PDEs on a manifold coupled with a geometric flow of its geometry, so that the Riemannian structure of the manifold itself is subject to continuous change. 
This analysis would benefit from a smoothed projection whose stability is uniformly bounded for the resulting family of evolving triangulated manifolds. 
This open problem is closely related to the question of how the commuting projection can be defined in situations of lower geometric regularity and how the operator norms of the smoothed projection depend \emph{en d\'etail} on the curvature properties of the Riemannian manifold.

\subsection{Structure of the manuscript}

The remainder of this article is structured as follows.
In Section~\ref{section:preliminaries}, we review the calculus of differential forms on Riemannian smooth manifolds.
In Section~\ref{section:triangulations}, we introduce triangulations, and in Section~\ref{section:finiteelementspaces}, we discuss abstract finite element spaces.
In Section~\ref{section:smoothing}, we discuss a class of localized smoothing operators on $\bbR^{n}$.
We construct the smoothed projections in Section~\ref{section:smoothedprojections}.
We discuss the a~priori error analysis of the intrinsic finite element method, including the piecewise Bramble--Hilbert lemma, in Section~\ref{section:applications}.
Finally, in Section~\ref{section:computationaltriangulations}, we discuss concrete examples of our abstract framework and the geometric errors associated with the computational problems.

\section{Differential forms on smooth manifolds}\label{section:preliminaries}

This section establishes notation and fundamental notions for analysis on manifolds.
Our notation follows standard conventions.
We assume that the reader is familiar with vector bundles and tensor fields on manifolds and, in particular, the calculus of differential forms; we refer to the monographs by Ahsan~\cite{ahsan2015tensors}, Hou and Hou~\cite{hou1997differential}, Lee~\cite{LeeSmooth}, and Lovett~\cite{lovett2019differential} as general references.

Some very general conventions are these.
For $r > 0$ and $x \in \bbR^{n}$, let $\Ball{x}{r}{} \subseteq \bbR^{n}$ denote the open Euclidean ball of radius $r$ centered at $x$.
For $k \in \bbN_{0}$ and $n \in \bbN_{0}$, let $\Sigma(k,n)$ be the set of strictly increasing functions from $\{ 1, \dots, k \}$ to $\{ 1, \dots, n \}$.
We emphasize that, by definition, $\Sigma(0,n) = \{\emptyset\}$, and $\Sigma(k,n) = \emptyset$ whenever $k > n$.
We adopt the convention that $\Sigma(k,n) := \emptyset$ if $k < 0$.

\subsection{Manifolds and differential forms}

Let $\Manifold$ be a connected smooth manifold of dimension $n$ without boundary.
We emphasize that at this stage, the manifold may also be any connected open subset of Euclidean space or any compact or non-compact connected smooth manifold. 
We use the following standard vector bundles.
We write $\Tangential\Manifold$ for the tangent bundle of $\Manifold$ and $\Tangential^{\dualityast}\Manifold$ for the cotangent bundle. 
In what follows, $\Tensor^{k}\Manifold$ denotes the $k$-th covariant tensor bundle of $\Manifold$, which is the $k$-th tensor power of the cotangent bundle. 
The $k$-th exterior power of the cotangent bundle, $\Lambda^{k}\Tangential^{\dualityast}\Manifold$, is a subbundle of $\Tensor^{k}\Manifold$.
We remark that $\Lambda^{k}\Tangential^{\dualityast}\Manifold$ is the zero bundle for $k \notin \{0,\dots,n\}$. 

We let $\Cont^{\infty}(\Manifold)$ be the space of smooth real-valued functions on $\Manifold$. 
Let $\Cont^{\infty}\Lambda^{k}(\Manifold)$ denote the space of smooth differential $k$-forms on $\Manifold$.
More generally, $\Cont^{\infty}\Tensor^{k}(\Manifold)$ denotes the space of smooth covariant $k$-tensor fields on $\Manifold$. 
Write $\bboldone_{\Manifold} \in \Cont^{\infty}(\Manifold)$ for the constant function with value one on $\Manifold$.
For $m \in \bbN_{0}$, we write $\Cont^{m}\Lambda^{k}(\Manifold)$ for the space of differential $k$-forms on $\Manifold$ whose coefficients in any local coordinate chart have $m$ continuous derivatives, and we define $\Cont^{m}\Tensor^{k}(\Manifold)$ analogously. 
Let $\Cont\Lambda^{k}(\Manifold) := \Cont^{0}\Lambda^{k}(\Manifold)$ be the space of continuous differential $k$-forms,
and let $\Cont^{\infty}_{c}\Lambda^{k}(\Manifold)$ be the space of those members of $\Cont^{\infty}\Lambda^{k}(\Manifold)$ whose support is a compact subset of $\Manifold$.

We review a few standard operations on differential forms.
A differential operator central to our discussion is the \emph{exterior derivative}
\begin{gather}\label{math:exteriorderivative}
	\cartan : \Cont^{\infty}\Lambda^{k}(\Manifold) \to \Cont^{\infty}\Lambda^{k+1}(\Manifold).
\end{gather}
The exterior derivative satisfies $\cartan \cartan {u} = 0$ for every smooth differential $k$-form ${u}$.
The \emph{exterior product} of ${u} \in \Cont^{\infty}\Lambda^{k}(\Manifold)$ and ${v} \in \Cont^{\infty}\Lambda^{\ell}(\Manifold)$ is a differential $(k+\ell)$-form and is written ${u} \wedge {v} \in \Cont^{\infty}\Lambda^{k+\ell}(\Manifold)$.

Occasionally, we work in coordinates.
If $\Domain \subseteq \Manifold$ is a coordinate domain with smooth coordinate functions $x^{1}, x^{2}, \dots, x^{n} : \Domain \to \bbR$, then any ${u} \in \Cont^{\infty}\Lambda^{k}(\Domain)$ has the \emph{standard representation} 
with respect to those coordinates
\begin{gather}\label{math:standardrepresentation}
	{u} = \sum_{\sigma \in \Sigma(k,n)} {u}_{\sigma} \cartanx^{\sigma}, \qquad \cartanx^{\sigma} := \cartanx^{\sigma(1)} \wedge \dots \wedge \cartanx^{\sigma(k)},
\end{gather}
where ${u}_{\sigma} \in \Cont^{\infty}(\Domain)$.
Recall that $\cartanx^{\emptyset} := 1$ in the case $k=0$.

\subsection{Measurable differential forms}

Our discussion involves fields with non-smooth coefficients.
There is an invariant notion of tensor fields, and thus of differential forms, with measurable coefficients.
In particular, the exterior product of two measurable differential forms is again measurable, and any measurable differential form has a standard representation~\eqref{math:standardrepresentation} with measurable coefficients in any local coordinate chart.

For $p \in [1,\infty]$, we define $\Lebesgue^{p}_{\loc}(\Manifold)$ as the vector space of measurable functions on $\Manifold$ that are locally $p$-integrable.
By definition, ${u} \in \Lebesgue^{p}_{\loc}(\Manifold)$ if and only if 
${u} \circ \varphi^{-1} \in \Lebesgue^{p}_{\loc}(\varphi(\Domain))$ for every smooth coordinate chart $\varphi : \Domain \to \varphi(\Domain) \subseteq \bbR^{n}$.
We write $\Lebesgue^{p}_{\loc}\Lambda^{k}(\Manifold)$ for the vector space of differential $k$-forms on $\Manifold$ that have locally ${p}$-integrable coefficients in every coordinate chart.
More generally, $\Lebesgue^{p}_{\loc}\Tensor^{k}(\Manifold)$ denotes the vector space of covariant $k$-tensors on $\Manifold$ that have locally ${p}$-integrable coefficients in every coordinate chart.
Here and in what follows, we tacitly identify tensor fields that agree almost everywhere. 

Given any oriented open subset $\Domain \subseteq \Manifold$ and any measurable differential $n$-form ${u}$, the integrability of ${u}$ and the value of the integral $\int_{\Domain} {u}$ are independent of any choice of local coordinates. %

The exterior derivative extends to differential forms with coefficients in Lebesgue spaces.
This requires the notion of a weak exterior derivative.
We call $w \in \Lebesgue^{1}_{\loc}\Lambda^{k+1}(\Manifold)$ the \emph{weak exterior derivative} of $u \in \Lebesgue^{1}_{\loc}\Lambda^{k}(\Manifold)$
if every point of $\Manifold$ has an oriented open neighborhood $\Domain \subseteq \Manifold$ such that for all $v \in \Cont^{\infty}_{c}\Lambda^{n-k-1}(\Domain)$, we have 
\begin{align}
	\int_{\Domain} u \wedge \cartan v = (-1)^{k+1} \int_{\Domain} w \wedge v.
\end{align}
The weak exterior derivative is unique up to equality almost everywhere, and we denote it by $\cartan u$.
Our particular interest is the following class of Sobolev spaces of differential forms: for any $p, q \in [1,\infty]$, we define
\begin{align}\label{math:localsobolevspace}
	\Sobdiff^{p,q}_{\loc}\Lambda^{k}(\Manifold)
	:=
	\left\{\;
	{u} \in \Lebesgue^{p}_{\loc}\Lambda^{k}(\Manifold)
	\suchthat*
	\text{ $u$ has a weak exterior derivative } \cartan{u} \in \Lebesgue^{q}_{\loc}\Lambda^{k+1}(\Manifold)
	\;\right\}
	.
\end{align}

\subsection{Metric constructions}

We assume that $\Manifold$ is equipped with a (smooth) Riemannian metric $\metric$.
In other words, $\metric$ is a smooth symmetric covariant tensor field of order $2$ on $\Manifold$
that can be expressed in any local coordinate chart as a smooth field of symmetric positive-definite bilinear forms.
The Riemannian metric induces numerous other structures on the manifold.

Any Riemannian metric $\metric$ induces an inner product $\langle\cdot,\cdot\rangle_{\metric}$ and its associated norm $\lVert\cdot\rVert_{\metric}$ on every fiber of the tangent bundle.
This induces an inner product and a norm on the cotangent bundle,
and both the inner product and the norm extend naturally to the fibers of tensor powers of the tangent and cotangent bundles.
We tacitly use the same symbol $\metric$ for these induced sections of bilinear pairings on the covariant tensor bundles.
One consequence is that we have a bilinear pairing 
\begin{gather}\label{math:bilinearpairing}
	\langle\cdot,\cdot\rangle_{\metric} :
	\Cont^{\infty}\Tensor^{k}(\Manifold) \times \Cont^{\infty}\Tensor^{k}(\Manifold)
	\to
	\Cont^{\infty}(\Manifold),
	\quad
	({u},{v})
	\mapsto
	\langle {u}, {v} \rangle_{\metric},
\end{gather}
and a pointwise norm 
\[
    \lVert{u}\rVert_{\metric} := \sqrt{ \langle {u},{u}\rangle_{\metric} }.
\]
Notice that $\langle{u},{u}\rangle_{\metric} = 0$ almost everywhere if and only if ${u} = 0$ almost everywhere.

Every Riemannian manifold has a canonical density $\density^{\Manifold}_{\metric}$.
This provides an invariant definition of the integral $\int f \density^{\Manifold}_{\metric}$ of a measurable function $f : \Manifold \to \bbR$, independent of coordinates whenever the integral is defined.

\begin{remark}
    Throughout this manuscript, the Riemannian metric enters the definitions of the various Banach and Hilbert spaces and their norms and inner products.
	We denote the standard Riemannian metric on Euclidean space by $\euclideanmetric$,
    but we will omit the Euclidean Riemannian metric from the notation whenever it serves the exposition.
\end{remark}

Since $\Manifold$ is connected,
the Riemannian metric $\metric$ induces a metric $\dist_{\metric} : \Manifold \times \Manifold \to [0,\infty)$,
where ``metric'' is understood here in the sense of metric spaces.\footnote{We will use the term \emph{metric} only in the sense of metric spaces. We call a \emph{metric} in the sense of differential geometry a \emph{Riemannian metric} or \emph{metric tensor}.}
Specifically, two points $x, y \in \Manifold$ have the \emph{geodesic distance}
\begin{align*}
	\dist_{\metric}(x,y) := \inf_{\gamma} \int_{0}^{1} \lVert\partial_t \gamma(t)\rVert_{\metric} \dif t,
\end{align*}
where the infimum is taken over all smooth\footnote{On a Riemannian smooth manifold, we may equivalently minimize over piecewise smooth or rectifiable paths without changing the value of the infimum.} paths $\gamma : [0,1] \to \Manifold$ from $x$ to $y$.
For any non-empty set $A \subseteq \Manifold$, its $\metric$-diameter
\[ 
    \diam_{\metric}(A) := \sup_{ x, y \in A } \dist_{\metric}(x,y)
\] 
is then understood with respect to this geodesic distance.
For each $x \in \Manifold$ and $r > 0$,
we let 
\[
    \Ball{x}{r}{\metric} := \left\{ y \in \Manifold \suchthat* \dist_{\metric}(x,y) < r \right\}
\]
denote the open $\metric$-ball of radius $r$ around $x$. 
More generally, for each $A \subseteq \Manifold$ and $r > 0$, we write
\[
    \Ball{A}{r}{\metric} := \bigcup_{x \in A}\Ball{x}{r}{\metric}
    .
\]

\subsection{Lebesgue spaces}

The Riemannian structure allows us to define Lebesgue spaces without the use of coordinates. 
For any $p \in [1,\infty]$, we define the Lebesgue space $\Lebesgue^{p}(\Manifold,\metric)$ of scalar fields in the usual way, namely as the Banach space of measurable scalar functions ${u} \in \Lebesgue^{1}_{\loc}(\Manifold)$ for which the norm
\begin{align*}
	\| {u} \|_{\Lebesgue^{p}(\Manifold,\metric)}
	:=
    \left( \int_{\Manifold} \absolutevalue{u}^{p} \density^{\Manifold}_{\metric} \right)^{\frac{1}{p}}
\end{align*}
is finite, with the obvious modification in the case $p=\infty$.
More generally, $\Lebesgue^{p}\Lambda^{k}(\Manifold,\metric)$ is the space of those differential forms ${u} \in \Lebesgue^{1}_{\loc}\Lambda^{k}(\Manifold)$ for which $\lVert{u}\rVert_{\metric} \in \Lebesgue^{p}(\Manifold,\metric)$.
We also write $\Lebesgue^{p}\Tensor^{k}(\Manifold,\metric)$ for the space of those $k$-tensors ${u} \in \Lebesgue^{1}_{\loc}\Tensor^{k}(\Manifold)$ for which $\lVert{u}\rVert_{\metric} \in \Lebesgue^{p}(\Manifold,\metric)$.
One can show that $\Lebesgue^{p}\Tensor^{k}(\Manifold,\metric)$ and $\Lebesgue^{p}\Lambda^{k}(\Manifold,\metric)$ are Banach spaces with norms
\[
    \| {u} \|_{\Lebesgue^{p}\Tensor^{k}(\Manifold,\metric)} := \| \lVert{u}\rVert_{\metric} \|_{\Lebesgue^{p}(\Manifold,\metric)},
    \qquad 
    \| {u} \|_{\Lebesgue^{p}\Lambda^{k}(\Manifold,\metric)} := \| \lVert{u}\rVert_{\metric} \|_{\Lebesgue^{p}(\Manifold,\metric)}.
\]
Given $p, q \in [1,\infty]$, we moreover define
\begin{align}\label{math:sobolevspace}
	\Sobdiff^{p,q}\Lambda^{k}(\Manifold,\metric)
	:=
	\left\{\;
	{u} \in \Lebesgue^{p}\Lambda^{k}(\Manifold,\metric)
	\suchthat*
	\text{ $u$ has a weak exterior derivative } \cartan{u} \in \Lebesgue^{q}\Lambda^{k+1}(\Manifold,\metric)
	\;\right\}
	.
\end{align}
These are Banach spaces equipped with the norms
\[
    \| u \|_{\Sobdiff^{p,q}\Lambda^{k}(\Manifold,\metric)}
    :=
    \| u \|_{\Lebesgue^{p}\Lambda^{k}(\Manifold,\metric)}
    +
    \| \cartan u \|_{\Lebesgue^{q}\Lambda^{k+1}(\Manifold,\metric)}
    .
\]

\begin{remark}
	It is easily seen that the Lebesgue spaces on non-compact manifolds generally depend on the Riemannian metric.
	For example, suppose that a non-compact manifold is endowed with two Riemannian metrics $\metric_{1}$ and $\metric_{2}$
	such that $\Manifold$ has finite volume with respect to $\metric_{1}$ but infinite volume with respect to $\metric_{2}$.
	For every $1 \leq p < \infty$, every nonzero constant function belongs to $\Lebesgue^{p}(\Manifold,\metric_{1})$ but not to $\Lebesgue^{p}(\Manifold,\metric_{2})$.
	However, the class of locally $p$-integrable differential forms is independent of the smooth metric.
\end{remark}

\subsection{Sobolev spaces} 
The Riemannian structure gives rise to an intrinsic definition of Sobolev spaces for tensor fields.
Central to this is the Levi-Civita connection.
We identify the Levi-Civita connection with the unique covariant derivative on the tangent bundle that is torsion-free and metric-compatible; we refer to the literature for a thorough discussion~\cite{LeeSmooth,LeeRiemannian}.
Importantly, this covariant derivative naturally extends to tensor bundles constructed from the tangent and cotangent bundles.

In particular, the Levi-Civita connection gives rise to the covariant derivatives of tensor fields.
If ${u} \in \Cont^{m}\Tensor^{k}(\Manifold)$, then $\nabla^{m} {u} \in \Cont^{0}\Tensor^{m+k}(\Manifold)$ denotes the $m$-th covariant derivative of ${u}$.
By convention, $\nabla^{0} u := u$. 
Note that $\nabla^{m} {u}$ is alternating in the last $k$ indices.
If ${u} \in \Cont^{m}\Lambda^{k}(\Manifold)$ has bounded covariant derivatives from order zero up to $m$, then we define 
\begin{align*}
	\lvert {u} \rvert_{\Cont^{m}\Lambda^{k}(\Manifold,\metric)}
	&:=
	\sup_{ x \in \Manifold } \left\lVert \nabla^{m}{u}_{|x} \right\rVert_{\metric}
    ,
	\qquad 
    \|{u}\|_{\Cont^{m}\Lambda^{k}(\Manifold,\metric)}
    :=
	\max_{ 0 \leq \ell \leq m } \lvert {u} \rvert_{\Cont^{\ell}\Lambda^{k}(\Manifold,\metric)}
    .
\end{align*}
If ${u} \in \Cont\Lambda^{k}(\Manifold)$ is bounded, then we define 
\begin{align*}
    \|{u}\|_{\Cont\Lambda^{k}(\Manifold,\metric)}
	&:=
	\lvert {u} \rvert_{\Cont\Lambda^{k}(\Manifold,\metric)}
	:=
	\sup_{ x \in \Manifold } \left\lVert {u}_{|x} \right\rVert_{\metric}
    .
\end{align*}
The covariant derivative also enables the definition of Sobolev spaces with norms that are independent of the choice of local coordinate charts. %
For $m \in \bbN_{0}$ and $p \in [1,\infty]$, the \emph{Sobolev space of differential forms} $\Sobolev^{m,p}\Lambda^{k}(\Manifold,\metric)$ is the subspace of $\Lebesgue^{p}\Lambda^{k}(\Manifold,\metric)$ whose members have $p$-integrable weak covariant derivatives up to order $m$,
and analogously, the \emph{Sobolev space of tensor fields} $\Sobolev^{m,p}\Tensor^{k}(\Manifold,\metric)$ is the subspace of $\Lebesgue^{p}\Tensor^{k}(\Manifold,\metric)$ whose members have $p$-integrable weak covariant derivatives up to order $m$.
These are Banach spaces. We introduce the norms and seminorms
\begin{gather}
	\| {u} \|_{\Sobolev^{m,p}\Lambda^{k}(\Manifold,\metric)}
	:=
	\sum_{\ell=0}^{m} \| \nabla^{\ell} {u} \|_{\Lebesgue^{p}\Tensor^{\ell+k}(\Manifold,\metric)},
	\qquad
	\lvert {u} \rvert_{\Sobolev^{m,p}\Lambda^{k}(\Manifold,\metric)}
	:=
	\| \nabla^{m} {u} \|_{\Lebesgue^{p}\Tensor^{m+k}(\Manifold,\metric)}.
\end{gather} 
Similarly, the local Sobolev space of differential forms $\Sobolev^{m,p}_{\loc}\Lambda^{k}(\Manifold)$ is the subspace of $\Lebesgue^{p}_{\loc}\Lambda^{k}(\Manifold)$
whose members have locally $p$-integrable weak covariant derivatives up to order $m$.
The local Sobolev space of tensor fields $\Sobolev^{m,p}_{\loc}\Tensor^{k}(\Manifold)$ is defined analogously. 

\begin{remark}
    For any $m \in \bbN_{0}$ and $p \in [1,\infty]$, we can characterize the local Sobolev space $\Sobolev^{m,p}_{\loc}\Lambda^{k}(\Manifold)$ as the space of differential $k$-forms whose coefficients in any local coordinate chart are locally integrable and have weak derivatives from order zero to $m$ that are $p$-integrable.

    We consider only Sobolev spaces of \emph{integer order} in this work.
	It is relatively straightforward to define Sobolev--Gagliardo--Slobodeckij spaces of scalar functions on manifolds,
    but the definition of fractional Sobolev spaces of tensor fields on manifolds is more delicate and beyond the scope of this manuscript. 
\end{remark}

\subsection{Pullbacks along diffeomorphisms}

We occasionally need to transform differential forms from one manifold to another.
Let $(\Manifold,\metric)$ and $(\otherManifold,\othermetric)$ be Riemannian manifolds, and let $\Phi : \Manifold \to \otherManifold$ be a diffeomorphism.
We write $\Jacobian\Phi$ for the Jacobian of $\Phi : \Manifold \to \otherManifold$, which at each point $x \in \Manifold$ is a linear mapping from the tangent fiber at $x$ to the tangent fiber at $\Phi(x)$. 
We define 
\[
    \left\| \Jacobian\Phi \right\|_{\Lebesgue^{\infty}} 
    := 
    \sup_{ \substack{ x \in \Manifold } } 
    \sup_{ \substack{ v \in T_{x}\Manifold \\ \lVert v \rVert_{\metric} \leq 1 } }
    \left\lVert \Jacobian\Phi_{|x} v \right\rVert_{\othermetric}
    .
\]
The absolute value of the Jacobian determinant $\absolutevalue{\det\Jacobian\Phi} \in \Cont^{\infty}(\Manifold)$ is a scalar field
defined at each point $x$ as the absolute value of the determinant of any matrix representation of $\Jacobian\Phi$ with respect to orthonormal bases of the fibers $\Tangential_{x}\Manifold$ and $\Tangential_{\Phi(x)}\otherManifold$.
If $\Manifold$ and $\otherManifold$ are oriented manifolds, we may define the signed determinant $\det\Jacobian\Phi$ of the Jacobian using positively oriented orthonormal bases of $T_x\Manifold$ and $T_{\Phi(x)}\otherManifold$. 

The Jacobian and the Jacobian determinant appear in the discussion of pullbacks. 
We write $\Phi^{\pullbackast} {u}$ for the pullback of any ${u} \in \Lebesgue^{1}_{\loc}\Lambda^{k}(\otherManifold)$, and notice that $\Phi^{\pullbackast} {u} \in \Lebesgue^{1}_{\loc}\Lambda^{k}(\Manifold)$.
To simplify the notation, the pullback along the inverse $\Phi^{-1}$ is written $\Phi^{-\pullbackast}$.

The following theorem summarizes the transformation behavior of differential forms under smooth pullbacks. 
The most important Banach spaces are preserved, and we have norm estimates with constants that depend on the regularity of the transformation.

\begin{theorem}\label{theorem:transformationssatz}
    Let $(\Manifold,\metric)$ and $(\otherManifold,\othermetric)$ be connected Riemannian manifolds, 
    and let $\Phi : \Manifold \to \otherManifold$ be a bi-Lipschitz diffeomorphism.
    Let $k \in \bbN_{0}$. 
    If ${u} \in \Cont\Lambda^{k}(\otherManifold)$ is bounded, then $\Phi^{\pullbackast} {u} \in \Cont\Lambda^{k}(\Manifold)$ and
	\begin{align}
		\label{math:transformationssatz:continuous}
		\| \Phi^{\pullbackast} {u} \|_{\Cont\Lambda^{k}(\Manifold,\metric)}
        & \leq
		\| \Jacobian\Phi \|_{\Lebesgue^{\infty}}^{k}
		\| {u} \|_{\Cont\Lambda^{k}(\otherManifold,\othermetric)}
		.
	\end{align}
    If $p \in [1,\infty]$ and ${u} \in \Lebesgue^{p}\Lambda^{k}(\otherManifold,\othermetric)$, then $\Phi^{\pullbackast} {u} \in \Lebesgue^{p}\Lambda^{k}(\Manifold,\metric)$ and 
	\begin{align}
		\label{math:transformationssatz:lp}
		\| \Phi^{\pullbackast} {u} \|_{\Lebesgue^{p}\Lambda^{k}(\Manifold,\metric)}
		 & \leq
		\| \Jacobian\Phi      \|^{k        }_{\Lebesgue^{\infty}} %
		\| \det \Jacobian\Phi^{-1} \|^{\frac{1}{p}}_{\Lebesgue^{\infty}(\otherManifold,\othermetric)}
		\| {u} \|_{\Lebesgue^{p}\Lambda^{k}(\otherManifold,\othermetric)}
		.
	\end{align}
    If $p,q \in [1,\infty]$ and ${u} \in \Sobdiff^{p,q}\Lambda^{k}(\otherManifold,\othermetric)$, then $\Phi^{\pullbackast} {u} \in \Sobdiff^{p,q}\Lambda^{k}(\Manifold,\metric)$ and
	\begin{align}
		\label{math:transformationssatz:weakexteriorderivative}
		\Phi^{\pullbackast} \cartan {u}
        = 
        \cartan \Phi^{\pullbackast} {u}
		.
	\end{align}
\end{theorem}

\begin{remark}
    The determinant norm in the estimate~\eqref{math:transformationssatz:lp} can be simplified using Hadamard's inequality, which states that
    $\| \det \Jacobian\Phi^{-1} \|_{\Lebesgue^{\infty}(\otherManifold,\othermetric)} \leq \| \Jacobian\Phi^{-1} \|^{n}_{\Lebesgue^{\infty}}$. 
\end{remark}

\subsection{Pullbacks and derivatives of higher order}

We have discussed the pullback as an operator on Lebesgue spaces of differential forms. 
The pullback along a diffeomorphism also operates on Sobolev spaces of differential forms. 
The relevant norm estimates not only involve the higher covariant derivatives of the differential form but also covariant derivatives of the Jacobian. 
We now discuss this more delicate pullback estimate. 

Since we need to discuss the regularity of the Jacobian, we examine successive covariant derivatives of the Jacobian. 
To that end, recall that the Jacobian exists as a vector bundle section on $\Manifold$, namely as a section of the homomorphism bundle $\Hom( \Tangential\Manifold, \Phi^{\pullbackast}\Tangential\otherManifold )$ from the tangent bundle of $\Manifold$ to the pullback of the tangent bundle of $\otherManifold$.
Equivalently, we can identify it with a section of the tensor product bundle $\Tangential^{\dualityast}\Manifold \otimes \Phi^{\pullbackast} \Tangential\otherManifold$, where the first factor is the cotangent bundle of $\Manifold$ and the second factor is the pullback of the tangent bundle of $\otherManifold$.
Here, a corresponding bundle inner product is given by the tensor product metric $\metric \otimes \Phi^{\pullbackast}\othermetric$, where the second factor is the pullback metric. 
Since $\Phi$ is smooth, $\Jacobian\Phi$ is a smooth section of that bundle.

Having formalized the Jacobian as a vector bundle section on $\Manifold$, we rigorously define its derivatives. 
The combination of the Levi-Civita connection on $(\Manifold,\metric)$ with the pullback of the Levi-Civita connection 
on $(\otherManifold,\othermetric)$ yields a connection on the bundle $\Hom( \Tangential\Manifold, \Phi^{\pullbackast}\Tangential\otherManifold )$ (and, equivalently, on the bundle $\Tangential^{\dualityast}\Manifold \otimes \Phi^{\pullbackast} \Tangential\otherManifold$). 
In this sense, the Jacobian admits covariant derivatives $\nabla^{m} \Jacobian\Phi$ of any order $m \in \bbN_{0}$. 
Set 
\begin{align*}
    \| \nabla^{m} \Jacobian\Phi \|_{\Lebesgue^{\infty}(\Manifold)} 
    &
    := 
    \sup_{ \substack{ x \in \Manifold \\ v_{0}, \dots, v_m \in T_{x}\Manifold \\ \prod_{j=0}^{m} \lVert v_{j} \rVert_{\metric} \leq 1 } }
    \left\lVert \nabla^{m}\Jacobian\Phi_{|x}(v_{0},\dots,v_m) \right\rVert_{\othermetric}
    .
\end{align*}
This quantifies the higher-order regularity of the diffeomorphism.
We can now state the interplay with the pullback of Sobolev differential forms.
The following estimate relies on a Fa\`a~di~Bruno formula for differential forms, which we state without proof.

\begin{theorem}\label{theorem:derivatives-of-pullbacks}
    Let $k, m \in \bbN_{0}$. %
    There exist polynomials $P^{(k)}_{m,j} : \bbR^{m-j+1} \to \bbR$, for $0 \leq j \leq m$, of the form 
    \[
        P^{(k)}_{m,j}( x_{1}, x_{2}, \dots, x_{m-j+1} )
        =
        \sum_{ \substack{ 
            a_{1},    a_{2},  \dots, a_{m-j+1} \in \bbN_{0} 
            \\
            \sum_{\ell=1}^{m-j+1} \ell a_{\ell} = m + k
        } }
        c_{a_{1},\dots,a_{m-j+1}}
        \prod_{i=1}^{m-j+1} x_{i}^{a_{i}}
        ,
        \quad 
        c_{a_{1},\dots,a_{m-j+1}} \in \bbR_{0}^{+}
        ,
    \]
    such that the following holds. 
    Let $(\Manifold,\metric)$ and $(\otherManifold,\othermetric)$ be connected Riemannian manifolds.
    Let $\Phi : \Manifold \to \otherManifold$ be a bi-Lipschitz diffeomorphism whose Jacobian has bounded covariant derivatives up to order $m$.
    For $p \in [1,\infty]$, there exists $\Ctrafosatz{m}{p}{n} > 0$, depending only on $m$, $p$, and $n$, 
    such that for all ${u} \in \Sobolev^{m,p}\Lambda^{k}(\otherManifold,\othermetric)$ 
    we have $\Phi^{\pullbackast} {u} \in \Sobolev^{m,p}\Lambda^{k}(\Manifold,\metric)$ 
    and 
    \begin{align*} %
        \begin{split}
        &
        \| \nabla^{m} \Phi^{\pullbackast} {u} \|_{\Lebesgue^{p}\Tensor^{m+k}(\Manifold,\metric)}
		\\&
        \leq
		\Ctrafosatz{m}{p}{n}
        \sum_{j=0}^{m}
		P^{(k)}_{m,j} \left( 
            \|              \Jacobian\Phi \|_{\Lebesgue^{\infty}}
            ,
            \| \nabla       \Jacobian\Phi \|_{\Lebesgue^{\infty}}
            ,
            \dots
            ,
            \| \nabla^{m-j} \Jacobian\Phi \|_{\Lebesgue^{\infty}}
        \right) 
        \| \det \Jacobian\Phi^{-1} \|^{\frac{1}{p}}_{\Lebesgue^{\infty}(\otherManifold,\othermetric)}
		\| \nabla^{j} {u} \|_{\Lebesgue^{p}\Tensor^{j+k}(\otherManifold,\othermetric)}
		.
        \end{split}
	\end{align*}
\end{theorem}

\begin{remark}
    A full exposition of the Fa\`a~di~Bruno formula is not within the scope of this manuscript, but we outline the origin of the formula. 
    The components of $\Phi^{\pullbackast} u$ are sums of products of the following type: 
    a component $u_{\sigma} \circ \Phi$ multiplied by $k$ entries of the Jacobian.
    Applying the Leibniz formula and the chain rule $m$ times yields terms containing $\nabla^{j} u_{\sigma}$ multiplied by $k+j$ entries of the Jacobian,
    with $m-j$ additional derivatives distributed among those entries. 
    In total, each of these Jacobian monomials has a weighted degree of derivatives equal to $k+m$. 
\end{remark}

\section{Smooth triangulations}\label{section:triangulations}

This section reviews smooth triangulations of manifolds and introduces several constants characterizing mesh regularity.
The mesh parameters in this manuscript include the adjacency and comparison constants, and the first-order regularity of the smooth simplices. Additionally, some estimates involve the higher-order shape constants.
We emphasize that these triangulations are defined intrinsically on the manifold.

\subsection{Simplices and triangulations}

For $\Dim \in \bbN_{0}$, the \emph{$\Dim$-dimensional reference simplex} $\referencesimplex_{\Dim} \subseteq \bbR^{\Dim}$ is the closed convex hull of the origin and the $\Dim$ standard coordinate vectors of $\bbR^{\Dim}$. We order its vertices $v_{0}, v_{1}, \dots, v_{\Dim}$ by listing the origin first, followed by the standard coordinate vectors.
For every non-empty subset of indices $I \subseteq \{0,\dots,\Dim\}$ with $\cardinality{I} = \ell + 1$, let $\incl_{I,\Dim} : \referencesimplex_{\ell} \to \referencesimplex_{\Dim}$ be the unique affine map that sends the vertices of $\referencesimplex_{\ell}$ to vertices of $\referencesimplex_{\Dim}$ indexed by $I$ whilst preserving the order.

A \emph{$\Dim$-dimensional simplex} in $\Manifold$ is a topological embedding $\incl : \referencesimplex_{\Dim} \to \Manifold$.
The \emph{vertices} of the simplex are the images of the origin and the $\Dim$ standard basis vectors.
As a practical convention, when $T \subset \Manifold$ is the image of a simplex, we denote the embedding by $\incl_{T} : \referencesimplex_{\Dim} \to \Manifold$.

If $\incl_{S} : \referencesimplex_{\ell} \to \Manifold$ is a simplex such that $\incl_{S} = \incl_{T} \circ \incl_{I,\Dim}$ for some non-empty index set $I \subseteq \{0,\dots,\Dim\}$, then we call $\incl_{S}$ a \emph{subsimplex} of $\incl_{T}$.
We also let $\incl_{S,T} : S \to T$ denote the canonical embedding of $S$ into $T$.

A \emph{finite triangulation} of $\Manifold$ is a finite collection $\Mesh$ of simplices satisfying the following properties:
\begin{enumerate}[(a)]
	\item
    Every point of $\Manifold$ is contained in the image of at least one $n$-dimensional simplex of $\Mesh$.
    \item 
    If $\incl_{T_{1}}, \incl_{T_{2}} \in \Mesh$ have the same image, then $\incl_{T_{1}} = \incl_{T_{2}}$. 
	\item
    For all $\incl_{T} \in \Mesh$ and all subsimplices $\incl_{S}$ of $\incl_{T}$, we have $\incl_{S} \in \Mesh$.
	\item
    For all $\incl_{T_{1}}, \incl_{T_{2}} \in \Mesh$, the intersection $T_{1} \cap T_{2}$ is either empty or the image of a common subsimplex of both $\incl_{T_{1}}$ and $\incl_{T_{2}}$.
\end{enumerate}
Given any $\Dim \in \bbN_{0}$, we let $\Mesh^{\Dim} \subseteq \Mesh$ denote the subset of $\Dim$-dimensional simplices in $\Mesh$. 

We are interested in smooth structures. 
If $\Manifold$ is a smooth manifold, then a simplex $\incl_{T} : \referencesimplex_{\Dim} \to \Manifold$ is called \emph{smooth} if it is a smooth embedding of a manifold with corners.\footnote{Here, a mapping from a simplex into a manifold is called \emph{smooth} if it can be extended to a smooth embedding of an open neighborhood in $\bbR^{\Dim}$ of the reference simplex into the manifold.}
A finite triangulation is called \emph{smooth} if all its simplices are smooth.

\begin{remark}
    Smooth triangulations~\cite{LeeSmooth} appear in some surface finite element methods with exact geometry, e.g., unit sphere triangulations via curved simplices. 
    We emphasize that the triangulation encodes only a piecewise notion of smoothness.
    Equivalently, a smooth triangulation is as a homeomorphism from a finite abstract simplicial complex onto a compact smooth manifold that is piecewise a smooth embedding.
    This abstract notion of simplicial complexes is standard in differential topology and is used in some proofs below. 
\end{remark}
    
\begin{remark}
    A simplex in any topological manifold is determined not only by its image $S$ but also by its embedding $\incl_{S} : \referencesimplex_{\Dim} \to \Manifold$.
    Different simplices may have the same image but represent different ways to embed the reference simplex into the manifold. In fact, infinitely many different smooth simplices may have the same image.
    However, if we fix a finite triangulation $\Mesh$, then each simplex $\incl_{S} : \referencesimplex_{\Dim} \to \Manifold$ is uniquely determined by its image $S$.
    Although we define simplices as embeddings, we often refer to them by their images only.
    When the triangulation $\Mesh$ is understood, our convention of representing a simplex $\incl_{S} : \referencesimplex_{\Dim} \to \Manifold$ by its image $S$ seems to aid readability. 
    This distinction is less prominent in the finite element literature, where triangulations consist of affine simplices and only affine maps from the reference simplex are considered.
    For a given affine simplex, there are only finitely many affine transformations from the reference simplex onto that simplex, corresponding to different orderings of the vertices. 
\end{remark}

\subsection{Metric features}

Consider a smooth manifold $\Manifold$ and a smooth finite triangulation $\Mesh$.
We now study the metric properties of triangulations and hence assume that $\Manifold$ is equipped with a Riemannian metric $\metric$. 
In what follows, the reference simplices $\referencesimplex_{\Dim}$ are always equipped with the Euclidean metric.

If $\incl_{T} : \referencesimplex_{\Dim} \to \Manifold$ is a smooth simplex of positive dimension, then we write 
\[ 
    h_{T} := \diam_{\metric}(T)
\]
for its $\metric$-diameter. 
For each vertex $V \in \Mesh^{0}$, we formally define $h_V$ as the minimum of $h_{T}$ over all positive-dimensional simplices $T \in \Mesh$ containing $V$, if such simplices exist, and $1$ otherwise.\footnote{The latter case occurs only in the formal situation of a zero-dimensional one-point manifold, which we include for formal completeness.}

We describe the regularity of triangulations using several quantities.
If $\incl_{T} : \referencesimplex_{\Dim} \to \Manifold$ is any smooth simplex, then we define the \emph{Lipschitz shape measure} $\shapemeasurelipschitz{\incl_{T}} \geq 0$ as the smallest constant such that 
\begin{subequations}\label{math:shapemeasurelipschitz}
\begin{align}
    \forall x, y \in \referencesimplex_{\Dim} \, : \,
    \dist_{\metric}( \incl_{T}(x), \incl_{T}(y) )
    \leq 
    \shapemeasurelipschitz{\incl_{T}} 
    h_{T}
    \dist( x, y ),
    \\
    \forall x, y \in \referencesimplex_{\Dim} \, : \,
    \dist( x, y )
    \leq 
    \shapemeasurelipschitz{\incl_{T}} 
    h_{T}^{-1}
    \dist_{\metric}( \incl_{T}(x), \incl_{T}(y) ).
\end{align}
\end{subequations}
The \emph{shape measure} $\shapemeasure{\incl_{T}} \geq 0$ of a smooth simplex $\incl_{T} : \referencesimplex_{\Dim} \to \Manifold$ is the smallest constant satisfying
\begin{gather}\label{math:shapemeasure}
    \| \Jacobian\incl_{T} \|_{\Lebesgue^{\infty}}
	\leq
	\shapemeasure{\incl_{T}}
	h_{T},
	\qquad
    \| \Jacobian(\incl_{T}^{-1}) \|_{\Lebesgue^{\infty}}
	\leq
	\shapemeasure{\incl_{T}}
	h_{T}^{-1}.
\end{gather}
Smooth simplices are orientable manifolds, and we transport the standard orientation of the reference simplex to $T$. Thus, diffeomorphisms between smooth simplices have a well-defined Jacobian determinant. 
We introduce the \emph{shape stretching measure} $\shapemeasuredeterminant{\incl_{T}} \geq 0$ of a smooth simplex $\incl_{T} : \referencesimplex_{\Dim} \to \Manifold$ as the smallest constant satisfying 
\begin{gather}\label{math:shapemeasuredeterminant}
    \| \determinant \Jacobian\incl_{T} \|_{\Lebesgue^{\infty}(\referencesimplex_{\Dim})}
	\leq
	\shapemeasuredeterminant{\incl_{T}}
	h_{T}^{\Dim},
	\qquad
    \| \determinant \Jacobian(\incl_{T}^{-1}) \|_{\Lebesgue^{\infty}(T)}
	\leq
	\shapemeasuredeterminant{\incl_{T}}
	h_{T}^{-\Dim}.
\end{gather}
We define these parameters for the entire triangulation by taking maxima over all simplices:
\begin{align*}
    \shapemeasurelipschitz{\Mesh} := \max_{ \substack{ \incl_{T} \in \Mesh } } \shapemeasurelipschitz{\incl_{T}}
    ,
    \qquad 
    \shapemeasure{\Mesh} := \max_{ \substack{ \incl_{T} \in \Mesh } } \shapemeasure{\incl_{T}}
    ,
    \qquad 
    \shapemeasuredeterminant{\Mesh} := \max_{ \substack{ \incl_{T} \in \Mesh } } \shapemeasuredeterminant	{\incl_{T}}
    .
\end{align*}
Intuitively, these mesh constants quantify the ``quality'' of the simplex shapes. 
We remark that $\shapemeasuredeterminant{\Mesh}$ can be bounded in terms of $\shapemeasure{\Mesh}$ due to Hadamard's inequality
and that $\shapemeasure{\Mesh}$ can be bounded in terms of $\shapemeasurelipschitz{\Mesh}$ via Rademacher's theorem. 

In addition, let $\cardinalitybound{\Mesh} \geq 1$ be the maximum number of simplices adjacent\footnote{Two simplices are called \emph{adjacent} if they intersect.} to any simplex of $\Mesh$, and let $\comparisonconstant{\Mesh} \geq 1$ be minimal such that for every $T \in \Mesh$ and every subsimplex $F \subseteq T$ we have $h_{T} \leq \comparisonconstant{\Mesh}h_{F}$:
\begin{align*}
    \cardinalitybound{\Mesh} 
    := 
    \max_{ T \in \Mesh } 
    \#\left\{ S \in \Mesh : S \cap T \neq \emptyset \right\},
    \qquad 
    \comparisonconstant{\Mesh} 
    := 
    \max_{ \substack{ T, F \in \Mesh \\ F \subseteq T } } 
    \frac{ h_{T} }{ h_{F} }.
\end{align*}

\begin{remark}
	On Euclidean domains (see~\cite{licht2017thesis}),
	the quantities $\cardinalitybound{\Mesh}$ and $\comparisonconstant{\Mesh}$ can be bounded in terms of the triangulation's shape measure.
    On Riemannian manifolds this need not hold because the curvature needs to be taken into account.
\end{remark}

The above shape quantities encode first-order information on the simplices. 
Later in this manuscript, we also use higher-order information on the simplices in the error analysis of finite element methods. 
To that end, we introduce higher-order shape constants. 
For $m \in \bbN$ and a smooth simplex $\incl_{T} : \referencesimplex_{\Dim} \to \Manifold$, let $\highershapeconstant{m}{T}$ be the smallest constant for which 
\begin{gather}
    \left\| \nabla^{m-1} \Jacobian \incl_{T} \right\|_{\Lebesgue^{\infty}}
	\leq
	\highershapeconstant{m}{T}
	h_{T}^{ m}
	.
\end{gather}
We define the analogous mesh-wide constant $\highershapeconstant{m}{\Mesh} := \max_{ \substack{ T \in \Mesh } } \highershapeconstant{m}{T}$.

Some of the first-order shape indicators are easily related in the Euclidean setting, but the situation is more complicated on manifolds. 
Since the reference simplex is convex, the Lipschitz constant of any smooth embedding $\incl_{T} : \referencesimplex_{\Dim} \to T$ equals the supremum norm of its Jacobian. 
The Lipschitz constant of its inverse $\incl_{T}^{-1} : T \to \referencesimplex_{\Dim}$ also bounds the Riemannian norm of the Jacobian $\Jacobian\incl_{T}^{-1}$ over the image $T$. 
However, the converse inequality is no longer trivial on a manifold.
The reason is that two points $x, y \in T$ may have length-minimizing paths within $\Manifold$ that do not wholly lie inside the simplex but may cross many other simplices. The local shape indicators are thus insufficient; see Figure~\ref{figure:example-with-triangle-and-circle}.

\begin{figure}%
    \centering
    \begin{tabular}{cc}
        \begin{tikzpicture}[
                x={(1cm,0cm)},
                y={(-0.55cm,-0.35cm)},
                z={(0cm,1cm)},
                scale=1.25
            ]

            \draw[->, gray!60] (0,0,0) -- (3.2,0,0) node[right] {$x$};
            \draw[->, gray!60] (0,0,0) -- (0,3.2,0) node[below left] {$y$};
            \draw[->, gray!60] (0,0,0) -- (0,0,3.2) node[above] {$z$};

            \coordinate (X) at (2.4,0,0);
            \coordinate (Y) at (0,2.4,0);
            \coordinate (Z) at (0,0,2.4);

            \draw[red, very thick, dashed] (Y) -- node[midway, below right, text=red  ] {$\hatT_{1}$} (Z);
            \draw[green, very thick]       (Z) -- node[midway, above right, text=green] {$\hatT_{2}$} (X);
            \draw[blue, very thick]        (X) -- node[midway,       below, text=blue ] {$\hatT_{3}$} (Y);

            \fill (X) circle (1.4pt);
            \fill (Y) circle (1.4pt);
            \fill (Z) circle (1.4pt);
        \end{tikzpicture}
        
        &

        \centering
        \begin{tikzpicture}[scale=1.25]
            \def\myradius{1.55}

            \draw[red, very thick, dashed] ( 32:\myradius) arc[start angle= 32,end angle=328,radius=\myradius];
            \draw[green, very thick]       (332:\myradius) arc[start angle=332,end angle=358,radius=\myradius];
            \draw[blue, very thick]        (  2:\myradius) arc[start angle=  2,end angle= 28,radius=\myradius];

            \node[red]   at (180:1.25) {$T_{1}$};
            \node[green] at (345:1.22) {$T_{2}$};
            \node[blue]  at ( 15:1.22) {$T_{3}$};
        \end{tikzpicture}
    
    \end{tabular}
    \caption{The one-dimensional affine triangulation on the left is transformed to a triangulation of the unit circle. 
        We equip each triangulation with the respective intrinsic path metrics. For the unit circle, that intrinsic path metric corresponds to the metric induced by the Euclidean metric tensor. 
        As the simplices $T_{2}$ and $T_{3}$ shrink and the red simplex (dashed lines) dominates,
        each mapping from $\hatT_{i}$ onto $T_{i}$ can be chosen such that each shape measure remains uniformly bounded, but the Lipschitz constants of their inverses cannot be uniformly bounded. 
        This shows that the Lipschitz shape measure $\shapemeasurelipschitz{T}$ on each simplex $T$ is not controlled by the shape measure $\shapemeasure{T}$ alone. 
        However, the failure of control is visible in the degenerate global comparison constants.}
    \label{figure:example-with-triangle-and-circle}
\end{figure}

We can avoid this difficulty by taking into account other shape indicators, as in the following theorem.
This is important for practical constructions: 
the Lipschitz constants of the inverse embeddings are difficult to control explicitly using local information alone, but they can be controlled with global information in terms of the other first-order shape indicators. 

\begin{theorem}\label{theorem:lipschitz-shape-indicator}
    Let $\Manifold$ be a connected compact smooth manifold of positive dimension with a smooth finite triangulation $\Mesh$.
    Then 
    \begin{equation} \label{math:lipschitz-shape-indicator} 
        \shapemeasure{\Mesh} 
        \leq 
        \shapemeasurelipschitz{\Mesh} 
        \leq 
        \sqrt{n+1} 
        \frac{
            \comparisonconstant{\Mesh} \sqrt{ 2n (n+1) } 
        }{
            1 - e^{ - \comparisonconstant{\Mesh} \sqrt{ 2n (n+1) } }
        }
        \comparisonconstant{\Mesh}^{2}
        \shapemeasure{\Mesh} 
        .
    \end{equation} 
\end{theorem}
\begin{proof}

    We first construct a reference triangulation that encodes the combinatorial structure of $\Mesh$.
    Let $N > 0$ be the number of its vertices, which are $v_{1},\dots,v_{N}$. 
    Consider the standard unit vectors $b_{1}, \dots, b_{N} \in \bbR^{N}$. 
    
    Suppose that $\incl_{T}$ is a $\Dim$-simplex and that its vertices have indices $i_{0},i_{1},\dots,i_{\Dim} \in \left\{ 1, \dots, N \right\}$. 
    Abbreviate $e_{0} := 0$ and let $e_{1}, \dots, e_{\Dim} \in \bbR^{\Dim}$ be the standard basis vectors.
    Define the corresponding reference map
    \begin{align*}
        \incl_{\canonical{T}} : \referencesimplex_{\Dim} \to \canonical{T} \subseteq \bbR^{N}
    \end{align*}
    as the unique affine mapping with $\incl_{\canonical{T}}( e_{j} ) = b_{i_{j}}$ for each $0 \leq j \leq \Dim$. 
    We define the reference triangulation and the associated manifold 
    \begin{align*}
        \widecanonical{\Mesh}    &:= \left\{ \incl_{\canonical{T}} \suchthat \incl_{T} \in \Mesh \right\},
        \qquad 
        \widecanonical{\Manifold} := \underlying{\widecanonical{\Mesh}}.
    \end{align*}
    The maps $\incl_{\canonical{T}}$ are affine, and they satisfy
    \begin{align*}
        \| \Jacobian\incl_{\canonical{T}}      \|_{\Lebesgue^{\infty}} \leq \sqrt{\Dim+1},
        \qquad 
        \| \Jacobian\incl_{\canonical{T}}^{-1} \|_{\Lebesgue^{\infty}} \leq 1.
    \end{align*}
    There exists a homeomorphism $\homeo : \widecanonical{\Manifold} \to \Manifold$
    that is uniquely defined by 
    \begin{align*}
        \forall \canonical{T} \in \widecanonical{\Mesh} 
        :
        \forall x \in \canonical{T}
        : 
        \homeo(x) = \incl_{T}( \incl_{\canonical{T}}^{-1}(x) )
        .
    \end{align*}
    By construction, $\homeo$ is piecewise smooth with respect to $\widecanonical{\Mesh}$ and $\homeo( b_{i} ) = v_{i}$ for all $1 \leq i \leq N$. 
    By the chain rule, 
    \begin{align*}
        \| \Jacobian \homeo    \restriction_{\canonical{T}} \|_{\Lebesgue^{\infty}} 
        \leq
        \| \Jacobian \incl_{T}                       \|_{\Lebesgue^{\infty}},
        \qquad 
        \| \Jacobian \homeo^{-1} \restriction_{      T} \|_{\Lebesgue^{\infty}} 
        \leq 
        \sqrt{\Dim+1} \| \Jacobian \incl_{T}^{-1} \restriction_{T} \|_{\Lebesgue^{\infty}}. 
    \end{align*}

    With this preparation done, we move towards proving~\eqref{math:lipschitz-shape-indicator}. 
    There is nothing to show if $n = 0$, so we assume $n \geq 1$. 
    We define a piecewise affine function $\meshfunc : \widecanonical{\Manifold} \to \bbR_{0}^{+}$ by requiring that at each vertex $\meshfunc( b_{i} ) = h_{v_{i}}$.
    If $T \in \Mesh$ and $\canonical{T} = \homeo^{-1}(T)$, then 
    \begin{gather*}
        \forall \canonical{x} \in \canonical{T} : \comparisonconstant{\Mesh} h_{T} \geq \meshfunc(\canonical{x}) \geq \comparisonconstant{\Mesh}^{-1} h_{T}
        ,
        \qquad 
        \| \nabla \meshfunc \restriction_{\canonical{T}} \| \leq \sqrt{ n(n+1) } h_{T}
        .
    \end{gather*}
    The coefficient $C_{g} := \sqrt{ n (n+1) }$ is due to each of the $n+1$ vertex heights in $\canonical{T}$ being $\sqrt{(n+1)/n}$.
    Consider any smooth path $\gamma : [0,B] \to \Manifold$ that lies within a single simplex $T \in \Mesh$. 
    Then $\Theta^{-1} \circ \gamma$ is a path within $\canonical{T}$ whose Euclidean length satisfies 
    \[
        \int_{0}^{B} \| \partial_{t} ( \Theta^{-1} \circ \gamma ) \| \dif t
        \leq 
        \sqrt{n+1} \| \Jacobian \incl_{T}^{-1} \restriction_{T} \|_{\Lebesgue^{\infty}}
        \int_{0}^{B} \| \partial_{t} \gamma \|_{\frakg} \dif t
        \leq 
        \sqrt{n+1} \shapemeasure{\Mesh} h_{T}^{-1}
        \int_{0}^{B} \| \partial_{t} \gamma \|_{\frakg} \dif t
        .
    \]
    We know that $\meshfunc(\canonical{x}) \leq \comparisonconstant{\Mesh} h_T$ for all $\canonical{x} \in \canonical{T}$ and $\canonical{T} \in \widecanonical{\Mesh}$. 
    We conclude that, more generally, any smooth path $\gamma : [0,B] \to \Manifold$  
    satisfies 
    \begin{gather}\label{math:bound-on-canonical-paths}
        \int_{0}^{B} \meshfunc( \Theta^{-1} \circ \gamma(t) ) \, \| \partial_t ( \Theta^{-1} \circ \gamma(t) ) \| \,\dif t
        \leq 
        \sqrt{n+1} \shapemeasure{\Mesh} \comparisonconstant{\Mesh}
        \int_{0}^{B} \| \partial_t \gamma(t) \|_{\frakg} \dif t
        .
    \end{gather}
    Consider now any $B > 0$ and a Lipschitz path $\canonical{\gamma} : [0,B] \to \widecanonical{\Manifold}$ with unit speed. 
    Then the composition $\meshfunc \circ \canonical{\gamma}$ is differentiable almost everywhere and for almost all $t \in (0,B)$:
    \[
        \absolutevalue{ \partial_{t}( \meshfunc( \canonical{\gamma}(t) ) ) }
        \leq 
        C_{g} \comparisonconstant{\Mesh}
        \meshfunc( \canonical{\gamma}(t) )
        .
    \]
    Gr\"onwall's inequality gives the lower bound for almost all $t \in (0,B)$:
    \[
        \meshfunc( \canonical{\gamma}(t) )
        \geq 
        \meshfunc( \canonical{\gamma}(0) )
        e^{ - C_{g} \comparisonconstant{\Mesh} t }
        .
    \]
    Hence, integration yields 
    \[
        \int_{0}^{B} \meshfunc( \canonical{\gamma}(t) ) \dif t 
        \geq 
        \meshfunc( \canonical{\gamma}(0) )
        \int_{0}^{B} 
        e^{ - C_{g} \comparisonconstant{\Mesh} t }
        \dif t 
        \geq 
        \meshfunc( \canonical{\gamma}(0) )
        \frac{
            1 - e^{ - C_{g} \comparisonconstant{\Mesh} B }
        }{
            C_{g} \comparisonconstant{\Mesh}
        }
        .
    \]
    Suppose that the path $\canonical{\gamma} : [0,B] \to \widecanonical{\Manifold}$ goes from $\canonical{x}$ to $\canonical{y}$ with $\canonical{x} \neq \canonical{y}$. 
    Since $B \geq \| \canonical{y} - \canonical{x} \|$ is a lower bound for the path length,
    \begin{align}\label{math:canonical-path-lower-bound}
        \int_{0}^{B} \meshfunc( \canonical{\gamma}(t) ) \dif t 
        &
        \geq 
        \meshfunc( \canonical{x} )
        \frac{
            1 - e^{ - C_{g} \comparisonconstant{\Mesh} \| \canonical{y} - \canonical{x} \| }
        }{
            C_{g} \comparisonconstant{\Mesh}
        }
        \geq 
        \meshfunc( \canonical{x} )
        \frac{
            1 - e^{ - C_{g} \comparisonconstant{\Mesh} \| \canonical{y} - \canonical{x} \| }
        }{
            C_{g} \comparisonconstant{\Mesh} \| \canonical{y} - \canonical{x} \|
        }
        \| \canonical{y} - \canonical{x} \|
        .
    \end{align}
    Suppose that $T \in \Mesh$ and $\canonical{x}, \canonical{y} \in \canonical{T}$.
    Using $\| \canonical{y} - \canonical{x} \| \leq \sqrt{2}$, we proceed by 
    \begin{align*} 
        \int_{0}^{B} \meshfunc( \canonical{\gamma}(t) ) \dif t 
        &
        \geq 
        \meshfunc( \canonical{x} )
        \frac{
            1 - e^{ - C_{g} \comparisonconstant{\Mesh} \sqrt{2} }
        }{
            C_{g} \comparisonconstant{\Mesh} \sqrt{2}
        }
        \| \canonical{y} - \canonical{x} \|
        .
    \end{align*}
    Recall $h_{T} \leq \meshfunc( \canonical{x} ) \comparisonconstant{\Mesh}$. 
    Therefore, 
    for all $T \in \Mesh$ and all $x, y \in T$, we find 
    \begin{align*}
        h_{T}
        \| \incl_{T}^{-1}(x) - \incl_{T}^{-1}(y) \|
        &
        \leq 
        \comparisonconstant{\Mesh} \meshfunc( \canonical{x} )
        \| \Theta^{-1}(x) - \Theta^{-1}(y) \|
        \leq 
        \sqrt{n+1} \shapemeasure{\Mesh} \comparisonconstant{\Mesh}^{2}
        \frac{
            C_{g} \comparisonconstant{\Mesh} \sqrt{2}
        }{
            1 - e^{ - C_{g} \comparisonconstant{\Mesh} \sqrt{2} }
        }
        \dist_{\frakg}( x, y )
        .
    \end{align*} 
    This shows the upper bound in~\eqref{math:lipschitz-shape-indicator}.
    The lower bound, stating that $\shapemeasure{T} \leq \shapemeasurelipschitz{T}$ for all $T \in \Mesh$, is a consequence of Rademacher's theorem for vector-valued functions on manifolds~\cite[Proposition~2.3]{sevilla2011problems}.
\end{proof} 

\subsection{Local neighborhoods and mesh size functions}

Let $\Mesh$ be a finite triangulation. 
For $T \in \Mesh$, let $T^{\patchstar}$ denote its closed local patch, defined as the union of all simplices adjacent to $T$.
Accordingly, its interior $\interior{T^{\patchstar}}$ denotes the open local patch around $T$.
The closed local patch $T^{\patchstar}$ is naturally equipped with a triangulation $\Mesh(T)$ in the obvious manner:
\[
    \Mesh(T) := \left\{ \incl_{T'} \in \Mesh \suchthat \image \incl_{T'} \subseteq T^{\patchstar} \right\}.
\]
We first verify that the local patch around a simplex contains an open neighborhood in a quantitative sense.
Then we introduce smooth mesh-size functions.

\begin{lemma}\label{lemma:neighborhood}
    Let $\Manifold$ be a compact smooth manifold with a smooth finite triangulation $\Mesh$.
    There exists $\neighborconstant > 0$ such that, for every simplex $T \in \Mesh$,
	\begin{align*}
        \Ball{T}{\neighborconstant h_{T}}{\metric} \subseteq T^{\patchstar}. 
	\end{align*}
    It suffices that $\shapemeasure{\Mesh} \comparisonconstant{\Mesh} \sqrt{n} \neighborconstant < 1$.
\end{lemma}
\begin{proof}
    Fix $T \in \Mesh$.
    We first consider a function $q : \Manifold \to \bbR$ by setting $q = 1$ on all vertices of $T$, setting $q = 0$ on all other vertices of $\Mesh$, and requiring that $q \circ \incl_{T'}$ is affine for each simplex $T' \in \Mesh$.
    This uniquely defines $q$, and the compatibility of the simplex parameterizations implies that $q$ is continuous.

    For any $T' \in \Mesh^{n}$ with $T' \cap T \neq \emptyset$, we thus estimate
    \begin{align*}
        \| \nabla q \restriction_{ T' } \|_{\Lebesgue^{\infty}\Tensor^{1}(T',\frakg)}
        \leq 
        \| \nabla( q \circ \incl_{T'} ) \|_{\Lebesgue^{\infty}\Tensor^{1}(\referencesimplex_{n})} \| \Jacobian\incl_{T'}^{-1} \|_{\Lebesgue^{\infty}}
        &
        \leq 
        \sqrt{n} \| \Jacobian\incl_{T'}^{-1} \|_{\Lebesgue^{\infty}}
        \\&
        \leq 
        \sqrt{n} \shapemeasure{\Mesh} h_{T'}^{-1}
        \leq 
        \sqrt{n} \shapemeasure{\Mesh} \comparisonconstant{\Mesh} h_{T}^{-1} 
        .
    \end{align*}
    This implies a global bound for the gradient of $q$ almost everywhere. If $x \in T$ and $y \in \Manifold \setminus T^{\patchstar}$, then integrating along a path from $x$ to $y$ shows
    \begin{align*} 
        \sqrt{n} \shapemeasure{\Mesh} \comparisonconstant{\Mesh} h_{T}^{-1} \dist_{\metric}( x, y )
        \geq
        \| \Jacobian q \|_{\Lebesgue^{\infty}} \dist_{\metric}( x, y )
        \geq
        \absolutevalue{ q(x) - q(y) } = 1.
    \end{align*}
    We conclude that $\Ball{T}{\neighborconstant h_{T}}{\metric} \subseteq T^{\patchstar}$ if $\sqrt{n} \shapemeasure{\Mesh} \comparisonconstant{\Mesh} \neighborconstant \leq 1$. 
\end{proof} 

\begin{theorem}\label{theorem:meshfunction}
    Let $\Manifold$ be a compact smooth manifold of dimension $n$ with a smooth finite triangulation $\Mesh$.
    There exist a positive smooth function $\meshfunc : \Manifold \to (0,\infty)$ and constants $\Cmeshsizevalue > 0$ and $\Cmeshsizelipschitz \geq 0$ such that 
	\begin{gather*}
        \forall T \in \Mesh : \forall x \in T : 
		\Cmeshsizevalue^{-1} h_{T} \leq \meshfunc(x) \leq \Cmeshsizevalue h_{T},
		\\
		\forall x, y \in \Manifold : 
        \absolutevalue{ \meshfunc(x) - \meshfunc(y) }
		\leq
		\Cmeshsizelipschitz \dist_{\metric}(x,y)
		.
	\end{gather*}
    We may choose $\Cmeshsizevalue = 1.01 \comparisonconstant{\Mesh}$ and $\Cmeshsizelipschitz = 1.01 \sqrt{n} \shapemeasurelipschitz{\Mesh} ( 1 - 1/\comparisonconstant{\Mesh} )$.
\end{theorem}

\begin{proof}
    We first define $\meshfunc_{\triangle} : \Manifold \to \bbR$ as follows. 
    For each vertex $V \in \Mesh^{0}$, we set the value $\meshfunc_{\triangle}(V) = h_V$.
    Having set the values at the vertices, we require that for each $T \in \Mesh$ the pullback $\incl_{T}^{\pullbackast}\meshfunc_{\triangle}$ is affine on the corresponding reference simplex of dimension $\dim(T)$.
    The compatibility of the simplex parameterizations implies that $\meshfunc_{\triangle}$ is continuous. By construction,
	\begin{align*}
		\forall T \in \Mesh : \forall x \in T : 
		\comparisonconstant{\Mesh}^{-1} h_{T} \leq \meshfunc_{\triangle}(x) \leq h_{T} 
        .
	\end{align*}
    The Jacobian of $\incl_{T}^{\pullbackast} \meshfunc_{\triangle}$ has norm bounded by $C_{0} h_{T}$, 
    where $C_{0} := \sqrt{n} (1 - 1/\comparisonconstant{\Mesh})$. 
    Hence, the Jacobian of $\meshfunc_{\triangle}$ on each simplex has a norm with respect to $\metric$ bounded by $\shapemeasure{\Mesh} C_{0}$.
    This local Lipschitz bound is uniform over all simplices $T \in \Mesh$.
    Integrating along smooth paths shows that the same bound is a global Lipschitz constant of $\meshfunc_{\triangle}$ with respect to $\dist_{\metric}$. 
	
    If $\meshfunc_{\triangle}$ is constant, then there is nothing else to show. 
    Otherwise, pick $\epsilon > 0$ and use~\cite[Theorem~1]{azagra2007smooth} to approximate $\meshfunc_{\triangle}$ by a smooth function $\meshfunc$ satisfying
    $\absolutevalue{ \meshfunc_{\triangle} - \meshfunc } \leq \epsilon$ uniformly over $\Manifold$ and the Lipschitz inequality
    $\absolutevalue{ \meshfunc(x) - \meshfunc(y) } \leq ( \shapemeasurelipschitz{\Mesh} C_{0} + \epsilon ) \dist_{\metric}(x,y)$ for all $x, y \in \Manifold$.
    By choosing $\epsilon > 0$ sufficiently small, we obtain the desired result. 
\end{proof}

\begin{remark}
    Unlike their Euclidean counterparts, local patches on manifolds need not be simplicial balls.
    For example, the boundary complex of any tetrahedron can be realized as a smooth finite triangulation of a $2$-sphere, but here the local patch of any triangle is the whole sphere.
    More complicated examples include triangulations of the real projective plane or the so-called Klein bottle;
    see also Figure~\ref{figure:sphere-and-projective-plane}.
\end{remark}

\begin{figure}
    \begin{tabular}{cc}
        \begin{tikzpicture}[
            scale=1.00,
            edge/.style={black,line width=.55pt},
            boundary/.style={ -{Stealth}, thick, shorten <= 3pt, shorten >= 3pt },
            vlabel/.style={ font=\small } %
        ]
            \coordinate (T)  at (0,3.464);
            \coordinate (UL) at (-1,1.732);
            \coordinate (LL) at (-2,0);
            \coordinate (B)  at (0,0);
            \coordinate (LR) at (2,0);
            \coordinate (UR) at (1,1.732);

            \draw[edge] (UL) -- (UR);
            \draw[edge] (UL) -- (B);
            \draw[edge] (UR) -- (B);

            \foreach \p in {T,UL,LL,B,LR,UR} \node[vertex] at (\p) {};

            \draw[-{Stealth}, red,   thick, shorten <= 4pt, shorten >= 4pt] (UL) -- (T);
            \draw[-{Stealth}, red,   thick, shorten <= 4pt, shorten >= 4pt] (UL) -- (LL);
            \draw[-{Stealth}, green, thick, shorten <= 4pt, shorten >= 4pt] (UR) -- (LR);
            \draw[-{Stealth}, green, thick, shorten <= 4pt, shorten >= 4pt] (UR) -- (T);
            \draw[-{Stealth}, blue,  thick, shorten <= 4pt, shorten >= 4pt]  (B) -- (LL);
            \draw[-{Stealth}, blue,  thick, shorten <= 4pt, shorten >= 4pt]  (B) -- (LR);
            
            \node[vlabel,above] at (T)  {$1$};
            \node[vlabel,left] at (LL)  {$1$};
            \node[vlabel,right] at (LR)  {$1$};
            
            \node[vlabel,left] at (UL)  {$2$};
            \node[vlabel,right] at (UR)  {$3$};
            \node[vlabel,below] at (B)  {$4$};
        \end{tikzpicture}
        &
        \begin{tikzpicture}[
            scale=1.00,
            edge/.style={black,line width=.55pt},
            boundary/.style={ -{Stealth}, thick, shorten <= 3pt, shorten >= 3pt },
            vlabel/.style={ font=\small } %
        ]
            \def\myradius{2.2}

            \coordinate (V0) at (  0:\myradius);
            \coordinate (V1) at ( 60:\myradius);
            \coordinate (V2) at (120:\myradius);
            \coordinate (V3) at (180:\myradius);
            \coordinate (V4) at (240:\myradius);
            \coordinate (V5) at (300:\myradius);

            \coordinate (IA) at (240:1);
            \coordinate (IB) at (115:1);
            \coordinate (IC) at (  0:1);

            \draw[edge] (V4) -- (IA) -- (IB) -- (V2);
            \draw[edge] (V5) -- (IC) -- (V1);
            \draw[edge] (V3) -- (IA) (V3) -- (IB);
            \draw[edge] (V5) -- (IA) (IA) -- (IC);
            \draw[edge] (IB) -- (IC) (IB) -- (V1);
            \draw[edge] (V0) -- (IC);

            \foreach \p in {V0,V1,V2,V3,V4,V5,IA,IB,IC} \node[vertex] at (\p) {};

            \node[vlabel,left]        at (V3) {$1$};
            \node[vlabel,below left]  at (V4) {$2$};
            \node[vlabel,below right] at (V5) {$3$};
            \node[vlabel,right]       at (V0) {$1$};
            \node[vlabel,above right] at (V1) {$2$};
            \node[vlabel,above left]  at (V2) {$3$};

            \node[vlabel,below]       at (IA) {$4$};
            \node[vlabel,above]       at (IB) {$5$};
            \node[vlabel,below right] at (IC) {$6$};

            \draw[-{Stealth}, red, thick,   shorten <= 3pt, shorten >= 3pt] (180:\myradius) arc[start angle=180,end angle=240,radius=\myradius];
            \draw[-{Stealth}, green, thick, shorten <= 3pt, shorten >= 3pt] (240:\myradius) arc[start angle=240,end angle=300,radius=\myradius];
            \draw[-{Stealth}, blue, thick,  shorten <= 3pt, shorten >= 3pt] (300:\myradius) arc[start angle=300,end angle=360,radius=\myradius];
            \draw[-{Stealth}, red, thick,   shorten <= 3pt, shorten >= 3pt] (  0:\myradius) arc[start angle=  0,end angle= 60,radius=\myradius];
            \draw[-{Stealth}, green, thick, shorten <= 3pt, shorten >= 3pt] ( 60:\myradius) arc[start angle= 60,end angle=120,radius=\myradius];
            \draw[-{Stealth}, blue, thick,  shorten <= 3pt, shorten >= 3pt] (120:\myradius) arc[start angle=120,end angle=180,radius=\myradius];
        \end{tikzpicture}
        \\
        
        \begin{tikzpicture}[
            scale=1.20,
            edge/.style={black,line width=.55pt},
            boundary/.style={ -{Stealth}, thick, shorten <= 3pt, shorten >= 3pt },
            vlabel/.style={ font=\small, inner sep=1pt }
        ]
            \foreach \x in {0,1,2,3}{
                \foreach \y in {0,1,2,3}{
                    \coordinate (v\x\y) at (\x,\y);
                }
            }

            \foreach \x in {0,1,2,3}{
                \foreach \y in {0,1,2,3}{
                    \node[vertex] at (v\x\y) {};
                }
            }

            \draw[edge] (v10) -- (v13);
            \draw[edge] (v20) -- (v23);

            \draw[edge] (v01) -- (v31);
            \draw[edge] (v02) -- (v32);

            \draw[edge] (v00) -- (v11);
            \draw[edge] (v10) -- (v21);
            \draw[edge] (v20) -- (v31);

            \draw[edge] (v01) -- (v12);
            \draw[edge] (v12) -- (v21);
            \draw[edge] (v21) -- (v32);

            \draw[edge] (v02) -- (v13);
            \draw[edge] (v12) -- (v23);
            \draw[edge] (v22) -- (v33);

            \draw[boundary,myred]   (v00) -- (v10);
            \draw[boundary,mygreen] (v10) -- (v20);
            \draw[boundary,myblue]  (v20) -- (v30);

            \draw[boundary,myred]   (v03) -- (v13);
            \draw[boundary,mygreen] (v13) -- (v23);
            \draw[boundary,myblue]  (v23) -- (v33);

            \draw[boundary,myantired]   (v00) -- (v01);
            \draw[boundary,myantigreen] (v01) -- (v02);
            \draw[boundary,myantiblue]  (v02) -- (v03);

            \draw[boundary,myantired]   (v30) -- (v31);
            \draw[boundary,myantigreen] (v31) -- (v32);
            \draw[boundary,myantiblue]  (v32) -- (v33);

            \node[vlabel,below left=2pt] at (v00) {$1$};
            \node[vlabel,left=2pt]       at (v01) {$2$};
            \node[vlabel,left=2pt]       at (v02) {$3$};
            \node[vlabel,above left=2pt] at (v03) {$1$};

            \node[vlabel,below=2pt]       at (v10) {$4$};
            \node[vlabel,below right=2pt] at (v11) {$5$};
            \node[vlabel,above left=2pt]  at (v12) {$6$};
            \node[vlabel,above=2pt]       at (v13) {$4$};

            \node[vlabel,below=2pt]       at (v20) {$7$};
            \node[vlabel,below right=2pt] at (v21) {$8$};
            \node[vlabel,above left=2pt]  at (v22) {$9$};
            \node[vlabel,above=2pt]       at (v23) {$7$};

            \node[vlabel,below right=2pt] at (v30) {$1$};
            \node[vlabel,right=2pt]       at (v31) {$2$};
            \node[vlabel,right=2pt]       at (v32) {$3$};
            \node[vlabel,above right=2pt] at (v33) {$1$};

        \end{tikzpicture}
        & 
        \begin{tikzpicture}[
            scale=1.20,
            edge/.style={black,line width=.55pt},
            boundary/.style={ -{Stealth}, thick, shorten <= 3pt, shorten >= 3pt },
            vlabel/.style={ font=\small, inner sep=1pt }
        ]
            \foreach \x in {0,1,2,3}{
                \foreach \y in {0,1,2,3}{
                    \coordinate (v\x\y) at (\x,\y);
                }
            }

            \foreach \x in {0,1,2,3}{
                \foreach \y in {0,1,2,3}{
                    \node[vertex] at (v\x\y) {};
                }
            }

            \draw[edge] (v10) -- (v13);
            \draw[edge] (v20) -- (v23);

            \draw[edge] (v01) -- (v31);
            \draw[edge] (v02) -- (v32);

            \draw[edge] (v01) -- (v10); %
            \draw[edge] (v02) -- (v13); %
            \draw[edge] (v12) -- (v21); %
            \draw[edge] (v20) -- (v31); %
            \draw[edge] (v23) -- (v32); %
            \draw[edge] (v01) -- (v12); %
            \draw[edge] (v10) -- (v21); %
            \draw[edge] (v12) -- (v23); %
            \draw[edge] (v21) -- (v32); %

            \draw[boundary,myred]   (v00) -- (v10);
            \draw[boundary,mygreen] (v10) -- (v20);
            \draw[boundary,myblue]  (v20) -- (v30);

            \draw[boundary,myred]   (v03) -- (v13);
            \draw[boundary,mygreen] (v13) -- (v23);
            \draw[boundary,myblue]  (v23) -- (v33);

            \draw[boundary,myantired]   (v00) -- (v01);
            \draw[boundary,myantigreen] (v01) -- (v02);
            \draw[boundary,myantiblue]  (v02) -- (v03);

            \draw[boundary,myantiblue]  (v31) -- (v30);
            \draw[boundary,myantigreen] (v32) -- (v31);
            \draw[boundary,myantired]   (v33) -- (v32);

            \node[vlabel,below left=2pt] at (v00) {$1$};
            \node[vlabel,left=2pt]       at (v01) {$2$};
            \node[vlabel,left=2pt]       at (v02) {$3$};
            \node[vlabel,above left=2pt] at (v03) {$1$};

            \node[vlabel,below=2pt]       at (v10) {$4$};
            \node[vlabel,below right=2pt] at (v11) {$5$};
            \node[vlabel,above left=2pt]  at (v12) {$6$};
            \node[vlabel,above=2pt]       at (v13) {$4$};

            \node[vlabel,below=2pt]       at (v20) {$7$};
            \node[vlabel,below right=2pt] at (v21) {$8$};
            \node[vlabel,above left=2pt]  at (v22) {$9$};
            \node[vlabel,above=2pt]       at (v23) {$7$};

            \node[vlabel,below right=2pt] at (v30) {$1$};
            \node[vlabel,right=2pt]       at (v31) {$3$};
            \node[vlabel,right=2pt]       at (v32) {$2$};
            \node[vlabel,above right=2pt] at (v33) {$1$};

        \end{tikzpicture}

    \end{tabular}
    \caption{%
    Triangulations of different manifolds, where edges with the same color are identified in accordance with their orientation. 
    Top left: a triangulation of the sphere by four triangles and four vertices, where each simplicial neighborhood of a triangle is the whole sphere.  
    Top right: a triangulation of the projective plane by ten triangles and six vertices. Since each triangle in these two triangulations shares an edge or a vertex with every other triangle, their simplicial neighborhoods cannot be homeomorphic to a two-dimensional disk. 
    Bottom left: a triangulation of the torus with eighteen triangles and nine vertices. The simplicial neighborhood of the triangle with vertices $5$, $6$, and $8$ is not a disk but has the topology of a punctured torus. However, it is orientable.
    Bottom right: a similar triangulation for the Klein bottle. The simplicial neighborhood of the triangle with vertices $5$, $6$, and $8$ includes the whole middle row (and the whole middle column), and thus cannot be orientable, which means that the simplicial neighborhood is not a disk.
    }\label{figure:sphere-and-projective-plane}
\end{figure}

Though the local patch may not be a topological ball, any simplex admits an open neighborhood within the local patch that transforms to an open neighborhood of the reference simplex in a quantitatively controlled manner. 
The following result shows that the reference transformation on a simplex can be extended to such an open neighborhood. 
It also establishes some auxiliary results of a rather technical nature.

\begin{theorem}\label{theorem:local-neighborhoods}
    Let $\Manifold$ be a compact smooth manifold with a smooth finite triangulation $\Mesh$.
    There exists $\rho > 0$, depending only on $n$ and $\cardinalitybound{\Mesh}$, such that the following holds.
    
    For each smooth $n$-dimensional simplex $\incl_{T} : \referencesimplex_{n} \to \Manifold$ of the triangulation, there exists a precompact open set $\subdomain_{T} \subseteq \Manifold$ together with a homeomorphism 
    \begin{align}\label{math:local-neighborhoods:patch-inclusion}
        \patchincl_{T} : \Ball{\referencesimplex_{n}}{\rho}{} \to \subdomain_{T} \subseteq \Manifold 
    \end{align}
    such that\footnote{Throughout this document, the notation $A \Subset B$ indicates that $\closure{A}$ is a compact subset of $B$.}
    \begin{align*}
        \subdomain_{T} \Subset \interior{T^{\patchstar}},
        \qquad  
        \patchincl_{T} \restriction_{ \referencesimplex_{n} } = \incl_{T}.
    \end{align*}
    If $S \in \Mesh^n$ satisfies $T \cap S \neq \emptyset$, then the restriction to $\widehat\subdomain_{T,S} := \patchincl_{T}^{-1}\left( \subdomain_{T} \cap S \right)$ is a diffeomorphism
    \begin{align}\label{math:local-neighborhoods:patch-inclusion-restriction}
        \patchincl_{T} : \widehat\subdomain_{T,S} \to \subdomain_{T} \cap S.
    \end{align}
    There exists a mapping $\umleitung_{T,S} : \widehat\subdomain_{T,S} \to \referencesimplex_{n}$ that is a diffeomorphism onto its image such that  
    \begin{align}\label{math:local-neighborhoods:umleitung}
        \patchincl_{T} \restriction_{ \widehat\subdomain_{T,S} } = \incl_{S} \circ \umleitung_{T,S}. 
    \end{align}
    There exists $\Cpatchjacobian{\Mesh} > 0$, depending only on $n$, $\cardinalitybound{\Mesh}$, $\shapemeasure{\Mesh}$, and $\comparisonconstant{\Mesh}$, such that
    \begin{gather}\label{math:local-neighborhoods:patch-inclusion-jacobian}
        \| \Jacobian\patchincl_{T}      \restriction_{\widehat\subdomain_{T,S}} \|_{\Lebesgue^{\infty}} \leq \Cpatchjacobian{\Mesh} h_{T}
        ,
        \qquad 
        \| \Jacobian\patchincl_{T}^{-1} \restriction_{\subdomain_{T} \cap S} \|_{\Lebesgue^{\infty}} \leq \Cpatchjacobian{\Mesh} h_{T}^{-1}
        \\
        \Lip( \patchincl_{T}                     ) \leq \Cpatchjacobian{\Mesh} h_{T}
        ,
        \qquad 
        \Lip( \patchincl_{T}^{-1} \restriction_{ \subdomain_{T} } ) \leq \Cpatchjacobian{\Mesh} h_{T}^{-1}
        .\label{math:local-neighborhoods:patch-inclusion-lipschitz}
    \end{gather}
    We have 
    \begin{align}\label{math:local-neighborhoods:patch-inclusion-higher}
        \| \nabla \Jacobian \patchincl_{T} \restriction_{ \widehat\subdomain_{T,S} } \|_{\Lebesgue^{\infty}} 
        \leq 
        \highershapeconstant{2}{\Mesh} \Cumleitung{\Mesh}^{2} \comparisonconstant{\Mesh}^{2} h_{T}^{2}
        +
        \shapemeasure{\Mesh} \Cumleitunghigher{\Mesh}{2} \comparisonconstant{\Mesh} h_{T}
        .
    \end{align}
    There exists $\Cumleitung{\Mesh} > 0$, depending only on $n$ and $\cardinalitybound{\Mesh}$, such that
    \begin{gather}\label{math:local-neighborhoods:umleitung-jacobian}
        \| \Jacobian\umleitung_{T,S}      \|_{\Lebesgue^{\infty}} \leq \Cumleitung{\Mesh}
        ,
        \qquad 
        \| \Jacobian\umleitung_{T,S}^{-1} \|_{\Lebesgue^{\infty}} \leq \Cumleitung{\Mesh} 
        ,
        \\
        \Lip( \umleitung_{T,S}      ) \leq \Cumleitung{\Mesh}
        ,
        \qquad 
        \Lip( \umleitung_{T,S}^{-1} ) \leq \Cumleitung{\Mesh}
        ,
        \qquad 
        \Lip( \Jacobian\umleitung_{T,S} ) \leq \Cumleitung{\Mesh} 
        .\label{math:local-neighborhoods:umleitung-lipschitz}
    \end{gather}
    There exists $\Cumleitunghigher{\Mesh}{m} > 0$, depending only on $m \in \bbN$, $n$, and $\cardinalitybound{\Mesh}$, such that
    \begin{align}\label{math:local-neighborhoods:umleitung-higher}
        \| \nabla^{m-1}\Jacobian\umleitung_{T,S} \|_{\Lebesgue^{\infty}} \leq \Cumleitunghigher{\Mesh}{m}
        .
    \end{align}
    There exists $\rho_{\umleitung} > 0$, depending only on $n$ and $\cardinalitybound{\Mesh}$, such that 
    \begin{align}\label{math:local-neighborhoods:umleitung-neighborhood}
        \Ball{ \incl_{S}^{-1}( T \cap S ) }{ \rho_{\umleitung} }{} 
        \cap 
        \referencesimplex_{n}  
        \subseteq 
        \umleitung_{T,S}( \widehat\subdomain_{T,S} )
        .
    \end{align}
    Finally, there exists $\Cfoo{\Mesh} \geq 1$, depending only on $n$ and $\cardinalitybound{\Mesh}$, with the following property:
    if $\hatF \subseteq \referencesimplex_{n}$ is a positive-dimensional subsimplex of the reference simplex
    and $S \in \Mesh^n$ satisfies $T \cap S \neq \emptyset$ but $\hatF \nsubseteq \widehat\subdomain_{T,S}$,
    then for every $0<\delta<\rho$ and $\hatx \in \hatF$ we have
    \begin{align}\label{math:local-neighborhoods:distance-to-boundary}
        \dist( \hatx, \partial\hatF ) > \Cfoo{\Mesh} \delta
        \quad\implies\quad 
        \dist( \hatx, \patchincl_{T}^{-1}\left( \subdomain_{T} \cap S \right) ) > \delta 
        .
    \end{align}
    Moreover, if $\hatv$ is a vertex of $\referencesimplex_{n}$ and $S \in \Mesh^n$ satisfies $T \cap S \neq \emptyset$ but $\hatv \notin \widehat\subdomain_{T,S}$,
    then 
    \begin{align}\label{math:local-neighborhoods:distance-of-vertex}
        \dist( \hatv, \patchincl_{T}^{-1}\left( \subdomain_{T} \cap S \right) ) > \rho
        . 
    \end{align}
\end{theorem}

\begin{proof}
    This proof is to be read as a continuation of the proof of Theorem~\ref{theorem:lipschitz-shape-indicator}. 

    Since the statement is trivial for $n = 0$, we assume $n \geq 1$. 
    Consider an $n$-dimensional smooth simplex $\incl_{T} : \referencesimplex_{n} \to \Manifold$. 
    Due to Lemma~3.2 in~\cite{michor2020manifolds},
    there exists an open neighborhood $\widehat\Domain_{T} \subseteq \bbR^{n}$ of $\referencesimplex_{n}$ and a smooth embedding 
    \[
        \incl_{T,\star} : \widehat\Domain_{T} \to \Manifold
    \]
    extending $\incl_{T} : \referencesimplex_{n} \to T$.
    After shrinking $\widehat\Domain_{T}$, we may assume without loss of generality that
    \[
        \incl_{T,\star}( \widehat\Domain_{T} ) \Subset \interior{T^{\patchstar}}
        .
    \]
    Define 
    \[
        \canonical\Domain_{T} := \homeo^{-1} \incl_{T,\star}( \widehat\Domain_{T} ) \subseteq \widecanonical{\Manifold}.
    \]
    We have a mapping $\incl_{\canonical{T},\star} : \widehat\Domain_{T} \to \canonical\Domain_{T}$ that extends $\incl_{\canonical{T}} : \referencesimplex_{n} \to \canonical{T}$ and is defined by 
    \[
        \incl_{\canonical{T},\star} := \homeo^{-1} \circ \incl_{T,\star}.
    \]
    This extension $\incl_{\canonical{T},\star}$ agrees with the affine map $\incl_{\canonical{T}}$ on $\referencesimplex_{n}$. 
    Since $\homeo$ is a piecewise diffeomorphism, $\incl_{\canonical{T},\star}^{-1}$ restricts to a diffeomorphism on $\canonical\Domain_{T} \cap \canonical{S}$ for each $\canonical{S} \in \widecanonical{\Mesh}$ adjacent to $\canonical{T}$.

    The triangulations $\Mesh(T)$ of the local patches around $T \in \Mesh^{n}$ belong to one of finitely many combinatorial classes, whose number can be bounded in terms of $n$ and $\cardinalitybound{\Mesh}$.
    For each such combinatorial class of local patches, we choose one of the mappings $\incl_{\canonical{T},\star}$, up to relabeling of the vertices. 
    
    This enables the following conclusion:
    there exists $\rho > 0$, depending only on $n$ and $\cardinalitybound{\Mesh}$, such that there exist mappings
    \begin{align*}
        \incl_{\canonical{T},\sharp} : \Ball{\referencesimplex_{n}}{\rho}{} \to \widecanonical{\Manifold}
    \end{align*}
    that are homeomorphisms onto their images and belong to a finite family of classes. 
    By construction, they extend $\incl_{\canonical{T}}$ to an open neighborhood of the reference simplex: 
    \[
        \incl_{\canonical{T},\sharp} \left( \Ball{\referencesimplex_{n}}{\rho}{} \right) \Subset \interior{\canonical{T}^{\patchstar}},
        \qquad
        \incl_{\canonical{T},\sharp} \restriction_{ \referencesimplex_{n} } = \incl_{\canonical{T}}.
    \]
    Write $\canonical{\subdomain}_{T} \subseteq \widecanonical{\Manifold}$ for the image of $\incl_{\canonical{T},\sharp}$. 
    For each $\canonical{S} \in \widecanonical{\Mesh}^{n}$ with $\canonical{T} \cap \canonical{S} \neq \emptyset$, write 
    \[
        \canonical{\subdomain}_{T,S} := \canonical{S} \cap \incl_{\canonical{T},\sharp}\left( \Ball{\referencesimplex_{n}}{\rho}{} \right),
        \qquad 
        \widehat{\subdomain}_{T,S} := \incl_{\canonical{T},\sharp}^{-1} ( \canonical{\subdomain}_{T,S} ).
    \]
    We have smooth diffeomorphisms 
    \[
        \incl_{\canonical{T},\sharp} : \widehat{\subdomain}_{T,S} \to \canonical{\subdomain}_{T,S}
        .
    \]
    The set $\widehat{\subdomain}_{T,S}$ is homeomorphic to a portion of the adjacent simplex $S$ and can be transformed into the reference simplex representation of $S$.
    For that, we define 
    \[
        \umleitung_{T,S} 
        := 
        \incl_{\canonical{S}}^{-1} \incl_{\canonical{T},\sharp} \restriction_{ \widehat\subdomain_{T,S} }
        .
    \]
    We now define the corresponding map and patches on the smooth manifold: 
    \[
        \patchincl_{T} := \homeo \circ \incl_{\canonical{T},\sharp},
        \qquad 
        \subdomain_{T} := \patchincl_{T}\left( \Ball{\referencesimplex_{n}}{\rho}{} \right),
        \qquad 
        \subdomain_{T,S} := \subdomain_{T} \cap S
        .
    \]
    By definition, 
    \begin{align*}
        \incl_{S} \circ \umleitung_{T,S}
        &= 
        \incl_{S} \circ \incl_{\canonical{S}}^{-1} \circ \incl_{\canonical{T},\sharp} \restriction_{ \widehat\subdomain_{T,S} }
        = 
        \homeo \circ \incl_{\canonical{S}} \circ \incl_{\canonical{S}}^{-1} \circ \incl_{\canonical{T},\sharp} \restriction_{ \widehat\subdomain_{T,S} }
        = 
        \homeo \circ \incl_{\canonical{T},\sharp} \restriction_{ \widehat\subdomain_{T,S} }
        = 
        \patchincl_{T} \restriction_{ \widehat\subdomain_{T,S} }
        .
    \end{align*}
    The existence of $\rho_{\umleitung} > 0$ in~\eqref{math:local-neighborhoods:umleitung-neighborhood} follows because we only have a finite family (up to relabeling) of $\incl_{\canonical{T},\sharp}$. 
    
    There exists $C_{\sharp} > 0$, depending only on $n$ and $\cardinalitybound{\Mesh}$, such that 
    \begin{gather*}
        \| \Jacobian \incl_{\canonical{T},\sharp}      \restriction_{ \widehat\subdomain_{T,S} } \|_{\Lebesgue^{\infty}} \leq C_{\sharp},
        \qquad 
        \| \Jacobian \incl_{\canonical{T},\sharp}^{-1} \restriction_{ \canonical\subdomain_{T,S} } \|_{\Lebesgue^{\infty}} \leq C_{\sharp},
        \qquad 
        \| \Jacobian \umleitung_{T,S}      \|_{\Lebesgue^{\infty}} \leq C_{\sharp},
        \qquad 
        \| \Jacobian \umleitung_{T,S}^{-1} \|_{\Lebesgue^{\infty}} \leq C_{\sharp},
        \\
        \Lip( \umleitung_{T,S}      ) \leq C_{\sharp},
        \qquad 
        \Lip( \umleitung_{T,S}^{-1} ) \leq C_{\sharp}.
    \end{gather*}
    The existence of $\Cumleitung{\Mesh}$ as in~\eqref{math:local-neighborhoods:umleitung-jacobian} and~\eqref{math:local-neighborhoods:umleitung-lipschitz} follows. 
    The existence of $\Cumleitunghigher{\Mesh}{m}$ as in~\eqref{math:local-neighborhoods:umleitung-higher} follows by similar arguments. 
    
    Since $\patchincl_{T} = \homeo \circ \incl_{\canonical{T},\sharp}$ is a composition, it satisfies the estimates 
    \begin{align*}
        \| \Jacobian \patchincl_{T} \restriction_{\widehat\subdomain_{T,S}} \|_{\Lebesgue^{\infty}} 
        &\leq 
        C_{\sharp} \| \Jacobian \homeo \restriction_{\canonical\subdomain_{T,S}} \|_{\Lebesgue^{\infty}}
        \leq 
        C_{\sharp} \shapemeasure{\Mesh} h_{S}
        \leq 
        C_{\sharp} \shapemeasure{\Mesh} \comparisonconstant{\Mesh} h_{T},
        \\
        \| \Jacobian \patchincl_{T}^{-1} \restriction_{\subdomain_{T,S}} \|_{\Lebesgue^{\infty}} 
        &\leq 
        C_{\sharp} \| \Jacobian\homeo^{-1} \restriction_{\subdomain_{T,S}} \|_{\Lebesgue^{\infty}}
        \leq 
        \sqrt{n+1} C_{\sharp} \shapemeasure{\Mesh} h_{S}^{-1}
        \leq 
        \sqrt{n+1} C_{\sharp} \shapemeasure{\Mesh} \comparisonconstant{\Mesh} h_{T}^{-1}. 
    \end{align*}
    Therefore,~\eqref{math:local-neighborhoods:patch-inclusion-jacobian} follows. 
    From $\patchincl_{T} \restriction \widehat{\subdomain}_{T,S} = \incl_{S} \circ \umleitung_{T,S}$ we conclude 
    \begin{align*}
        \nabla \Jacobian \patchincl_{T} \restriction \widehat{\subdomain}_{T,S}
        =
        \nabla\left( \Jacobian\incl_{S} \circ \Jacobian\umleitung_{T,S} \right)
        = 
        \nabla\Jacobian\incl_{S} \left( \Jacobian\umleitung_{T,S} \otimes \Jacobian\umleitung_{T,S} \right)
        +
        \Jacobian\incl_{S} \nabla\Jacobian\umleitung_{T,S}
        ,
    \end{align*}
    therefore 
    \begin{align*}
        \| \nabla \Jacobian \patchincl_{T} \restriction \widehat{\subdomain}_{T,S} \|_{\Lebesgue^{\infty}}
        &
        \leq 
        \| \nabla\Jacobian\incl_{S} \|_{\Lebesgue^{\infty}} \| \Jacobian\umleitung_{T,S} \|_{\Lebesgue^{\infty}}^{2}
        +
        \| \Jacobian\incl_{S} \nabla\Jacobian\umleitung_{T,S} \|_{\Lebesgue^{\infty}}
        \\&
        \leq 
        \highershapeconstant{2}{\Mesh} \Cumleitung{\Mesh}^{2} h_{S}^{2}
        +
        \shapemeasure{\Mesh} \Cumleitunghigher{\Mesh}{2} h_{S}
        .
    \end{align*}
    Since $h_{S} \leq \comparisonconstant{\Mesh} h_{T}$, this shows~\eqref{math:local-neighborhoods:patch-inclusion-higher}.

    Next, to prove~\eqref{math:local-neighborhoods:patch-inclusion-lipschitz}, we first recall that $\patchincl_{T}$ is continuous with a bounded Jacobian almost everywhere, and its domain $\widehat\subdomain_{T}$ is convex. Consequently, $\| \Jacobian \patchincl_{T} \|_{\Lebesgue^{\infty}}$ is the Lipschitz bound of the inclusion. 
    Recall $\patchincl_{T}^{-1} = \incl_{\canonical{T},\sharp}^{-1} \circ \homeo^{-1}$, where $\incl_{\canonical{T},\sharp}^{-1}$ satisfies a uniform Lipschitz bound. 
    Consider any Lipschitz path $\canonical{\gamma} : [0,B] \to \canonical{\Manifold}$ with unit speed between $\canonical{x}, \canonical{y} \in \canonical{T}^{\patchstar}$. 
    As in the proof of Theorem~\ref{theorem:lipschitz-shape-indicator}, using~\eqref{math:canonical-path-lower-bound},
    \begin{align} 
        \int_{0}^{B} \meshfunc( \canonical{\gamma}(t) ) \dif t 
        &
        \geq 
        \meshfunc( \canonical{x} )
        \frac{
            1 - e^{ - C_{g} \comparisonconstant{\Mesh} \| \canonical{y} - \canonical{x} \| }
        }{
            C_{g} \comparisonconstant{\Mesh} \| \canonical{y} - \canonical{x} \|
        }
        \| \canonical{y} - \canonical{x} \|
        .
    \end{align}
    Using $\| \canonical{y} - \canonical{x} \| \leq 3\sqrt{2}$, we now proceed by 
    \begin{align*} 
        \int_{0}^{B} \meshfunc( \canonical{\gamma}(t) ) \dif t 
        &
        \geq 
        \meshfunc( \canonical{x} )
        \frac{
            1 - e^{ - C_{g} \comparisonconstant{\Mesh} 3\sqrt{2} }
        }{
            C_{g} \comparisonconstant{\Mesh} 3\sqrt{2}
        }
        \| \canonical{y} - \canonical{x} \|
        .
    \end{align*}
    We recall~\eqref{math:bound-on-canonical-paths}. With $x := \homeo(\canonical{x})$ and $y := \homeo(\canonical{y})$, we find 
    \begin{align*}
        \meshfunc( \canonical{x} )
        \| \canonical{y} - \canonical{x} \|
        \leq 
        \frac{
            C_{g} \comparisonconstant{\Mesh} 3\sqrt{2}
        }{
            1 - e^{ - C_{g} \comparisonconstant{\Mesh} 3\sqrt{2} }
        }
        \int_{0}^{B} \meshfunc( \canonical{\gamma}(t) ) \dif t 
        \leq 
        \frac{
            C_{g} \comparisonconstant{\Mesh} 3\sqrt{2}
        }{
            1 - e^{ - C_{g} \comparisonconstant{\Mesh} 3\sqrt{2} }
        }
        \sqrt{n+1} \shapemeasure{\Mesh} \comparisonconstant{\Mesh}^{2} 
        \dist_{\frakg}( x, y )
        .
    \end{align*}
    With $\meshfunc( \canonical{x} ) \geq h_{T} \comparisonconstant{\Mesh}^{-2}$, we obtain the Lipschitz inequality that yields~\eqref{math:local-neighborhoods:patch-inclusion-lipschitz}.

    To prove the final statement, we begin by observing that there exists $C_{n,\circ}$, depending only on $n$, with the following property: 
    if $\canonical{F} \in \widecanonical{\Mesh}$ and $\canonical{K} \in \widecanonical{\Mesh}$ with $\canonical{F} \nsubseteq \canonical{K}$,
    and if $\delta > 0$ and $\canonical{x} \in \canonical{F}$ with $\dist_{\bbR^{N}}( \canonical{x}, \partial\canonical{F} ) > C_{n,\circ} \delta$, then $\dist_{\bbR^{N}}( \canonical{x}, \canonical{K} ) > \delta$.
    
    Assume first that $\hat{F} \subseteq \referencesimplex_{n}$ is a positive-dimensional subsimplex of the $n$-dimensional reference simplex, and let $\hatx \in \hat{F}$.
    Write $F := \incl_T(\hatF)$ and $\canonical{F} := \incl_{\canonical{T}}( \hat{F} )$, and write $\canonical{x} := \incl_{\canonical{T}}( \hatx )$.
    We find 
    \[
        \dist_{\bbR^{n}}( \hatx, \partial\hatF )
        \leq 
        \| \Jacobian \incl_{\canonical{T}}^{-1} \|_{\Lebesgue^{\infty}} 
        \dist_{\bbR^{N}}( \canonical{x}, \partial\canonical{F} )
        .
    \]
    Let $S \in \Mesh^n$ with $T \cap S \neq \emptyset$ but $F \nsubseteq S$. Write $\canonical S := \homeo^{-1}(S)$.
    Because $\patchincl_T = \incl_T$ on $\referencesimplex_{n}$, the condition $\hatF \nsubseteq \widehat\subdomain_{T,S}$ is equivalent to $F\nsubseteq S$, and hence to $\canonical F \nsubseteq \canonical S$.
    We notice that 
    \begin{align*}
        \dist_{\bbR^{N}}( \canonical{x}, \canonical{S} )
        \leq 
        \dist_{\bbR^{N}}( \canonical{x}, \canonical{\subdomain}_{T,S} )
        \leq 
        \| \Jacobian \incl_{\canonical{T},\sharp} \|_{\Lebesgue^{\infty}} 
        \dist_{\bbR^{n}}( \hatx, \widehat\subdomain_{T,S} )
        .
    \end{align*}
    We conclude that for every $0 < \delta < \rho$,
    \begin{align*}
        \dist_{\bbR^{n}}( \hatx, \partial\hatF )
        > C_{n,\circ}
        \| \Jacobian \incl_{\canonical{T},\sharp} \|_{\Lebesgue^{\infty}}
        \| \Jacobian \incl_{\canonical{T}}^{-1} \|_{\Lebesgue^{\infty}} 
        \delta 
        \qquad 
        \implies 
        \qquad 
        \dist_{\bbR^{n}}( \hatx, \widehat\subdomain_{T,S} )
        > 
        \delta
        .
    \end{align*}
    We use that 
    \[
        C_{n,\circ} \| \Jacobian \incl_{\canonical{T},\sharp} \|_{\Lebesgue^{\infty}}
        \| \Jacobian \incl_{\canonical{T}}^{-1} \|_{\Lebesgue^{\infty}}
        \leq 
        C_{n,\circ} C_{\sharp}^{2} =: \Cfoo{\Mesh}.
    \]
    Consider any vertex $\hatV$ of the reference simplex, write $\canonical{V} := \incl_{\canonical T}(\hatV)$, and suppose that $S \in \Mesh^{n}$ with $T \cap S \neq \emptyset$ but $\canonical{S}$ does not touch $\canonical{V}$. 
    Then $\canonical{V}$ has distance from $\canonical{S}$ at least $1$. 
    The Lipschitz estimate for $\incl_{\canonical{T},\sharp}$ therefore gives a distance of at least $C_{\sharp}^{-1} > 0$ between $\hatV$ and $\widehat\subdomain_{T,S}$. 
    The desired statement follows if we stipulate $\rho < C_{\sharp}^{-1}$. 
    The proof is complete. 
\end{proof}

\section{Finite element spaces on triangulations}\label{section:finiteelementspaces}

This section describes finite element differential forms on intrinsic triangulations of manifolds.
The basic idea is to define finite element spaces on reference simplices first.
We then transport them to the smooth simplices of a smooth triangulation and define the global finite element space by imposing compatibility conditions.

In addition, we define interpolants: first on a reference simplex, then on smooth simplices, and finally globally on triangulations. 
We establish bounds for operator norms in terms of the mesh regularity that will be important later in the manuscript.

\subsection{Finite element spaces}

We first define spaces of polynomial differential forms on reference simplices.
We then transport these spaces to smooth simplices by pullback.
Given a smooth triangulation of a manifold, global finite element spaces are defined by imposing single-valued tangential traces.

Let $\Dim \in \bbN_{0}$ and let $\referencesimplex_{\Dim} \subseteq \bbR^{\Dim}$ be the $\Dim$-dimensional reference simplex.
For $k, r \in \bbZ$, let $\calP_{r}\Lambda^{k}(\referencesimplex_{\Dim})$ denote the space of differential $k$-forms on $\referencesimplex_{\Dim}$ whose coefficients are polynomials of degree at most $r$.
Additionally, we define the space of trimmed polynomial differential forms by
\begin{gather*}
	\calP_{r}^{-}\Lambda^{k}(\referencesimplex_{\Dim})
	:=
	\calP_{r-1}\Lambda^{k}(\referencesimplex_{\Dim})
	+
	\kappa
	\calP_{r-1}\Lambda^{k+1}(\referencesimplex_{\Dim}),
\end{gather*}
where $\kappa : \Cont^{\infty}\Lambda^{k+1}(\referencesimplex_{\Dim}) \to \Cont^{\infty}\Lambda^{k}(\referencesimplex_{\Dim})$ denotes contraction with the radial vector field $X(x) := x$. 

As we call $\calP_{r}^{-}\Lambda^{k}(\referencesimplex_{\Dim})$ a \emph{trimmed space of polynomial differential forms}, we call $\calP_{r}\Lambda^{k}(\referencesimplex_{\Dim})$ a \emph{full space of polynomial differential forms}.
We use the convention that $\calP_{r}\Lambda^{k}(\referencesimplex_{\Dim})$ is trivial if $r<0$ or if $k \notin \{0,\dots,\Dim\}$.
By a similar convention, $\calP_{r}^{-}\Lambda^{k}(\referencesimplex_{\Dim})$ is trivial if $r < 0$ or if $k \notin \{0,\dots,\Dim\}$.

We refer to~\cite{AFW1,arnold2009geometric} for the algebraic properties of these spaces, some of which we now summarize.
For $r \geq 1$, we have the inclusions and identities 
\begin{gather*}
    \calP_{r-1}\Lambda^{k}(\referencesimplex_{\Dim})
    \subseteq
    \calP_{r}^{-}\Lambda^{k}(\referencesimplex_{\Dim})
    \subseteq
    \calP_{r}\Lambda^{k}(\referencesimplex_{\Dim})
    ,
    \\
    \calP_{r}^{-}\Lambda^{0}(\referencesimplex_{\Dim}) = \calP_{r}\Lambda^{0}(\referencesimplex_{\Dim})
    ,
    \qquad 
    \calP_{r}^{-}\Lambda^{\Dim}(\referencesimplex_{\Dim}) = \calP_{r-1}\Lambda^{\Dim}(\referencesimplex_{\Dim})
    ,
    \\
    \cartan \calP_{r}\Lambda^{k}(\referencesimplex_{\Dim})
	=
	\cartan \calP_{r}^{-}\Lambda^{k}(\referencesimplex_{\Dim})
    ,
	\qquad
	\calP_{r}\Lambda^{k}(\referencesimplex_{\Dim}) \cap \ker\cartan
	=
	\calP_{r+1}^{-}\Lambda^{k}(\referencesimplex_{\Dim}) \cap \ker\cartan
    .
\end{gather*}
In addition, the spaces $\calP_{r}\Lambda^{k}(\referencesimplex_{\Dim})$ and $\calP_{r}^{-}\Lambda^{k}(\referencesimplex_{\Dim})$ are invariant under pullback by affine automorphisms of $\referencesimplex_{\Dim}$.

We transport these spaces to the smooth simplices of a smooth triangulation.
Let us assume that $\Manifold$ is a connected compact smooth manifold of dimension $n$, and let $\Mesh$ be a smooth finite triangulation of $\Manifold$. 
For any $\Dim$-dimensional smooth simplex $\incl_{T} : \referencesimplex_{\Dim} \to \Manifold$, we define 
\begin{align}
    \calP_{r}\Lambda^{k}(T)
    :=
    \incl_{T}^{-\pullbackast}
    \calP_{r}\Lambda^{k}(\referencesimplex_{\Dim}),
    \qquad 
    \calP_{r}^{-}\Lambda^{k}(T)
    :=
    \incl_{T}^{-\pullbackast}
    \calP_{r}^{-}\Lambda^{k}(\referencesimplex_{\Dim}).
    \label{math:physical-polynomial-space}
\end{align}
These spaces of polynomial differential forms are preserved if we take traces to subsimplices. 
If $S, T \in \Mesh$ with $S \subseteq T$, then pullbacks along the canonical embedding $\incl_{S,T}:S\longrightarrow T$ induce trace operators
\begin{align}
    \trace_{T,S}
    :
    \calP_{r}\Lambda^{k}(T)
    \longrightarrow
    \calP_{r}\Lambda^{k}(S),
    \qquad 
    \trace_{T,S}
    :
    \calP_{r}^{-}\Lambda^{k}(T)
    \longrightarrow
    \calP_{r}^{-}\Lambda^{k}(S).
    \label{math:physical-trace}
\end{align}
These mappings are surjective.

We introduce global finite element spaces on smooth triangulations by imposing single-valued trace conditions.
For the full polynomial spaces, let
\begin{align}
    \calP_{r}\Lambda^{k}(\Mesh) \label{math:global-polynomial-space:full}
    :=
    \left\{ 
        ({u}_{T})_{ T \in \Mesh^{n} }
        \in 
        \prod_{ T \in \Mesh^{n} } \calP_{r}\Lambda^{k}(T)
        \suchthat* 
        \forall T_{1}, T_{2} \in \Mesh^{n}, T_{1} \cap T_{2} \neq \emptyset 
        : 
        \trace_{T_{1},T_{1} \cap T_{2}} {u}_{T_{1}} = \trace_{T_{2},T_{1} \cap T_{2}} {u}_{T_{2}}
    \right\}
    .
\end{align}
Similarly, we define
\begin{align} \label{math:global-polynomial-space:trimmed}
    \calP_{r}^{-}\Lambda^{k}(\Mesh)
    :=
    \left\{ 
        ({u}_{T})_{ T \in \Mesh^{n} }
        \in 
        \prod_{ T \in \Mesh^{n} } \calP_{r}^{-}\Lambda^{k}(T)
        \suchthat* 
        \forall T_{1}, T_{2} \in \Mesh^{n}, T_{1} \cap T_{2} \neq \emptyset 
        : 
        \trace_{T_{1},T_{1} \cap T_{2}} {u}_{T_{1}} = \trace_{T_{2},T_{1} \cap T_{2}} {u}_{T_{2}}
    \right\}
    .
\end{align}
Since the finite element spaces are single-valued on subsimplices, traces are defined unambiguously.
For each $S \in \Mesh$, there are well-defined linear mappings
\begin{gather*}
	\trace_{S} : \calP_{r}\Lambda^{k}(\Mesh) \to \calP_{r}\Lambda^{k}(S),
	\qquad
	\trace_{S} : \calP_{r}^{-}\Lambda^{k}(\Mesh) \to \calP_{r}^{-}\Lambda^{k}(S).
\end{gather*}

We consider differential complexes of finite element spaces. 
These are combined from the arrows 
\begin{gather}
    \begin{split}
        \cartan : \calP_{r+1}    \Lambda^{k}(\Mesh) \to \calP_{r  }    \Lambda^{k+1}(\Mesh),
        \qquad
        \cartan : \calP_{r  }    \Lambda^{k}(\Mesh) \to \calP_{r  }^{-}\Lambda^{k+1}(\Mesh),
        \\
        \cartan : \calP_{r+1}^{-}\Lambda^{k}(\Mesh) \to \calP_{r  }    \Lambda^{k+1}(\Mesh),
        \qquad
        \cartan : \calP_{r  }^{-}\Lambda^{k}(\Mesh) \to \calP_{r  }^{-}\Lambda^{k+1}(\Mesh),
    \end{split} \label{math:finite-element-arrows}
\end{gather}
where $r \geq 1$.
If we combine these mappings such that consecutive spaces match, then we obtain a
finite element de~Rham complex on the triangulation $\Mesh$ having the following form:
\begin{align}\label{math:finite-element-de-rham-complex}
    \begin{CD}
        0 \to \bbR \to 
        \calP\Lambda^{0}(\Mesh)
        @>\cartan>>
        \calP\Lambda^{1}(\Mesh)
        @>\cartan>>
        \dots 
        @>\cartan>>
        \calP\Lambda^{n}(\Mesh)
        \to 0
    \end{CD}
\end{align}
We will always assume the spaces contain at least the lowest-order spaces:
\begin{align*}
    \calP_{1}^{-}\Lambda^{k}(\Mesh) \subseteq \calP\Lambda^{k}(\Mesh)
    .
\end{align*}

\subsection{Whitney forms and their cohomology}

The lowest-order trimmed polynomial differential forms admit a particularly important topological interpretation.
For every $0\leq k\leq n$, we call $\calP_{1}^{-}\Lambda^{k}(\Mesh)$ the space of \emph{Whitney $k$-forms} on the triangulation $\Mesh$.
On every simplex $T \in \Mesh^{n}$, the exterior derivative preserves this class of spaces:
\[
    \cartan
    \calP_{1}^{-}\Lambda^{k}(T)
    \subseteq
    \calP_{1}^{-}\Lambda^{k+1}(T).
\]
As the exterior derivative commutes with the traces, the elementwise exterior derivative preserves the single-valued trace conditions;
this defines mappings
\[
    \cartan:
    \calP_{1}^{-}\Lambda^{k}(\Mesh)
    \longrightarrow
    \calP_{1}^{-}\Lambda^{k+1}(\Mesh).
\]
Consequently, these spaces can be assembled into a differential complex
\begin{gather*}
    \begin{CD}
        0
        @>>>
        \calP_{1}^{-}\Lambda^{0}(\Mesh)
        @>\cartan>>
        \calP_{1}^{-}\Lambda^{1}(\Mesh)
        @>\cartan>>
        \dots
        @>\cartan>>
        \calP_{1}^{-}\Lambda^{n}(\Mesh)
        @>>>
        0.
    \end{CD}
\end{gather*}
This differential complex is known as the complex of Whitney forms. 
The following well-known theorem determines its cohomology; see, for example,~\cite{whitney2005geometric,AFW1,arnold2009geometric}.

\begin{theorem}\label{theorem:whitney-cohomology}
    Let $\Mesh$ be a finite smooth triangulation of the compact smooth manifold $\Manifold$.
    For all $k \in \bbZ$, %
    \begin{align*}
        \frakH^{k}_{\calW}(\Mesh) 
        =
        \betti_{k}(\Mesh)
        ,
        \text{ where }
        \frakH^{k}_{\calW}(\Mesh) 
        :=
        \frac{
            \ker \left( \cartan : \calP_{1}^{-}\Lambda^{k}(\Mesh) \to \calP_{1}^{-}\Lambda^{k+1}(\Mesh) \right)
        }{
            \rng \left( \cartan : \calP_{1}^{-}\Lambda^{k-1}(\Mesh) \to \calP_{1}^{-}\Lambda^{k}(\Mesh) \right)
        }
        .
    \end{align*}
\end{theorem}
    
\begin{remark}
    The complex of Whitney forms, on a finite triangulation, is isomorphic to the cochain complex of the triangulation. 
    Consequently, its $k$-th cohomology equals the $k$-th simplicial cohomology of the triangulation with real coefficients, whose dimension equals the $k$-th Betti number. 
    Later in the manuscript, the Whitney complex will serve as the cohomological reference for the higher-order finite element spaces.
\end{remark}

\subsection{Weak exterior derivatives of piecewise smooth forms}

We have defined the finite element spaces $\calP_{r}\Lambda^{k}(\Mesh)$ and $\calP_{r}^{-}\Lambda^{k}(\Mesh)$ by a compatibility condition.
We want to develop an equivalent definition in terms of the weak exterior derivative.

Let $\Manifold$ be a compact connected Riemannian smooth manifold of dimension $n$ with metric tensor $\metric$, and let $\Mesh$ be a smooth finite triangulation of $\Manifold$.
We notice that we have a canonical embedding 
\[
    \prod_{ T \in \Mesh^{n} } \Cont^{\infty}\Lambda^{k}(T) \subseteq \Lebesgue^{1}_{\loc}\Lambda^{k}(\Manifold)
\]
since the left-hand side can be identified with the differential $k$-forms that are piecewise smooth with respect to the triangulation. 
We define the piecewise exterior derivative 
\begin{align}
    \cartan_{\Mesh}
    : 
    \prod_{ T \in \Mesh^{n} } \Cont^{\infty}\Lambda^{k}(T)
    \to 
    \prod_{ T \in \Mesh^{n} } \Cont^{\infty}\Lambda^{k+1}(T)
    ,
    \qquad 
    \left( u_{T} \right)_{ T \in \Mesh^{n} }
    \mapsto 
    \left( \cartan u_{T} \right)_{ T \in \Mesh^{n} }
    .
    \label{math:piecewise-exterior-derivative}
\end{align}
We want to study when the piecewise exterior derivative coincides with the weak exterior derivative. 
The following theorem settles this question.

\begin{theorem}\label{theorem:weak-exterior-derivative-of-piecewise-forms}
    Let 
    \[
        u = (u_{T})_{ T \in \Mesh^{n} } \in \prod_{ T \in \Mesh^{n} } \Cont^{\infty}\Lambda^{k}(T).
    \]
    Then the following are equivalent.
    \begin{enumerate}
        \item
        The differential form is single-valued on the simplices of the triangulation, i.e.,
        \begin{align}
            \forall S \in \Mesh : \forall T_{1}, T_{2} \in \Mesh^{n} : S \subseteq T_{1} \cap T_{2} \implies 
            \trace_{T_{1},S} u_{T_{1}}
            =
            \trace_{T_{2},S} u_{T_{2}}
            .
            \label{math:single-valued-traces-on-simplices}
        \end{align}
        \item
        The differential form is single-valued on the faces of the triangulation, i.e.,
        \begin{align}
            \forall F \in \Mesh^{n-1} : \forall T_{1}, T_{2} \in \Mesh^{n} : F = T_{1} \cap T_{2} \implies 
            \trace_{T_{1},F} u_{T_{1}}
            =
            \trace_{T_{2},F} u_{T_{2}}
            .
            \label{math:single-valued-traces-on-faces}
        \end{align}
        \item
        We have $u \in \Sobdiff^{1,1}_{\loc}\Lambda^{k}(\Manifold)$. 
    \end{enumerate}
    If any of these conditions holds, then $\cartan u = \cartan_{\Mesh} u$. 
\end{theorem}
\begin{proof}
    It is clear that~\eqref{math:single-valued-traces-on-simplices} implies~\eqref{math:single-valued-traces-on-faces}.
    Conversely, suppose that~\eqref{math:single-valued-traces-on-faces} holds.
    Let $S \in \Mesh$ and $T_{1}, T_{2} \in \Mesh^{n}$ with $S \subseteq T_{1} \cap T_{2}$. 
    Since $\Mesh$ is a triangulation of a smooth manifold, it is face-connected~\cite[Satz 11.2.9]{stockerzieschang}:
    there exist $K_{0}, K_{1}, \dots, K_{m} \in \Mesh^{n}$ such that 
    \begin{align*}
        K_{0} = T_{1}, \qquad K_m = T_{2}, \qquad \forall 1 \leq i \leq m : S \subseteq K_{i} \cap K_{i-1} \in \Mesh^{n-1}.
    \end{align*}
    For all $1 \leq i \leq m$, we thus observe 
    \begin{align*}
        \trace_{K_{i  },S} u_{K_{i  }}
        =
        \trace_{K_{i} \cap K_{i-1},S} \trace_{K_{i  }, K_{i} \cap K_{i-1}} u_{K_{i}}
        =
        \trace_{K_{i} \cap K_{i-1},S} \trace_{K_{i-1}, K_{i} \cap K_{i-1}} u_{K_{i-1}}
        =
        \trace_{K_{i-1},S} u_{K_{i-1}}
        .
    \end{align*}
    We infer that $\trace_{T_{1},S} u_{T_{1}} = \trace_{T_{2},S} u_{T_{2}}$, meaning that~\eqref{math:single-valued-traces-on-simplices} holds.
    
    We equip all simplices of $\Mesh$ with an arbitrary orientation. 
    If $T \in \Mesh^{n}$ and $F \in \Mesh^{n-1}$ with $F \subseteq T$, then we let $o(T,F) := 1$ if $T$ induces the orientation of $F$ and $o(T,F) := -1$ otherwise.

    Assume that $\Domain \subseteq \Manifold$ is a domain in the interior of the local patch of a vertex of $\Mesh$. 
    The interior of that local patch is orientable.\footnote{The open local patch of a vertex is known as an open vertex star and is contractible. A contractible topological space is simply connected. By~\cite[Theorem~15.43]{LeeSmooth}, a connected smooth manifold that is simply connected is orientable.} 
    Hence, if $T \in \Mesh^{n}$ with $\Domain \cap T \neq \emptyset$, we can uniquely define $o(\Domain,T) := 1$ if the orientations of $\Domain$ and $T$ agree and define $o(\Domain,T) := -1$ otherwise.
    For any $v \in \Cont^{\infty}_{c}\Lambda^{n-k-1}(\Manifold)$ with $\supp v \Subset \Domain$, we find
    \begin{align*}
        &
        \int_{\Domain} u \wedge \cartan v
        =
        \sum_{ \substack{ T \in \Mesh^{n} \\ T \cap \Domain \neq \emptyset } }
        o(\Domain,T) \int_{T} u_{T} \wedge \cartan v
        \\&\quad 
        =
        (-1)^{k+1}
        \sum_{ \substack{ T \in \Mesh^{n} \\ T \cap \Domain \neq \emptyset } }
        o(\Domain,T) \int_{T} \cartan u_{T} \wedge v
        +
        (-1)^{k}
        \sum_{ \substack{ T \in \Mesh^{n} \\ T \cap \Domain \neq \emptyset } }
        \sum_{ \substack{ F \in \Mesh^{n-1} \\ F \subseteq T } }
        o(\Domain,T) o(T,F) \int_{F} \trace_{T,F} u_{T} \wedge v
        \\&\quad 
        =
        (-1)^{k+1}
        \int_{\Domain} \cartan_{\Mesh} u \wedge v
        +
        (-1)^{k}
        \sum_{ \substack{ T \in \Mesh^{n} \\ T \cap \Domain \neq \emptyset } }
        \sum_{ \substack{ F \in \Mesh^{n-1} \\ F \subseteq T } }
        o(\Domain,T) o(T,F) \int_{F} \trace_{T,F} u_{T} \wedge v
        .
    \end{align*}
    
    On the one hand, 
    assume $u$ single-valued. 
    Since each face is contained in exactly two $n$-simplices, the face integrals cancel each other. By definition, the piecewise exterior derivative $\cartan_{\Mesh} u$ is the weak exterior derivative of $u$ over the open set $\Domain$. That shows $u \in \Sobdiff^{1,1}_{\loc}\Lambda^{k}(\Manifold)$ and that $\cartan u = \cartan_{\Mesh} u$ is the weak exterior derivative.
    
    On the other hand,
    assume that $u$ has a weak exterior derivative. 
    By testing over the interior of $n$-simplices, we conclude that $\cartan_{\Mesh} u$ must be the weak exterior derivative. 
    Let $T_{1}, T_{2} \in \Mesh^{n}$ and $F \in \Mesh^{n-1}$ such that $F = T_{1} \cap T_{2}$. 
    Let $\Domain \subseteq \Manifold$ be a domain in the interior of the union $T_{1} \cup T_{2}$. Then $\Domain$ is orientable, being within the interior of some local vertex patch. 
    For any $v \in \Cont^{\infty}_{c}\Lambda^{n-k-1}(\Domain)$, accounting for the orientation,
    \begin{align*}
        &
        \int_{\Domain} u \wedge \cartan v
        =
        (-1)^{k+1}
        \int_{\Domain} \cartan u \wedge v
        +
        (-1)^{k}
        \int_{F} 
        \left( 
            o(\Domain,T_1) o(T_1,F) \trace_{T_1,F} u_{T_1}
            +
            o(\Domain,T_2) o(T_2,F) \trace_{T_2,F} u_{T_2}
        \right) \wedge v
        .
    \end{align*}
    Since $v$ was arbitrary, $\trace_{T_{1},F} u_{T_{1}} = \trace_{T_{2},F} u_{T_{2}}$. 
    It follows that $u$ is single-valued. 
\end{proof}

We first introduced finite element spaces by imposing compatibility conditions between elementwise spaces. 
Theorem~\ref{theorem:weak-exterior-derivative-of-piecewise-forms} characterizes the global finite element spaces intrinsically in terms of weak derivatives:
\begin{align}
    \calP_{r}\Lambda^{k}(\Mesh)
    &=
    \left\{
        u
        \in
        \Sobdiff^{1,1}_{\loc}\Lambda^{k}(\Manifold)
        \suchthat*
        \forall T \in \Mesh^{n} 
        : 
        u_{T} \in \calP_{r}\Lambda^{k}(T)
    \right\},
    \label{math:weak-characterization:full}
    \\
    \calP_{r}^{-}\Lambda^{k}(\Mesh)
    &=
    \left\{
        u
        \in
        \Sobdiff^{1,1}_{\loc}\Lambda^{k}(\Manifold)
        \suchthat*
        \forall T \in \Mesh^{n} 
        :
        u_{T} \in \calP_{r}^{-}\Lambda^{k}(T)
    \right\}.
    \label{math:weak-characterization:trimmed}
\end{align}

\subsection{Inverse inequalities}

We collect a few inverse inequalities that compare the norms of finite element differential forms.
For any $p, q \in [1,\infty]$ with $p \leq q$, there exists $\Clebesguenormcomparison{n}{k}{r}{p}{q} > 0$ such that 
\begin{align}\label{math:inverse-inequality-lebesgue}
    \| u \|_{\Lebesgue^{q}\Lambda^{k}(\referencesimplex_{n})}
    \leq 
    \Clebesguenormcomparison{n}{k}{r}{p}{q}
    \| u \|_{\Lebesgue^{p}\Lambda^{k}(\referencesimplex_{n})}
    ,
    \qquad 
    u \in \calP_{r}\Lambda^{k}(\referencesimplex_{n})
    .
\end{align}
In addition, for any $p \in [1,\infty]$ there exists $\Clebesguenormlipschitz{n}{k}{r}{p} > 0$ such that 
\begin{align}\label{math:inverse-inequality-gradient}
    \| \nabla u \|_{\Lebesgue^{\infty}\Tensor^{k+1}(\referencesimplex_{n})}
    \leq 
    \Clebesguenormlipschitz{n}{k}{r}{p} 
    \| u \|_{\Lebesgue^{p}\Lambda^{k}(\referencesimplex_{n})}
    ,
    \qquad 
    u \in \calP_{r}\Lambda^{k}(\referencesimplex_{n})
    .
\end{align}
Here, $\Clebesguenormcomparison{n}{k}{r}{p}{q}$ and $\Clebesguenormlipschitz{n}{k}{r}{p}$ depend only on $n$, $k$, $r$, $p$ and $q$.
Because $\calP_{r}^{-}\Lambda^{k}(\referencesimplex_{n}) \subseteq \calP_{r}\Lambda^{k}(\referencesimplex_{n})$, these inequalities apply to trimmed spaces as well.

Note that we can always estimate these constants in terms of the endpoint cases, so that these constants are bounded independently of the Lebesgue exponents. The constants may be chosen so that 
\begin{align*}
    \Clebesguenormcomparison{n}{k}{r}{p}{q}
    \leq 
    (n!)^{-(1-\frac{1}{p}+\frac{1}{q})}
    \Clebesguenormcomparison{n}{k}{r}{1}{\infty}
    \leq 
    \Clebesguenormcomparison{n}{k}{r}{1}{\infty}
    ,
    \qquad 
    \Clebesguenormlipschitz{n}{k}{r}{p}
    \leq 
    (n!)^{-(1-\frac{1}{p})}
    \Clebesguenormlipschitz{n}{k}{r}{1}
    \leq 
    \Clebesguenormlipschitz{n}{k}{r}{1}
    .
\end{align*}

\subsection{Degrees of freedom and canonical interpolants}

Choose an orientation for every simplex of the triangulation. 
If $F, T \in \Mesh$ with $F \subseteq T$ and $r \geq 1$, 
then the degrees of freedom for the full polynomial space $\calP_{r}\Lambda^{k}(T)$ associated with $F$ are the linear functionals
\begin{align}
    u
    \longmapsto
    \int_{F}
    \trace_{T,F} u
    \wedge\eta,
    \qquad
    \eta
    \in
    \calP_{r+k-\dim(F)}^{-}
    \Lambda^{\dim(F)-k}(F).
    \label{math:degrees-of-freedom:full}
\end{align}
Similarly, the degrees of freedom for the trimmed polynomial space $\calP_{r}^{-}\Lambda^{k}(T)$ associated with $F$ are
\begin{align}
    u
    \longmapsto
    \int_{F}
    \trace_{T,F}u
    \wedge\eta,
    \qquad
    \eta
    \in
    \calP_{r+k-\dim(F)-1}
    \Lambda^{\dim(F)-k}(F).
    \label{math:degrees-of-freedom:trimmed}
\end{align}
Recall that polynomial spaces with negative polynomial degree are understood to be trivial.
These functionals extend to the finite element spaces $\calP_{r}\Lambda^{k}(\Mesh)$ and $\calP_{r}^{-}\Lambda^{k}(\Mesh)$, respectively, in the obvious manner.
We define the sets of functionals 
\begin{gather*}
	\DegreeOfFreedom_{r}\Lambda^{k}(F)     := \left\{\; u \mapsto \int_{F} \trace_{F} u \wedge \eta \suchthat* \eta \in \calP_{r-\dim(F)+k}^{-}\Lambda^{\dim(F)-k}(F) \;\right\},
	\\
	\DegreeOfFreedom_{r}^{-}\Lambda^{k}(F) := \left\{\; u \mapsto \int_{F} \trace_{F} u \wedge \eta \suchthat* \eta \in \calP_{r-\dim(F)+k-1}\Lambda^{\dim(F)-k}(F) \;\right\}.
\end{gather*}
These functionals constitute the duals of the finite element spaces:
\begin{align}
    \calP_{r}\Lambda^{k}(\Mesh)^{\ast}
    =
    \bigoplus_{ S \in \Mesh } \DegreeOfFreedom_{r}\Lambda^{k}(S),
    \qquad 
    \calP_{r}^{-}\Lambda^{k}(\Mesh)^{\ast}
    =
    \bigoplus_{ S \in \Mesh } \DegreeOfFreedom_{r}^{-}\Lambda^{k}(S).
\end{align}    
In particular, these degrees of freedom are unisolvent for the corresponding finite element spaces.
They therefore determine canonical interpolation operators by requiring that the interpolant and the original differential form have the same degrees of freedom.

We begin with the interpolant on a single simplex. We recall the standard result; see~\cite[Section~4]{AFW1}.
There exist projections for $r \geq 1$,
\begin{gather*}
	\Interpolant_{T}^{k,r}   : \Cont\Lambda^{k}(T) \to \calP_{r}    \Lambda^{k}(T) \subseteq \Cont\Lambda^{k}(T),
	\qquad 
	\Interpolant_{T}^{k,r,-} : \Cont\Lambda^{k}(T) \to \calP_{r}^{-}\Lambda^{k}(T) \subseteq \Cont\Lambda^{k}(T),
\end{gather*}
defined by requiring that for every $u \in \Cont\Lambda^{k}(T)$ and for every subsimplex $F$ of $T$, 
\begin{gather*}
	\int_{F} \trace_{T,F} \Interpolant_{T}^{k,r} u \wedge \eta
	=
	\int_{F} \trace_{T,F} u \wedge \eta,
	\quad
	\eta \in \calP_{r+k-\dim(F)}^{-}\Lambda^{\dim(F) - k}(F),
	\\
	\int_{F} \trace_{T,F} \Interpolant_{T}^{k,r,-} u \wedge \eta
	=
	\int_{F} \trace_{T,F} u \wedge \eta,
	\quad
	\eta \in \calP_{r+k-\dim(F)-1}\Lambda^{\dim(F) - k}(F)
    .
\end{gather*}
We recall the following results for the interpolation operators on a single simplex. 

\begin{theorem}\label{theorem:canonical-interpolant-bound}
	Let $T \in \Mesh$ with $\Dim := \dim(T)$ and $r \geq 1$. 
    The projections
	\begin{gather*}
		\Interpolant_{T}^{k,r}   : \Cont\Lambda^{k}(T) \to \calP_{r}    \Lambda^{k}(T) \subseteq \Cont\Lambda^{k}(T),
		\qquad 
		\Interpolant_{T}^{k,r,-} : \Cont\Lambda^{k}(T) \to \calP_{r}^{-}\Lambda^{k}(T) \subseteq \Cont\Lambda^{k}(T)
	\end{gather*}
    satisfy the following properties. 
    
	There exists a constant $\Cinterpol{\Dim}{k}{r} > 0$, depending only on $\Dim$, $k$, and $r$, such that
	\begin{gather*}
		\| \Interpolant_{T}^{k,r  } {u} \|_{\Cont\Lambda^{k}(T)} \leq \Cinterpol{\Dim}{k}{r} \shapemeasure{\Mesh}^{2k} \| {u} \|_{\Cont\Lambda^{k}(T)}
		,
		\qquad
		\| \Interpolant_{T}^{k,r,-} {u} \|_{\Cont\Lambda^{k}(T)} \leq \Cinterpol{\Dim}{k}{r} \shapemeasure{\Mesh}^{2k} \| {u} \|_{\Cont\Lambda^{k}(T)}
		.
	\end{gather*}
	If ${u} \in \Cont\Lambda^{k}(T)$ with $\cartan{u} \in \Cont\Lambda^{k+1}(T)$, 
    then the interpolants commute with the exterior derivative:
	\begin{align*}
    	\cartan \Interpolant_{T}^{k,r+1,-} {u} = \Interpolant_{T}^{k+1,r+1} \cartan {u} = \Interpolant_{T}^{k+1,r  } \cartan {u},
		\qquad
		\cartan \Interpolant_{T}^{k,r  ,-} {u} = \Interpolant_{T}^{k+1,r  } \cartan {u} = \Interpolant_{T}^{k+1,r,-} \cartan {u}.
	\end{align*}
    Whenever $F \in \Mesh$ with $F \subseteq T$, we have 
    \begin{align} \label{math:canonical-interpolant-trace}
		\Interpolant_{F}^{k,r}   \trace_{T,F} = \trace_{T,F} \Interpolant_{T}^{k,r}
		, 
        \qquad
		\Interpolant_{F}^{k,r,-} \trace_{T,F} = \trace_{T,F} \Interpolant_{T}^{k,r,-}
		.
	\end{align}
\end{theorem}

\begin{proof}
    The unisolvence, boundedness, restriction properties, and commuting identities are standard properties of the canonical degrees of freedom for polynomial differential forms on the reference simplex.
    We refer to~\cite[Section~4]{AFW1}.
    Via pullback along the reference transformation, the theorem follows. 
\end{proof}

As per~\eqref{math:canonical-interpolant-trace}, the trace of a canonical interpolant onto a subsimplex depends only on the trace of the differential form onto that subsimplex.
We now define the global canonical finite element interpolants by 
\begin{gather*}
    \Interpolant_{\Mesh}^{k,r}   : \Cont\Lambda^{k}(\Manifold) \to \calP_{r}    \Lambda^{k}(\Mesh), 
    \quad 
    {u} \mapsto \left( \Interpolant_{T}^{k,r}   {u} \right)_{ T \in \Mesh^{n} },
	\\
	\Interpolant_{\Mesh}^{k,r,-} : \Cont\Lambda^{k}(\Manifold) \to \calP_{r}^{-}\Lambda^{k}(\Mesh), 
    \quad 
    {u} \mapsto \left( \Interpolant_{T}^{k,r,-} {u} \right)_{ T \in \Mesh^{n} }.
\end{gather*}
The trace identities~\eqref{math:canonical-interpolant-trace} show that the local interpolants have matching traces on common subsimplices and are thus well-defined as members of the global finite element space. 
It is clear from the definitions that these mappings are well-defined. 
We summarize their properties as follows. 

\begin{theorem}\label{theorem:canonical-interpolant-bound-global}
	The mappings 
	\begin{gather*}
		\Interpolant_{\Mesh}^{k,r}   : \Cont\Lambda^{k}(\Manifold) \to \calP_{r}    \Lambda^{k}(\Mesh),
		\qquad 
		\Interpolant_{\Mesh}^{k,r,-} : \Cont\Lambda^{k}(\Manifold) \to \calP_{r}^{-}\Lambda^{k}(\Mesh)
	\end{gather*}
    satisfy the following properties. 
    
    If $S \in \Mesh$ and $u \in \Cont\Lambda^{k}(\Manifold)$, then 
    \begin{align*}
		\trace_{S} \left( \Interpolant_{\Mesh}^{k,r  } u \right) = \Interpolant_{S}^{k,r  } \trace_{S} u
		, 
        \qquad
		\trace_{S} \left( \Interpolant_{\Mesh}^{k,r,-} u \right) = \Interpolant_{S}^{k,r,-} \trace_{S} u
	\end{align*}
	and, writing $\Dim := \dim(S)$, it holds that 
	\begin{gather*}
		\| \Interpolant_{S}^{k,r  } \trace_{S} {u} \|_{\Cont\Lambda^{k}(S)}
        \leq
        \Cinterpol{\Dim}{k}{r} \shapemeasure{\Mesh}^{2k} 
        \| \trace_{S} {u} \|_{\Cont\Lambda^{k}(S)}
        ,
        \qquad
        \| \Interpolant_{S}^{k,r,-} \trace_{S} {u} \|_{\Cont\Lambda^{k}(S)}
        \leq
        \Cinterpol{\Dim}{k}{r} \shapemeasure{\Mesh}^{2k} 
        \| \trace_{S} {u} \|_{\Cont\Lambda^{k}(S)}
        .
	\end{gather*}
	If ${u} \in \Cont\Lambda^{k}(\Manifold)$ with $\cartan{u} \in \Cont\Lambda^{k+1}(\Manifold)$, 
    then the interpolants commute with the exterior derivative:
	\begin{align*}
		\cartan \Interpolant_{\Mesh}^{k,r  } {u} = \Interpolant_{\Mesh}^{k,r  ,-} \cartan {u} = \Interpolant_{\Mesh}^{k+1,r,-} \cartan {u},
		\qquad
		\cartan \Interpolant_{\Mesh}^{k,r+1} {u} = \Interpolant_{\Mesh}^{k,r+1,-} \cartan {u} = \Interpolant_{\Mesh}^{k+1,r  } \cartan {u}.
	\end{align*}
\end{theorem}

\begin{proof}
    The result is clear from Theorem~\ref{theorem:canonical-interpolant-bound}, but we briefly discuss the operator norm estimate. 
    We find local bounds via standard arguments based on reference transformations, adapted to the manifold situation.
    Utilizing the pullback estimates in Theorem~\ref{theorem:transformationssatz},
    we find that any $T \in \Mesh$ with $\Dim := \dim(T)$ and $u \in \Cont\Lambda^{k}(\Manifold)$ satisfy 
    \begin{align*}
        \|
            \Interpolant_{T}^{k,r}u
        \|_{\Cont\Lambda^{k}(T,\metric)}
        &\leq
        \Cinterpol{\Dim}{k}{r}
        \|
            \Jacobian\incl_{T}^{-1}
        \|_{\Lebesgue^{\infty}}^{k}
        \|
            \Jacobian\incl_{T}
        \|_{\Lebesgue^{\infty}}^{k}
        \|
            u
        \|_{\Cont\Lambda^{k}(T,\metric)},
        \\
        \|
            \Interpolant_{T}^{k,r,-}u
        \|_{\Cont\Lambda^{k}(T,\metric)}
        &\leq
        \Cinterpol{\Dim}{k}{r}
        \|
            \Jacobian\incl_{T}^{-1}
        \|_{\Lebesgue^{\infty}}^{k}
        \|
            \Jacobian\incl_{T}
        \|_{\Lebesgue^{\infty}}^{k}
        \|
            u
        \|_{\Cont\Lambda^{k}(T,\metric)}.
    \end{align*}
    The desired bounds for the operator norms now follow from the definition of the mesh constants. 
\end{proof}

Returning to the global finite element spaces, we obtain a commutative diagram 
\begin{align} \label{math:canonical-interpolant-is-cochain-map}
    \begin{CD}
        \Cont^{\infty}\Lambda^{0}(\Manifold)
        @>\cartan>>
        \Cont^{\infty}\Lambda^{1}(\Manifold)
        @>\cartan>>
        \dots
        @>\cartan>>
        \Cont^{\infty}\Lambda^{n}(\Manifold)
        \\
        @V\Interpolant_{\Mesh}^{0}VV
        @V\Interpolant_{\Mesh}^{1}VV
        @.
        @V\Interpolant_{\Mesh}^{n}VV
        \\
        \calP\Lambda^{0}(\Mesh)
        @>\cartan>>
        \calP\Lambda^{1}(\Mesh)
        @>\cartan>>
        \dots
        @>\cartan>>
        \calP\Lambda^{n}(\Mesh)
    \end{CD}
\end{align}
of the smooth de~Rham complex and the finite element de~Rham complex.
Here, the specific vertical interpolants depend on the choice of finite element spaces in each degree.
The main topic of this manuscript is how to extend this commuting interpolant to much rougher spaces.

\subsection{Finite element systems}

It is possible to embed the discussion of finite element spaces in the framework of finite element systems (FES); see also~\cite{christiansen2008construction,christiansen2009foundations,christiansen2016constructions,gillette2019trimmed,StructPresDisc}.
Assume that $\Mesh$ is a triangulation. Recall the exterior derivatives 
$ 
    \cartan: \Cont^{\infty}\Lambda^{k}(T) \to \Cont^{\infty}\Lambda^{k+1}(T)
$ 
on each $T \in \Mesh$.
For each $S,T \in \Mesh$ with $S \subseteq T$, the pullback along the inclusion $\incl_{S,T} \colon S \to T$ gives rise to the operator 
$
    \trace^{k}_{T,S}: \Cont^{\infty}\Lambda^{k}(T) \to \Cont^{\infty}\Lambda^{k}(S)
$.
An \emph{element system} on $\Mesh$ is a family of subspaces $\ES\Lambda^{k}(T) \subseteq \Cont^{\infty}\Lambda^{k}(T)$ for each $k \in \bbZ$ and each $T \in \Mesh$
such that for all $T \in \Mesh$ there are mappings
$
	\cartan: \ES\Lambda^{k}(T) \to \ES\Lambda^{k+1}(T)
$,
and for all $S, T \in \Mesh$ with $S \subseteq T$ there are restrictions
$
	\incl_{S,T}^{\pullbackast}: \ES\Lambda^{k}(T) \to \ES\Lambda^{k}(S)
$.
For example, the spaces $\Cont^{\infty}\Lambda^{k}(T)$ constitute an element system on the mesh $\Mesh$, and this element system is called the \emph{smooth} element system.
A \emph{finite element system} is defined as an element system whose spaces are all finite-dimensional.
The \emph{conforming space} of the element system is
\begin{gather*}
	\ES\Lambda^{k}(\Mesh)
	:=
	\left\{\;
    u \in \textstyle{ \bigoplus_{T \in \Mesh} } \ES\Lambda^{k}(T)
        \suchthat*
        \forall S, T \in \Mesh, S \subseteq T : \incl_{S,T}^{\pullbackast} u_{T} = u_{S}
	\;\right\}.
\end{gather*}
For example, $\Cont^{\infty}\Lambda^{k}(\Mesh)$ denotes the conforming space of the smooth element system.
The trace $\trace_{S} \colon \ES\Lambda^{k}(\Mesh) \to \ES\Lambda^{k}(S)$ is well-defined for any $S \in \Mesh$.
Since $\Mesh$ triangulates an $n$-dimensional manifold, each member of $\ES\Lambda^{k}(\Mesh)$ is uniquely determined by its values on the $n$-dimensional simplices. 
Hence, it is convenient to use %
\begin{gather*}
	\ES\Lambda^{k}(\Mesh)
	\cong 
	\left\{\;
        u \in \textstyle{ \bigoplus_{ \substack{ T \in \Mesh^{n} } } } \ES\Lambda^{k}(T)
        \suchthat*
        \forall T_{1}, T_{2} \in \Mesh^{n}, T_{1} \cap T_{2} \neq \emptyset : \trace_{T_{1},T_{1} \cap T_{2}} u_{T_{1}} = \trace_{T_{2},T_{1} \cap T_{2}} u_{T_{2}}
	\;\right\}.
\end{gather*}
Further conditions characterize practically useful conforming spaces.
With some abuse of notation, write $\partial T$ for the triangulation of the boundary of $T \in \Mesh$.
The element system is said to \emph{admit extensions} if for each $T \in \Mesh$, every $u \in \ES\Lambda^{k}(\partial T)$ arises by taking the traces of some member of $\ES\Lambda^{k}(T)$. 
We define
\begin{align*}
	\mathring\ES\Lambda^{k}(T)
	:=
	\left\{\;
        u \in \ES\Lambda^{k}(T) 
        \suchthat*
        \forall S \in \Mesh, S \subsetneq T : \incl_{S,T}^{\pullbackast} u = 0
	\;\right\}
    .
\end{align*}
It is clear that for a finite element system,
\begin{gather*}
    \dim \ES\Lambda^{k}(\Mesh) \leq \textstyle{ \sum_{ T \in \Mesh } } \dim\mathring\ES\Lambda^{k}(T).
\end{gather*}
This upper bound is an equality if the element system admits extensions~\cite[Proposition~5.14]{StructPresDisc}.
The element system is called \emph{locally exact} if the following sequence is exact for each $T \in \Mesh$:
\begin{gather*}
	\begin{CD}
		0 \to \bbR @>>>
		\ES\Lambda^{0}(T) @>\cartan>> \ES\Lambda^{1}(T) @>\cartan>> \cdots @>\cartan>> \ES\Lambda^{\dim(T)}(T)
		\to 0.
	\end{CD}
\end{gather*}
An element system is called \emph{compatible} if it admits extensions and is locally exact.
For example, the smooth element system is compatible~\cite[Corollary~5.18]{StructPresDisc}.
If the finite element system is compatible~\cite[Proposition~5.16]{StructPresDisc}, 
then one can show that the differential complex of vector spaces 
\begin{align*}
	\begin{CD}
		0 \to \ES\Lambda^{0}(\Mesh) @>\cartan>> \ES\Lambda^{1}(\Mesh) @>\cartan>> \dots @>\cartan>> \ES\Lambda^{n}(\Mesh) \to 0
	\end{CD}
\end{align*}
has cohomology isomorphic to the simplicial cohomology of the triangulation.
If the finite element system admits extensions, then local exactness can be expressed equivalently in the following way~\cite[Proposition~5.17]{StructPresDisc}.
An element system that admits extensions is locally exact if and only if for each $T \in \Mesh$, the space $\ES\Lambda^{0}(T)$ contains the constant functions and the following sequence is exact:
\begin{gather*}
	\begin{CD}
		0 \to
		\mathring\ES\Lambda^{0}(T) @>\cartan>> \mathring\ES\Lambda^{1}(T) @>\cartan>> \cdots @>\cartan>> \mathring\ES\Lambda^{\dim(T)}(T)
		@>\int>> \bbR \to 0.
	\end{CD}
\end{gather*}
In all relevant applications, the element system is compatible. 
Compatible finite element systems allow the construction of geometrically decomposed bases, and their conforming de~Rham complexes have cohomology spaces of the correct dimensions.

Given a finite element system, an \emph{interpolator} is a collection of projection operators $\Interpolant_{T}^{k} : \Cont^{\infty}\Lambda^{k}(T) \to \ES\Lambda^{k}(T)$,
where $k \in \bbZ$ and $T \in \Mesh$, such that for any $S \in \Mesh$ with $S \subseteq T$ it holds that 
$
	\trace_{T,S} \Interpolant_{T}^{k} = \Interpolant_{S}^{k} \trace_{T,S}
$.
Then 
\begin{align*}
    \Interpolant_{\Mesh}^{k} \colon \textstyle{ \bigoplus_{T \in \Mesh} } \Cont^{\infty}\Lambda^{k}(T) \to \textstyle{ \bigoplus_{T \in \Mesh} } \ES\Lambda^{k}(T),
	\quad
	\left( u_{T} \right)_{T \in \Mesh} \mapsto \left( \Interpolant_{T}^{k} u_{T} \right)_{T \in \Mesh}
\end{align*}
defines a \emph{global interpolant}.
The resulting global interpolant restricts to a mapping
$
	\Interpolant_{\Mesh}^{k} \colon \Cont^{\infty}\Lambda^{k}(\Mesh) \to \ES\Lambda^{k}(\Mesh)
$.
A general way of constructing interpolants relies on a formalization of what is commonly known as degrees of freedom in the finite element context.

In what follows, $\Cont^{\infty}\Lambda^{k}(T)^{\dualityast}$ is the algebraic dual space of $\Cont^{\infty}\Lambda^{k}(T)$.
Note that if we have simplices $S, T \in \Mesh$ with $S \subseteq T$, then $\Cont^{\infty}\Lambda^{k}(S)^{\dualityast}$ embeds into $\Cont^{\infty}\Lambda^{k}(T)^{\dualityast}$ by dualization of the pullback.
A \emph{system of degrees of freedom} is a choice of a subspace $\Xi^{k}(T) \subseteq \Cont^{\infty}\Lambda^{k}(T)^{\dualityast}$ for each $k$ and $T$.
Define
\begin{align*}
    \Xi^{k}(\Mesh) := \textstyle{ \bigoplus_{T \in \Mesh} } \Xi^{k}(T).
\end{align*}
For every simplex $T \in \Mesh$, there is a bilinear pairing
\begin{align*}
    \Cont^{\infty}\Lambda^{k}(T) \times \textstyle{ \bigoplus_{S \subseteq T} } \Xi^{k}(S) \to \bbR, \quad (u,\xi) \mapsto \xi(u),
\end{align*}
and the system of degrees of freedom is called \emph{unisolvent} on the finite element system if 
\begin{gather*}
    \ES\Lambda^{k}(T) \times \textstyle{ \bigoplus_{S \subseteq T} } \Xi^{k}(S) \to \bbR, \quad (u,\xi) \mapsto \xi(u)
\end{gather*}
is a non-degenerate pairing for any $T \in \Mesh$.
Informally, the members of $\ES\Lambda^{k}(T)$ are uniquely determined by their degrees of freedom.
If the finite element system admits extensions, then one can show~\cite[Proposition~5.35]{StructPresDisc} that a system of degrees of freedom is unisolvent if and only if all $T \in \Mesh$ admit a non-degenerate bilinear pairing
\begin{gather*}
	\mathring\ES\Lambda^{k}(T) \times \Xi^{k}(T) \to \bbR, \quad (u,\xi) \mapsto \xi(u).
\end{gather*}
One can show that a finite element system has a unisolvent system of degrees of freedom if and only if it admits extensions~\cite[Proposition~5.37]{StructPresDisc}.
Any unisolvent system of degrees of freedom induces an interpolator by
\begin{gather*}
	\forall T \in \Mesh
    \,:\,
    \forall u \in \Cont^{\infty}\Lambda^{k}(T)
    \,:\,
    \forall \xi \in \textstyle{ \bigoplus_{S \subseteq T} } \Xi^{k}(S)
    \,:\,
    \xi\left( \Interpolant_{T}^{k} u \right) = \xi\left( u \right)
    .
\end{gather*}
In fact, one can show the converse: a finite element system has an interpolator if and only if it has a unisolvent system of degrees of freedom; see~\cite[Proposition~5.37]{StructPresDisc}.
The following result concerns interpolators that commute with the exterior derivative~\cite[Proposition~5.41]{StructPresDisc}:
The interpolator induced by a unisolvent system of degrees of freedom commutes with the exterior derivative if and only if
\begin{gather*}
	\forall T \in \Mesh
    \,:\,
    \forall \xi \in \Xi^{k}(T)
    \,:\,
    \xi \circ \cartan \in \textstyle{ \bigoplus_{ S \subseteq T } } \Xi^{k-1}(S).
\end{gather*}
One can show that this \emph{commutativity property} gives rise to a differential complex of the degrees of freedom.

\section{Smoothing operators on domains}\label{section:smoothing}

We introduce a smoothing operator on $\bbR^{n}$ with variable mollification radius. 
The operator smooths a given differential form within a fixed open set but leaves the form unchanged sufficiently far from that set.
We prove bounds in Lebesgue norms and show that the smoothing operator commutes with the exterior derivative.
This is inspired by variable-radius smoothing operators in the literature~\cite{Schoeberl2005multilevel,AFW1,christiansen2008construction,ern2016mollification}.

Mollification via convolution is a well-known concept.
In what follows, $\molli : \bbR^{n} \to \bbR$ is a fixed smooth non-negative function with support in $\Ball{0}{1}{}$ and integral $1$:
\[
    \molli \in \Cont_{c}^{\infty}(\bbR^{n}),
    \qquad
    \molli \geq 0,
    \qquad 
    \supp\molli \Subset \Ball{0}{1}{},
    \qquad
    \int_{\bbR^{n}}\molli(z)\,\dif z=1.
\]
We call $\molli$ the \emph{mollifier}.
For any $\epsilon > 0$, we define the \emph{scaled mollifier} by
\begin{gather}\label{math:smoothing:scaledmollifier}
	\molli_{\epsilon} : \bbR^{n} \to \bbR,
	\quad
	y \mapsto \epsilon^{-n} \molli( \epsilon^{\inv} y ).
\end{gather}
We note that $\molli_{\epsilon}$ is smooth, non-negative, has support in $\Ball{0}{\epsilon}{}$, and has integral $1$.
For any locally integrable function ${u} \in \Lebesgue^{1}_{\loc}(\bbR^{n})$, we denote its convolution with $\molli_{\epsilon}$ by
\begin{gather}\label{math:smoothing:smoothingoffunctions}
	\left( {u} \star \molli_{\epsilon} \right)(x)
	:=
    \int_{\bbR^{n}} \molli_{\epsilon}( y - x ) {u}(y) \;\dif y
	.
\end{gather}
More generally, we mollify differential forms ${u} \in \Lebesgue^{1}_{\loc}\Lambda^{k}(\bbR^{n})$ coefficientwise:
\begin{gather}\label{math:smoothing:smoothingofforms}
	\left( {u} \star \molli_{\epsilon} \right)(x)
	:=
	\sum_{\sigma \in \Sigma(k,n)}
	\int_{\bbR^{n}}
    \molli_{\epsilon}( y - x ) {u}_{\sigma}(y)
	\;\dif y
	\cartanx^{\sigma}
	.
\end{gather}
We recall the following result on the regularization of differential forms.

\begin{theorem}\label{math:smoothing:standardresult}
	Let $p \in [1,\infty]$ and ${u} \in \Lebesgue^{p}\Lambda^{k}(\bbR^{n})$.
	Then ${u} \star \molli_{\epsilon} \in \Cont^{\infty}\Lambda^{k}(\bbR^{n})$,
    and ${u} \star \molli_{\epsilon}$ converges to ${u}$ almost everywhere as $\epsilon$ goes to zero.
	If $p < \infty$, then ${u} \star \molli_{\epsilon}$ converges to ${u}$ in $\Lebesgue^{p}\Lambda^{k}(\bbR^{n})$ as $\epsilon$ goes to zero.
\end{theorem}

\begin{proof}
    This is proven in~\cite[Theorem~C.19]{leoni2017first} for scalar functions. We use the result coefficientwise.
\end{proof}

Smoothing via convolution defines the smoothed field by averaging within a fixed radius of each point. 
The generalized smoothing operators let the averaging radius vary smoothly over the domain. 

We assume that we have a smooth compactly supported non-negative function $\skalierer : \bbR^{n} \to [0,\infty)$, 
\begin{align}\label{math:scaler:definition}
    \skalierer \in \Cont^{\infty}_{c}(\bbR^{n}), \qquad \forall x \in \bbR^{n} : \skalierer(x) \geq 0,
\end{align}
satisfying 
\begin{align}\label{math:scaler:lipschitz-bound}
    \max_{x \in \bbR^{n}} \| \nabla\skalierer(x) \| < 1.
\end{align}
Define the $\skalierer$-neighborhood of any set $A \subseteq \bbR^{n}$ by
\begin{gather}\label{math:smoothing:generalized-ball}
	\Umfeld{A}{\skalierer}{} := A \cup \bigcup_{ \substack{ x \in A \\ \skalierer(x) > 0 } } \Ball{x}{\skalierer(x)}{}
	. %
\end{gather}
We introduce the following position-dependent perturbations of the identity, parameterized by $y \in \bbR^{n}$:
\begin{gather}\label{math:smoothing:translationoperator}
	\Phi_{\skalierer,y}
	:
	\bbR^{n} \to \bbR^{n},
	\quad
	x
	\mapsto
	x + \skalierer(x) y
	.
\end{gather}
Their Jacobians take the form
\begin{align*}
	\Jacobian \Phi_{\skalierer,y} = \Id + y \otimes \nabla\skalierer.
\end{align*}
These transformations define a family of diffeomorphisms.

\begin{lemma}\label{lemma:Phi-is-diffeomorphism}
    For all $y \in \Ball{0}{1}{}$, the mapping $\Phi_{\skalierer,y} : \bbR^{n} \to \bbR^{n}$ is a bi-Lipschitz diffeomorphism. 
    Moreover, 
    \begin{gather}
        \| \Jacobian \Phi_{\skalierer,y} \|_{\Lebesgue^{\infty}}
        \leq 
        1 + \|\nabla\skalierer\|_{\Lebesgue^{\infty}(\bbR^{n})}
        , 
        \qquad 
        \| \Jacobian \Phi_{\skalierer,y}^{-1} \|_{\Lebesgue^{\infty}}
        \leq 
        ( 1 - \|\nabla\skalierer\|_{\Lebesgue^{\infty}(\bbR^{n})} )^{-1}
        \label{math:Phi-is-diffeomorphism:jacobians}
        \\
        1 - \|\nabla\skalierer\|_{\Lebesgue^{\infty}(\bbR^{n})}
        \leq 
        \det( \Jacobian \Phi_{\skalierer,y} )
        \leq 
        1 + \|\nabla\skalierer\|_{\Lebesgue^{\infty}(\bbR^{n})}
        . \label{math:Phi-is-diffeomorphism:determinant}
    \end{gather} 
    For all $x \in \bbR^{n}$ with $\skalierer(x) > 0$, the mapping $y \mapsto \Phi_{\skalierer,y}(x)$ is an embedding of $\Ball{0}{1}{}$ into $\bbR^{n}$.
\end{lemma}
\begin{proof}
    Fix $y \in \Ball{0}{1}{}$. 
    Due to~\eqref{math:scaler:lipschitz-bound} and the determinant identity
    \begin{align}\label{math:determinant-identity}
        \det( \Id + y \otimes \nabla\skalierer_{|x} ) = 1 + y \cdot \nabla\skalierer_{|x},
    \end{align}
    we conclude that $\Phi_{\skalierer,y} : \bbR^{n} \to \bbR^{n}$ is locally a diffeomorphism for any $y \in \Ball{0}{1}{}$.
    Moreover, its determinant satisfies 
    \begin{align*}
        1 - \|\nabla\skalierer\|_{\Lebesgue^{\infty}(\bbR^{n})}
        \leq 
        \det( \Id + y \otimes \nabla\skalierer_{|x} )
        \leq 
        1 + \|\nabla\skalierer\|_{\Lebesgue^{\infty}(\bbR^{n})}
        .
    \end{align*}
    We show that $\Phi_{\skalierer,y}$ is bijective. 
    For each $z \in \bbR^{n}$, define $T_{z,y}(x) := z - \skalierer(x) y$. 
    Any $x \in \bbR^{n}$ is a fixed point of $T_{z,y}$ if and only if $x + \skalierer(x) y = z$. 
    For any $x, x' \in \bbR^{n}$ we find 
    \[
        \| T_{z,y}(x) - T_{z,y}(x') \|
        \leq
        \max_{a \in \bbR^{n}} \| \nabla\skalierer(a) \| 
        \cdot 
        \| x - x' \|.
    \]
    The Lipschitz bound~\eqref{math:scaler:lipschitz-bound} and the Banach fixed point theorem show that $T_{z,y}$ has a unique fixed point $x \in \bbR^{n}$.
    We conclude that for every $z \in \bbR^{n}$ there exists a unique $x \in \bbR^{n}$ such that $\Phi_{\skalierer,y}(x) = z$. 
    Since $\Phi_{\skalierer,y}$ is bijective and a local diffeomorphism, 
    we use the inverse function theorem and conclude that it is a diffeomorphism.
    Taking derivatives and using Lipschitz estimates immediately yields the Jacobian estimates~\eqref{math:Phi-is-diffeomorphism:jacobians}.
    Lastly, if $x \in \bbR^{n}$ with $\skalierer(x) > 0$, then $y \mapsto x + \skalierer(x) y$ is obviously a diffeomorphism from the unit ball onto its image in $\bbR^{n}$. 
\end{proof}

For any ${u} \in \Cont^{\infty}\Lambda^{k}(\bbR^{n})$, which we write in the standard representation~\eqref{math:standardrepresentation}, the pullback along $\Phi_{\skalierer,y}$ satisfies
\begin{align*}
	( \Phi_{\skalierer,y}^{\pullbackast} {u} )_{|x}
	=
	\sum_{ \sigma \in \Sigma(k,n) }
	{u}_{\sigma}\left( x + \skalierer(x) y \right)
	\cdot
	\left( \Phi_{\skalierer,y}^{\pullbackast} \cartanx^{\sigma} \right)_{|x}
	.
\end{align*}
Note that $( \Phi_{\skalierer,y}^{\pullbackast} \cartanx^{\sigma} )_{|x}$ depends smoothly on $x$ and $y$, and that
\begin{align*}
	\left\lVert ( \Phi_{\skalierer,y}^{\pullbackast} \cartanx^{\sigma} )_{|x} \right\rVert
	\leq
	\left( 1 + \lVert y \rVert \cdot \lVert \nabla\skalierer (x) \rVert \right)^{k}
	\leq
	\left( 1 + \lVert y \rVert \cdot \| \nabla\skalierer \|_{\Lebesgue^{\infty}(\bbR^{n})} \right)^{k}
	.
\end{align*}
With these observations in place, we define the localized smoothing operator.
If $u \in \Lebesgue^{1}_{\loc}\Lambda^{k}(\bbR^{n})$, then we set 
\begin{gather}\label{math:smoothing:definition}
	\left( \Regu_{\skalierer}^{k} {u} \right)_{|x}
	:=
	\int_{ \bbR^{n} } \molli(y) \left( \Phi_{\skalierer,y}^{\pullbackast} {u} \right)_{|x} \;\dif y
    .
\end{gather}
We can write the smoothed differential form as 
\[
    \left( \Regu_{\skalierer}^{k}
	{u} \right)_{|x}
    = 
    \begin{cases}
        \int_{ \bbR^{n} } \molli(y) \left( \Phi_{\skalierer,y}^{\pullbackast} {u} \right)_{|x} \;\dif y 
        & \text{ if } \skalierer(x) > 0, 
        \\
        u_{|x}
        & \text{ if } \skalierer(x) = 0.
    \end{cases}
\]
This takes the explicit form 
\begin{gather}\label{math:smoothing:definition:rewriting}
	\left( \Regu_{\skalierer}^{k}
	{u} \right)_{|x}
	=
	\sum_{ \sigma \in \Sigma(k,n) }
	\int_{ \bbR^{n} }
	\molli(y)
	{u}_{\sigma}\left( x + \skalierer(x) y \right)
	\cdot
	\left( \Phi_{\skalierer,y}^{\pullbackast} \cartanx^{\sigma} \right)_{|x}
	\;\dif y
    .
\end{gather}

\begin{remark}
    We briefly comment on measure-theoretic reasons why~\eqref{math:smoothing:definition:rewriting} is well-defined and independent of any pointwise representative of the measurable differential form $u$, up to equivalence almost everywhere. 
    If $x \in \bbR^{n}$ with $\skalierer(x) > 0$, then the integral does not change if we change $u$ on a negligible set. 
    If $x \in \bbR^{n}$ with $\skalierer(x) = 0$, then we first notice that $\Phi_{\skalierer,y}(x) = x$ and $\Jacobian\Phi_{\skalierer,y|x} = \Id$ because $\skalierer$ is non-negative, hence $\nabla\skalierer(x) = 0$. 
    It then follows that $\left( \Regu_{\skalierer}^{k} {u} \right)_{|x} = u(x)$ and hence $\Regu_{\skalierer}^{k} {u} = u$ almost everywhere on the level set $\skalierer^{-1}(0)$. 
\end{remark}

\begin{proposition}\label{proposition:bounded-regularizer}
    Let $p \in [1,\infty]$. 
    We have a linear operator 
    \begin{gather}\label{math:bounded-regularizer:Lploc}
        \Regu_{\skalierer}^{k} : \Lebesgue^{p}_{\loc}\Lambda^{k}(\bbR^{n}) \to \Lebesgue^{p}_{\loc}\Lambda^{k}(\bbR^{n})
        .
    \end{gather}
    We have a bounded linear operator 
    \begin{gather}\label{math:bounded-regularizer:Lp}
        \Regu_{\skalierer}^{k} : \Lebesgue^{p}\Lambda^{k}(\bbR^{n}) \to \Lebesgue^{p}\Lambda^{k}(\bbR^{n})
        .
    \end{gather}
    For every ${u} \in \Lebesgue^{p}_{\loc}\Lambda^{k}(\bbR^{n})$, we have $\Regu_{\skalierer}^{k} u = u$ almost everywhere on the zero level set of $\skalierer$.

    If $A \subseteq \bbR^{n}$ is measurable and ${u} \in \Lebesgue^{p}_{\loc}\Lambda^{k}(\bbR^{n})$, then 
	\begin{gather}\label{math:bounded-regularizer:lpbound}
		\| \Regu_{\skalierer}^{k} {u} \|_{\Lebesgue^{p}\Lambda^{k}(A)}
        \leq
		\left( 1 + \| \nabla\skalierer \|_{\Lebesgue^{\infty}(\bbR^{n})} \right)^{k}
		\left( 1 - \|\nabla\skalierer\|_{\Lebesgue^{\infty}(\bbR^{n})} \right)^{-\frac{1}{p}}
		\| {u} \|_{\Lebesgue^{p}\Lambda^{k}(\Umfeld{A}{\skalierer}{})}
	\end{gather}
    provided that the right-hand side is finite.    
\end{proposition}
\begin{proof}
    Let ${u} \in \Lebesgue^{p}_{\loc}\Lambda^{k}(\bbR^{n})$. 
    It is already clear that $\Regu_{\skalierer}^{k} {u} = {u}$ almost everywhere on the zero level set $\skalierer^{-1}(0)$.
    All other statements follow if we prove~\eqref{math:bounded-regularizer:lpbound}. 
    
    Let $A \subseteq \bbR^{n}$ be a measurable set. 
    Then $\Umfeld{A}{\skalierer}{}$ is a measurable set.
    For every $x \in \bbR^{n}$,
    \begin{align*}
		\lVert \Regu_{\skalierer}^{k} ( {u} )_{|x} \rVert
		\leq
		\left( 1 + \| \nabla\skalierer \|_{\Lebesgue^{\infty}(\bbR^{n})} \right)^{k}
		\int_{ \bbR^{n} }
		\molli(y)
		\lVert {u}\left( x + \skalierer(x) y \right) \rVert
		\;\dif y
		.
	\end{align*}
    The estimate~\eqref{math:bounded-regularizer:lpbound} is obvious when $p = \infty$, so we discuss only $p \in [1,\infty)$ in detail.
    On the one hand, we already see that 
    \begin{align*}
		\int_{A \cap \skalierer^{-1}(0)}
		\lVert \Regu_{\skalierer}^{k} ( {u} )_{|x} \rVert^{p}
		\dif x
        \leq 
        \int_{A \cap \skalierer^{-1}(0)}
		\lVert {u}\left( x \right) \rVert^{p}
		\;\dif x
		.
	\end{align*}
	On the other hand, Jensen's inequality and Fubini's theorem lead to 
	\begin{align*}
		\int_{A \setminus \skalierer^{-1}(0)}
		\lVert \Regu_{\skalierer}^{k} ( {u} )_{|x} \rVert^{p}
		\dif x
		 &
		\leq
		\left( 1 + \| \nabla\skalierer \|_{\Lebesgue^{\infty}(A)} \right)^{pk}
		\int_{A \setminus \skalierer^{-1}(0)}
		\left\lVert
            \int_{ \bbR^{n} }
            \molli(y)
            \lVert{u}\left( x + \skalierer(x) y \right)\rVert
		\;\dif y
		\right\rVert^{p}
		\dif x
		\\&
		\leq
		\left( 1 + \| \nabla\skalierer \|_{\Lebesgue^{\infty}(A)} \right)^{pk}
		\int_{A \setminus \skalierer^{-1}(0)}
		\int_{ \bbR^{n} }
		\molli(y)
		\lVert {u}\left( x + \skalierer(x) y \right) \rVert^{p}
		\;\dif y
		\dif x
		\\&
		=
		\left( 1 + \| \nabla\skalierer \|_{\Lebesgue^{\infty}(A)} \right)^{pk}
		\int_{ \bbR^{n} }
		\molli(y)
		\int_{A \setminus \skalierer^{-1}(0)}
		\lVert {u}\left( x + \skalierer(x) y \right) \rVert^{p}
		\;\dif x
		\dif y
		.
	\end{align*}
    The integrand is non-zero only for $y \in \Ball{0}{1}{}$. Then $\Phi_{\skalierer,y}$ is a diffeomorphism. By the change of variables $z := \Phi_{\skalierer,y}(x)$,
	\begin{align*}
		\int_{A \setminus \skalierer^{-1}(0)} \lVert{u}\left( x + \skalierer(x) y \right)\rVert^{p} \;\dif x
        &
		=
		\int_{\Phi_{\skalierer,y}( A \setminus \skalierer^{-1}(0) )} \lVert{u}(z)\rVert^{p} \left\lvert\det( \Id + y \otimes \nabla\skalierer_{|\Phi_{\skalierer,y}^{-1}(z)} )^{-1}\right\rvert \;\dif z
        \\&
		\leq
		\left( 1 - \|\nabla\skalierer\|_{\Lebesgue^{\infty}(\bbR^{n})} \right)^{-1}
		\int_{\Umfeld{A}{\skalierer}{}} \lVert{u}(z)\rVert^{p} \;\dif z
		.
	\end{align*}
    Here we used Lemma~\ref{lemma:Phi-is-diffeomorphism}.
    Finally, recall that $\molli$ has unit integral. 
    The desired result is thus evident.
\end{proof}

The smoothing operator is inspired by the standard convolution-based mollification. 
Consequently, it smooths the differential form on the open set where $\skalierer$ is positive while preserving any pre-existing smoothness.

\begin{lemma}\label{lemma:smoothing:retain-and-improve-smoothness}
    Let ${u} \in \Lebesgue^{1}_{\loc}\Lambda^{k}(\bbR^{n})$.
    Then 
    \begin{align}\label{math:smoothing:retain-and-improve-smoothness:smoothness}
        \Regu_{\skalierer}^{k} {u} 
        \in 
        \Cont^{\infty}\Lambda^{k}( \{ x \in \bbR^{n} \suchthat \skalierer(x) > 0 \} )
        .
    \end{align}
    If ${u} \in \Lebesgue^{p}\Lambda^{k}(\bbR^{n})$ for some $p \in [1,\infty]$ and $x \in \bbR^{n}$ satisfies $\skalierer(x) > 0$, then 
    \begin{align}\label{math:smoothing:retain-and-improve-smoothness:bound}
        \| \Regu_{\skalierer}^{k} ( {u} )_{|x} \|
        &
        \leq 
        \tbinom{n}{k}
        \volume{\Ball{0}{1}{}}^{1-1/p}
        \left\| \molli \right\|_{\Lebesgue^{\infty}(\bbR^{n})}
        \left( 1 + \| \nabla\skalierer \|_{\Lebesgue^{\infty}(\bbR^{n})} \right)^{k}
        \skalierer(x)^{-n/p}
        \left\| {u} \right\|_{\Lebesgue^{p}\Lambda^{k}(\Ball{x}{\skalierer(x)}{})}
        .
    \end{align}
    If $\Domain \subseteq \bbR^{n}$ is open and ${u} \in \Lebesgue^{1}_{\loc}\Lambda^{k}(\bbR^{n})$ with ${u} \in \Cont^{m}\Lambda^{k}(\Domain)$ for some $m \in \bbN_{0}$, then
    \begin{align}\label{math:smoothing:retain-and-improve-smoothness:retain}
        \Regu_{\skalierer}^{k} {u} \in \Cont^{m}\Lambda^{k}(\Domain)
        .
    \end{align}
\end{lemma}

\begin{proof}
    We use the standard representations of ${u}$ and $\Regu_{\skalierer}^{k}{u}$:
	\[
		{u}(x) = \sum_{ \sigma \in \Sigma(k,n) } {u}_{\sigma}(x) \cartanx^{\sigma},
		\qquad
        \Regu_{\skalierer}^{k} {u}(x) = \sum_{ \tau \in \Sigma(k,n) } (\Regu_{\skalierer}^{k} {u})_{\tau}(x) \cartanx^{\tau}.
	\]
    Assume $x \in \bbR^{n}$ with $\skalierer(x) > 0$. 
    With an affine change of variables $z := x + \skalierer(x) y$, we rewrite the integral representation~\eqref{math:smoothing:definition:rewriting} as 
    \begin{align*}
        \Regu_{\skalierer}^{k} ( {u} )_{|x}
        &
        =
        \sum_{ \sigma \in \Sigma(k,n) }
        \int_{ \bbR^{n} }
        \molli(y) {u}_{\sigma}\left( x + \skalierer(x) y \right)
        \cdot
        ( \Phi_{\skalierer,y}^{\pullbackast} \cartanx^{\sigma} )_{|x} \;\dif y
        \\&
        =
        \sum_{ \sigma \in \Sigma(k,n) }
        \int_{ \bbR^{n} }
        \skalierer(x)^{-n} \molli( \skalierer(x)^{-1}( z - x ) ) {u}_{\sigma}(z)
        \cdot
        ( \Phi_{\skalierer,\skalierer(x)^{-1}( z - x )}^{\pullbackast} \cartanx^{\sigma} )_{|x} \;\dif z
        .
    \end{align*}
    We write this as an integral operator with a matrix-valued kernel:
	\[
        \Regu_{\skalierer}^{k} {u}(x)
		=
		\sum_{ \substack{ \sigma, \tau \in \Sigma(k,n) } }
		\int_{\bbR^{n}}
        \calK_{\tau,\sigma}(x,z)\,{u}_{\sigma}(z)\,\dif z \;\cartanx^{\tau}
        .
	\]
	Here, we use smooth functions 
    \[
        \calK_{\tau,\sigma} : \{ (x,z) \in \bbR^{n} \times \bbR^{n} \suchthat \skalierer(x) > 0 \} \to \bbR
    \]
    which satisfy 
    \[
        \forall x \in \bbR^{n} : \forall z \in \bbR^{n} 
        : 
        \skalierer(x) > 0 \wedge \calK_{\tau,\sigma}(x,z) \neq 0
        \implies 
        z \in \Ball{x}{\skalierer(x)}{} 
        .
    \]
    Fix $x_{0} \in \bbR^{n}$ with $\skalierer(x_{0}) > 0$, and let $\subdomain \subseteq \bbR^{n}$ be a precompact open neighborhood of $x_{0}$.
    By shrinking the neighborhood, we assume that for some $a, b > 0$ we have for all $x' \in \subdomain$ the estimate $a \leq \skalierer(x') \leq b$. 
    Thus, 
    \[
        V \Subset \{ x \in \bbR^{n} \suchthat \skalierer(x) > 0 \}, 
        \qquad 
        \forall x' \in \subdomain \, : \, \forall \sigma,\tau \in \Sigma(k,n) \, : \, \supp \calK_{\tau,\sigma}(x',\cdot) \subseteq \closure{\Ball{x'}{b}{}}
        .
    \]
    In particular, 
    \[
        \subdomain \times \Ball{\subdomain}{b}{} \Subset \{ (x,z) \in \bbR^{n} \times \bbR^{n} \suchthat \skalierer(x) > 0 \}.
    \]
    Hence, $\calK_{\tau,\sigma}(x',z)$ is smooth on $\subdomain \times \Ball{\subdomain}{b}{}$ with each derivative uniformly bounded. 
    We conclude that $\Regu_{\skalierer}^{k}{u}$ is smooth in a neighborhood of $x_{0}$ on which $\skalierer$ takes positive values, implying~\eqref{math:smoothing:retain-and-improve-smoothness:smoothness}. 
    
    Furthermore, 
    \begin{align*}
        &
        \| \Regu_{\skalierer}^{k} {u}_{|x} \|
        \leq 
        \sum_{ \sigma \in \Sigma(k,n) }
        \int_{ \Ball{x}{\skalierer(x)}{} }
        \left\| 
            \skalierer(x)^{-n} \molli( \skalierer(x)^{-1}( z - x ) ) {u}_{\sigma}(z)
            \cdot
            ( \Phi_{\skalierer,\skalierer(x)^{-1}( z - x )}^{\pullbackast} \cartanx^{\sigma} )_{|x}
        \right\|
         \;\dif z
        \\&\quad 
        \leq 
        \left( 1 + \| \nabla\skalierer \|_{\Lebesgue^{\infty}(\bbR^{n})} \right)^{k}
        \left\| \molli \right\|_{\Lebesgue^{\infty}(\bbR^{n})}
        \skalierer(x)^{-n}
        \sum_{ \sigma \in \Sigma(k,n) }
        \int_{ \Ball{x}{\skalierer(x)}{} } \left\lvert {u}_{\sigma}(z) \right\rvert \;\dif z
        \\&\quad 
        \leq 
        \left( 1 + \| \nabla\skalierer \|_{\Lebesgue^{\infty}(\bbR^{n})} \right)^{k}
        \left\| \molli \right\|_{\Lebesgue^{\infty}(\bbR^{n})}
        \skalierer(x)^{-n}
        \sum_{ \sigma \in \Sigma(k,n) }
        \volume{\Ball{0}{1}{}}^{1-1/p}
        \skalierer(x)^{n(1-1/p)}
        \left\| {u}_{\sigma} \right\|_{\Lebesgue^{p}(\Ball{x}{\skalierer(x)}{})} 
        .
    \end{align*}
    This implies the desired inequality~\eqref{math:smoothing:retain-and-improve-smoothness:bound}. 
    
    Suppose that $\Domain$ is open, ${u} \in \Cont^{m}\Lambda^{k}(\Domain)$, and $x \in \Domain$.
    If $\skalierer(x) > 0$, then $\Regu_{\skalierer}^{k} {u}$ is smooth in an open neighborhood of $x$.
    If $\skalierer(x) = 0$, then we choose $\delta > 0$ so small that 
    \[ 
		\Ball{\Ball{x}{\delta}{}}{\skalierer}{} \Subset \Domain.
	\]
    Indeed, we may choose $\delta > 0$ so that $( 1 + \| \nabla\skalierer \|_{\Lebesgue^{\infty}(\bbR^{n} ) } ) \delta < \dist( x, a )$ for all $a \in \bbR^{n} \setminus \Domain$. 
    Then the integral 
    \[
        \sum_{ \sigma \in \Sigma(k,n) }
        \int_{ \bbR^{n} }
        \molli(y) {u}_{\sigma}\left( x' + \skalierer(x') y \right)
        \left( \Phi_{\skalierer,y}^{\pullbackast} \cartanx^{\sigma} \right)_{|x'} \;\dif y
    \]
    has an integrand in $(x',y)$ that has bounded continuous derivatives up to order $m$ within $\Ball{x}{\delta}{} \times \Ball{0}{1}{}$.
    We conclude that $\Regu_{\skalierer}^{k}{u}$ has $m$ continuous derivatives in a neighborhood of $x$, implying~\eqref{math:smoothing:retain-and-improve-smoothness:retain}.
\end{proof}

We conclude the discussion of the Euclidean smoothing operator by proving that $\Regu_{\skalierer}^{k}$ commutes with the exterior derivative.

\begin{lemma}\label{lemma:smoothing:commutativity}
	If ${u} \in \Sobdiff^{1,1}_{\loc}\Lambda^{k}(\bbR^{n})$, 
    then $\Regu_{\skalierer}^{k} {u} \in \Sobdiff^{1,1}_{\loc}\Lambda^{k}(\bbR^{n})$ 
    and $\cartan \Regu_{\skalierer}^{k} {u} = \Regu_{\skalierer}^{k+1} \cartan {u}$. 
\end{lemma} 

\begin{proof}
	Let ${u} \in \Lebesgue^{1}_{\loc}\Lambda^{k}(\bbR^{n})$ such that $\cartan{u} \in \Lebesgue^{1}_{\loc}\Lambda^{k+1}(\bbR^{n})$.
    The exterior derivative of $\Regu_{\skalierer}^{k} {u}$ exists in the sense of distributions. 
    Let ${v} \in \Cont^{\infty}\Lambda^{n-k-1}(\bbR^{n})$ have compact support. 
    Since ${u}$ is locally integrable, the Fubini--Tonelli theorem and an integration by parts lead to 
	\begin{align*}
		\int \Regu_{\skalierer}^{k} {u} \wedge \cartan_{x} {v}_{|x}
		&=
		\int \int_{\bbR^{n}} \molli(y) \left( \Phi_{\skalierer,y}^{\pullbackast} {u} \right)_{|x} \dif y \wedge \cartan_{x} {v}_{|x}
		\\&=
		\int_{\bbR^{n}} \int \molli(y) \left( \Phi_{\skalierer,y}^{\pullbackast} {u} \right)_{|x} \wedge \cartan_{x} {v}_{|x} \dif y
        =
		(-1)^{k+1}
		\int_{\bbR^{n}} \int \molli(y) \cartan_{x} \left( \Phi_{\skalierer,y}^{\pullbackast} {u} \right)_{|x} \wedge {v}_{|x} \dif y
		.
	\end{align*} 
    In the last step, we have applied integration by parts because $\Phi_{\skalierer,y}^{\pullbackast} {u}$ has a weak derivative and $v$ is compactly supported. 
    Since $\Phi_{\skalierer,y}$ is a smooth diffeomorphism,
    Theorem~\ref{theorem:transformationssatz} implies that $\cartan \Phi_{\skalierer,y}^{\pullbackast} {u} = \Phi_{\skalierer,y}^{\pullbackast} \cartan {u}$.
    Unwinding the changes, we find
	\begin{align*}
		\int_{\bbR^{n}} \int \molli(y) \left( \Phi_{\skalierer,y}^{\pullbackast} \cartan {u} \right)_{|x} \wedge {v}_{|x} \dif y
		 & =
		\int \int_{\bbR^{n}} \molli(y) \left( \Phi_{\skalierer,y}^{\pullbackast} \cartan {u} \right)_{|x} \dif y \wedge {v}_{|x}
		=
		\int \Regu_{\skalierer}^{k+1}\left( \cartan {u} \right) \wedge {v}_{|x}
		.
	\end{align*}
    Comparing this identity with the definition of the weak exterior derivative proves the result.
\end{proof}

\section{Smoothed projections}\label{section:smoothedprojections}

We are now in a position to define smoothed projections on compact manifolds.
The smoothed projection is composed of three stages.
In the first stage, we use localized smoothing operators.
In the second stage, the canonical interpolant maps the smoothed differential form into the finite element space.
These two stages yield a uniformly bounded commuting quasi-interpolant that is not yet idempotent.
In the third stage, we invert the restriction of the quasi-interpolant to the finite element space, thereby obtaining the desired projection.

\subsection{Geometric preparations} 

Let $\Manifold$ be a connected compact $n$-dimensional Riemannian smooth manifold without boundary. 
Let $\metric$ denote its Riemannian metric.
Let $\calA = \{\Psi_{i}\}_{i \in I}$ be an atlas of $\Manifold$ given by a collection of smooth coordinate charts
\begin{align*}
	\Psi_{i} : \Domain_{i} \subseteq \Manifold \to \widekarte\Domain_{i} \subseteq \bbR^{n}.
\end{align*}
The atlas $\calA$ is called \emph{uniformly $\Cont^{m}$} if for all $1 \leq \ell \leq m$ there exists a constant $\Catlasregularity{\ell} \geq 0$ satisfying for all $i \in I$
\begin{align*}
    \| \nabla^{\ell-1} \Jacobian\Psi_{i} \|_{\Lebesgue^{\infty}} 
    \leq 
    \Catlasregularity{\ell}
    ,
	\qquad
    \| \nabla^{\ell-1} \Jacobian\Psi_{i}^{-1} \|_{\Lebesgue^{\infty}} 
    \leq 
    \Catlasregularity{\ell}
    .
\end{align*}
The atlas $\calA$ is called \emph{uniformly bi-Lipschitz} if for some constant $\Catlaslipschitz > 0$, called the \emph{Lipschitz constant of the atlas}, 
and all $ i \in I $ and all $x, y \in \Domain_{i}$ we have
\begin{subequations} \label{math:quasiisometrybounds}
	\begin{gather}
		\Catlaslipschitz^{-1}
		\dist_{\metric}(x,y)
		\leq
		\left\lVert \Psi_{i}(x) - \Psi_{i}(y) \right\rVert
		\leq
		\Catlaslipschitz \dist_{\metric}(x,y)
		.
	\end{gather}
\end{subequations}
Thus, the coordinate charts in a uniformly bi-Lipschitz atlas have uniformly bounded metric distortion.
This implies that such coordinate charts map balls into balls:
\begin{align*}
    \forall 
    x \in \Domain_{i}
    &\,:\,
    \Psi_{i}( \Ball{x}{\delta}{\metric} \cap \Domain_{i} )
	\subseteq
    \Ball{\Psi_{i}(x)}{\Catlaslipschitz \delta}{} \cap \widekarte\Domain_{i}
    ,
	\\
	\forall 
    \karte{x} \in \widekarte\Domain_{i}
    &\,:\,
    \Psi_{i}^{-1}( \Ball{\karte{x}}{\delta}{} \cap \widekarte\Domain_{i} )
	\subseteq
    \Ball{\Psi_{i}^{-1}(\karte{x})}{\Catlaslipschitz \delta}{\metric} \cap \Domain_{i} 
    .
\end{align*}
It also implies that the charts have bounded Jacobians:
\begin{align*}
    \| \Jacobian\Psi_{i} \|_{\Lebesgue^{\infty}} \leq \Catlaslipschitz,
	\qquad
    \| \Jacobian\Psi_{i}^{-1} \|_{\Lebesgue^{\infty}} \leq \Catlaslipschitz.
\end{align*}

Suppose that $\calA = \{\Psi_{i}\}_{i \in I}$ is a finite uniformly bi-Lipschitz atlas whose coordinate charts we write as $\Psi_{i} : \Domain_{i} \subseteq \Manifold \to \widekarte{\Domain}_{i} \subseteq \bbR^{n}$. 
A \emph{bump system} subordinate to $\calA$ is a collection $\{ \karte\skalierer_{i} \}_{i \in I}$ of smooth non-negative functions 
\[
    \karte\skalierer_{i} : \widekarte{\Domain}_{i} \rightarrow \bbR_{0}^{+}
\]
subject to the following conditions:
\begin{enumerate}
    \item $\supp\karte\skalierer_{i} \Subset \widekarte{\Domain}_{i}$,
	\item $\max \karte\skalierer_{i} = 1$,
    \item We have a cover of $\Manifold$ given by the family of open sets
    \[ 
        \interior{ \Psi^{-1}_{i} \left\{\; \karte{x} \in \widekarte{\Domain}_{i} \suchthat* \karte{\skalierer}_{i}(\karte{x}) = 1 \;\right\} }, \qquad i \in I.
    \]
\end{enumerate}
Given such a bump system, define the smooth functions 
\[
    \skalierer_{i} : \Manifold \to \bbR_{0}^{+}, \qquad x \mapsto 
    \begin{cases}
        \karte\skalierer_{i}( \Psi_{i}(x) ) & \text{ if } x \in \Domain_{i}
        \\
        0 & \text{ otherwise. }
    \end{cases}
\]
We also indicate how separated the supports of the bump functions are from the chart boundaries:
\begin{align}\label{math:atlas-radius}
    \atlasradius 
    := 
    \inf_{ i \in I }
    \dist_{\metric}\left( 
        \supp \skalierer_{i}
        ,
        \Manifold \setminus \Domain_{i}
    \right)
    .
\end{align}

\begin{theorem}\label{theorem:existence-of-good-atlas}
    Let $\Manifold$ be a connected compact Riemannian smooth manifold without boundary. 
    Then $\Manifold$ admits a finite uniformly bi-Lipschitz atlas with precompact coordinate images in $\bbR^{n}$. 
    If $m \in \bbN$, then we can assume in addition that the atlas is uniformly $\Cont^{m}$.
\end{theorem}
\begin{proof}
    Let $\Manifold$ be a compact Riemannian manifold, and let $a \in \Manifold$. 
    Choose a coordinate chart $\Psi : \Domain \to \widekarte\Domain$, where $\Domain$ is an open neighborhood of $a$ in $\Manifold$ and $\widekarte\Domain \subseteq \bbR^{n}$ is open.
    Choose $\delta > 0$ such that $\Ball{\Psi(a)}{\delta}{} \Subset \widekarte\Domain$.
    Then $a$ has an open neighborhood $\subdomain_{\delta} \Subset \Domain$ defined by $\Psi(\subdomain_{\delta}) = \Ball{\Psi(a)}{\delta}{}$. 
    
    The mapping $\Psi^{-1}$ maps any straight line segment within that ball onto a smooth path within $\Manifold$, 
    and it has a Jacobian on $\Ball{\Psi(a)}{\delta}{} \Subset \widekarte\Domain$ whose operator norm is uniformly bounded by some constant $C > 0$.
    We conclude that $\Psi^{-1} : \Ball{\Psi(a)}{\delta}{} \to \subdomain_{\delta}$ is Lipschitz. 
    By construction, $\Psi^{-1} : \Ball{\Psi(a)}{\delta}{} \to \subdomain_{\delta}$ and $\Psi : \subdomain_{\delta} \to \Ball{\Psi(a)}{\delta}{}$ have bounded derivatives. 
    
    Recall that, on a compact connected Riemannian smooth manifold, the distance between any two points is attained by a minimizing geodesic path.
    For any $\epsilon > 0$, the points in the geodesic ball $\Ball{a}{\epsilon}{\metric}$ have their distances attained by paths within $\Ball{a}{4\epsilon}{\metric}$.
    Since the geodesic open balls generate the topology of the manifold, we can choose $\epsilon > 0$ so small that $\Ball{a}{4\epsilon}{\metric} \subseteq \subdomain_{\delta}$. 
    Suppose that $x, y \in \Ball{a}{\epsilon}{\metric}$, and let $\gamma : [0,1] \to \Ball{a}{4\epsilon}{\metric} \subseteq \Manifold$ be a minimizing geodesic from $x$ to $y$. 
    Then
    \begin{align*}
        \left\lVert \Psi(x) - \Psi(y) \right\rVert
        &\leq 
        \int_{0}^{1} \lVert \partial_t ( \Psi \circ \gamma)(t) \rVert_{\bbR^{n}} \dif t
        \leq 
        \| \Jacobian\Psi \restriction_{ \subdomain_{\delta} } \|_{\Lebesgue^{\infty}} \cdot \int_{0}^{1} \lVert \partial_t \gamma(t) \rVert_{\metric} \dif t
        .
    \end{align*}
    Hence, the restriction $\Psi : \Ball{a}{\epsilon}{\metric} \to \Psi( \Ball{a}{\epsilon}{\metric} )$ is Lipschitz. 
    Its inverse is Lipschitz because it is a restriction of the Lipschitz map $\Psi^{-1} : \Ball{\Psi(a)}{\delta}{} \to \subdomain_{\delta}$.
    This shows that $\Psi : \Ball{a}{\epsilon}{\metric} \to \Psi( \Ball{a}{\epsilon}{\metric} )$ bi-Lipschitz. 
    Obviously, $\Ball{a}{\epsilon}{\metric}$ is precompact.
    
    Since $\Manifold$ is compact, finitely many such neighborhoods cover $\Manifold$. 
    The existence of a finite uniformly bi-Lipschitz uniformly $\Cont^{m}$ atlas is proven. 
\end{proof}

\begin{theorem}\label{theorem:bumpsystem}
    Let $\Manifold$ be a connected compact Riemannian smooth manifold without boundary. 
    Let $\calA = \{\Psi_{i}\}_{i \in I}$ be a finite uniformly bi-Lipschitz atlas.
    There exists a bump system subordinate to $\calA$.
\end{theorem}
\begin{proof}
    Consider any bi-Lipschitz mapping $\Psi_{i} : \Domain_{i} \subseteq \Manifold \to \widekarte\Domain_{i} \subseteq \bbR^{n}$ from the atlas. 
    The open sets $\Domain_{i}$ constitute an open cover of $\Manifold$. 
    Applying the shrinking lemma twice, we choose non-empty open subsets $\subdomain_{i} \Subset \Domain_{i}$ that cover $\Manifold$, and non-empty open subsets $\subsubdomain_{i} \Subset \subdomain_{i}$ that again cover the manifold.
    Write $\widekarte\subdomain_{i} := \Psi_{i}(\subdomain_{i})$ and $\widekarte\subsubdomain_{i} := \Psi_{i}(\subsubdomain_{i})$.
    Since $I$ is finite, we fix $\delta > 0$ such that $\Ball{\widekarte\subsubdomain_{i}}{\delta}{} \subseteq \widekarte\subdomain_{i}$ and $\Ball{\widekarte\subdomain_{i}}{\delta}{} \Subset \widekarte\Domain_{i}$ for every $i \in I$.
    Let $\bboldone_{\widekarte\subdomain_{i}} : \bbR^{n} \to \{0,1\}$ be the indicator function of $\widekarte\subdomain_{i}$.
    For any $0 < \lambda < \delta$, define $\karte\skalierer_{i} := \bboldone_{\widekarte\subdomain_{i}} \star \molli_{\lambda}$ as the mollification with radius $\lambda > 0$. 
    Then $\karte\skalierer_{i}$ is non-negative and smooth. It satisfies 
    \[
        \image \karte\skalierer_{i} \subseteq [0,1],
        \quad 
        \max \karte\skalierer_{i} = 1,
        \quad 
        \supp \karte\skalierer_{i} \Subset \widekarte\Domain_{i},
        \quad 
        \widekarte\subsubdomain_{i} \subseteq \karte\skalierer_{i}^{-1}(1).
    \]
    This is the desired function associated with the index $i$.
\end{proof}

In light of those results, we henceforth assume for the remainder of this section that
\begin{align*}
	\calA = \{ \Psi_{i} : \Domain_{i} \to \widekarte\Domain_{i} \}_{i \in I}
\end{align*}
is a finite uniformly bi-Lipschitz atlas with Lipschitz constant $\Catlaslipschitz \geq 1$, and that 
\begin{align*}
    \left\{ \karte\skalierer_{i} : \widekarte\Domain_{i} \rightarrow \bbR_{0}^{+} \right\}_{i \in I}
\end{align*}
is a bump system subordinate to $\calA$. Whenever convenient, we identify $\karte\skalierer_{i}$ with its extension-by-zero to a smooth function on $\bbR^{n}$. 
The number of charts is denoted by
\begin{gather}
    \NumberOfCharts := \cardinality{I}.
\end{gather}
After relabeling, we may assume that $I = \{ 1, \dots, \NumberOfCharts \}$.

\subsection{Commuting regularization}

We fix a finite smooth triangulation $\Mesh$ of the manifold $\Manifold$.
Let $\meshfunc$ be the smooth mesh-size function provided by Theorem~\ref{theorem:meshfunction}.
We abbreviate
\begin{align*}
	\karte\meshfunc_{i}
	:=
	\meshfunc \circ \Psi_{i}^{-1}
	.
\end{align*}
By the chain rule and the uniform atlas bounds,
\begin{gather*}
	\| \nabla \karte\meshfunc_{i} \|_{ \Lebesgue^{\infty}(\widekarte\Domain_{i}) }
	\leq
	\| \Jacobian \Psi_{i}^{-1} \|_{ \Lebesgue^{\infty}(\widekarte\Domain_{i}) }
	\| \nabla \meshfunc \|_{ \Lebesgue^{\infty}(\Domain_{i},\metric) }
	,
	\qquad
	\| \karte\meshfunc_{i} \|_{ \Lebesgue^{\infty}(\widekarte\Domain_{i}) } \leq \Cmeshsizevalue \max_{T \in \Mesh} h_{T}
	.
\end{gather*}
For $\epsilon > 0$ sufficiently small, the following holds uniformly for every $i\in I$:
\begin{gather}\label{math:epsilon-conditions-for-local-smoothing-operators}
	    \| \nabla( \epsilon \karte\skalierer_{i} \karte\meshfunc_{i} ) \|_{ \Lebesgue^{\infty}(\widekarte\Domain_{i}) } < 1
        ,
	    \qquad 
        \epsilon \| \karte\skalierer_{i} \karte\meshfunc_{i} \|_{ \Lebesgue^{\infty}(\widekarte\Domain_{i}) } 
        \leq 
        \epsilon \Catlaslipschitz \Cmeshsizevalue \cdot \max_{ T \in \Mesh} h_{T} < \atlasradius
        .
\end{gather}
Note that the second condition implies $\Umfeld{\supp\karte\skalierer_{i}}{ \epsilon \karte\skalierer_{i} \karte\meshfunc_{i} }{} \Subset \widekarte\Domain_{i}$. 

Generalizing the notation~\eqref{math:smoothing:generalized-ball} from the Euclidean setting to manifolds, we define the $\skalierer$-neighborhood of any set $A \subseteq \Manifold$ by
\begin{gather}\label{math:smoothing:generalized-ball-in-manifold}
	\Umfeld{A}{\skalierer}{\metric} := A \cup \bigcup_{ \substack{ x \in A \\ \skalierer(x) > 0 } } \Ball{x}{\skalierer(x)}{\metric}.
\end{gather}
We define localized smoothing operators $\karte \Regu^{k}_{i}$ within each coordinate chart: 
\begin{gather}\label{math:localsmoothing-on-chart}
	\karte \Regu_{i}^{k} :
	\Lebesgue^{1}_{\loc}\Lambda^{k}(\widekarte\Domain_{i}) \to \Lebesgue^{1}_{\loc}\Lambda^{k}(\widekarte\Domain_{i}),
	\quad
	{u} \mapsto
	\begin{cases}
		\int_{ \bbR^{n} } \molli(y) 
        ( \Phi_{\epsilon \karte\skalierer_{i} \karte\meshfunc_{i},y}^{\pullbackast} {u} )_{|\karte{x}} \;\dif y
		& \text{ if } 
        \karte{x} \in \supp\karte\skalierer_{i},
		\\
		{u}_{|\karte{x}}
		& \text{ if } 
        \karte{x} \notin \supp\karte\skalierer_{i}.
		\end{cases}
\end{gather}
Here, we implicitly extend $\karte\skalierer_{i}\karte\meshfunc_{i}$, initially defined on $\widekarte\Domain_{i}$ by zero to a function on $\bbR^n$.
We transfer this operator to the manifold using pullback along the coordinate charts. This gives mappings 
\begin{gather}\label{math:localsmoothing}
	\Regu_{i}^{k} :
	\Lebesgue^{1}_{\loc}\Lambda^{k}(\Manifold) \to \Lebesgue^{1}_{\loc}\Lambda^{k}(\Manifold),
	\quad
	{u} \mapsto
	\begin{cases}
        ( \Psi_{i}^{\pullbackast} \karte\Regu_{i}^{k} \Psi_{i}^{-\pullbackast} {u} )_{|x}
		& \text{ if } 
        x \in \Domain_{i}
        ,
		\\
		{u}_{|x}
		& \text{ if } 
        x \notin \Domain_{i}
        .
		\end{cases}
\end{gather}
Notice that the smoothing radius in the coordinate chart is controlled by the bump function $\karte\skalierer_{i}$ as well as by the mesh size function $\karte\meshfunc_{i}$.
We first summarize the properties of this local smoothing operator. 

\begin{proposition}\label{proposition:local-smoothing-manifold}
    Let $p \in [1,\infty]$. 
    We have a linear operator 
    \begin{gather}\label{math:local-smoothing-manifold:Lploc}
        \Regu_{i}^{k} : \Lebesgue^{p}_{\loc}\Lambda^{k}(\Manifold) \to \Lebesgue^{p}_{\loc}\Lambda^{k}(\Manifold)
        .
    \end{gather}
    We have a bounded linear operator 
    \begin{gather}\label{math:local-smoothing-manifold:Lp}
        \Regu_{i}^{k} : \Lebesgue^{p}\Lambda^{k}(\Manifold) \to \Lebesgue^{p}\Lambda^{k}(\Manifold)
        .
    \end{gather}
    If ${u} \in \Sobdiff^{1,1}_{\loc}\Lambda^{k}(\Manifold)$, then $\Regu_{i}^{k} {u} \in \Sobdiff^{1,1}_{\loc}\Lambda^{k}(\Manifold)$ and 
    \begin{gather}\label{math:local-smoothing-manifold:commutativity}
        \cartan \Regu_{i}^{k} {u} = \Regu_{i}^{k+1} \cartan {u}
        .
    \end{gather}
    If $\Domain \subseteq \Manifold$ is open and ${u} \in \Lebesgue^{1}_{\loc}\Lambda^{k}(\Manifold)$ with ${u} \restriction_{\Domain} \in \Cont^{m}\Lambda^{k}(\Domain)$ for some $m \in \bbN_{0}$, then
    \begin{align}\label{math:local-smoothing-manifold:smoothing}
        \Regu_{i}^{k} {u} \in \Cont^{m}\Lambda^{k}(\Domain)
        .
    \end{align}
    If ${u} \in \Lebesgue^{1}_{\loc}\Lambda^{k}(\Manifold)$, then 
    $\Regu_{i}^{k} u = u$ almost everywhere on the zero level set of $\skalierer_{i}$ and 
    \begin{align}\label{math:local-smoothing-manifold:retains-regularity}
        \Regu_{i}^{k} {u} 
        \in 
        \Cont^{\infty}\Lambda^{k}( \{ x \in \Manifold \suchthat \skalierer_{i}(x) > 0 \} ). 
    \end{align}
    If ${u} \in \Lebesgue^{p}\Lambda^{k}(\Manifold)$ and $x \in \Manifold$ with $\skalierer_{i}(x) > 0$, then 
    \begin{align} \label{math:local-smoothing-manifold:continuous-point-estimate}
        \| \Regu_{i}^{k} {u} _{|x} \|
        \leq 
        \Catlaslipschitz^{2k+\frac{n}{p}}
        \tbinom{n}{k}
        \left\| \molli \right\|_{\Lebesgue^{\infty}(\bbR^{n})}
        \volume{\Ball{0}{1}{}}^{1-\frac{1}{p}}
        \frac{ \left( 1 + \epsilon \| \nabla(\karte\meshfunc_{i}\karte\skalierer_{i}) \|_{\Lebesgue^{\infty}} \right)^{k} }
        { \left( \epsilon \meshfunc(x) \skalierer_{i}(x) \right)^{\frac{n}{p}} } %
        \left\| {u} \right\|_{\Lebesgue^{p}\Lambda^{k}\left( \Ball{x}{ \epsilon \Catlaslipschitz \meshfunc(x) \skalierer_{i}(x) }{\metric}, \metric \right) }
        .
    \end{align}
    If ${u} \in \Lebesgue^{p}_{\loc}\Lambda^{k}(\Manifold)$ and $A \subseteq \Manifold$ is measurable, then 
    \begin{gather}\label{math:local-smoothing-manifold:lp-estimate}
        \| \Regu_{i}^{k} {u} \|_{\Lebesgue^{p}\Lambda^{k}(A)}
        \leq
        \Catlaslipschitz^{2k+2\frac{n}{p}}
        \left( 1 + \epsilon \| \nabla( \karte\meshfunc_{i} \karte\skalierer_{i} ) \|_{\Lebesgue^{\infty}( \widekarte\Domain_{i} ) } \right)^{k}
        \left( 1 - \epsilon \| \nabla( \karte\meshfunc_{i} \karte\skalierer_{i} ) \|_{\Lebesgue^{\infty}( \widekarte\Domain_{i} ) } \right)^{-\frac{1}{p}}
        \| {u} \|_{\Lebesgue^{p}\Lambda^{k}\left(\Umfeld{A}{\epsilon\Catlaslipschitz\meshfunc\skalierer_{i}}{\metric},\metric\right)}
    \end{gather}
    provided that the right-hand side is finite.
\end{proposition}
\begin{proof}
    These results follow from a combination of Proposition~\ref{proposition:bounded-regularizer}, Lemma~\ref{lemma:smoothing:retain-and-improve-smoothness} and Lemma~\ref{lemma:smoothing:commutativity}, together with Theorem~\ref{theorem:transformationssatz} and Theorem~\ref{theorem:derivatives-of-pullbacks}. 
    We only address the quantitative estimates explicitly. 
    
    Set $\karte{x} := \Psi_{i}(x)$. 
    The estimate~\eqref{math:local-smoothing-manifold:continuous-point-estimate} is a consequence of Theorem~\ref{theorem:transformationssatz} and Lemma~\ref{lemma:smoothing:retain-and-improve-smoothness} via 
    \begin{align*}
        &
        \| \Regu_{i}^{k} {u}_{|x} \|_{\metric}
        =
        \lVert \Psi_{i}^{\pullbackast} \Regu_{ \epsilon \karte\skalierer_{i} \karte\meshfunc_{i} }^{k} \Psi_{i}^{-\pullbackast} {u}_{|x} \rVert_{\metric}
        \leq 
        \Catlaslipschitz^{k}
        \lVert \Regu_{ \epsilon \karte \skalierer_{i} \karte\meshfunc_{i} }^{k} \Psi_{i}^{-\pullbackast} {u}_{| \karte{x} } \rVert
        \\&\quad 
        \leq 
        \Catlaslipschitz^{k}
        \tbinom{n}{k}
        \left\| \molli \right\|_{\Lebesgue^{\infty}(\bbR^{n})}
        \volume{\Ball{0}{1}{}}^{1-\frac{1}{p}}
            ( 
                1 + \epsilon \| \nabla(\karte\meshfunc_{i}\karte\skalierer_{i}) \|_{\Lebesgue^{\infty}} 
            )^{k} 
            ( 
            \epsilon 
            \karte\skalierer_{i}\left( \karte{x} \right) 
            \karte\meshfunc_{i}\left( \karte{x} \right) 
            )^{-\frac{n}{p}} 
        \| \Psi_{i}^{-\pullbackast} {u} \|_{ \Lebesgue^{p}\Lambda^{k}\left( \Ball{\karte{x}}{\epsilon\karte\meshfunc_{i}(\karte{x})\karte\skalierer_{i}(\karte{x})}{}\right) }
    \end{align*}
    and the inequality 
    \begin{align*}
        \| \Psi_{i}^{-\pullbackast} {u} \|_{ \Lebesgue^{p}\Lambda^{k}\left( \Ball{\karte{x}}{\epsilon\karte\meshfunc_{i}(\karte{x})\karte\skalierer_{i}(\karte{x})}{} \right) } 
        &
        \leq 
        \Catlaslipschitz^{k+\frac{n}{p}}
        \| {u} \|_{ \Lebesgue^{p}\Lambda^{k}\left( \Ball{x}{ \epsilon \Catlaslipschitz \meshfunc(x) \skalierer_{i}(x) }{\metric}, \metric \right) } 
        .
    \end{align*}
    Similarly, suppose that $A \subseteq \Domain_{i}$ is measurable and write $\widekarte{A} := \Psi_{i}( A \cap \Domain_{i} )$. 
    Estimate~\eqref{math:local-smoothing-manifold:lp-estimate} follows from Theorem~\ref{theorem:transformationssatz} and Proposition~\ref{proposition:bounded-regularizer} via 
    \begin{align*}
        &
        \| \Regu_{i}^{k} {u} \|_{\Lebesgue^{p}\Lambda^{k}(A,\metric)}
        \leq 
        \| \Psi_{i}^{\pullbackast} \Regu_{\epsilon\karte\skalierer_{i}\karte\meshfunc_{i}}^{k} \Psi_{i}^{-\pullbackast} {u} \|_{\Lebesgue^{p}\Lambda^{k}(A,\metric)}
        \leq 
        \Catlaslipschitz^{k+\frac{n}{p}}
        \| \Regu_{\epsilon\karte\skalierer_{i}\karte\meshfunc_{i}}^{k} \Psi_{i}^{-\pullbackast} {u} \|_{\Lebesgue^{p}\Lambda^{k}(\widekarte{A})}
        \\&\qquad 
        \leq 
        \Catlaslipschitz^{k+\frac{n}{p}}
        ( 1 + \epsilon \| \nabla(\karte\skalierer_{i}\karte\meshfunc_{i}) \|_{\Lebesgue^{\infty}(\widekarte\Domain_{i})} )^{k}
        ( 1 - \epsilon \| \nabla(\karte\skalierer_{i}\karte\meshfunc_{i}) \|_{\Lebesgue^{\infty}(\widekarte\Domain_{i})} )^{-\frac{1}{p}}
        \| \Psi_{i}^{-\pullbackast} {u} \|_{ \Lebesgue^{p}\Lambda^{k}\left( \Umfeld{\widekarte{A}}{\epsilon\karte\skalierer_{i}\karte\meshfunc_{i}}{}\right) }
    \end{align*}
    and the upper bound  
    \begin{align*}
        \| \Psi_{i}^{-\pullbackast} {u} \|_{ \Lebesgue^{p}\Lambda^{k}\left( \Umfeld{\widekarte{A}}{ \epsilon \karte\skalierer_{i} \karte\meshfunc_{i}}{} \right) } 
        &
        \leq 
        \Catlaslipschitz^{k+\frac{n}{p}}
        \| {u} \|_{ \Lebesgue^{p}\Lambda^{k}\left( \Umfeld{A}{ \epsilon \Catlaslipschitz \meshfunc \skalierer_{i} }{\metric}, \metric \right) } 
        .
    \end{align*}
    Since $\Regu_{i}^{k} {u}$ does not change ${u}$ outside $\Domain_{i}$, the desired inequality holds for any measurable $A \subseteq \Manifold$. 
    The desired estimates follow. 
\end{proof}

We define the global commuting smoothing operator as the composition of the local smoothing operators associated with the chosen coordinate charts:
\begin{align*}
    \Regu^{k} := \Regu_{\NumberOfCharts}^{k} \circ \Regu_{\NumberOfCharts-1}^{k} \circ \cdots \circ \Regu_{2}^{k} \circ \Regu_{1}^{k}
	.
\end{align*}
This composition maps locally integrable differential forms to locally integrable differential forms.
Loosely speaking, the composed operator smooths a neighborhood of every point at least once, with a radius comparable to the local mesh size, while subsequent local operators preserve the acquired smoothness.
The main properties of this operator are summarized as follows.

\begin{theorem}\label{theorem:global-smoothing}
    Let $\epsilon > 0$ be small enough.
    Let $p \in [1,\infty]$. 
    We have a linear operator 
    \begin{gather}\label{math:global-smoothing-manifold:Lploc}
        \Regu^{k} : \Lebesgue^{p}_{\loc}\Lambda^{k}(\Manifold) \to \Cont^{\infty}\Lambda^{k}(\Manifold)
        .
    \end{gather}
    We have a bounded linear operator 
    \begin{gather}\label{math:global-smoothing-manifold:Lp}
        \Regu^{k} : \Lebesgue^{p}\Lambda^{k}(\Manifold) \to \Cont^{\infty}\Lambda^{k}(\Manifold) \subseteq \Lebesgue^{p}\Lambda^{k}(\Manifold)
        .
    \end{gather}
    If ${u} \in \Sobdiff^{1,1}_{\loc}\Lambda^{k}(\Manifold)$, then 
    \begin{gather}\label{math:global-smoothing-manifold:commutativity}
        \cartan \Regu^{k} {u} = \Regu^{k+1} \cartan {u}
        .
    \end{gather}
    For every ${u} \in \Lebesgue^{p}\Lambda^{k}(\Manifold,\metric)$ and $T \in \Mesh^{n}$, 
    it holds that 
	\begin{align*}
		\| \Regu^{k}{u} \|_{\Cont^{}\Lambda^{k}(T)}
		\leq
        C_{\alpha}
        \left( 1 + C_{\beta} \epsilon \right)^{ \NumberOfCharts k }
        \left( 1 - C_{\beta} \epsilon \right)^{-(\NumberOfCharts-1) \frac{1}{p}}
        \left( \epsilon h_{T} \right)^{-\frac{n}{p}}
        \| {u} \|_{\Lebesgue^{p}\Lambda^{k}(T^{\patchstar})},
	\end{align*}
    where 
    \begin{align*}
        C_{\alpha} 
        &
        := 
        \Catlaslipschitz^{ 2 \NumberOfCharts k + ( 2 \NumberOfCharts - 1 ) \frac n p } 
        \tbinom{n}{k} 
        \volume{\Ball{0}{1}{}}^{1-\frac{1}{p}} 
        \|\molli\|_{\Lebesgue^{\infty}(\bbR^{n})}
        \left( \comparisonconstant{\Mesh} \Cmeshsizevalue \right)^{ \frac{n}{p} }
        ,
        \\ 
        C_{\beta} 
        &
        := 
        \Cmeshsizevalue 
        \max_{T \in \Mesh} h_{T} \cdot 
        \max_{j \in I} \| \nabla \karte\skalierer_{j} \|_{\Lebesgue^{\infty}(\bbR^{n})} + \Catlaslipschitz \Cmeshsizelipschitz
        .
    \end{align*} 
\end{theorem}
\begin{proof}
    Let $u \in \Lebesgue^{1}_{\loc}\Lambda^{k}(\Manifold)$. 
    For every $x \in \Manifold$, there exists $1\leq i \leq \NumberOfCharts$ such that $\skalierer_{i}(x) = 1$.
    It then follows that $x$ has an open neighborhood $\subdomain \subseteq \Manifold$ such that
    \begin{align*}
        \Regu_{i}^{k} \Regu_{i-1}^{k} \dots \Regu_{1}^{k} {u} \in \Cont^{\infty}\Lambda^{k}(\subdomain).
    \end{align*}
    Since every subsequent local smoothing operator preserves smoothness, and the choice of $x$ was arbitrary, we deduce that $\Regu^{k} {u} \in \Cont^{\infty}\Lambda^{k}(\Manifold)$.
    Successive application of Proposition~\ref{proposition:local-smoothing-manifold} yields~\eqref{math:global-smoothing-manifold:Lploc},~\eqref{math:global-smoothing-manifold:commutativity}, and~\eqref{math:global-smoothing-manifold:Lp}.

    We show the uniform operator norm bound step by step. 
    We define the open cover of $\Manifold$ by 
	\begin{align*}
        \Inside_{i} := \interior{ \left\{\; x \in {\Domain}_{i} \suchthat* \skalierer_{i}(x) = 1 \;\right\} }, 
        \quad 
        i \in I.
	\end{align*}
    Let $\delta > 0$ be the Lebesgue radius of the open cover $(\Inside_{i})_{i \in I}$ on the metric space $\Manifold$,
    meaning that for each $x \in \Manifold$ there exists $i \in I$ satisfying $\Ball{x}{\delta}{\metric} \subseteq \Inside_{i}$. 
	For the remainder of the proof, fix $T \in \Mesh^{n}$ and define 
    \begin{align*}
        A_{i} 
        &
        := 
        \left\{ x \in T \suchthat* \Ball{x}{\delta}{\metric} \subseteq \Inside_{i} \right\}, 
        \quad 
        i \in I
        .
    \end{align*}
    Note that $T$ is the union of the $A_{i}$, that each $A_{i}$ is measurable, and that 
    \begin{align*}
        A_{i} 
        = 
        \left\{ z \in T \suchthat* \dist( z, \Manifold \setminus \Inside_{i} ) \geq \delta \right\}
        .
    \end{align*}

	Consider now any of the sets $A_{i}$.
	We begin with $G_{\NumberOfCharts+1} := A_{i}$ and then define recursively
    \begin{align*} 
        G_{\NumberOfCharts} &:= \Umfeld {G_{\NumberOfCharts+1}}{ \Catlaslipschitz \epsilon \skalierer_{\NumberOfCharts}\meshfunc}{\metric},
        \\
        \vdots 
        \\
        G_{j}               &:= \Umfeld{ G_{j+1} }{ \Catlaslipschitz\epsilon\skalierer_{j}\meshfunc }{\metric},
        \\
        \vdots 
        \\
        G_{1}               &:= \Umfeld{ G_{2} }{ \Catlaslipschitz\epsilon\skalierer_{1}\meshfunc }{\metric}. 
	\end{align*} 
    We estimate the sizes of the sets in this sequence recursively.
    
	Recall that $\neighborconstant > 0$ is the neighborhood parameter given by Lemma~\ref{lemma:neighborhood}, so that $\Ball{T}{\neighborconstant h_{T}}{\metric} \subseteq T^{\patchstar}$ for all $T \in \Mesh$.
    By definition, $\meshfunc \leq \comparisonconstant{\Mesh} \Cmeshsizevalue h_{T}$ on $T^{\patchstar}$.
    Using $\skalierer_{\NumberOfCharts} \leq 1$ and the mesh-size bound: 
	\begin{align*}
		G_{\NumberOfCharts}
		&
		=
		\Umfeld{A_{i}}{\Catlaslipschitz \epsilon\skalierer_{\NumberOfCharts}\meshfunc}{\metric}
		\subseteq
		\Ball{A_{i}}{\Catlaslipschitz \epsilon \Cmeshsizevalue h_{T}}{\metric}
		.
	\end{align*}
	In particular, $G_{\NumberOfCharts} \subseteq T^{\patchstar}$ provided that 
    \begin{gather}\label{math:epsilon-global-smoothing:eins}
         \Catlaslipschitz \Cmeshsizevalue \epsilon \leq \neighborconstant
         .
    \end{gather}

    We already know that $G_{\NumberOfCharts} \subseteq T^{\patchstar}$ if $\Catlaslipschitz \epsilon \comparisonconstant{\Mesh} \Cmeshsizevalue \leq \neighborconstant$.
    Suppose that, for some $1 < j \leq \NumberOfCharts$, we have shown that the following implication holds:
	if $(\NumberOfCharts-j+1) \Catlaslipschitz \epsilon \comparisonconstant{\Mesh} \Cmeshsizevalue \leq \neighborconstant$, then 
	\begin{align*}
		G_{j}
		\subseteq
		\Ball{ A_{i} }{ (\NumberOfCharts-j+1) \Catlaslipschitz \epsilon \comparisonconstant{\Mesh} \Cmeshsizevalue h_{T} }{\metric}
		\subseteq
		T^{\patchstar}
		.
	\end{align*}
	It then follows that, using $\skalierer_{j-1} \leq 1$, 
	\begin{align*}
		G_{j-1}
		 &
		=
		\Umfeld{ G_{j} }{ \Catlaslipschitz \epsilon\skalierer_{j-1}\meshfunc }{\metric}
		\\&
		\subseteq
        \Umfeld{ 
            \Ball{ A_{i} }{ 
                (\NumberOfCharts-j+1) \Catlaslipschitz \epsilon \comparisonconstant{\Mesh} \Cmeshsizevalue h_{T} 
            }{\metric}
        }{ \Catlaslipschitz \epsilon\skalierer_{j-1}\meshfunc }{\metric}
		\\&
		\subseteq
        \Ball{ 
            \Ball{ A_{i} }{ 
                (\NumberOfCharts-j+1) \Catlaslipschitz \epsilon \comparisonconstant{\Mesh} \Cmeshsizevalue h_{T} 
            }{\metric}
        }{ \Catlaslipschitz \epsilon \comparisonconstant{\Mesh} \Cmeshsizevalue h_{T} }{\metric}
		\\&
		\subseteq
		\Ball{ A_{i} }{ (\NumberOfCharts-(j-1)+1) \Catlaslipschitz \epsilon \comparisonconstant{\Mesh} \Cmeshsizevalue h_{T} }{\metric}
		.
	\end{align*}
	In particular,
	if $(\NumberOfCharts-(j-1)+1) \Catlaslipschitz \epsilon \comparisonconstant{\Mesh} \Cmeshsizevalue \leq \neighborconstant$, then 
	\begin{align*}
		G_{j-1}
		\subseteq
		\Ball{ A_{i} }{ (\NumberOfCharts-(j-1)+1) \Catlaslipschitz \epsilon \comparisonconstant{\Mesh} \Cmeshsizevalue h_{T} }{\metric}
		\subseteq
		T^{\patchstar}
		.
	\end{align*}
	To summarize, if 
    \begin{gather}\label{math:epsilon-global-smoothing:zwei}
        \NumberOfCharts \Catlaslipschitz \epsilon \comparisonconstant{\Mesh} \Cmeshsizevalue \leq \neighborconstant,      
    \end{gather}
    then 
	\begin{align*}
		A_{i} \subseteq G_{\NumberOfCharts} \subseteq \dots \subseteq G_{1}
        \subseteq 
        \Ball{ A_{i} }{ \NumberOfCharts \Catlaslipschitz \epsilon \comparisonconstant{\Mesh} \Cmeshsizevalue h_{T} }{\metric} 
        \subseteq 
        T^{\patchstar}
        .
	\end{align*}
	In addition, we choose $\epsilon > 0$ so small that 
    \begin{gather}\label{math:epsilon-global-smoothing:drei}
        \NumberOfCharts \Catlaslipschitz \epsilon \comparisonconstant{\Mesh} \Cmeshsizevalue \cdot \max_{S \in \Mesh} h_{S} \leq \delta,
    \end{gather}
    which ensures $G_{1} \subseteq \Ball{A_{i}}{\delta}{\metric} \subseteq \Inside_{i}$. 
    We henceforth assume that $\epsilon$ is so small that these conditions hold. 
    This permits a choice of $\epsilon > 0$ that is independent of the simplex.
	    
    We now prove the norm estimates when $u \in \Lebesgue^{p}\Lambda^{k}(\Manifold)$. 
    First, repeated application of~\eqref{math:local-smoothing-manifold:lp-estimate} with $p=\infty$, together with preservation of continuity, shows that for $i+1 \leq j \leq \NumberOfCharts$,
    \begin{align}\label{math:global-smoothing:intermediate-pointwise-estimates}
	    &
		\| \Regu^{k}_{j} \cdots \Regu^{k}_{2} \Regu^{k}_{1}{u} \|
		_{\Cont\Lambda^{k}(G_{j+1})}
        \leq
        \Catlaslipschitz^{2k}
		\left( 1 + \epsilon \| \nabla(\karte\skalierer_{j}\karte\meshfunc_{j}) \|_{\Lebesgue^{\infty}(\widekarte\Domain_{j})} \right)^{k}
		\| \Regu^{k}_{j-1} \cdots \Regu^{k}_{2} \Regu^{k}_{1}{u} \|
		_{\Cont\Lambda^{k}(G_{j})}
			.
	\end{align}
	We have $G_{i+1} \subseteq \Inside_{i} \cap T^{\patchstar}$. Therefore, 
    \[
        \forall x \in G_{i+1} : \skalierer_{i}(x) = 1
        ,
        \qquad 
        \forall x \in G_{i+1} : \meshfunc(x) \geq \Cmeshsizevalue^{-1} \comparisonconstant{\Mesh}^{-1} h_{T}
        . 
    \]
    Estimate~\eqref{math:local-smoothing-manifold:continuous-point-estimate} leads to 
	\begin{align}
        \begin{split}
        &
		\| \Regu^{k}_{i} \cdots \Regu^{k}_{2} \Regu^{k}_{1}{u} \|_{\Cont^{}\Lambda^{k}( G_{i+1} )}
		\\&
        \leq
		\tbinom{n}{k}
		\Catlaslipschitz^{2k+\frac n p}
        \volume{\Ball{0}{1}{}}^{1-\frac{1}{p}}
        \|\molli\|_{\Lebesgue^{\infty}(\bbR^{n})}
        \frac{
            ( 1 + \epsilon \| \nabla(\karte\skalierer_{i}\karte\meshfunc_{i}) \|_{\Lebesgue^{\infty}(\widekarte\Domain_{i})} )^{k}
        }{
		    ( \epsilon \comparisonconstant{\Mesh}^{-1} \Cmeshsizevalue^{-1} h_{T} )^{\frac{n}{p}}
        }
		\| \Regu^{k}_{i-1} \cdots \Regu^{k}_{2} \Regu^{k}_{1}{u} \|_{\Lebesgue^{p}\Lambda^{k}( G_{i} )}
        .
        \end{split} \label{math:global-smoothing:intermediate-smoothing}
	\end{align}
	Another sequence of estimates is based on Proposition~\ref{proposition:local-smoothing-manifold}. 
    For $1 \leq j \leq i-1$, we have 
	\begin{align} 
		\begin{split}
        &
		\| \Regu^{k}_{j} \cdots \Regu^{k}_{2} \Regu^{k}_{1}{u} \|
		_{\Lebesgue^{p}\Lambda^{k}(G_{j+1})}
		\leq
        \Catlaslipschitz^{2k+2\frac{n}{p}}
		\frac{ 
            ( 1 + \epsilon \| \nabla(\karte\skalierer_{j}\karte\meshfunc_{j}) \|_{\Lebesgue^{\infty}(\bbR^{n})} )^{k}
		}{
            ( 1 - \epsilon \| \nabla(\karte\skalierer_{j}\karte\meshfunc_{j}) \|_{\Lebesgue^{\infty}(\bbR^{n})} )^{\frac{1}{p}}
        }
		\| \Regu^{k}_{j-1} \cdots \Regu^{k}_{2} \Regu^{k}_{1}{u} \|
		_{\Lebesgue^{p}\Lambda^{k}(G_{j})}
		.
		\end{split} \label{math:global-smoothing:intermediate-lp-estimate}
	\end{align}
    We get $\epsilon \|\nabla(\karte\skalierer_{j}\karte\meshfunc_{j})\|_{\Lebesgue^{\infty}} < 1$ due to~\eqref{math:epsilon-conditions-for-local-smoothing-operators}.
    Explicitly, 
    \begin{align*}
		\| \nabla(\karte\skalierer_{j}\karte\meshfunc_{j}) \|_{\Lebesgue^{\infty}(\bbR^{n})}
		&
		\leq
		\Cmeshsizevalue \max_{T \in \Mesh} h_{T} %
		\| \nabla \karte\skalierer_{j} \|_{\Lebesgue^{\infty}(\bbR^{n})}
		+
		\| \karte\skalierer_{j} \|_{\Lebesgue^{\infty}(\bbR^{n})}
		\Catlaslipschitz \Cmeshsizelipschitz
		\\&
		\leq
		\Cmeshsizevalue \max_{T \in \Mesh} h_{T} %
		\| \nabla \karte\skalierer_{j} \|_{\Lebesgue^{\infty}(\bbR^{n})}
		+
		\Catlaslipschitz \Cmeshsizelipschitz
		.
	\end{align*}
    Therefore, $C_{\beta} \epsilon < 1$ is a sufficient condition for the bound on the gradient. 
    
    Recall that $T = A_{1} \cup \dots \cup A_{\NumberOfCharts}$. Obviously,
	\begin{align*}
		\| \Regu^{k} {u} \|_{\Cont^{}\Lambda^{k}(T)}
        & 
        \leq
        \max_{ \substack{ 1 \leq i \leq \NumberOfCharts } } %
        \| \Regu^{k}{u} \|_{\Cont\Lambda^{k}(A_{i})}
		.
	\end{align*}
    Note that there are most $\NumberOfCharts$ pointwise estimates~\eqref{math:global-smoothing:intermediate-pointwise-estimates}, one application of~\eqref{math:global-smoothing:intermediate-smoothing}, and at most $\NumberOfCharts$ applications of~\eqref{math:global-smoothing:intermediate-lp-estimate}.
	The combination of all these inequalities implies the desired result. 
    The proof is complete. 
\end{proof}

\begin{remark}
    We summarize the conditions on $\epsilon > 0$ that we use for Theorem~\ref{theorem:global-smoothing}. 
    The individual local smoothing operators require~\eqref{math:epsilon-conditions-for-local-smoothing-operators}, so that Proposition~\ref{proposition:local-smoothing-manifold} holds. 
    Moreover,~\eqref{math:epsilon-global-smoothing:zwei} and~\eqref{math:epsilon-global-smoothing:drei} hold are required for Theorem~\ref{theorem:global-smoothing}. 
\end{remark} %

\subsection{Commuting quasi-interpolation}

Since we have fixed a smooth triangulation $\Mesh$ of the manifold $\Manifold$, we introduce a finite element de~Rham complex of the form~\eqref{math:finite-element-de-rham-complex} discussed earlier,
\begin{gather*}
	\begin{CD}
		\calP\Lambda^{0}(\Mesh)
		@>\cartan>>
		\calP\Lambda^{1}(\Mesh)
		@>\cartan>>
		\dots
		@>\cartan>>
		\calP\Lambda^{n}(\Mesh)
		,
	\end{CD}
\end{gather*}
composed of derivatives in~\eqref{math:finite-element-arrows} such that each term contains the Whitney forms.
We recall that we have canonical interpolants
\begin{gather*}
	\Interpolant_{\Mesh}^{k} : \Cont\Lambda^{k}(\Manifold) \to \calP\Lambda^{k}(\Mesh)
\end{gather*}
that commute with the exterior derivative on the smooth de~Rham complex~\eqref{math:canonical-interpolant-is-cochain-map}.
We define the \emph{smoothed quasi-interpolant}
\begin{align*}
	\smoothedinterpol^{k} : \Lebesgue^{1}_{\loc}\Lambda^{k}(\Manifold) \to \calP\Lambda^{k}(\Mesh)
\end{align*}
as the composition 
\begin{align*}
    \smoothedinterpol^{k} := \Interpolant_{\Mesh}^{k} \circ \Regu^{k}. 
\end{align*}
The properties of the smoothed quasi-interpolant are summarized in the following result.

\begin{theorem}
    Assume $\epsilon > 0$ is small enough as in Theorem~\ref{theorem:global-smoothing}. 
	Let $p \in [1,\infty]$. 
    We have a linear operator
	\begin{align*}
		\smoothedinterpol^{k} : \Lebesgue^{p}_{\loc}\Lambda^{k}(\Manifold) \to \calP\Lambda^{k}(\Mesh).
	\end{align*}
    We have a bounded operator 
	\begin{align*}
		\smoothedinterpol^{k} : \Lebesgue^{p}\Lambda^{k}(\Manifold,\metric) \to \calP\Lambda^{k}(\Mesh).
	\end{align*}
    If ${u} \in \Sobdiff^{1,1}_{\loc}\Lambda^{k}(\Manifold)$, then
    \begin{align*}
		\cartan \smoothedinterpol^{k} {u} = \smoothedinterpol^{k+1} \cartan {u}.
	\end{align*}
    Assume that $\calP\Lambda^{k}(\Mesh) = \calP_{r}\Lambda^{k}(\Mesh)$ or $\calP\Lambda^{k}(\Mesh) = \calP_{r}^{-}\Lambda^{k}(\Mesh)$. 
    If ${u} \in \Lebesgue^{p}\Lambda^{k}(\Manifold,\metric)$ and $T \in \Mesh^{n}$, then
    \begin{align*}
		\| \smoothedinterpol^{k}{u} \|_{\Cont\Lambda^{k}(T)}
		&\leq
        \Cinterpol{n}{k}{r} 
        \shapemeasure{\Mesh}^{2k} 
        C_{\alpha}
        \left( 1 + C_{\beta} \epsilon \right)^{ \NumberOfCharts k }
        \left( 1 - C_{\beta} \epsilon \right)^{-\frac{\NumberOfCharts-1}{p}}
		\left( \epsilon h_{T} \right)^{-\frac{n}{p}}
        \| {u} \|_{\Lebesgue^{p}\Lambda^{k}(T^{\patchstar})}
        ,
        \\
		\| \smoothedinterpol^{k}{u} \|_{\Lebesgue^{p}\Lambda^{k}(T)}
		&\leq
        \Cinterpol{n}{k}{r} 
        \shapemeasure{\Mesh}^{2k}
		C_{\alpha}
        \left( 1 + C_{\beta} \epsilon \right)^{ \NumberOfCharts k }
        \left( 1 - C_{\beta} \epsilon \right)^{-\frac{\NumberOfCharts-1}{p}}
		\epsilon^{-\frac{n}{p}}
        \left( \frac{\shapemeasuredeterminant{\Mesh}}{n!} \right)^{\frac{1}{p}}
        \| {u} \|_{\Lebesgue^{p}\Lambda^{k}(T^{\patchstar})}
        ,
	\end{align*}
    where the constants $C_{\alpha}$ and $C_{\beta}$ are defined in Theorem~\ref{theorem:global-smoothing}.
\end{theorem}

\begin{proof}
	That $\smoothedinterpol^{k}$ commutes with the exterior derivative follows by 
    \begin{align*}
        \cartan \smoothedinterpol^{k} u
        &=
        \cartan \Interpolant_{\Mesh}^{k} \Regu^{k} u
        =
        \Interpolant_{\Mesh}^{k+1} \cartan \Regu^{k} u
        =
        \Interpolant_{\Mesh}^{k+1} \Regu^{k+1} \cartan u
        =
        \smoothedinterpol^{k+1} \cartan u.
    \end{align*}    
    Now let ${u} \in \Lebesgue^{p}\Lambda^{k}(\Manifold,\metric)$ and $T \in \Mesh^{n}$. 
    The local $\Lebesgue^{p}$ and supremum norms satisfy
    \begin{align*}
		\| \smoothedinterpol^{k} {u} \|_{\Lebesgue^{p}\Lambda^{k}(T)}
		\leq
        \volume_{\metric}(T)^{\frac{1}{p}}
		\| \smoothedinterpol^{k} {u} \|_{\Cont\Lambda^{k}(T)}
		\leq 
        \shapemeasuredeterminant{\Mesh}^{\frac{1}{p}}
        (n!)^{-\frac{1}{p}}
        h_{T}^{\frac n p}
		\| \smoothedinterpol^{k} {u} \|_{\Cont\Lambda^{k}(T)}
		.
	\end{align*}
    Assume that $\calP\Lambda^{k}(\Mesh) = \calP_{r}\Lambda^{k}(\Mesh)$ or $\calP\Lambda^{k}(\Mesh) = \calP_{r}^{-}\Lambda^{k}(\Mesh)$. 
    The local stability estimate for the canonical interpolant in Theorem~\ref{theorem:canonical-interpolant-bound} gives
	\begin{align*}
		\| \smoothedinterpol^{k} {u} \|_{\Cont\Lambda^{k}(T)}
		=
		\| \Interpolant_{\Mesh}^{k} \Regu^{k} {u} \|_{\Cont\Lambda^{k}(T)}
		\leq
		\Cinterpol{n}{k}{r} 
        \shapemeasure{\Mesh}^{2k}
		\| \Regu^{k} {u} \|_{\Cont\Lambda^{k}(T)}.
	\end{align*}
	As we assume that $\epsilon > 0$ is small enough, Theorem~\ref{theorem:global-smoothing} applies. 
    The desired estimates follow.
    Together with the uniformly finite overlap of patches, the global boundedness is apparent.
\end{proof}

\subsection{Projection property} 

The smoothed quasi-interpolant is uniformly bounded and commutes with the exterior derivative, but it is generally not a projection onto the finite element space.
Though its restriction to the finite element space is not the identity, its deviation from the identity on that space can be controlled.

The following theorem is based on the ideas in~\cite[Proposition~5.2]{christiansen2008smoothed} but substantially adapted for the non-Euclidean setting. 
We use the Sch\"oberl trick~\cite{Schoeberl2005multilevel}, which modifies the smoothed quasi-interpolant and leads to a smoothed projection, and the proof includes the complete technical details.

\begin{theorem}\label{theorem:fehlerschranke}
	Assume that $\calP\Lambda^{k}(\Mesh)$ is $\calP_{r}\Lambda^{k}(\Mesh)$ or $\calP_{r}^{-}\Lambda^{k}(\Mesh)$,
    and that the uniformly bi-Lipschitz atlas $\calA$ is uniformly $\Cont^{2}$. 
    
    If $\epsilon > 0$ is small enough, then for all $p \in [1,\infty]$ and ${u}_{h} \in \calP\Lambda^{k}(\Mesh)$ we have
	\begin{align*}
		\| \smoothedinterpol^{k} {u}_{h} - {u}_{h} \|_{\Lebesgue^{p}\Lambda^{k}(\Manifold,\metric)} 
        \leq 
        \frac 1 2 
        \| {u}_{h} \|_{\Lebesgue^{p}\Lambda^{k}(\Manifold,\metric)}
		.
	\end{align*}
	Here, the bound on $\epsilon > 0$ depends only on $\Manifold$, the atlas $\calA$, the bump system, the polynomial degree, and the regularity of $\Mesh$. 
\end{theorem}

\begin{proof} 
	Fix an $n$-dimensional simplex $T \in \Mesh^{n}$.
    Let $\rho > 0$ be as in Theorem~\ref{theorem:local-neighborhoods}. We recall the embedding 
    \[
        \patchincl_{T} \colon \Ball{\referencesimplex_{n}}{\rho}{} \to T^{\patchstar} \subset \Manifold.
    \]
    In particular, $\patchincl_{T}$ agrees with $\incl_{T}$ on $\referencesimplex_{n}$. 
    We use the following notation. 
    Write 
    \[
        \widehat\subdomain_{T} := \Ball{\referencesimplex_{n}}{\rho}{}
        , 
        \qquad 
        \subdomain_{T} := \patchincl_{T}\left( \Ball{\referencesimplex_{n}}{\rho}{} \right)
        .
    \]
    For each $S \in \Mesh^{n}$ with $S \cap T \neq \emptyset$, we define the sets 
    \begin{align*}
        \widehat\subdomain_{T,S} := \patchincl_{T}^{-1}( S \cap \subdomain_{T} )
        , 
        \qquad 
        \subdomain_{T,S} := S \cap \subdomain_{T}
        .
    \end{align*} 
    For each $i \in I$ we write 
    \begin{align*}
        \widehat\Domain_{i} := \patchincl_{T}^{-1}( \Domain_{i} \cap \subdomain_{T} ).
    \end{align*}
    For each $F \in \Mesh$ with $F \subseteq T$, write 
    \begin{gather*}
        \hatF := \incl_{T}^{-1}(F) \subseteq \referencesimplex_{n}.
    \end{gather*}
    See also Figure~\ref{figure:setting-in-fehlerschranke} for an illustration.
    The proof is completed in several stages. 

    Fix $u_{h} \in \calP\Lambda^{k}(\Mesh)$. 
    First, we transform the error into a reference configuration. 
    By a scaling argument based on Theorem~\ref{theorem:transformationssatz}, the definitions~\eqref{math:shapemeasure} and~\eqref{math:shapemeasuredeterminant}, and standard Lebesgue space embeddings,
    \begin{align*}
		\| {u}_{h} - \smoothedinterpol^{k} {u}_{h} \|_{\Lebesgue^{p}\Lambda^{k}(T,\metric)}
		&
        \leq
		\| \Jacobian\incl_{T}^{-1} \|_{\Lebesgue^{\infty}}^{k}
        \| \determinant \Jacobian(\incl_{T}) \|_{\Lebesgue^{\infty}(\referencesimplex_{n})}^{\frac{1}{p}}
        \| \patchincl_{T}^{\pullbackast} {u}_{h} - \patchincl_{T}^{\pullbackast} \smoothedinterpol^{k} {u}_{h} \|_{\Lebesgue^{p}\Lambda^{k}(\referencesimplex_{n})}
		\\& 
        \leq
		\shapemeasure{\incl_{T}}^{k}
        \shapemeasuredeterminant{\incl_{T}}^{\frac{1}{p}}
        h_{T}^{-k+\frac n p}
        \left( n! \right)^{-\frac{1}{p}}
		\| \patchincl_{T}^{\pullbackast} {u}_{h} - \patchincl_{T}^{\pullbackast} \smoothedinterpol^{k} {u}_{h} \|_{\Lebesgue^{\infty}\Lambda^{k}(\referencesimplex_{n})}
        .
	\end{align*}
    Next, we rewrite the error locally in terms of the degrees of freedom on the reference simplex. 
    On the $n$-dimensional reference simplex, we can use the canonical interpolant:
    \begin{align*}
        ( \patchincl_{T}^{\pullbackast} \smoothedinterpol^{k} {u}_{h} ) \restriction_{\referencesimplex_{n}}
        =
        \incl_{T}^{\pullbackast} \smoothedinterpol^{k} {u}_{h}
        =
        \incl_{T}^{\pullbackast} \Interpolant_{\Mesh}^{k} \Regu^{k} {u}_{h}
        =
        \Interpolant_{\referencesimplex_{n}}^{k} \incl_{T}^{\pullbackast} \Regu^{k} {u}_{h}
        =
        \Interpolant_{\referencesimplex_{n}}^{k}
        \left(
            \patchincl_{T}^{\pullbackast} \Regu^{k} {u}_{h} 
        \right)
        .
    \end{align*} 
    Recall that $\calP\Lambda^{k}(\Mesh) \in \left\{ \calP_{r}\Lambda^{k}(\Mesh), \calP_{r}^{-}\Lambda^{k}(\Mesh) \right\}$. 
    By the definition of the interpolant in terms of degrees of freedom~\eqref{math:degrees-of-freedom:full} and~\eqref{math:degrees-of-freedom:trimmed}, we have the following upper bound:\footnote{For the full space, $\calP^-_{r+k-\dim(F)}\Lambda^{\dim(F)-k}(\hatF)$ is precisely the canonical test space. For the trimmed space, its canonical test space $\calP_{r+k-\dim(F)-1}\Lambda^{\dim(F)-k}(\hatF)$ is contained in this space, so enlarging the supremum is legitimate.}
    \begin{align}\label{math:fehler-estimated-by-dof}
        \left\| 
            \patchincl_{T}^{\pullbackast} {u}_{h} 
            - 
            \Interpolant_{\referencesimplex_{n}}^{k}
            \patchincl_{T}^{\pullbackast} \Regu^{k} {u}_{h} 
        \right\|_{\Cont\Lambda^{k}(\referencesimplex_{n})}
		\leq
		C
		\sup_{ \substack{ 
            F \in \Mesh, \; F \subseteq T,
            \\
            \eta \in \calP^{-}_{r+k-\dim(F)}\Lambda^{\dim(F) - k}(\hatF)
            \\ 
            \| \eta \|_{\Lebesgue^{1}\Lambda^{\dim(F) - k}(\hatF)} \leq 1 
        } }
        \left\lvert 
            \int_{\hatF} \trace_{\referencesimplex_{n},\hatF}( \patchincl_{T}^{\pullbackast} {u}_{h} - \patchincl_{T}^{\pullbackast} \Regu^{k} {u}_{h} ) \wedge \eta 
        \right\rvert
        ,
	\end{align}
    where the constant only depends on $n$ and $r$. 
    Since $\Regu^{k} u_{h}$ is smooth on $\Manifold$, we know that $\patchincl_{T}^{\pullbackast} \Regu^{k} u_{h}$ is piecewise smooth and that the trace of $\patchincl_{T}^{\pullbackast} {u}_{h} - \patchincl_{T}^{\pullbackast} \Regu^{k} {u}_{h}$ on any subsimplex of $\referencesimplex_{n}$ is well-defined. 
    We want to bound the face integrals of the error $\patchincl_{T}^{\pullbackast} {u}_{h} - \patchincl_{T}^{\pullbackast} \Regu^{k} {u}_{h}$ against the test forms of the degrees of freedom on that face. 
    We reduce these to the essential suprema of $\patchincl_{T}^{\pullbackast} {u}_{h}$ and of its piecewise derivative in the reference configuration. 

    To proceed, we next rewrite the regularization operator in the reference configuration. 
    Abbreviate 
    \[
        \hatu_{h} := \patchincl_{T}^{\pullbackast}(u_{h})
        .
    \]
    If $\hatx \in \widehat\subdomain_{T}$ such that 
    \[
        \hatx \in \widehat\Domain_{i}
        ,
        \qquad 
        \Ball{\hatx}{\Cpatchjacobian{\Mesh} \Catlaslipschitz \Cmeshsizevalue \comparisonconstant{\Mesh} \epsilon}{} \subseteq \Ball{\referencesimplex_{n}}{\rho}{}
        ,
    \]
    then the definition of the $i$-th local smoothing operator yields 
    \begin{align*}
        ( \patchincl_{T}^{\pullbackast} \Regu^{k}_{i} u_{h} )( \hatx )
        =
        \int_{ \bbR^{n} } \molli(y) 
        ( 
            \patchincl_{T}^{\pullbackast} \Psi_{i}^{\pullbackast}
            \Phi_{\epsilon \karte\skalierer_{i} \karte\meshfunc_{i},y}^{\pullbackast} 
                \Psi_{i}^{-\pullbackast} \patchincl_{T}^{-\pullbackast} \hatu_{h} 
        )_{|\hatx} 
        \;\dif y
        ,
    \end{align*}
    and the value of $\patchincl_{T}^{\pullbackast} \Regu^{k}_{i} u_{h}$ at any such $\hatx \in \widehat\Domain_{i}$ depends
    only on the values of $\patchincl_{T}^{\pullbackast} u_{h}$ within the ball
    \[
        \Ball{\hatx}{ \Cpatchjacobian{\Mesh} \Catlaslipschitz \Cmeshsizevalue \comparisonconstant{\Mesh} \epsilon}{}.
    \]
    Indeed, any path starting at $\Psi_{i}(\patchincl_{T}(\hatx))$ of length 
    \[ 
        \epsilon \skalierer_{i}( \patchincl_{T}(\hatx) ) \meshfunc( \patchincl_{T}(\hatx) ) \|y\| \leq \Cmeshsizevalue \comparisonconstant{\Mesh} h_{T} \epsilon
    \]
    is mapped to a path from $\hatx$ of length at most $\Cpatchjacobian{\Mesh} \Catlaslipschitz \Cmeshsizevalue \comparisonconstant{\Mesh} \epsilon$
    along the locally Lipschitz mapping $\patchincl_{T}^{-1} \circ \Psi_{i}^{-1}$. 
    This allows us to rewrite the local and global regularization operators within the reference configuration. 
    For any $\hatx \in \widehat\subdomain_{T}$ with $\Ball{\hatx}{\Cpatchjacobian{\Mesh} \Catlaslipschitz \Cmeshsizevalue \comparisonconstant{\Mesh} \epsilon }{} \subseteq \widehat\subdomain_{T}$, and any $y \in \Ball{0}{1}{}$, define 
    \begin{align*}
        \Phi_{i,\epsilon,y}
        :=
        \Phi_{\epsilon \karte\skalierer_{i} \karte\meshfunc_{i},y}
        ,
        \qquad 
        \hat\Phi_{i,\epsilon,y}(\hatx) 
        := 
        \begin{cases}
            \patchincl_{T}^{-1}
            \Psi_{i}^{-1}
            \left( 
                \Psi_{i}\patchincl_{T}(\hatx) 
                +
                \epsilon 
                \skalierer_{i}(\patchincl_{T}(\hatx)) 
                \meshfunc(\patchincl_{T}(\hatx)) 
                y
            \right)
            &
            \text{ if }
            \hatx \in \widehat\Domain_{i},
            \\
            \hatx 
            &
            \text{ if }
            \hatx \notin \widehat\Domain_{i}.
        \end{cases}
    \end{align*}
    For almost all such $\hatx$, the expression $\hat\Phi_{i,\epsilon,y}(\hatx)$ is smooth in $y$ and $\hatx$ with bounded Jacobians,
    \begin{align}\label{math:jacobian-of-phi-bounds}
        \| \Jacobian_{y} \hat\Phi_{i,\epsilon,y}(\hatx) \|
        \leq 
        \Cpatchjacobian{\Mesh} 
        \Catlaslipschitz 
        \Cmeshsizevalue \comparisonconstant{\Mesh} \epsilon 
        ,
        \qquad 
        \| \Jacobian_{\hatx} \hat\Phi_{i,\epsilon,y}(\hatx) \|
        \leq 
        \Cpatchjacobian{\Mesh}^{2}
        \Catlaslipschitz^{2} 
        \left( 1 
        +
        \| \nabla( \karte\skalierer_{i} \karte\meshfunc_{i} ) \|_{\infty}
        \epsilon \right)
        .
    \end{align}
    In particular, 
    \[
        \| \hat\Phi_{i,\epsilon,y}(\hatx) - \hatx \| \leq \Cpatchjacobian{\Mesh} \Catlaslipschitz \Cmeshsizevalue \comparisonconstant{\Mesh} \epsilon 
        .
    \]
    We conclude that if $\hatx \in \widehat\subdomain_{T}$ and 
    \begin{equation}\label{math:fehlerschranke:first-containment}
        \Ball{\hatx}{ \NumberOfCharts \Cpatchjacobian{\Mesh} \Catlaslipschitz \Cmeshsizevalue \comparisonconstant{\Mesh} \epsilon }{} \subseteq \widehat\subdomain_{T}
        ,
    \end{equation}
    then we define recursively 
    \begin{align}\label{math:sequence-of-points}
        \hata_{i} 
        := 
        \hat\Phi_{i+1,\epsilon,y_{i+1}} 
        \cdots 
        \hat\Phi_{\NumberOfCharts-1,\epsilon,y_{\NumberOfCharts-1}}
        \hat\Phi_{\NumberOfCharts,\epsilon,y_{\NumberOfCharts}}
        ( \hatx )
        ,
        \qquad 
        0 \leq i \leq \NumberOfCharts
        ,
    \end{align}
    are well-defined and satisfy 
    \[
        \hata_{\NumberOfCharts} = \hatx, 
        \qquad 
        \hata_{0}, \hata_{1}, \hata_{2}, \dots, \hata_{\NumberOfCharts} 
        \in 
        \Ball{\hatx}{ \NumberOfCharts \Cpatchjacobian{\Mesh} \Catlaslipschitz \Cmeshsizevalue \comparisonconstant{\Mesh} \epsilon }{}.
    \]
    To ensure that~\eqref{math:fehlerschranke:first-containment} holds for all $\hatx \in \referencesimplex_{n}$, we henceforth assume that $\epsilon > 0$ is so small that 
    \begin{gather}\label{math:epsilon-restriction-eins}
        \NumberOfCharts \Cpatchjacobian{\Mesh} \Catlaslipschitz \Cmeshsizevalue \comparisonconstant{\Mesh} \epsilon < \rho 
        .
    \end{gather}
    Since $\patchincl_{T}$ is only piecewise smooth, we need to study the existence of Jacobians in more detail. 
	Abbreviate
	\[
		a_{i} := \patchincl_{T}(\hata_{i}), \qquad 0 \leq i \leq \NumberOfCharts.
	\]
    Recall that $\skalierer_{i}$ is smooth and non-negative. 
    In the case $\skalierer_{i}(a_{i}) = 0$, we have $a_{i-1} = a_{i}$ and $\Jacobian\hat\Phi_{i,\epsilon,y_{i}|\hata_{i}} = \Id$, as per the following argument:
    if $a_{i} \notin \Domain_{i}$, then $\Jacobian\hat\Phi_{i,\epsilon,y_{i}|\hata_{i}} = \Id$, 
    and 
    if $a_{i} \in \Domain_{i}$, then for any $\hatb \in \widehat\subdomain_{T}$ we have 
	\begin{align*}
        &
        \patchincl_{T}^{-1}\Psi_{i}^{-1} \Phi_{\epsilon \karte\skalierer_{i} \karte\meshfunc_{i},y_{i}} \Psi_{i}\patchincl_{T}(\hatb) 
        -
        \patchincl_{T}^{-1}\Psi_{i}^{-1} \Phi_{\epsilon \karte\skalierer_{i} \karte\meshfunc_{i},y_{i}} \Psi_{i}\patchincl_{T}(\hata_{i})
        -
        (\hatb-\hata_{i})
        \\&\quad 
        =
        \patchincl_{T}^{-1}\Psi_{i}^{-1}
        \left( 
            \Psi_{i}\patchincl_{T}(\hatb) 
            +
            \epsilon ( \karte\skalierer_{i} \karte\meshfunc_{i} ) ( \Psi_{i}\patchincl_{T}(\hatb) ) \cdot y_{i}
        \right)
        -
        \patchincl_{T}^{-1}\Psi_{i}^{-1} \Psi_{i}\patchincl_{T}(\hatb)
        ,
	\end{align*}
    implying for any $\hatb$ sufficiently close to $\hata_{i}$ that 
	\begin{align*}
        \lim_{ \hatb \to \hata_{i} } 
		\frac{
            \| 
                \hat\Phi_{i,\epsilon,y_{i}}(\hatb) 
				-
                \hat\Phi_{i,\epsilon,y_{i}}(\hata_{i})
				-
                (\hatb-\hata_{i})
            \|
        }{
            \| \hatb - \hata_{i} \|
        }
        \leq 
        \epsilon \Catlaslipschitz^{2} \Cpatchlipschitz{\Mesh}^{2}
        \lim_{ \hatb \to \hata_{i} }
		\frac{
            \| 
                ( \karte\skalierer_{i} \karte\meshfunc_{i} ) ( \Psi_{i}\patchincl_{T}(\hatb) )
            \|
        }{
            \| \Psi_{i}\patchincl_{T}(\hatb) - \Psi_{i}\patchincl_{T}(\hata_{i}) \|
        }
        =
        0
        ,
	\end{align*} 
    leading to $a_{i-1} = a_i$ and $\Jacobian\hat\Phi_{i,\epsilon,y_i|\hata_i}=\Id$, irrespective of whether $\patchincl_{T}$ is differentiable at $\hata_i$ or not.
	In the case $\skalierer_{i}(a_{i}) > 0$, Lemma~\ref{lemma:Phi-is-diffeomorphism} ensures that the mapping $y_{i} \mapsto \Phi_{i,\epsilon,y_{i}}(\Psi_i(a_{i}))$
	is a diffeomorphism from $\Ball{0}{1}{}$ onto its image in $\widekarte\Domain_{i}$.
    Since $\patchincl_{T}$ and its inverse are differentiable except at interfaces, 
    we know that for almost all $y_{i} \in \Ball{0}{1}{}$ the mapping $\patchincl_{T}$ is smooth in an open neighborhood of $\hata_{i-1} = \hat\Phi_{i,\epsilon,y_{i}}( \hata_{i} )$. 
    In combination, 
    we thus find that the Jacobians 
    $\Jacobian \hat\Phi_{1,\epsilon,y_{1}|\hata_{1}}, \Jacobian \hat\Phi_{2,\epsilon,y_{2}|\hata_{2}}, \dots, \Jacobian \hat\Phi_{\NumberOfCharts,\epsilon,y_{\NumberOfCharts}|\hata_{\NumberOfCharts}}$
    and their inverses exist for almost all $( y_1, y_2, \dots, y_{\NumberOfCharts} ) \in \Ball{0}{1}{}^{\NumberOfCharts}$ and almost all $\hatx \in \widehat\subdomain_{T}$ that satisfy~\eqref{math:fehlerschranke:first-containment}.
    
    Equipped with that, we write the global regularization operator in the reference configuration as 
    \begin{align}\label{math:regularization-operator-in-reference-configuration}
        \left( \patchincl_{T}^{\pullbackast} \Regu^{k} u_{h} \right)_{|\hatx} 
        =
        \underset{\bbR^{n}\times\bbR^{n}\times\cdots\times\bbR^{n}}{\int\int\cdots\int}
        \left(
            \hat\Phi_{\NumberOfCharts,\epsilon,y_{\NumberOfCharts}}^{\pullbackast}
            \cdots
            \hat\Phi_{              2,\epsilon,y_{              2}}^{\pullbackast}
            \hat\Phi_{              1,\epsilon,y_{              1}}^{\pullbackast}
            \hatu_{h}
            \right)_{|\hatx}
		\prod_{i=1}^{\NumberOfCharts} \molli(y_{i})
		\;
        \dif y_{\NumberOfCharts} \cdots \;\dif y_{2} \;\dif y_{1}
        ,
    \end{align}
    where the integrand is defined for almost all $\hatx \in \widehat\subdomain_{T}$ that satisfy~\eqref{math:fehlerschranke:first-containment}.
    Given $y_{1}, y_{2}, \dots, y_{\NumberOfCharts} \in \Ball{0}{1}{}$, abbreviate 
    \begin{gather*}
        \bfy := ( y_{1}, y_{2}, \dots, y_{\NumberOfCharts} )
        ,
        \qquad 
        \hat\Phi_{\epsilon,\bfy} := \hat\Phi_{1,\epsilon,y_{1}} \circ \hat\Phi_{2,\epsilon,y_{2}} \circ \cdots \circ \hat\Phi_{\NumberOfCharts,\epsilon,y_{\NumberOfCharts}}
        .
    \end{gather*}

	We proceed by studying the error on each individual reference subsimplex.
    Consider any $F \in \Mesh$ with $F \subseteq T$. 
    Suppose that $\hatx \in \interior{\hatF}$. 
	Since $\patchincl_{T}$ is only piecewise differentiable, 
    we again pay attention to the Jacobians of the regularization operators in the tangential directions along $\hatF$.
    
    By the definition of bump system, there exists a maximal $j \in \{1,\dots,\NumberOfCharts\}$ for which $\skalierer_{j}(a_j) > 0$. 
    On the one hand, 
    for almost all $\bfy \in \Ball{0}{1}{}^{\NumberOfCharts}$ the Jacobians $\Jacobian\hat\Phi_{1,\epsilon,y_{1}|\hata_{1}}, \dots, \Jacobian\hat\Phi_{j-1,\epsilon,y_{j-1}|\hata_{j-1}}$ exist. 
    On the other hand, 
    since we then have $\skalierer_{j+1}(a_{j+1}) = \dots = \skalierer_{\NumberOfCharts}(a_{\NumberOfCharts}) = 0$, 
    we must have $\hata_j = \dots = \hata_{\NumberOfCharts} = \hatx$ as well as 
    $\Jacobian\hat\Phi_{j+1,\epsilon,y_{j+1}|\hata_{j+1}} = \dots = \Jacobian\hat\Phi_{\NumberOfCharts,\epsilon,y_{\NumberOfCharts}|\hata_{\NumberOfCharts}} = \Id$. 
    The Jacobian $\Jacobian\hat\Phi_{j,\epsilon,y_{j}|\hata_{j}} = \Jacobian\hat\Phi_{j,\epsilon,y_{j}|\hatx}$ exists in the directions of the tangential space $\Tangential_{\hatx}\hatF$,
    and the Jacobian in those directions exists and is continuous at $\hatx$. 
    We conclude that~\eqref{math:regularization-operator-in-reference-configuration} leads to a well-defined trace representation 
    \begin{align}
        \trace_{\referencesimplex_{n},\hatF} 
        ( \patchincl_{T}^{\pullbackast} \Regu^{k} u_{h} )_{|\hatx} 
        =
        \underset{\bbR^{n}\times\bbR^{n}\times\cdots\times\bbR^{n}}{\int\int\cdots\int}
        \trace_{\referencesimplex_{n},\hatF} 
        (
            \hat\Phi_{\NumberOfCharts,\epsilon,y_{\NumberOfCharts}}^{\pullbackast}
            \cdots
            \hat\Phi_{              2,\epsilon,y_{              2}}^{\pullbackast}
            \hat\Phi_{              1,\epsilon,y_{              1}}^{\pullbackast}
            \hatu_{h}
        )_{|\hatx}
		\prod_{i=1}^{\NumberOfCharts} \molli(y_{i})
		\;
        \dif y_{\NumberOfCharts} \cdots \;\dif y_{2} \;\dif y_{1}
        .
    \end{align}
	We define
	\begin{align*}
        \delta := \delta(\epsilon) := \NumberOfCharts \Cpatchjacobian{\Mesh} \Catlaslipschitz \Cmeshsizevalue \comparisonconstant{\Mesh} \epsilon,
        \qquad 
        \hatF_{c,\delta} := \hatF \cap \closure{ \Ball{\partial\hatF}{\Cfoo{\Mesh} \delta}{} },
        \qquad 
        \hatF_{\delta} := \hatF \setminus \hatF_{c,\delta}. 
	\end{align*}
	If $\dim(F) = 0$, then by assumption on $\rho$ in Theorem~\ref{theorem:local-neighborhoods} we already have for all $S \in \Mesh^{n}$ with $S \cap T \neq \emptyset$ but $F \cap S = \emptyset$ the distance condition,
	\[
        \dist\Big( \hatF, \patchincl_{T}^{-1}( \subdomain_{T} \cap S) \Big) 
        > 
        \rho > \delta
        .
	\]
    If $\dim(F) > 0$, then Theorem~\ref{theorem:local-neighborhoods} gives us a constant $\Cfoo{\Mesh} > 0$ with the following property:
    If $\rho > \delta > 0$ and $\hatx \in \hatF$, and if $S \in \Mesh^{n}$ with $S \cap T \neq \emptyset$ and $F \nsubseteq S$, 
    then 
	\[
        \dist( \hatx, \partial\hatF ) > \Cfoo{\Mesh} \delta
        \quad
        \implies
        \quad 
        \dist\left( \hatx, \patchincl_{T}^{-1}(\subdomain_{T} \cap S) \right) 
        > 
        \delta. 
	\]
    We split the integral, 
    \begin{align*}
        \int_{\hatF} \trace_{\referencesimplex_{n},\hatF} \left( \patchincl_{T}^{\pullbackast} {u}_{h} - \patchincl_{T}^{\pullbackast} \Regu^{k} {u}_{h} \right) \wedge \eta 
        = 
        \int_{\hatF_{c,\delta}} \trace_{\referencesimplex_{n},\hatF} \left( \patchincl_{T}^{\pullbackast} {u}_{h} - \patchincl_{T}^{\pullbackast} \Regu^{k} {u}_{h} \right) \wedge \eta 
        +
        \int_{\hatF_{\delta}} \trace_{\referencesimplex_{n},\hatF} \left( \patchincl_{T}^{\pullbackast} {u}_{h} - \patchincl_{T}^{\pullbackast} \Regu^{k} {u}_{h} \right) \wedge \eta 
        ,
    \end{align*}
    and estimate the two terms separately. 

	First, we estimate the difference close to the boundary.
    If $\dim(F) = 0$, then $\hatF_{c,\delta} = \emptyset$, so the first integral vanishes, and hence we only consider the case $\dim(F) > 0$.
	We have 
	\begin{align*}
        &
        \bigl\lvert
        \int_{\hatF_{c,\delta}}
            \trace_{\referencesimplex_{n},\hatF} \left( \patchincl_{T}^{\pullbackast} {u}_{h} - \patchincl_{T}^{\pullbackast} \Regu^{k} {u}_{h} \right) \wedge \eta
        \bigr\rvert
        \\&\qquad\quad
        \leq 
        \max_{ \substack{ S \in \Mesh^{n} \\ T \cap S \neq \emptyset } }
        \esssup\limits_{ \substack{ \bfy \in \Ball{0}{1}{}^{\NumberOfCharts} } }
        \left( 1 + \| \Jacobian \hat\Phi_{\epsilon,\bfy} \restriction_{\hatF} \|_{\Lebesgue^{\infty}(\hatF_{c,\delta})}^{k} \right)
        \| \hatu_{h} \|_{\Lebesgue^{\infty}\Lambda^{k}(\widehat\subdomain_{T,S})}
        \| \eta \|_{\Lebesgue^{1}\Lambda^{\dim(F)-k}(\hatF_{c,\delta})}
        .
    \end{align*}
    Using a rough geometry estimate,
    \begin{align*}
        \| \eta \|_{\Lebesgue^{1}\Lambda^{\dim(F)-k}(\hatF_{c,\delta})}
        \leq 
        \volume( \hatF_{c,\delta} )
        \| \eta \|_{\Lebesgue^{\infty}\Lambda^{\dim(F)-k}(\hatF)}
        \leq 
        \volume( \partial\hatF ) \Cfoo{\Mesh} \delta 
        \cdot 
        \| \eta \|_{\Lebesgue^{\infty}\Lambda^{\dim(F)-k}(\hatF)}
        ,
    \end{align*}
    and an inverse inequality on the reference $\dim(F)$-simplex, we obtain 
    \begin{align*}
        \| \eta \|_{\Lebesgue^{\infty}\Lambda^{\dim(F)-k}(\hatF)}
        \leq 
        \Clebesguenormcomparison{\dim(F)}{\dim(F)-k}{r+k-\dim(F)}{1}{\infty} 
        (\dim(F)+1)^{k+n}
        \| \eta \|_{\Lebesgue^{1}\Lambda^{\dim(F)-k}(\hatF)}
        .
    \end{align*}
    Thus, the estimate for the integral over $\hatF_{c,\delta}$ follows. 

    The integral over $\hatF_{\delta}$ requires a more delicate study. 
    After translation and rotation, we choose an orthonormal coordinate system $z_{1}, \dots, z_n : \bbR^{n} \to \bbR$ such that $\hatF$ lies within the coordinate subspace of the first $\dim(F)$ coordinates:
    \[
        \hat{F} \subseteq \left\{ \; x \in \bbR^{n} \suchthat* z_{\dim(F)+1}(x) = \cdots = z_{n}(x) = 0 \; \right\}.
    \]
    We write $\hatu_{h}$ using the standard representation~\eqref{math:standardrepresentation} in these coordinates,
    \[
        \hat{u}_{h} = \sum_{\sigma \in \Sigma(k,n)} \hat{u}_{h,\sigma} \cartanz^{\sigma}, \qquad \cartanz^{\sigma} := \cartanz^{\sigma(1)} \wedge \cdots \wedge \cartanz^{\sigma(k)}.
    \]
    Since $u_{h} \in \calP\Lambda^{k}(\Mesh)$, the following holds: if $\sigma \in \Sigma(k,n)$ with $\rng\sigma \subseteq \{1,\dots,\dim(F)\}$, then $\hatu_{h,\sigma}$ is continuous at any point of the relative interior of $\hatF$.
	It follows that if $\hatx \in \hatF_{\delta}$, then we can write 
    \begin{align*} 
        & 
        \trace_{\referencesimplex_{n},\hatF} 
        ( \patchincl_{T}^{\pullbackast} {u}_{h} - \patchincl_{T}^{\pullbackast} \Regu^{k} {u}_{h} )_{|\hatx}
        \\&\quad
        =
        \sum_{ \sigma \in \Sigma(k,n) }
        \underset{\bbR^{n}\times\bbR^{n}\times\cdots\times\bbR^{n}}{\int\int\cdots\int}
        \trace_{\referencesimplex_{n},\hatF} 
	        (
	            \hatu_{h,\sigma} \cartanz^{\sigma}
	            -
	            \hat\Phi^{\pullbackast}_{\epsilon,\bfy}
	            \left(\hatu_{h,\sigma} \cartanz^{\sigma}\right)
	        )_{|\hatx}
		\prod_{i=1}^{\NumberOfCharts} \molli(y_{i})
		\;
        \dif y_{\NumberOfCharts} \cdots \;\dif y_{2} \;\dif y_{1}
        \\&\quad
        =
        \sum_{ \sigma \in \Sigma(k,\dim(F)) }
        \underset{\bbR^{n}\times\bbR^{n}\times\cdots\times\bbR^{n}}{\int\int\cdots\int}
        (
            \hatu_{h,\sigma}(\hatx)
            -
            \hatu_{h,\sigma}( \hat\Phi_{\epsilon,\bfy}(\hatx) )
        )
        \cartanz^{\sigma}
        \prod_{i=1}^{\NumberOfCharts} \molli(y_{i})
		\;
        \dif y_{\NumberOfCharts} \cdots \;\dif y_{2} \;\dif y_{1}
        \\&\qquad\qquad 
        +
        \sum_{ \sigma \in \Sigma(k,n) }
        \underset{\bbR^{n}\times\bbR^{n}\times\cdots\times\bbR^{n}}{\int\int\cdots\int}
        \hatu_{h,\sigma} ( \hat\Phi_{\epsilon,\bfy}(\hatx) )
        \trace_{\referencesimplex_{n},\hatF} 
        (
            \cartanz^{\sigma}
            -
            \hat\Phi^{\pullbackast}_{\epsilon,\bfy}
            \cartanz^{\sigma}
        )_{|\hatx}
		\prod_{i=1}^{\NumberOfCharts} \molli(y_{i})
		\;
        \dif y_{\NumberOfCharts} \cdots \;\dif y_{2} \;\dif y_{1}
        .
    \end{align*}
	We use that $\hat\Phi_{\epsilon,\bfy}(\hatx)$ is within one of the sets $\widehat\subdomain_{T,S}$ for some $S \in \Mesh^{n}$ with $F \subseteq S$.
    By definition of $\delta$, 
    \begin{align*}
        \| \hatx - \hat\Phi_{\epsilon,\bfy}(\hatx) \| < \delta
        .
    \end{align*}
    Provided that $\epsilon > 0$ is small enough, we can assume that 
    \begin{align}\label{math:epsilon-restriction-zwei}
        \Lip( \umleitung_{T,S}) \delta 
        \leq 
        \Cumleitung{\Mesh} \delta
        < 
        \rho_{\umleitung} 
        .
    \end{align}
    Therefore, 
    \[
         \| \umleitung_{T,S}(\hatx) - \umleitung_{T,S}( \hat\Phi_{\epsilon,\bfy}(\hatx) ) \| 
         \leq \Lip(\umleitung_{T,S}) \| \hatx - \hat\Phi_{\epsilon,\bfy}(\hatx) \| \leq \Lip(\umleitung_{T,S}) \delta < \rho_{\umleitung}
         ,
    \]
    and hence $\umleitung_{T,S}(\hatx)$ and $\umleitung_{T,S}( \hat\Phi_{\epsilon,\bfy}(\hatx) )$ are connected by a straight line segment in $\umleitung_{T,S}( \widehat\subdomain_{T,S} )$. 
    Transporting this line segment back along $\umleitung_{T,S}^{-1}$ gives a smooth path from $\hatx$ to $\hat\Phi_{\epsilon,\bfy}(\hatx)$ 
    contained within $\widehat\subdomain_{T,S}$ of length at most $\Lip(\umleitung_{T,S}) \Lip(\umleitung_{T,S}^{-1}) \delta$.
    Consequently, the integrals in the first sum are estimated via 
    \begin{align*}
        \left\lvert
            \hatu_{h,\sigma}(\hatx)
            -
            \hatu_{h,\sigma}( \hat\Phi_{\epsilon,\bfy}(\hatx) )
        \right\rvert
        &
        \leq 
        \max_{ \substack{ S \in \Mesh^{n}, \, S \cap T \neq \emptyset } } 
        \left\| \nabla \hatu_{h,\sigma} \restriction_{\widehat\subdomain_{T,S}} \right\|_{\Lebesgue^{\infty}\Tensor^{1}(\widehat\subdomain_{T,S})}
        \Cumleitung{\Mesh}^{2} 
        \delta
        \\& 
        \leq 
        \max_{ \substack{ S \in \Mesh^{n}, \, S \cap T \neq \emptyset } } 
        \left\| \nabla \hatu_{h} \restriction_{\widehat\subdomain_{T,S}} \right\|_{\Lebesgue^{\infty}\Tensor^{k+1}(\widehat\subdomain_{T,S})}
        \Cumleitung{\Mesh}^{2} 
        \delta
        .
    \end{align*}

    We still need to estimate the integrals in the second sum, that is, the terms $\trace_{\referencesimplex_{n},\hatF} ( \cartanz^{\sigma} - \hat\Phi^{\pullbackast}_{\epsilon,\bfy} \cartanz^{\sigma} )_{|\hatx}$ integrated with respect to $\bfy$. 
    For that, it suffices to estimate the norms of the Jacobian difference $\Id - \Jacobian \hat\Phi_{\epsilon,\bfy|\hatx}$ that exists on the tangential space $\Tangential_{\hatx} \hatF$. Indeed, a telescoping argument yields  
    \begin{align*}
        \| 
            (
                \cartanz^{\sigma}
                -
                \hat\Phi_{\epsilon,\bfy}^{\pullbackast}
                \cartanz^{\sigma}
            )_{|\hatx}
            \restriction_{\Tangential_{\hatx} \hatF}
        \|
        \leq 
        \|
            (\Id-\Jacobian\hat\Phi_{\epsilon,\bfy|\hatx})
            \restriction_{\Tangential_{\hatx}\hatF}
        \|
        \cdot 
        \sum_{i=1}^{k}
		\|
			\Jacobian\hat\Phi_{\epsilon,\bfy|\hatx}
			\restriction_{\Tangential_{\hatx}\hatF}
		\|^{i-1}
		.
    \end{align*}
    With another telescoping argument, we get 
    \begin{align*}
        \Big( 
            \Id - \Jacobian \hat\Phi_{\epsilon,\bfy|\hatx} 
        \Big) 
        \restriction_{\Tangential_{\hatx} \hatF}
        &
        =
        \Big( 
            \Id 
            - 
            \Jacobian \hat\Phi_{1,\epsilon,y_{1}|\hata_{1}}
            \Jacobian \hat\Phi_{2,\epsilon,y_{2}|\hata_{2}}
            \cdots 
            \Jacobian \hat\Phi_{\NumberOfCharts,\epsilon,y_{\NumberOfCharts}|\hata_{\NumberOfCharts}}
        \Big) 
        \restriction_{\Tangential_{\hatx} \hatF}
        \\&
        =
        \sum_{i=1}^{\NumberOfCharts}
        \Jacobian \hat\Phi_{1,\epsilon,y_{1}|\hata_{1}}
        \cdots 
        \Jacobian \hat\Phi_{i-1,\epsilon,y_{i-1}|\hata_{i-1}}
        \Big( 
            \Id 
            - 
            \Jacobian \hat\Phi_{i,\epsilon,y_{i}|\hata_{i}}
        \Big) 
        \restriction_{\Tangential_{\hatx} \hatF}
        .
    \end{align*}
    Consider any $1 \leq i \leq \NumberOfCharts$.
    In the case that either $\hata_{i} \notin \widehat\Domain_{i}$ or $\hata_{i} \in \widehat\Domain_{i}$ but $\skalierer_{i}( a_{i} ) = 0$, 
    we already know that the Jacobian difference vanishes:
    \begin{align*}
        \Id - \Jacobian \hat\Phi_{i,\epsilon,y_{i}|\hata_{i}} = 0
        .
    \end{align*}
    So we focus on the case $a_{i} \in \Domain_{i}$ and $\skalierer_{i}(a_{i}) > 0$. By the preceding almost everywhere argument, $\patchincl_{T}$ is differentiable at $\hata_{i-1}$.
    Abbreviate 
    \[
        \mychart_{i} := \Psi_{i} \circ \patchincl_{T}
        ,
		\qquad
        \Phi_{i,\epsilon,y_{i}} := \Phi_{\epsilon \karte\skalierer_{i} \karte\meshfunc_{i},y_{i}}
        .
    \]
    Thus, $\hat\Phi_{i,\epsilon,y_{i}} = \mychart_{i}^{-1} \circ \Phi_{i,\epsilon,y_{i}} \circ \mychart_{i}$.
    This allows us to calculate 
    \begin{align*}
        \Jacobian \hat\Phi_{i,\epsilon,y_{i}|\hata_{i}} \restriction_{\Tangential_{\hatx} \hatF}
        -
        \Id \restriction_{\Tangential_{\hatx} \hatF}
        &
        =
        \Jacobian \mychart^{-1}_{i| \Phi_{i,\epsilon, y_{i}} \mychart_{i}(\hata_{i}) }
        \Jacobian \Phi_{i,\epsilon, y_{i}| \mychart_{i}(\hata_{i})}
        \Jacobian \mychart_{i|\hata_{i}} \restriction_{\Tangential_{\hatx} \hatF}
        -
        \Id \restriction_{\Tangential_{\hatx} \hatF}
        \\& 
        =
        \Jacobian \mychart^{-1}_{i| \mychart_{i}(\hata_{i-1}) }
        \Jacobian \Phi_{i,\epsilon, y_{i}| \mychart_{i}(\hata_{i})}
        \Jacobian \mychart_{i|\hata_{i}} \restriction_{\Tangential_{\hatx} \hatF}
        -
        \Id \restriction_{\Tangential_{\hatx} \hatF}
        \\& 
        =
        \Jacobian \mychart^{-1}_{i| \mychart_{i}(\hata_{i-1}) }
        \Jacobian \Phi_{i,\epsilon, y_{i}| \mychart_{i}(\hata_{i})}
        \Jacobian \mychart_{i|\hata_{i}} \restriction_{\Tangential_{\hatx} \hatF}
        -
        \Jacobian \mychart^{-1}_{i|\mychart_{i}(\hata_{i-1})}
        \Jacobian \mychart_{i|\hata_{i-1}} \restriction_{\Tangential_{\hatx} \hatF}
        \\& 
        =
        \Jacobian \mychart^{-1}_{i| \mychart_{i}(\hata_{i-1}) }
        ( 
            \Jacobian \Phi_{i,\epsilon, y_{i}| \mychart_{i}(\hata_{i})}
            - 
            \Id 
        )
        \Jacobian \mychart_{i|\hata_{i}} \restriction_{\Tangential_{\hatx} \hatF}
        +
        \Jacobian \mychart^{-1}_{i| \mychart_{i}(\hata_{i-1}) }
        ( 
            \Jacobian \mychart_{i| \hata_{i  } } 
            - 
            \Jacobian \mychart_{i| \hata_{i-1} } 
        )
        \restriction_{\Tangential_{\hatx} \hatF}
        .
	\end{align*}
    Since $\Jacobian \Phi_{i,\epsilon, y_{i}| \mychart_{i}(\hata_{i})} - \Id = y_{i} \otimes \epsilon \nabla( \karte\skalierer_{i} \karte\meshfunc_{i} )_{ | \mychart_{i}(\hata_{i}) }$, we get 
    \begin{align*}
        \left\| 
            \Jacobian \mychart^{-1}_{i| \mychart_{i}(\hata_{i-1}) } 
            \left( \Jacobian \Phi_{i,\epsilon, y_{i}| \mychart_{i}(\hata_{i})} - \Id \right)
            \Jacobian \mychart_{i|\hata_{i}}
            \restriction_{\Tangential_{\hatx} \hatF}
        \right\|
        \leq 
        \epsilon 
        \Catlaslipschitz^{2} \Cpatchjacobian{\Mesh}^{2}
        \left\| \nabla( \karte\skalierer_{i} \karte\meshfunc_{i} ) \right\|_{\Lebesgue^{\infty}(\bbR^{n})}
        .
    \end{align*}
    In the case that $\mychart_{i}$ is not differentiable at $\hata_{i}$ we have $\hata_{i} = \dots = \hata_{\NumberOfCharts} = \hatx$, 
    and the Jacobian exists in the tangential directions of the face $\hatF$, so that 
    \begin{align*}
        &
        \Jacobian \mychart^{-1}_{i| \mychart_{i}(\hata_{i-1}) }
        ( 
            \Jacobian \mychart_{i| \hata_{i  } } 
            - 
            \Jacobian \mychart_{i| \hata_{i-1} } 
        )
        \restriction_{\Tangential_{\hatx} \hatF}
        =
        \Jacobian \mychart^{-1}_{i| \mychart_{i}(\hata_{i-1}) }
        ( 
            \Jacobian \mychart_{i| \hatx } 
            - 
            \Jacobian \mychart_{i| \hata_{i-1} } 
        )
        \restriction_{\Tangential_{\hatx} \hatF}
        .
    \end{align*}
    In the case that $\mychart_{i}$ is differentiable at $\hata_{i}$, we find 
    \begin{align*}
        &
        \Jacobian \mychart^{-1}_{i| \mychart_{i}(\hata_{i-1}) }
        ( 
            \Jacobian \mychart_{i| \hata_{i  } } 
            - 
            \Jacobian \mychart_{i| \hata_{i-1} } 
        )
        \restriction_{\Tangential_{\hatx} \hatF}
        \\&\quad 
        =
        \Jacobian \mychart^{-1}_{i| \mychart_{i}(\hata_{i-1}) }
        ( 
            \Jacobian \mychart_{i| \hata_{i  } } 
            - 
            \Jacobian \mychart_{i| \hatx } 
        )
        \restriction_{\Tangential_{\hatx} \hatF}
        +
        \Jacobian \mychart^{-1}_{i| \mychart_{i}(\hata_{i-1}) }
        ( 
            \Jacobian \mychart_{i| \hatx } 
            - 
            \Jacobian \mychart_{i| \hata_{i-1} } 
        )
        \restriction_{\Tangential_{\hatx} \hatF}
        .
    \end{align*}
    By assumption on $\delta$, there exist $S, S' \in \Mesh^{n}$ such that 
    \begin{align*}
        T \cap S \neq \emptyset \neq T \cap S'
        , 
        \qquad
        \hatF \subseteq \widehat\subdomain_{T,S} \cap \widehat\subdomain_{T,S'}
        ,
        \qquad 
        \hata_{i} \in \widehat\subdomain_{T,S}
        ,
        \qquad 
        \hata_{i-1} \in \widehat\subdomain_{T,S'}
        .
    \end{align*}
    Write 
    \begin{gather*}
        \subsubdomain_{T,S,i} := \subdomain_{T,S} \cap \Domain_{i},
        \qquad 
        \widekarte\subsubdomain_{T,S,i} := \Psi_{i} \subsubdomain_{T,S,i}, 
        \qquad 
        \widehat\subsubdomain_{T,S,i} := \patchincl_{T}^{-1} \subsubdomain_{T,S,i}, 
        \\
        \subsubdomain_{T,S',i} := \subdomain_{T,S'} \cap \Domain_{i},
        \qquad 
        \widekarte\subsubdomain_{T,S',i} := \Psi_{i} \subsubdomain_{T,S',i}, 
        \qquad 
        \widehat\subsubdomain_{T,S',i} := \patchincl_{T}^{-1} \subsubdomain_{T,S',i}.
    \end{gather*}
    By construction, $a_{i} \in \subsubdomain_{T,S,i}$ and $a_{i-1} \in \subsubdomain_{T,S',i}$.
    Provided that $\epsilon > 0$ is small enough, we can assume that 
    \begin{align}\label{math:epsilon-restriction-drei}
        \left( \Cpatchjacobian{\Mesh} + \shapemeasure{\Mesh} \Cumleitung{\Mesh} \right) \comparisonconstant{\Mesh} h_T \delta < \atlasradius,
    \end{align}
    where $\atlasradius$ is defined as in~\eqref{math:atlas-radius}.
    This already implies that $\hatx \in \widehat\Domain_{i}$. 
    From $\hatx \in \widehat\subdomain_{T,S} \cap \widehat\subdomain_{T,S'}$ we also get $\hatx \in \widehat\subsubdomain_{T,S,i} \cap \widehat\subsubdomain_{T,S',i}$. 
    
    We make a critical observation regarding the Jacobian of $\mychart_{i}$ at $\hatx$. 
    The mappings $\mychart_{i} \restriction_{ \widehat\subsubdomain_{T,S,i} }$ and $\mychart_{i} \restriction_{ \widehat\subsubdomain_{T,S',i} }$ are smooth, and their derivatives extend continuously to $\hatF$. Both mappings agree along $\hatF$. 
    Hence, their Jacobians agree along the face $\hatF$ in the tangential directions of that face. 

    Recall $\mychart_{i} \restriction_{ \widehat\subsubdomain_{T,S,i} } = \Psi_{i} \patchincl_{T} \restriction_{ \widehat\subsubdomain_{T,S,i} } = \Psi_{i} \incl_{S} \umleitung_{T,S}$. 
    The points $\umleitung_{T,S}(\hatx)$ and $\umleitung_{T,S}(\hata_{i})$ have distance at most $\Lip(\umleitung_{T,S}) \| \hatx - \hata_{i} \| < \rho_{\umleitung}$,
    and therefore these points are connected by a straight line segment in $\umleitung_{T,S}( \widehat\subdomain_{T,S} )$. 
    Transporting this line segment along $\incl_{S}$ 
    gives a smooth path from $\patchincl_{T}(\hatx)$ to $a_{i} = \patchincl_{T}(\hata_{i})$ 
    contained within $\subdomain_{T,S}$
    of length at most $\shapemeasure{\Mesh} \Lip(\umleitung_{T,S}) \comparisonconstant{\Mesh} h_{T} \delta$.
    Since we have chosen $\epsilon > 0$ small enough in~\eqref{math:epsilon-restriction-drei}, this path is also in $\Domain_{i}$, and hence in $\subsubdomain_{T,S,i}$. 
    Transporting the line segment along $\umleitung_{T,S}^{-1}$ gives a path within $\widehat\subsubdomain_{T,S,i}$ of length at most $\Cumleitung{\Mesh}^{2} \| \hatx - \hata_{i} \|$,
    integrating along which gives 
    \begin{align*}
        \left\|
            \Jacobian \mychart_{i|\hata_{i}} \restriction_{\Tangential_{\hatx} \hatF}
            -
            \Jacobian \mychart_{i|\hatx}     \restriction_{\Tangential_{\hatx} \hatF}
        \right\|
        &
        \leq 
        \| \nabla \Jacobian \mychart_{i} \restriction_{ \widehat\subsubdomain_{T,S,i} } \|_{\Lebesgue^{\infty}}
        \Cumleitung{\Mesh}^{2} \delta
        .
    \end{align*}
    A completely analogous argument yields 
    \begin{align*}
        \left\|
            \Jacobian \mychart_{i|\hata_{i-1}} \restriction_{\Tangential_{\hatx} \hatF}
            -
            \Jacobian \mychart_{i|\hatx}       \restriction_{\Tangential_{\hatx} \hatF}
        \right\|
        &
        \leq 
        \| \nabla \Jacobian \mychart_{i} \restriction_{ \widehat\subsubdomain_{T,S',i} } \|_{\Lebesgue^{\infty}}
        \Cumleitung{\Mesh}^{2} \delta
        .
    \end{align*}
    Several Jacobians are easily estimated:
    \begin{gather*}
        \left\| 
            \Jacobian\mychart_{i} \restriction_{ \widehat\subsubdomain_{T,S,i} }
        \right\|_{\Lebesgue^{\infty}}
        \leq 
        \Catlaslipschitz 
        \Cpatchjacobian{\Mesh}
        h_{T}
        ,
        \qquad 
        \left\| 
            \Jacobian\mychart_{i} \restriction_{ \widehat\subsubdomain_{T,S',i} }
        \right\|_{\Lebesgue^{\infty}}
        \leq 
        \Catlaslipschitz 
        \Cpatchjacobian{\Mesh}
        h_{T}
        ,
        \\ 
        \left\| 
            \Jacobian\mychart_{i}^{-1} \restriction_{ \widekarte\subsubdomain_{T,S,i} }
        \right\|_{\Lebesgue^{\infty}}
        \leq 
        \Catlaslipschitz 
        \Cpatchjacobian{\Mesh}
        h_{T}^{-1}
        ,
        \qquad 
        \left\| 
            \Jacobian\mychart_{i}^{-1} \restriction_{ \widekarte\subsubdomain_{T,S',i} }
        \right\|_{\Lebesgue^{\infty}}
        \leq 
        \Catlaslipschitz 
        \Cpatchjacobian{\Mesh}
        h_{T}^{-1}
        .
    \end{gather*}
    We compute the derivative of the Jacobian of $\mychart_{i}$ over $\widehat\subsubdomain_{T,S,i}$ and over $\widehat\subsubdomain_{T,S',i}$ in the usual manner:
    For $\hatz$ in either of these two subdomains,
    \begin{gather*}
        \nabla \Jacobian \mychart_{i|\hatz}
        =
        \nabla\left( \Jacobian\Psi_{i|\patchincl_{T}(\hatz)} \circ \Jacobian\patchincl_{T|\hatz}  \right)
        = 
        \nabla\Jacobian\Psi_{i|\patchincl_{T}(\hatz)} \left( \Jacobian\patchincl_{T|\hatz} \otimes \Jacobian\patchincl_{T|\hatz} \right)
        +
        \Jacobian\Psi_{i|\patchincl_{T}(\hatz)} \nabla\Jacobian\patchincl_{T|\hatz}
        .
    \end{gather*}
    This results in the estimates 
    \begin{align*}
        \left\| 
            \nabla \Jacobian \mychart_{i} \restriction_{\widehat\subsubdomain_{T,S,i}}
        \right\|_{\Lebesgue^{\infty}} 
        &
        \leq 
        \Cpatchjacobian{\Mesh}^{2} 
        \left\| 
            \nabla\Jacobian\Psi_{i} 
        \right\|_{\Lebesgue^{\infty}} 
        h_{T}^{2}
        + 
        \Cpatchjacobianhigher{\Mesh}{2}
        \left\| 
            \Jacobian\Psi_{i} 
        \right\|_{\Lebesgue^{\infty}}
        h_{T} 
        ,
        \\
        \left\| 
            \nabla \Jacobian \mychart_{i} \restriction_{\widehat\subsubdomain_{T,S',i}}
        \right\|_{\Lebesgue^{\infty}} 
        &
        \leq 
        \Cpatchjacobian{\Mesh}^{2} 
        \left\| 
            \nabla\Jacobian\Psi_{i} 
        \right\|_{\Lebesgue^{\infty}} 
        h_{T}^{2}
        + 
        \Cpatchjacobianhigher{\Mesh}{2}
        \left\| 
            \Jacobian\Psi_{i} 
        \right\|_{\Lebesgue^{\infty}}
        h_{T} 
        ,
    \end{align*}
    where $\Cpatchjacobianhigher{\Mesh}{2} := \highershapeconstant{2}{\Mesh} \Cumleitung{\Mesh}^{2} \comparisonconstant{\Mesh}^{2} \diam(\Manifold) + \shapemeasure{\Mesh} \Cumleitunghigher{\Mesh}{2} \comparisonconstant{\Mesh}$.
    Taking all this together, we have estimated~\eqref{math:fehler-estimated-by-dof} in terms of essential suprema of $\hatu_{h}$ and its piecewise derivative. 

    Numerous estimates for the essential suprema are in order. 
    Suppose that $S \in \Mesh^{n}$ with $T \cap S \neq \emptyset$. 
    On the one hand, the essential supremum of $\hatu_{h}$ is bounded by the following chain of inequalities,
    \begin{align*}
        \| \hatu_{h} \|_{\Lebesgue^{\infty}\Lambda^{k}(\widehat\subdomain_{T,S})}
        &\leq 
        \| \Jacobian\patchincl_{T} \|_{\Lebesgue^{\infty}}^{k}
        \| u_{h} \|_{\Lebesgue^{\infty}\Lambda^{k}(\subdomain_{T,S})}
        \\&\leq 
        \| \Jacobian\patchincl_{T} \|_{\Lebesgue^{\infty}}^{k}
        \| u_{h} \|_{\Lebesgue^{\infty}\Lambda^{k}(S)}
        \leq 
        \| \Jacobian\patchincl_{T} \|_{\Lebesgue^{\infty}}^{k}
        \| \Jacobian\incl_{S}^{-1} \|_{\Lebesgue^{\infty}}^{k} 
        \| \incl_{S}^{\pullbackast} u_{h} \|_{\Lebesgue^{\infty}\Lambda^{k}(\referencesimplex_{n})}
        \\ 
        \| \incl_{S}^{\pullbackast} u_{h} \|_{\Lebesgue^{\infty}\Lambda^{k}(\referencesimplex_{n})}
        &
        \leq 
        \Clebesguenormcomparison{n}{k}{r}{p}{\infty}
        \| \incl_{S}^{\pullbackast} u_{h} \|_{\Lebesgue^{p}\Lambda^{k}(\referencesimplex_{n})}
        \\
        \| \incl_{S}^{\pullbackast} u_{h} \|_{\Lebesgue^{p}\Lambda^{k}(\referencesimplex_{n})}
        &
        \leq 
        \| \Jacobian\incl_{S} \|_{\Lebesgue^{\infty}}^{k} 
        \| \det\Jacobian\incl_{S}^{-1} \|_{\Lebesgue^{\infty}(S)}^{\frac{1}{p}} 
        \| u_{h} \|_{\Lebesgue^{p}\Lambda^{k}(S)}
        ,
    \end{align*}
    yielding 
    \begin{align*}
        \| \hatu_{h} \|_{\Lebesgue^{\infty}\Lambda^{k}(\widehat\subdomain_{T,S})}
        &
        \leq 
        \| \Jacobian\patchincl_{T} \|_{\Lebesgue^{\infty}}^{k}
        \| \Jacobian\incl_{S}^{-1} \|_{\Lebesgue^{\infty}}^{k} 
        \Clebesguenormcomparison{n}{k}{r}{p}{\infty}
        \| \Jacobian\incl_{S} \|_{\Lebesgue^{\infty}}^{k} 
        \| \det\Jacobian\incl_{S}^{-1} \|_{\Lebesgue^{\infty}(S)}^{\frac{1}{p}} 
        \| u_{h} \|_{\Lebesgue^{p}\Lambda^{k}(S)}
        \\&
        \leq 
        \Cpatchjacobian{\Mesh}^{k} 
        \shapemeasure{\Mesh}^{2k}
        \Clebesguenormcomparison{n}{k}{r}{p}{\infty}
        \shapemeasuredeterminant{\Mesh}^{\frac{1}{p}}
        \comparisonconstant{\Mesh}^{\frac{n}{p}} h_{T}^{k-\frac{n}{p}}
        \| u_{h} \|_{\Lebesgue^{p}\Lambda^{k}(S)}.
    \end{align*}
    On the other hand, the first-order terms are estimated as follows. 
    From Theorem~\ref{theorem:local-neighborhoods} recall that 
    \begin{gather*}
        \hatu_{h} \restriction_{ \widehat\subdomain_{T,S} } 
        = 
        \umleitung_{T,S}^{\pullbackast} \incl_{S}^{\pullbackast} u_{h}
        .
    \end{gather*}
    Therefore, 
    Theorem~\ref{theorem:derivatives-of-pullbacks} yields
    \begin{align*}
        \| \nabla \hatu_{h} \|_{ \Lebesgue^{\infty}\Tensor^{k+1}(\widehat\subdomain_{T,S}) }
        &
        \leq 
        \Ctrafosatz{1}{\infty}{n}
        P_{1,0} \left( 
            \|              \Jacobian\umleitung_{T,S} \|_{\Lebesgue^{\infty}}
            ,
            \| \nabla       \Jacobian\umleitung_{T,S} \|_{\Lebesgue^{\infty}}
        \right) 
        \| \incl_{S}^{\pullbackast} u_{h} \|_{ \Lebesgue^{\infty}\Tensor^{k}(\referencesimplex_{n}) }
        \\&\qquad \qquad \qquad 
        +
        \Ctrafosatz{1}{\infty}{n}
        P_{1,1} \left( 
            \|              \Jacobian\umleitung_{T,S} \|_{\Lebesgue^{\infty}}
        \right) 
        \| \nabla ( \incl_{S}^{\pullbackast} u_{h} ) \|_{ \Lebesgue^{\infty}\Tensor^{k+1}(\referencesimplex_{n}) }
        .
    \end{align*}
    Using the inverse inequalities, we obtain
    \begin{align*}
        \| \incl_{S}^{\pullbackast} u_{h} \|_{ \Lebesgue^{\infty}\Lambda^{k}(\referencesimplex_{n}) }
        &
        \leq 
        \Clebesguenormcomparison{n}{k}{r}{p}{\infty}
        \| \incl_{S}^{\pullbackast} u_{h} \|_{ \Lebesgue^{p}\Lambda^{k}(\referencesimplex_{n}) }
        ,
        \\
        \| \nabla( \incl_{S}^{\pullbackast} u_{h} ) \|_{ \Lebesgue^{\infty}\Tensor^{k+1}(\referencesimplex_{n}) }
        &
        \leq 
        \Clebesguenormlipschitz{n}{k}{r}{p}
        \| \incl_{S}^{\pullbackast} u_{h} \|_{ \Lebesgue^{p}\Lambda^{k}(\referencesimplex_{n}) }
        ,
        \\
        \| \incl_{S}^{\pullbackast} u_{h} \|_{\Lebesgue^{p}\Lambda^{k}(\referencesimplex_{n})}
        &\leq 
        \| \Jacobian\incl_{S} \|_{\Lebesgue^{\infty}}^{k} 
        \| \det\Jacobian\incl_{S}^{-1} \|_{\Lebesgue^{\infty}(S)}^{\frac{1}{p}} 
        \| u_{h} \|_{\Lebesgue^{p}\Lambda^{k}(S)}
        .
    \end{align*}
    Abbreviate \[C' := \Ctrafosatz{1}{\infty}{n} \left( 
            P_{1,0} \left( 
                \Cumleitunghigher{\Mesh}{1}
                ,
                \Cumleitunghigher{\Mesh}{2}
            \right) 
            \Clebesguenormcomparison{n}{k}{r}{p}{\infty}
            + 
            P_{1,1} \left( 
                \Cumleitunghigher{\Mesh}{1}
            \right) 
            \Clebesguenormlipschitz{n}{k}{r}{p}
        \right).
        \] We find 
    \begin{align*}
        \| \nabla \hatu_{h} \|_{ \Lebesgue^{\infty}\Tensor^{k+1}(\widehat\subdomain_{T,S}) }
        &
        \leq 
        C' 
        \| \Jacobian\incl_{S} \|_{\Lebesgue^{\infty}}^{k} 
        \| \det\Jacobian\incl_{S}^{-1} \|_{\Lebesgue^{\infty}(S)}^{\frac{1}{p}}
        \| u_{h} \|_{\Lebesgue^{p}\Lambda^{k}(S)}
        \\&
        \leq 
        C' 
        \shapemeasure{\Mesh}^{k} 
        \shapemeasuredeterminant{\Mesh}^{\frac{1}{p}} 
        \comparisonconstant{\Mesh}^{k+\frac n p}
        h_{T}^{k-\frac n p}
        \| u_{h} \|_{\Lebesgue^{p}\Lambda^{k}(S)}
        .
    \end{align*}
    These inequalities give the remaining local bounds associated with $u_{h}$. 

    We combine these inequalities and the assumptions on $\epsilon > 0$. 
    In summary,
    \begin{align*}
		\left\lvert 
            \int_{\hatF} \trace_{\referencesimplex_{n},\hatF}( \patchincl_{T}^{\pullbackast} {u}_{h} - \patchincl_{T}^{\pullbackast} \Regu^{k} {u}_{h} ) \wedge \eta 
        \right\rvert
        \leq 
        C
        \epsilon 
        h_{T}^{k-\frac{n}{p}}
        \| {u}_{h} \|_{\Lebesgue^{p}\Lambda^{k}(T^{\patchstar},\metric)}
        \| \eta \|_{\Lebesgue^{1}\Lambda^{\dim(F)-k}(\hatF)}
        .
	\end{align*}
    Together with the scaling argument at the beginning of the proof, we obtain the local estimate 
    \begin{align}\label{math:fehlerschranke:local-error}
        \| u_h - \smoothedinterpol^{k} u_h \|_{\Lebesgue^{p}\Lambda^k(T,\metric)}
        \leq
        C_{\star} \epsilon
        \| u_h \|_{\Lebesgue^{p}\Lambda^k(T^{\patchstar},\metric)}.
    \end{align}
    We remark that constants in the inverse inequalities can be chosen independent of the Lebesgue exponent $p$, 
    and hence the constant $C_{\star}$ can be assumed bounded uniformly in $p \in [1,\infty]$. 

    We know that each $n$-simplex in $\Mesh$ is adjacent to at most $\cardinalitybound{\Mesh}$ other $n$-simplices in the triangulation. 
    Raising to the $p$-th power and summing~\eqref{math:fehlerschranke:local-error} over the $n$-simplices (or taking the maximum if $p=\infty$) and using this bounded overlap property gives
    \begin{align*}
        \| {u}_{h} - \smoothedinterpol^{k} {u}_{h} \|_{\Lebesgue^{p}\Lambda^{k}(\Manifold,\metric)}
        \leq 
        C_{\star} \epsilon 
        \cardinalitybound{\Mesh}^{\frac{1}{p}}
        \| {u}_{h} \|_{\Lebesgue^{p}\Lambda^{k}(\Manifold,\metric)}
        .
    \end{align*}
    Choose $\epsilon > 0$ to satisfy all preceding smallness conditions and 
    \begin{align}\label{math:epsilon-restriction-ultimate}
        2 C_{\star} \cardinalitybound{\Mesh} \epsilon < 1
        .
    \end{align}
    Since $\cardinalitybound{\Mesh}^{1/p} \leq \cardinalitybound{\Mesh}$, the theorem follows simultaneously for every $p \in [1,\infty]$.
\end{proof}

\begin{remark}
    We summarize the assumptions on $\epsilon$ for Theorem~\ref{theorem:fehlerschranke}.
    First, we require the assumptions in Theorem~\ref{theorem:global-smoothing}. 
    We have imposed the conditions~\eqref{math:epsilon-restriction-eins},~\eqref{math:epsilon-restriction-zwei},~\eqref{math:epsilon-restriction-drei}, and~\eqref{math:epsilon-restriction-ultimate}, 
    \begin{align*}
        D \epsilon < \rho
        , 
        \quad 
        \Cumleitung{\Mesh} D \epsilon < \rho_{\umleitung}
        ,
        \quad 
        2 C_{\star} \cardinalitybound{\Mesh} \epsilon < 1
        ,
        \quad 
        \left( \Cpatchjacobian{\Mesh} + \shapemeasure{\Mesh} \Cumleitung{\Mesh} \right) \comparisonconstant{\Mesh} D H \epsilon 
        < \atlasradius
        ,
    \end{align*}
    with the abbreviations $D := \NumberOfCharts \Cpatchjacobian{\Mesh} \Catlaslipschitz \Cmeshsizevalue \comparisonconstant{\Mesh}$ and $H := \max_{ T \in \Mesh } h_T$. 
\end{remark}

\begin{remark}
    The last part of the proof of Theorem~\ref{theorem:fehlerschranke} is intrinsically reliant on the $C^{2}$ structure of the manifold and the second-order regularity of the triangulation.
    In comparison to the proof in~\cite{christiansen2008smoothed}, one might wonder whether a global smoothing operator could be constructed without using local coordinate charts;
    for example, one could attempt smoothing by averaging over geodesic balls. The obstacle is that any continuity estimate would require comparing values within geodesic balls, 
    thus re-introducing local coordinate systems. 
\end{remark}

\begin{figure}
\centering 
\begin{tikzpicture}[scale=3]

    \def\myrho{0.5}
    \pgfmathsetmacro{\rs}{\myrho/sqrt(2)}

    \coordinate (O) at (0,0);
    \coordinate (E1) at (1,0);
    \coordinate (E2) at (0,1);

    \fill[gray!5] (0,-\myrho) 
        -- (1,-\myrho)
        arc[start angle=-90,end angle=45,radius=\myrho]
        -- (\rs,{1+\rs})
        arc[start angle=45,end angle=180,radius=\myrho]
        -- (-\myrho,0)
        arc[start angle=180,end angle=270,radius=\myrho]
        -- cycle;

    \draw[dotted,thick] (0,-\myrho) 
        -- (1,-\myrho)
        arc[start angle=-90,end angle=45,radius=\myrho]
        -- (\rs,{1+\rs})
        arc[start angle=45,end angle=180,radius=\myrho]
        -- (-\myrho,0)
        arc[start angle=180,end angle=270,radius=\myrho]
        -- cycle;

    \draw[dashed,thick] (O) .. controls (-0.20,-0.12) and ( 0.25,-0.26) .. (-0.34,-0.36);
    \draw[dashed,thick] (O) .. controls (-0.15, 0.25) and (-0.32,0.20) .. (-0.49,0.13);

    \draw[dashed,thick] (E1) .. controls (1.12,-0.08) and (  0.10,-0.18) .. (1.19,-0.45);
    \draw[dashed,thick] (E1) .. controls (1.16,-0.05) and (1.32,0.23) .. (1.45,0.20);

    \draw[dashed,thick] (E2) .. controls (-0.08,1.12) and (-0.18,0.68) .. (-0.49,0.70);
    \draw[dashed,thick] (E2) .. controls (-0.08,1.42) and (-0.18,1.28) .. (-0.29,1.39);
    \draw[dashed,thick] (E2) .. controls ( 0.06,1.15) and ( 0.14,1.31) .. ( 0.90,0.80);

    \draw[thick] (O) -- (E1);
    \draw[thick] (O) -- (E2);
    \draw[thick,red] (E1) -- (E2);

    \fill (O)  circle (0.012);
    \fill (E1) circle (0.012);
    \fill (E2) circle (0.012);

    \draw[<->] (0.5,-0.01) -- (0.5,-\myrho) node[midway,left] {$\rho$};
    
    \node at (0.4,0.3) {$\referencesimplex_{2}$};
    \node at (0.7,0.7) {$\widehat\subdomain_{T,S}$};
    \node[red] at (0.8,0.4) {$\hatF$};
    \node at (1.2,-0.2) {$\widehat\subdomain_{T,S'}$};

\end{tikzpicture}
\qquad 
\begin{tikzpicture}[scale=5]

    \begin{scope}[shift={(0,0.5)}]

        \coordinate (O)  at (0,0);
        \coordinate (E1) at (1,0);
        \coordinate (E2) at (0,1);

        \begin{scope}
            \clip (O) -- (E1) -- (E2) -- cycle;

            \fill[gray!5]
                (0,0) -- (1,0) -- (0.82,0.18) 
                .. controls (0.78,0.16) and (0.72,0.13) ..
                (0.63,0.17) 
                .. controls (0.48,0.24) and (0.36,0.14) ..
                (0.22,0.21)
                .. controls (0.16,0.24) and (0.12,0.26) ..
                (0,0.22)
                -- cycle;
        \end{scope}

        \node at (0.3,0.1) {$\subdomain_{T,S}$};
        \node[red] at (0.6,0.06) {$\incl_{S}^{-1}(F)$};

        \draw[dotted, thick]
            (0,0.22)
            .. controls (0.12,0.26) and (0.16,0.24) .. 
            (0.22,0.21)
            .. controls (0.36,0.14) and (0.48,0.24) .. 
            (0.63,0.17)
            .. controls (0.72,0.13) and (0.78,0.16) .. 
            (0.82,0.18);

        \draw[thick] (E1) -- (E2);
        \draw[thick] (O) -- (E2);
        \draw[thick,red] (O) -- (E1);

        \fill (O)  circle (0.012);
        \fill (E1) circle (0.012);
        \fill (E2) circle (0.012);

    \end{scope}
\end{tikzpicture}
\caption{Geometric setting in the proof of Theorem~\ref{theorem:fehlerschranke}. 
Left: an open neighborhood of the triangle $T$ is transformed to an open $\rho$-neighborhood of the reference triangle. Any subset belonging to a neighboring triangle $S$ is transformed onto a set $\widehat\subdomain_{T,S}$. Right: for example, the set $\widehat\subdomain_{T,S}$ is transformed along $\umleitung_{T,S}$ into a subset of the reference simplex as a representation of $S$.}
\label{figure:setting-in-fehlerschranke}
\end{figure}

\begin{proposition}\label{proposition:austrianinverse}
	Suppose that $\epsilon > 0$ satisfies the smallness assumptions of Theorem~\ref{theorem:fehlerschranke}.
	The restriction of the smoothed interpolant to the finite element space 
	\begin{align*}
		\smoothedinterpol^{k} : \calP\Lambda^{k}(\Mesh) \to \calP\Lambda^{k}(\Mesh)
	\end{align*}
    is invertible. We denote its inverse by 
	\begin{align*}
		\discreteinverse^{k} : \calP\Lambda^{k}(\Mesh) \to \calP\Lambda^{k}(\Mesh).
	\end{align*}
    For any $p \in [1,\infty]$, we have 
    \begin{align}\label{math:austrianinverse:norm}
         \| \discreteinverse^{k} v_{h} \|_{\Lebesgue^{p}\Lambda^{k}(\Manifold,\metric)}
         \leq 
         2
         \| v_{h} \|_{\Lebesgue^{p}\Lambda^{k}(\Manifold,\metric)}
         ,
         \quad 
         v_{h} \in \calP\Lambda^{k}(\Mesh).
    \end{align}
    Moreover, $\discreteinverse^{k}$ commutes with the exterior derivative. 
\end{proposition}

\begin{proof}
    When $\epsilon > 0$ is small enough,
    $\smoothedinterpol^{k}$ is close enough to the identity operator to ensure that $\smoothedinterpol^{k} : \calP\Lambda^{k}(\Mesh) \to \calP\Lambda^{k}(\Mesh)$ is invertible.
    Indeed, injectivity follows by the reverse triangle inequality:
    \begin{align}\label{math:smoothed-interpolant-is-coercive}
        \| \smoothedinterpol^{k} {u} \|_{\Lebesgue^{p}\Lambda^{k}(\Manifold,\metric)}
        \geq 
        \left\lvert
            \| {u} \|_{\Lebesgue^{p}\Lambda^{k}(\Manifold,\metric)}
            -
            \| {u} - \smoothedinterpol^{k} {u} \|_{\Lebesgue^{p}\Lambda^{k}(\Manifold,\metric)}
        \right\rvert
        \geq 
        \frac{1}{2}
        \| {u} \|_{\Lebesgue^{p}\Lambda^{k}(\Manifold,\metric)}
        ,
		\quad
		{u} \in \calP\Lambda^{k}(\Mesh)
        .
    \end{align}
    We let $\discreteinverse^{k} : \calP\Lambda^{k}(\Mesh) \to \calP\Lambda^{k}(\Mesh)$ be its inverse.
    Inequality~\eqref{math:austrianinverse:norm} follows by~\eqref{math:smoothed-interpolant-is-coercive}, 
    and that $\discreteinverse^{k}$ commutes with the exterior derivative follows by 
	\begin{align*}
		\cartan \discreteinverse^{k} {u}
		=
		\discreteinverse^{k+1} \smoothedinterpol^{k+1} \cartan \discreteinverse^{k} {u}
		=
		\discreteinverse^{k+1} \cartan \smoothedinterpol^{k} \discreteinverse^{k} {u}
		=
		\discreteinverse^{k+1} \cartan {u}
        ,
		\quad
		{u} \in \calP\Lambda^{k}(\Mesh)
        .
	\end{align*}
    This completes the proof.
\end{proof}

We define the \emph{smoothed projection}
\begin{align*}
	\smoothedproj^{k} : \Lebesgue^{p}\Lambda^{k}(\Manifold,\metric) \to \calP\Lambda^{k}(\Mesh), \quad {u} \mapsto ( \discreteinverse^{k} \circ \smoothedinterpol^{k} ) {u}.
\end{align*}
This operator commutes with the exterior derivative, acts as the identity on $\calP\Lambda^{k}(\Mesh)$, and is bounded on the spaces $\Lebesgue^{p}\Lambda^{k}(\Manifold,\metric)$ with operator norm independent of the Lebesgue exponent. %

\begin{theorem}\label{theorem:smoothed-projection}
    For every $p \in [1,\infty]$, the smoothed projection defines a bounded linear map
	\begin{align*}
		\smoothedproj^{k} : \Lebesgue^{p}\Lambda^{k}(\Manifold,\metric) \to \calP\Lambda^{k}(\Mesh).
	\end{align*}
    For every $p \in [1,\infty]$ and every ${u} \in \Lebesgue^{p}\Lambda^{k}(\Manifold,\metric)$, we have
	\begin{align*}
		\| \smoothedproj^{k} {u} \|_{\Lebesgue^{p}\Lambda^{k}(\Manifold,\metric)}
		\leq
		C 
		\| {u} \|_{\Lebesgue^{p}\Lambda^{k}(\Manifold,\metric)}.
	\end{align*}
	Here, $C > 0$ depends only on $\Manifold$, the polynomial degree, and the regularity of $\Mesh$.
    If ${u} \in \Sobdiff^{1,1}_{\loc}\Lambda^{k}(\Manifold)$, then 
	\begin{align*}
		\cartan \smoothedproj^{k} {u} = \smoothedproj^{k+1} \cartan {u}
		.
	\end{align*}
	If ${u} \in \calP\Lambda^{k}(\Mesh)$, then $\smoothedproj^{k} {u} = {u}$.
\end{theorem}

\section{Application to the Hodge--Laplace equation and error estimates}\label{section:applications}

We review the Hodge--Laplace equation on manifolds and the Galerkin theory of Hilbert complexes.
This leads to error estimates for the conforming intrinsic finite element method, which is the prototypical application of the smoothed projection.

Let $\Manifold$ be a connected compact smooth $n$-dimensional manifold without boundary, equipped with a Riemannian metric $\metric$.
Recall that the Riemannian metric induces a natural density $\density^{\Manifold}_{\metric}$.

The space $\Lebesgue^{2}\Lambda^{k}(\Manifold,\metric)$ is a Hilbert space with inner product
\begin{gather}\label{math:l2scalarproduct}
	\langle {u}, {v} \rangle_{\metric}
	:=
	\langle {u}, {v} \rangle_{\Lebesgue^{2}\Lambda^{k}(\Manifold,\metric)}
	:=
	\int_{\Manifold} \langle{u},{v}\rangle_{\metric} \density^{\Manifold}_{\metric},
	\qquad 
    {u}, {v} \in \Lebesgue^{2}\Lambda^{k}(\Manifold,\metric)
    .
\end{gather}
We abbreviate $H\Lambda^{k}(\Manifold,\metric) := \Sobdiff^{2,2}\Lambda^{k}(\Manifold,\metric)$.
Definitions imply that $\cartan {u} \in H\Lambda^{k+1}(\Manifold,\metric)$ for each ${u} \in H\Lambda^{k}(\Manifold,\metric)$.
The graph inner product
\begin{gather*}
	\llangle {u}, {v} \rrangle_{\metric}
	:=
	\langle {u}, {v} \rangle_{H\Lambda^{k  }(\Manifold,\metric)}
	:=
	\langle {u}, {v} \rangle_{\Lebesgue^{2}\Lambda^{k  }(\Manifold,\metric)}
	+
	\langle \cartan{u}, \cartan{v} \rangle_{\Lebesgue^{2}\Lambda^{k+1}(\Manifold,\metric)},
	\quad
	{u}, {v} \in H\Lambda^{k}(\Manifold,\metric)
\end{gather*}
makes $H\Lambda^{k}(\Manifold,\metric)$ a Hilbert space.
The exterior derivative defines a closed, densely defined unbounded operator
\begin{align*}
    \cartan : H\Lambda^{k}(\Manifold,\metric) \subseteq \Lebesgue^{2}\Lambda^{k}(\Manifold,\metric) \to \Lebesgue^{2}\Lambda^{k+1}(\Manifold,\metric).
\end{align*}
The adjoint of this operator (in the sense of unbounded operators between Hilbert spaces) is again closed and densely defined. 
Write $H^{\dualityast}\Lambda^{k+1}(\Manifold,\metric) \subseteq \Lebesgue^{2}\Lambda^{k+1}(\Manifold,\metric)$ for its domain. 
The adjoint operator is 
\[
    \cartan^{\dualityast} : H^{\dualityast}\Lambda^{k+1}(\Manifold,\metric) \subseteq \Lebesgue^{2}\Lambda^{k+1}(\Manifold,\metric) \to \Lebesgue^{2}\Lambda^{k}(\Manifold,\metric)
    .
\] 
Recall the $\Lebesgue^{2}$ de~Rham complex:
\begin{gather}\label{math:l2derhamcomplex:primary}
	\begin{CD}
		0 \to
		H\Lambda^{0}(\Manifold,\metric)
		@>{\cartan}>>
		H\Lambda^{1}(\Manifold,\metric)
		@>{\cartan}>>
		\dots
		@>{\cartan}>>
		H\Lambda^{n}(\Manifold,\metric)
		\to
		0
		.
	\end{CD}
\end{gather}
We also consider the adjoint $\Lebesgue^{2}$ de~Rham complex
\begin{gather}\label{math:l2derhamcomplex:adjoint}
	\begin{CD}
		0 \gets
		H^{\dualityast}\Lambda^{0}(\Manifold,\metric)
        @<{\cartan^{\dualityast}}<<
		H^{\dualityast}\Lambda^{1}(\Manifold,\metric)
        @<{\cartan^{\dualityast}}<<
		\dots
		@<{\cartan^{\dualityast}}<<
		H^{\dualityast}\Lambda^{n}(\Manifold,\metric)
		\gets
		0
		.
	\end{CD}
\end{gather}
The following notation will be helpful:
\begin{align}
	\frakB^{k}(\Manifold,\metric)
	&
    :=
	\left\{\;
	{u} \in H\Lambda^{k}(\Manifold,\metric)
	\suchthat*
	\exists {w} \in H\Lambda^{k-1}(\Manifold,\metric) : \cartan {w} = {u}
	\;\right\}
	,
    \\
	\frakZ^{k}(\Manifold,\metric)
	&
    :=
	\left\{\;
	{u} \in H\Lambda^{k}(\Manifold,\metric)
	\suchthat*
	\cartan {u} = 0
	\;\right\}
	,
	\\
	\frakZ^{k}(\Manifold,\metric)^{\perp_{\metric}}
	&
    :=
	\left\{\;
	{u} \in \Lebesgue^{2}\Lambda^{k}(\Manifold,\metric)
	\suchthat*
	{u} \perp_{\Lebesgue^{2}\Lambda^{k}(\Manifold,\metric)} \frakZ^{k}(\Manifold,\metric)
	\;\right\}
	.
\end{align}
Obviously, $\frakB^{k}(\Manifold,\metric) \subseteq \frakZ^{k}(\Manifold,\metric)$, and this inclusion generally need not be an equality.

One can show that the operators in the primary complex~\eqref{math:l2derhamcomplex:primary} and in the adjoint complex~\eqref{math:l2derhamcomplex:adjoint} have closed ranges.
Thus, there exists a constant $\Cpoincareweber{k} > 0$ such that
\begin{align*}
	\forall v \in \frakZ^{k}(\Manifold,\metric)^{\perp_{\metric}} \cap H\Lambda^{k}(\Manifold,\metric)
    \,:\,
    \| v \|_{\Lebesgue^{2}\Lambda^{k  }(\Manifold,\metric)}
	\leq \Cpoincareweber{k}
	\| \cartan v \|_{\Lebesgue^{2}\Lambda^{k+1}(\Manifold,\metric)}
    .
\end{align*}

We define the space of \emph{harmonic $k$-forms}
\begin{align*}
	\frakH^{k}(\Manifold,\metric)
	:=
	\left\{
	\xi \in H\Lambda^{k}(\Manifold,\metric) \suchthat* \cartan \xi = 0 \text{ and } \xi \perp_{\metric} \frakB^{k}(\Manifold,\metric)
	\;\right\}
	.
\end{align*}
As the differential operators have closed ranges, we have the orthogonal Hodge decomposition
\begin{align*}
    \Lebesgue^{2}\Lambda^{k}(\Manifold,\metric)
	&
    =
	\frakB^{k}(\Manifold,\metric)
	\oplus
	\frakH^{k}(\Manifold,\metric)
	\oplus
	\frakZ^{k}(\Manifold,\metric)^{\perp_{\metric}}
    ,
    \\
    H\Lambda^{k}(\Manifold,\metric)
	&
    =
	\frakB^{k}(\Manifold,\metric)
	\oplus
	\frakH^{k}(\Manifold,\metric)
	\oplus
	\left( \frakZ^{k}(\Manifold,\metric)^{\perp_{\metric}} \cap H\Lambda^{k}(\Manifold,\metric) \right)
	.
\end{align*}
A central result states that the dimension of $\frakH^{k}(\Manifold,\metric)$ is finite and equals the $k$-th Betti number of $\Manifold$, that is, 
\begin{gather}
	\dim \frakH^{k}(\Manifold,\metric) = \betti_{k}(\Manifold)
	.
\end{gather}
The Betti number $\betti_{k}(\Manifold)$ is a purely topological concept.

\begin{remark}\label{remark:bettinumbers}
	The Betti numbers $\betti_{k}(\Manifold)$ are topological invariants of the compact manifold $\Manifold$.
    The $k$-th Betti number is the dimension of the $k$-th singular homology group of $\Manifold$ with real coefficients; see~\cite{spanier1995algebraic} for a thorough discussion.
    For example, $\betti_{0}(\Manifold)$ is the number of connected components of $\Manifold$ and, accordingly, $\frakH^0(\Manifold,\metric)$ is the space of locally constant functions.
    If $\Manifold$ is a compact orientable surface of genus $g \in \bbN_{0}$, then the first Betti number is $\betti_{1}(\Manifold) = 2g$.
\end{remark}

We define the Hodge--Laplace operator
\begin{align*}
    \Laplace_{k} : \dom(\Laplace_{k}) \subseteq \Lebesgue^{2}\Lambda^{k}(\Manifold,\metric) \to \Lebesgue^{2}\Lambda^{k}(\Manifold,\metric),
	\quad
    {u} \mapsto \cartan^{\dualityast} \cartan u + \cartan \cartan^{\dualityast} u
\end{align*}
as an unbounded operator with domain 
\begin{align*}
    \dom(\Laplace_{k}) 
    := 
    \left\{ 
        u \in H\Lambda^{k}(\Manifold,\metric) \cap H^{\dualityast}\Lambda^{k}(\Manifold,\metric)
        \suchthat 
        \cartan u \in H^{\dualityast}\Lambda^{k+1}(\Manifold,\metric),
        \; 
        \cartan^{\dualityast} u \in H\Lambda^{k-1}(\Manifold,\metric)
    \right\}
    .
\end{align*}
The operator $\Laplace_{k}$ is a self-adjoint, closed, densely defined operator with closed range.
Its kernel, which is also the orthogonal complement of its range, is $\frakH^{k}(\Manifold,\metric)$.

Given $f \in \Lebesgue^{2}\Lambda^{k}(\Manifold,\metric)$, the Hodge--Laplace equation seeks the unknown $u \in \dom(\Laplace_{k})$ and the Lagrange multiplier $p \in \frakH^{k}(\Manifold,\metric)$ such that
\begin{align}\label{math:hodgelaplace:strongformulation}
	\Laplace_{k} u + p = f, \quad u \perp_{\metric} \frakH^{k}(\Manifold,\metric).
\end{align}
We remark that $\Laplace_{k}$ has a bounded generalized inverse $\Laplace_{k}^{\dagger} : \Lebesgue^{2}\Lambda^{k}(\Manifold,\metric) \to \Lebesgue^{2}\Lambda^{k}(\Manifold,\metric)$. 
Its operator norm is bounded in terms of the Poincar\'e--Friedrichs constant. 
Its kernel is $\frakH^{k}(\Manifold,\metric)$. 
Since $u = \Laplace_{k}^{\dagger} f$ in~\eqref{math:hodgelaplace:strongformulation}, we regard $\Laplace_{k}^{\dagger}$ as the solution operator of the Hodge--Laplace equation.
On compact manifolds, $\Laplace_{k}$ has a discrete spectrum and $\Laplace_{k}^{\dagger}$ is compact.

While the strong formulation~\eqref{math:hodgelaplace:strongformulation} leads to theoretical insights, the following weak formulation is more amenable to finite element methods.
Given $f \in \Lebesgue^{2}\Lambda^{k}(\Manifold,\metric)$, we seek $\sigma \in H\Lambda^{k-1}(\Manifold,\metric)$, $u \in H\Lambda^{k}(\Manifold,\metric)$, and $p \in \frakH^{k}(\Manifold,\metric)$ such that
\begin{subequations}\label{math:hodgelaplace:weakformulation}
	\begin{alignat}{2}
		\langle \sigma, \tau \rangle_{\metric}
		-
		\langle u, \cartan \tau \rangle_{\metric}
		                                            & = 0,
		                                            & \quad \tau                           & \in H\Lambda^{k-1}(\Manifold,\metric),
		\label{math:hodgelaplace:weakformulation:a}
		\\
		\langle \cartan \sigma, v \rangle_{\metric}
		+
		\langle \cartan u, \cartan v \rangle_{\metric}
		+
		\langle p, v \rangle_{\metric}
		                                            & =
		\langle f, v \rangle_{\metric},
		                                            & \quad
		v                                           & \in H\Lambda^{k}(\Manifold,\metric),
		\label{math:hodgelaplace:weakformulation:b}
		\\
		\langle u, q \rangle_{\metric}               & = 0,
		                                            & \quad
		q                                           & \in \frakH^{k}(\Manifold,\metric)
		.
		\label{math:hodgelaplace:weakformulation:c}
	\end{alignat}
\end{subequations}
Here, the first equation is equivalent to $u \in \dom(\cartan^{\dualityast})$ with $\sigma = \cartan^{\dualityast} u$.
Together with the second equation and $\sigma\in H\Lambda^{k-1}(\Manifold,\metric)$, this enforces $u \in \dom(\Laplace_{k})$.
This weak formulation is the starting point for our Galerkin theory of Hilbert complexes.
Problem~\eqref{math:hodgelaplace:weakformulation} has a unique solution $(\sigma,u,p)$ that satisfies
\begin{gather*}
	\| \sigma \|_{H\Lambda^{k-1}(\Manifold,\metric)}
	+
	\| u \|_{H\Lambda^{k}(\Manifold,\metric)}
	+
	\|p\|_{\Lebesgue^{2}\Lambda^{k}(\Manifold,\metric)}
	\leq
	C 
    \| f \|_{\Lebesgue^{2}\Lambda^{k}(\Manifold,\metric)}
    .
\end{gather*}
Here, the constant $C$ depends only on the Poincar\'e--Friedrichs constants $\Cpoincareweber{k-1}$ and $\Cpoincareweber{k}$.

\begin{remark}
    We have not assumed that the manifold $\Manifold$ is orientable, but we briefly discuss how integration works on oriented Riemannian manifolds. 
    Suppose that $\Manifold$ is oriented. 
    The global integral of any compactly supported locally integrable $n$-form is well-defined on any oriented smooth manifold and independent of coordinates, irrespective of whether the manifold carries a Riemannian metric. 
    However, any smooth Riemannian metric $\metric$ on the oriented manifold $\Manifold$, together with the chosen orientation, determines a canonical smooth \emph{Riemannian volume form} $\volume^{\Manifold}_{\metric} \in \Lebesgue^{1}_{\loc}\Lambda^{n}(\Manifold)$.
    For every locally integrable $n$-form $u \in \Lebesgue^{1}_{\loc}\Lambda^{n}(\Manifold)$, there exists a unique locally integrable function $u_{\metric} \in \Lebesgue^{1}_{\loc}(\Manifold)$, 
    up to equality almost everywhere, such that $u = u_{\metric} \volume^{\Manifold}_{\metric}$.
    With that representation,
        $\int_{\Manifold} u = \int_{\Manifold} u_{\metric} \, \volume^{\Manifold}_{\metric}$.
    The Riemannian metric also induces a canonical density $\density^{\Manifold}_{\metric}$ such that 
        $\int_{\Manifold} f \volume^{\Manifold}_{\metric} = \int_{\Manifold} f \density^{\Manifold}_{\metric}$.
    The Hilbert spaces $\Lebesgue^{2}\Lambda^{k}(\Manifold,\metric)$ can be defined via the density or the volume form,
    both leading to the same space. 
    In addition, we recall that the Riemannian metric $\metric$ and the orientation induce the Hodge star operator
    \begin{gather}\label{math:hodgestar}
        \star_{\metric} : \Lebesgue^{1}_{\loc}\Lambda^{k}(\Manifold) \to \Lebesgue^{1}_{\loc}\Lambda^{n-k}(\Manifold),
    \end{gather}
    which is uniquely characterized by the identity  
    \[
        {u} \wedge \star_{\metric} {v} = \langle {u},{v} \rangle_{\metric} \volume^{\Manifold}_{\metric},
        \qquad 
        {u}, {v} \in \Lebesgue^{1}_{\loc}\Lambda^{k}(\Manifold)
    \]
    almost everywhere. 
    In particular, the Riemannian volume form and the Hodge star satisfy $\volume^{\Manifold}_{\metric} = \star_{\metric} \bboldone_{\Manifold}$.
\end{remark}

\subsection{Intrinsic triangulation and finite element de~Rham complex}

Let $\Mesh$ be a finite smooth triangulation of the compact manifold $\Manifold$.
We fix a finite element de~Rham complex
\begin{gather*}
	\begin{CD}
		0 \to
		\calP\Lambda^{0}(\Mesh)
		@>{\cartan}>>
		\calP\Lambda^{1}(\Mesh)
		@>{\cartan}>>
		\dots
		@>{\cartan}>>
		\calP\Lambda^{n}(\Mesh)
		\to
		0
		.
	\end{CD}
\end{gather*}
These spaces are equipped with the Hilbert space structures given by the $\Lebesgue^{2}$ de~Rham complex.
We introduce the subspaces
\begin{align*}
	\frakB^{k}(\Mesh)
	&:=
	\left\{\;
        {u}_{h} \in \calP\Lambda^{k}(\Mesh)
        \suchthat*
        \exists {v}_{h} \in \calP\Lambda^{k-1}(\Mesh) : \cartan {v}_{h} = {u}_{h}
	\;\right\}
	,
    \\
    \frakZ^{k}(\Mesh)
	&:=
	\left\{\;
        {u}_{h} \in \calP\Lambda^{k}(\Mesh)
        \suchthat*
        \cartan {u}_{h} = 0
	\;\right\}
	,
	\\
	\frakZ^{k}(\Mesh)^{\perp_{h}}
	&:=
	\left\{\;
        {u}_{h} \in \calP\Lambda^{k}(\Mesh)
        \suchthat*
        {u}_{h} \perp_{\Lebesgue^{2}\Lambda^{k}(\Manifold,\metric)} \frakZ^{k}(\Mesh) 
	\;\right\}
	.
\end{align*}
The space of discrete harmonic forms is
\begin{align*}
	\frakH^{k}(\Mesh,\metric)
	:=
	\left\{\;
	\xi_{h} \in \calP\Lambda^{k}(\Mesh)
	\suchthat*
	\cartan \xi_{h} = 0, \xi_{h} \perp_{\metric} \frakB^{k}(\Mesh)
	\;\right\}
	.
\end{align*}
Together with the smoothed projections, we have the following commutative diagram:
\begin{gather*}
	\begin{CD}
		H\Lambda^{0}(\Manifold,\metric)
		@>\cartan>>
        H\Lambda^{1}(\Manifold,\metric)
		@>\cartan>>
		\dots
		@>\cartan>>
		H\Lambda^{n}(\Manifold,\metric)
		\\
		@V{\smoothedproj^{0}}VV @V{\smoothedproj^{1}}VV @. @V{\smoothedproj^{n}}VV
		\\
		\calP\Lambda^{0}(\Mesh)
		@>\cartan>>
		\calP\Lambda^{1}(\Mesh)
		@>\cartan>>
		\dots
		@>\cartan>>
		\calP\Lambda^{n}(\Mesh)
		. 
	\end{CD}
\end{gather*}
Define 
\begin{align}\label{math:norm-of-smoothed-projection}
    \|\smoothedproj^{k}\| 
    := 
    \sup_{ \substack{ v \in \Lebesgue^{2}\Lambda^{k}(\Manifold,\metric) \\ \|v\|_{\Lebesgue^{2}\Lambda^{k}(\Manifold,\metric)} \leq 1 } }
    \|\smoothedproj^{k} v \|_{\Lebesgue^{2}\Lambda^{k}(\Manifold,\metric)}
    .
\end{align}
The Poincar\'e--Friedrichs constant of the finite element de~Rham complex is bounded in terms of the original Poincar\'e--Friedrichs constant and the norm~\eqref{math:norm-of-smoothed-projection} of the smoothed projection:
Theorem~3.6 of~\cite{AFW2} shows that, with $\Cdiscretepoincareweber{k} := \Cpoincareweber{k} \|\smoothedproj^{k}\|$, we have 
\begin{align*}
	\| v_{h} \|_{\Lebesgue^{2}\Lambda^{k}(\Manifold,\metric)}
	\leq
    \Cdiscretepoincareweber{k}
	\cdot
	\| \cartan v_{h} \|_{\Lebesgue^{2}\Lambda^{k+1}(\Manifold,\metric)}
\end{align*}
for all $v_{h} \in \calP\Lambda^{k}(\Mesh)$ that are orthogonal to the kernel of $\cartan : \calP\Lambda^{k}(\Mesh) \to \calP\Lambda^{k+1}(\Mesh)$.

Furthermore, the harmonic spaces are close in the following sense. 
Let $P_{\frakH_{h}}$ and $P_{\frakH}$ denote the orthogonal projections onto $\frakH^{k}(\Mesh,\metric)$ and $\frakH^{k}(\Manifold,\metric)$, respectively. 
Theorem~3.5 of~\cite{AFW2} shows that
\begin{gather} 
    \label{math:contraction_physical_to_discrete_harmonics}
	\| q - P_{\frakH_{h}} q \|_{\Lebesgue^{2}\Lambda^{k}(\Manifold,\metric)}
	\leq
	\| q - \smoothedproj^{k} q \|_{\Lebesgue^{2}\Lambda^{k}(\Manifold,\metric)},
	\quad
	q \in \frakH^{k}(\Manifold,\metric),
	\\
	\label{math:contraction_discrete_to_physical_harmonics}
	\| q_{h} - P_{\frakH} q_{h} \|_{\Lebesgue^{2}\Lambda^{k}(\Manifold,\metric)}
	\leq
	\| P_{\frakH} q_{h} - \smoothedproj^{k} P_{\frakH} q_{h} \|_{\Lebesgue^{2}\Lambda^{k}(\Manifold,\metric)},
	\quad
	q_{h} \in \frakH^{k}(\Mesh,\metric).
\end{gather}
We study the dimensions of the finite element cohomology spaces. 

The lowest-order finite element de~Rham complex, the complex of Whitney forms, realizes the Betti numbers of the manifold on cohomology; see Theorem~\ref{theorem:whitney-cohomology}. The smoothed projection allows us to extend that observation to higher-order finite element de~Rham complexes.

\begin{proposition}
    We have $\dim \frakH^{k}(\Manifold,\metric) = \dim\frakH^{k}(\Mesh,\metric)$.
\end{proposition}
\begin{proof} %
    We recall the smoothed projections $\smoothedproj^{k} : \Lebesgue^{2}\Lambda^{k}(\Manifold,\metric) \to \calP\Lambda^{k}(\Mesh)$ onto the finite element spaces.
    In addition, we consider the smoothed projections $\smoothedproj^{k}_{\calW} : \calP\Lambda^{k}(\Mesh) \to \calP^{-}_{1}\Lambda^{k}(\Mesh)$ onto the Whitney forms.
    This leads to a commutative diagram:
    \begin{align*}
        \begin{CD}
            0
            @>>>
            H\Lambda^{0}(\Manifold,\metric)
            @>{\cartan}>>
            H\Lambda^{1}(\Manifold,\metric)
            @>{\cartan}>>
            \cdots
            @>{\cartan}>>
            H\Lambda^{n}(\Manifold,\metric)
            @>>>
            0
            \\
            @. @V{\smoothedproj^{0}}VV @V{\smoothedproj^{1}}VV @. @V{\smoothedproj^{n}}VV
            \\
            0
            @>>>
            \calP\Lambda^{0}(\Mesh)
            @>{\cartan}>>
            \calP\Lambda^{1}(\Mesh)
            @>{\cartan}>>
            \cdots
            @>{\cartan}>>
            \calP\Lambda^{n}(\Mesh)
            @>>>
            0
            \\
            @. @V{\smoothedproj^{0}_{\calW}}VV @V{\smoothedproj^{1}_{\calW}}VV @. @V{\smoothedproj^{n}_{\calW}}VV
            \\
            0
            @>>>        
            \calP^{-}_{1}\Lambda^{0}(\Mesh)
            @>{\cartan}>>
            \calP^{-}_{1}\Lambda^{1}(\Mesh)
            @>{\cartan}>>
            \cdots
            @>{\cartan}>>
            \calP^{-}_{1}\Lambda^{n}(\Mesh)
            @>>>
            0
            .
        \end{CD}
    \end{align*}
    We analyze the vertical operators on cohomology. 
    Since the middle row is a subcomplex of the top row and the top row admits a commuting projection onto the middle row, 
    we conclude that $\smoothedproj^{k} \colon H\Lambda^{k}(\Manifold,\metric) \to \calP\Lambda^{k}(\Mesh)$ is surjective on cohomology.
    By a completely analogous argument, $\smoothedproj^{k}_{\calW} \colon \calP\Lambda^{k}(\Mesh) \to \calP^{-}_{1}\Lambda^{k}(\Mesh)$ is surjective on cohomology. 
    From Theorem~\ref{theorem:whitney-cohomology} we conclude that 
    \begin{align*}
        \dim \frakH^{k}_{\calW}(\Mesh) 
        \leq 
        \dim \frakH^{k}(\Mesh,\metric) 
        \leq 
        \dim \frakH^{k}(\Manifold,\metric)
    \end{align*}
    but also $\dim \frakH^{k}_{\calW}(\Mesh) = \dim \frakH^{k}(\Manifold,\metric)$. 
    The desired identity is thus clear. 
\end{proof}

\begin{remark}
    The above proof relies on our knowledge of the cohomology of Whitney forms. 
    An alternative proof that the finite element de~Rham complex realizes the Betti numbers is conditional on the mesh being sufficiently fine.
    If we assume that $\Id - \smoothedproj^{k}$ is a contraction on $\frakH^{k}(\Manifold,\metric)$, which holds in practice for sufficiently fine meshes, 
    then~\eqref{math:contraction_discrete_to_physical_harmonics} and~\eqref{math:contraction_physical_to_discrete_harmonics} imply that the orthogonal projections 
    $P_{\frakH_{h}} : \frakH^{k}(\Manifold,\metric) \to \frakH^{k}(\Mesh,\metric)$ and $P_{\frakH} : \frakH^{k}(\Mesh,\metric) \to \frakH^{k}(\Manifold,\metric)$ are injective. 
    It then follows that $\dim \frakH^{k}(\Manifold,\metric) = \dim\frakH^{k}(\Mesh,\metric)$.
    We notice that this approach not only establishes the equality of dimensions but also provides a quantitative estimate of the distance between the continuous and discrete harmonic forms.
    As yet another alternative, this identity of dimensions can be shown by purely algebraic means (see~\cite{licht2016complexes}).
\end{remark}

\subsection{Intrinsic finite element formulation}

Consider now the following discrete weak formulation of the Hodge--Laplace equation:
given $f \in \Lebesgue^{2}\Lambda^{k}(\Manifold,\metric)$, find $\sigma_{h} \in \calP\Lambda^{k-1}(\Mesh)$, $u_{h} \in \calP\Lambda^{k}(\Mesh)$, and $p_{h} \in \frakH^{k}(\Mesh,\metric)$ such that
\begin{subequations}
	\begin{alignat}{2}
		\langle \sigma_{h}, \tau_{h} \rangle_{\metric}
		-
		\langle u_{h}, \cartan \tau_{h} \rangle_{\metric}
		                                                    & = 0,
		                                                    & \quad \tau_{h}                & \in \calP\Lambda^{k-1}(\Mesh),
		\\
		\langle \cartan \sigma_{h}, v_{h} \rangle_{\metric}
		+
		\langle \cartan u_{h}, \cartan v_{h} \rangle_{\metric}
		+
		\langle p_{h}, v_{h} \rangle_{\metric}
		                                                    & =
		\langle f, v_{h} \rangle_{\metric},
		                                                    & \quad
		v_{h}                                               & \in \calP\Lambda^{k}(\Mesh),
		\\
		\langle u_{h}, q_{h} \rangle_{\metric} & = 0,
		                                                    & \quad
		q_{h}                                               & \in \frakH^{k}(\Mesh,\metric)
		.
	\end{alignat}
\end{subequations}
By the Galerkin theory of Hilbert complexes, this discrete problem has a unique solution with %
\begin{gather*}
	\| \sigma_{h} \|_{H\Lambda^{k-1}(\Manifold,\metric)}
	+
	\| u_{h} \|_{H\Lambda^{k}(\Manifold,\metric)}
	+
	\| p_{h} \|_{\Lebesgue^{2}\Lambda^{k}(\Manifold,\metric)}
	\leq
	C \| f \|_{\Lebesgue^{2}\Lambda^{k}(\Manifold,\metric)},
\end{gather*}
where the stability constant $C > 0$ depends only on the discrete Poincar\'e--Friedrichs constants.

We now discuss the error theory of this Galerkin method based on~\cite{AFW2}, for which we make additional definitions.
We use the following abbreviations for the best-approximation errors:
\begin{subequations}\label{math:bestapproximationerror}
	\begin{gather}
		\label{math:bestapproximationerror:w}
		E( u ) := \inf_{ v \in \calP\Lambda^{k}(\Mesh) }\| u - v \|_{\Lebesgue^{2}\Lambda^{k}(\Manifold,\metric)}, \quad u \in \Lebesgue^{2}\Lambda^{k}(\Manifold,\metric),
		\\
		\label{math:bestapproximationerror:v}
		E_{\cartan}( u ) := \inf_{ v \in \calP\Lambda^{k}(\Mesh) }\| u - v \|_{H\Lambda^{k}(\Manifold,\metric)}, \quad u \in H\Lambda^{k}(\Manifold,\metric).
	\end{gather}
\end{subequations}
Note that the latter is bounded by the componentwise best-approximation errors:
\begin{align*}
	E_{\cartan}( u )
    & \leq
	\inf_{ v \in \calP\Lambda^{k}(\Mesh) } \left( \| u - v \|_{\Lebesgue^{2}\Lambda^{k}(\Manifold,\metric)} + \| \cartan(u - v) \|_{\Lebesgue^{2}\Lambda^{k+1}(\Manifold,\metric)} \right)
	\\&\leq
	\| u - \smoothedproj^{k} u \|_{\Lebesgue^{2}\Lambda^{k}(\Manifold,\metric)} + \| \cartan(u - \smoothedproj^{k} u) \|_{\Lebesgue^{2}\Lambda^{k+1}(\Manifold,\metric)}
	\\&=
	\| u - \smoothedproj^{k} u \|_{\Lebesgue^{2}\Lambda^{k}(\Manifold,\metric)} + \| \cartan u - \smoothedproj^{k+1} \cartan u \|_{\Lebesgue^{2}\Lambda^{k+1}(\Manifold,\metric)}
    \leq
	\left( 1 + \| \smoothedproj^{k} \| \right)
    E( u )
	+
	\left( 1 + \| \smoothedproj^{k+1} \| \right)
    E( \cartan u )
	.
\end{align*}
We introduce the approximation quantities  
\begin{gather*}
	\coeff_{\mu} 
    := 
    \sup\left\{ 
        \| q - \smoothedproj^{k} q \|_{\Lebesgue^{2}\Lambda^{k}(\Manifold,\metric)}
        \suchthat 
        q \in \frakH^{k}(\Manifold,\metric), \ \|q\|_{\Lebesgue^{2}\Lambda^{k}(\Manifold,\metric)} \leq 1 
    \right\}
    ,
    \\
    \coeff_{\delta}  := \left\|( \Id -  \smoothedproj^{k}) \Laplace_{k}^{\dagger} \right\|,
	\qquad
    \coeff_{\eta} :=
    \max_{j \in \{0,1\}}\left\{
        \left\| ( \Id -  \smoothedproj^{k-j+1} ) \cartan               \Laplace_{k-j}^{\dagger} \right\|
        ,
        \left\| ( \Id -  \smoothedproj^{k+j-1} ) \cartan^{\dualityast} \Laplace_{k+j}^{\dagger} \right\|
    \right\}
    .
\end{gather*}
The Galerkin solution is a quasi-optimal approximation of the continuous solution; see~\cite[Theorem~3.9]{AFW2}:
\begin{align*}
	 &
	\| \sigma - \sigma_{h} \|_{H\Lambda^{k-1}(\Manifold,\metric)}
	+
	\| u - u_{h} \|_{H\Lambda^{k}(\Manifold,\metric)}
	+
	\| p - p_{h} \|_{\Lebesgue^{2}\Lambda^{k}(\Manifold,\metric)}
	\\
	& \qquad\qquad\qquad\qquad\qquad
	\leq C\bigg(
        E_{\cartan}( \sigma )
        +
        E_{\cartan}( u )
        +
        E_{\cartan}( p )
        +
        \coeff_{\mu}
        E_{\cartan}( P_{\frakB(\Manifold,\metric)} u )
	\bigg),
\end{align*}
where $P_{\frakB(\Manifold,\metric)} : \Lebesgue^{2}\Lambda^{k}(\Manifold,\metric) \to \frakB^{k}(\Manifold,\metric)$ is the orthogonal projection.

The Galerkin theory of Hilbert complexes also gives estimates for the individual error components, which are useful because the individual components generally converge at different rates.
By~\cite[Theorem~3.11]{AFW2}, there exists a constant $C > 0$, depending only on $\Cpoincareweber{k-1}$, $\Cpoincareweber{k}$, and the $\Lebesgue^{2}$-operator bound for $\smoothedproj$, such that
\begin{align*}
	\| \cartan( \sigma - \sigma_{h} ) \|_{\Lebesgue^{2}\Lambda^{k}(\Manifold,\metric)} & \leq C E( \cartan \sigma ),
	\\
	\| \sigma - \sigma_{h} \|_{\Lebesgue^{2}\Lambda^{k-1}(\Manifold,\metric)}          & \leq C\left( E( \sigma ) + \coeff_{\eta} E( \cartan \sigma ) \right),
	\\
	\| p - p_{h} \|_{\Lebesgue^{2}\Lambda^{k}(\Manifold,\metric)}                      & \leq C \left( E(p) + \coeff_{\mu} E( \cartan \sigma ) \right),
	\\
	\| \cartan( u - u_{h} ) \|_{\Lebesgue^{2}\Lambda^{k+1}(\Manifold,\metric)}           & \leq C\left( E(\cartan u) + \coeff_{\eta}\left( E(\cartan \sigma) + E(p) \right) \right),
	\\
	\| u - u_{h} \|_{\Lebesgue^{2}\Lambda^{k}(\Manifold,\metric)}                      & \leq C\left( E(u) + \coeff_{\eta} \left( E(\cartan u) + E(\sigma) \right) + (\coeff_{\eta}^{2} + \coeff_{\delta}) \left( E(\cartan \sigma) + E(p) \right) + \coeff_{\mu} E(P_{\frakB(\Manifold,\metric)}u) \right).
\end{align*}
We have bounded the Galerkin error terms by combinations of best-approximation errors,
but we emphasize that the estimates obtained so far do not yet immediately yield convergence rates in terms of the mesh size and the regularity of the solution.
For that, we need to estimate the best-approximation error terms~\eqref{math:bestapproximationerror} and the error coefficients $\coeff_{\mu}$, $\coeff_{\delta}$, and $\coeff_{\eta}$ in terms of the mesh size. 
This matter is non-trivial on a manifold, and therefore the well-known polynomial approximation estimates are not intrinsically available; after all, there is no natural intrinsic sense of polynomials on a manifold.
This is the subject of the remainder of this section.

\subsection{Best approximation error estimates}

We estimate the best-approximation errors. We first focus on the approximation of differential forms with strong regularity assumptions. 
The approximation of differential forms with weaker regularity assumptions is discussed afterward.

\begin{theorem}\label{theorem:bestapproximationestimate:sobolev}
    Let $m, r \in \bbN$ with $1 \leq m \leq r+1$ and $m > n/2$. 
    If $v \in \Sobolev^{m,2}\Lambda^{k}(\Manifold,\metric)$, then $v \in \Cont\Lambda^{k}(\Manifold)$ and 
    \begin{align*}
        \inf_{ v_{h} \in \calP_{r}\Lambda^{k}(\Mesh) }
		\| v - v_{h} \|_{\Lebesgue^{2}\Lambda^{k}(\Manifold,\metric)}
		&
        \leq
		C
		\max_{ T \in \Mesh^{n} } h_{T}^{m} \cdot %
        \| v \|_{\Sobolev^{m,2}\Lambda^{k}(\Manifold,\metric)}
		\\
        \| v - \Interpolant_{\Mesh}^{k,r} v \|_{\Lebesgue^{2}\Lambda^{k}(\Manifold,\metric)}
		&
        \leq
		C
		\max_{ T \in \Mesh^{n} } h_{T}^{m} \cdot %
        \| v \|_{\Sobolev^{m,2}\Lambda^{k}(\Manifold,\metric)}
		\\
        \| v - \Interpolant_{\Mesh}^{k,r+1,-} v \|_{\Lebesgue^{2}\Lambda^{k}(\Manifold,\metric)}
		&
        \leq
		C
		\max_{ T \in \Mesh^{n} } h_{T}^{m} \cdot %
        \| v \|_{\Sobolev^{m,2}\Lambda^{k}(\Manifold,\metric)}
		.
	\end{align*}
    Here, $C > 0$ depends only on $n$, $k$, $r$, $m$, the first-order shape parameters, and the shape constants of $\Mesh$ up to order $m+1$.
\end{theorem}
\begin{proof}
    In this proof, $C > 0$ denotes a generic constant that depends only on $m$, $k$, $n$, $r$, and the shape parameters of $\Mesh$ up to order $m+1$ and whose value may change from line to line. 
    Let $v \in \Sobolev^{m,2}\Lambda^{k}(\Manifold,\metric)$ be fixed but arbitrary. 
    By the Sobolev embedding theorem, $m > n/2$ implies $v \in \Cont\Lambda^{k}(\Manifold)$, and there exists a constant $C_{\Manifold,m}$ that only depends on the manifold and the Sobolev index such that 
    \[
        \| v \|_{\Cont\Lambda^{k}(\Manifold)}
        \leq 
        C_{\Manifold,m}
        \| v \|_{\Sobolev^{m,2}\Lambda^{k}(\Manifold,\metric)}
        .
    \]

    Write $v_{I} := \Interpolant_{\Mesh}^{k,r} v \in \calP_{r}\Lambda^{k}(\Mesh)$ for the global canonical finite element interpolant of $v$ into the space $\calP_{r}\Lambda^{k}(\Mesh)$.
    The interpolant is well-defined due to continuity of $v$. 
    By definition, 
	\begin{align*}
		\inf_{ v_{h} \in \calP_{r}\Lambda^{k}(\Mesh) }
		\| v - v_{h} \|_{\Lebesgue^{2}\Lambda^{k}(\Manifold,\metric)}
		\leq
		\| v - v_I \|_{\Lebesgue^{2}\Lambda^{k}(\Manifold,\metric)}
		.
	\end{align*}
    It remains to bound the approximation error of the canonical interpolant. 
    
	Let $T \in \Mesh^{n}$ be an arbitrary $n$-dimensional simplex.
    The definition of $\calP_{r}\Lambda^{k}(T)$ as a pullback together with the pullback estimate in Theorem~\ref{theorem:transformationssatz} yields %
	\begin{align*}%
		\| v - v_I \|_{\Lebesgue^{2}\Lambda^{k}(T,\metric)}
		&
        \leq
        \| \Jacobian\incl_{T}^{-1} \|_{\Lebesgue^{\infty}(T,\metric)}^{k}
        \| \det \Jacobian\incl_{T} \|_{\Lebesgue^{\infty}(\referencesimplex_{n})}^{1/2}
		\| \incl_{T}^{\pullbackast} v - \incl_{T}^{\pullbackast} v_I \|_{\Lebesgue^{2}\Lambda^{k}(\referencesimplex_{n})}
        \leq
        C h_{T}^{ - k + \frac n 2 }
		\| \incl_{T}^{\pullbackast} v - \incl_{T}^{\pullbackast} v_I \|_{\Lebesgue^{2}\Lambda^{k}(\referencesimplex_{n})}
		.
	\end{align*}
    By definition of the canonical interpolant on each simplex, $\incl_{T}^{\pullbackast} v_I$ equals the canonical interpolant of 
    \[ 
        \incl_{T}^{\pullbackast} v \in \Sobolev^{m,2}\Lambda^{k}(\referencesimplex_{n}) \cap \Cont\Lambda^{k}(\referencesimplex_{n})
    \]
    on the reference simplex. 
    Due to the boundedness of the operator  
    $
        \Interpolant_{\referencesimplex_{n}}^{k,r}
        :
        \Cont\Lambda^{k}(\referencesimplex_{n})
        \to 
        \calP_{r}\Lambda^{k}(\referencesimplex_{n})
    $
    and the Bramble--Hilbert lemma on the reference simplex (see~\cite[4.3.8]{brennerscott},~\cite[Lemma~11.9]{ern2021finite}), 
    \begin{align*}
		&
		\| \incl_{T}^{\pullbackast} v - \incl_{T}^{\pullbackast} v_I \|_{\Lebesgue^{2}\Lambda^{k}(\referencesimplex_{n})}
		\leq
		C
        \lvert \incl_{T}^{\pullbackast} v \rvert_{\Sobolev^{m,2}\Lambda^{k}(\referencesimplex_{n})}
        ,
	\end{align*}
    where $C > 0$ depends only on $m$, $n$ and $r$. 
    An application of Theorem~\ref{theorem:derivatives-of-pullbacks} ensures that
	\begin{align*} 
        &
        \lvert \incl_{T}^{\pullbackast} v \rvert_{\Sobolev^{m,2}\Lambda^{k}(\referencesimplex_{n})}
		\\&\quad 
        \leq
        C
        \sum_{j=0}^{m}
		P^{(k)}_{m,j} \left( 
            \|              \Jacobian\incl_{T} \|_{\Lebesgue^{\infty}}
            ,
            \| \nabla       \Jacobian\incl_{T} \|_{\Lebesgue^{\infty}}
            ,
            \dots
            ,
            \| \nabla^{m-j} \Jacobian\incl_{T} \|_{\Lebesgue^{\infty}}
        \right) 
        \| \det \Jacobian\incl_{T}^{-1} \|^{\frac{1}{2}}_{\Lebesgue^{\infty}(T,\metric)}
		\| \nabla^{j} {v} \|_{\Lebesgue^{2}\Tensor^{k+j}(T,\metric)}
		\\&\quad 
        \leq
        C
        h_{T}^{ k + m - \frac n 2 }
        \sum_{j = 0}^{m}
        \| \nabla^{j} {v} \|_{\Lebesgue^{2}\Tensor^{k+j}(T,\metric)}
		.
	\end{align*}
    Squaring and summing the local estimates yields the desired inequality:
    \begin{align*}
    	\| v - v_I \|_{\Lebesgue^{2}\Lambda^{k}(\Manifold,\metric)}^{2}
    	\leq 
        C
        \sum_{ T \in \Mesh^{n} } 
        h_{T}^{2m}
        \| v \|_{\Sobolev^{m,2}\Lambda^{k}(T,\metric)}^{2}
    	\leq 
        C
        \left( \max_{ T \in \Mesh^{n} } h_{T}^{2m} \right) 
        \| v \|_{\Sobolev^{m,2}\Lambda^{k}(\Manifold,\metric)}^{2}.
    \end{align*}
    The analogous inequality for the interpolant into the space $\calP_{r+1}^{-}\Lambda^{k}(\Mesh)$ follows because that space already contains $\calP_{r}\Lambda^{k}(\Mesh)$.
    This proves the desired statement. 
\end{proof}

The canonical interpolant enables convergence estimates but requires sufficient regularity so that the degrees of freedom can be evaluated. 
Quasi-interpolation operators on non-smooth fields serve the same purpose but require less regularity. 
We describe a quasi-interpolation operator that satisfies a local piecewise Bramble--Hilbert lemma on square-integrable differential forms.
Building on work by Scott and Zhang~\cite{scott1990finite}, Ern and Guermond~\cite{ern2017finite}, and Oswald~\cite{oswald1993bpx}, 
quasi-interpolation operators have been generalized to differential forms of low regularity on Euclidean domains~\cite{licht2021local,licht2023averaging}.

\begin{theorem}\label{theorem:ern-guermond-oswald}
    There exists a bounded linear operator 
    \begin{align}\label{math:ern-guermond-oswald:operator}
        \Pi^{k} : \Lebesgue^{2}\Lambda^{k}(\Manifold) \to \calP\Lambda^{k}(\Mesh)
    \end{align}
    with the following properties. $\Pi^{k}$ is a projection:
    \begin{align}\label{math:ern-guermond-oswald:projection}
        \Pi u_{h} = u_{h}, \qquad u_{h} \in \calP\Lambda^{k}(\Mesh).
    \end{align}
    There exists $C_{\Pi} > 0$ such that if $T \in \Mesh^{n}$ and $u \in \Lebesgue^{2}\Lambda^{k}(\Manifold)$, then 
    \begin{align}\label{math:ern-guermond-oswald:stability}
        \| \Pi^{k} u \|_{\Lebesgue^{2}\Lambda^{k}(T,\metric)}
        \leq 
        C_{\Pi}
        \sum_{ \substack{ T' \in \Mesh^{n} \\ T \cap T' \neq \emptyset } }
        \| u \|_{\Lebesgue^{2}\Lambda^{k}(T',\metric)}
        .
    \end{align}
    Here, $C_{\Pi} > 0$ depends only on $n$, $k$, $r$, and the first-order shape parameters.
    
    Suppose that $m, \ell \in \bbN_{0}$ and
    \begin{align}
        \calP_{m-1}\Lambda^{k}(\referencesimplex_{n}) \subseteq \calP\Lambda^{k}(\referencesimplex_{n}),
        \qquad 
        \calP_{\ell-1}\Lambda^{k+1}(\referencesimplex_{n}) \subseteq \calP\Lambda^{k+1}(\referencesimplex_{n}).
    \end{align}
    There exists $C > 0$ such that if $T \in \Mesh^{n}$ and $u \in H\Lambda^{k}(\Manifold)$ satisfies 
    \begin{align*}
        \forall T' \in \Mesh^{n} : ( u \restriction_{T'}, \cartan u \restriction_{T'} ) \in \Sobolev^{m,2}\Lambda^{k}(T') \times \Sobolev^{\ell,2}\Lambda^{k+1}(T'),
    \end{align*}
    then 
    \begin{align}\label{math:ern-guermond-oswald:error}
        \| u - \Pi^{k} u \|_{\Lebesgue^{2}\Lambda^{k}(T,\metric)}
        \leq 
        C
        \sum_{ \substack{ T' \in \Mesh^{n} \\ T \cap T' \neq \emptyset } }
        \left( h_{T}^{m}
        \| u \|_{\Sobolev^{m,2}\Lambda^{k}(T',\metric)}
        +
        h_{T}^{\ell + 1}
        \| \cartan u \|_{\Sobolev^{\ell,2}\Lambda^{k+1}(T',\metric)} \right)
        .
    \end{align}
    Here, $C > 0$ depends only on $n$, $k$, $r$, $m$, $\ell$, the first-order shape parameters, and the shape constants of $\Mesh$ up to order $\max(m,\ell)+1$.
\end{theorem}
\begin{proof}
    Any face $F \in \Mesh^{n-1}$ is contained in exactly two distinct simplices $T_1, T_2 \in \Mesh^{n}$. Write $\subsubdomain_{F} := T_1 \cup T_2$.
    Define the reference face patch 
    \begin{align*}
        \widehat\subsubdomain
        := 
        \{ (x_1,\dots,x_n) \in [-1,1] \times [0,1]^{n-1} \suchthat \lvert x_1 \rvert + x_2 + \cdots + x_n \leq 1 \}
        ,
    \end{align*}
    and let $\widehat\Mesh$ be the obvious triangulation with two $n$-simplices $\referencesimplex_{n,1} := \referencesimplex_{n}$ and $\referencesimplex_{n,2}$. 
    Using the reference mappings $\incl_{T_1}$ and $\incl_{T_2}$,
    construct a bi-Lipschitz map $\theta_{F} : \widehat\subsubdomain \to \subsubdomain_{F}$ 
    that restricts to diffeomorphisms $\theta_{F} : \referencesimplex_{n,1} \to T_1$ and $\theta_{F} : \referencesimplex_{n,2} \to T_2$ 
    and such that for some constant $\widehat{C} > 0$ that only depends on $\shapemeasure{\Mesh}$, $\comparisonconstant{\Mesh}$ and $n$, 
    \begin{align*}
        \| \Jacobian \theta_{F} \|_{L^{\infty}} \leq \widehat{C} h_{F},
        \qquad 
        \| \Jacobian \theta_{F}^{-1} \|_{L^{\infty}} \leq \widehat{C} h_{F}^{-1},
        \qquad 
        \| \det \Jacobian \theta_{F} \|_{L^{\infty}} \leq \widehat{C} h_{F}^{n}, 
        \qquad 
        \| \det \Jacobian \theta_{F}^{-1} \|_{L^{\infty}} \leq \widehat{C} h_{F}^{-n}. 
    \end{align*}
    Let $\Pi^{k}_{\widehat\subsubdomain} : \Lebesgue^{2}\Lambda^{k}(\widehat\subsubdomain) \to \calP\Lambda^{k}(\widehat\Mesh)$ and $\Pi^{k}_{\referencesimplex} : \Lebesgue^{2}\Lambda^{k}(\referencesimplex_{n}) \to \calP\Lambda^{k}(\referencesimplex_{n})$ be orthogonal projections.
    For any $T \in \Mesh^{n}$ and any $F \in \Mesh^{n-1}$, define 
    \begin{align*}
        \Pi^{k}_{T} := \incl_{T}^{-\pullbackast} \Pi^{k}_{\referencesimplex} \incl_{T}^{\pullbackast}
        ,
        \qquad 
        \Pi^{k}_{F} := \theta_{F}^{-\pullbackast} \Pi^{k}_{\widehat\subsubdomain} \theta_{F}^{\pullbackast}
        .
    \end{align*}
    These mappings have operator norms bounded in terms of $\shapemeasure{\Mesh}$, $\shapemeasuredeterminant{\Mesh}$, and $\comparisonconstant{\Mesh}$.

    For any $S \in \Mesh$, we fix $T_{S} \in \Mesh^{n}$ such that $S \subseteq T_{S}$. 
    Given $u \in \Lebesgue^{2}\Lambda^{k}(\Manifold)$, we define $\Pi^{k} u \in \calP_{r}\Lambda^{k}(\Mesh)$ (or in $\calP^{-}_{r}\Lambda^{k}(\Mesh)$) by the following condition:
    for all $S \in \Mesh$ and all $\eta \in \calP^{-}_{r+k-\dim(S)}\Lambda^{\dim(S)-k}(S)$ (or in $\calP_{r+k-\dim(S)-1}\Lambda^{\dim(S)-k}(S)$) we require 
    \begin{align*}
        \int_{S} \eta \wedge \left( \trace_{S} \Pi^{k} u - \trace_{T_{S},S} \Pi^{k}_{T_{S}} u \right) = 0.
    \end{align*}
    It is clear that this defines a bounded linear operator $\Pi^{k} : \Lebesgue^{2}\Lambda^{k}(\Manifold) \to \calP_{r}\Lambda^{k}(\Mesh)$
    that satisfies~\eqref{math:ern-guermond-oswald:projection} and~\eqref{math:ern-guermond-oswald:stability} with a constant bounded in terms of $\shapemeasure{\Mesh}$, $\shapemeasuredeterminant{\Mesh}$, $n$, and the polynomial degree $r$. 
    
    Let $T \in \Mesh^{n}$. 
    First, by the triangle inequality, 
    \begin{align*}
        \left\| 
            u
            - 
            \Pi^{k} u
        \right\|_{\Lebesgue^{2}\Lambda^{k}(T,\metric)}
		\leq 
        \left\| 
            \Pi^{k}_{T} u
            - 
            u
        \right\|_{\Lebesgue^{2}\Lambda^{k}(T,\metric)}
		+
        \left\| 
            \Pi^{k}_{T} u
            - 
            \Pi^{k} u
        \right\|_{\Lebesgue^{2}\Lambda^{k}(T,\metric)}
		.
    \end{align*}
    The difference $u - \Pi^{k}_{T} u$ is a local quasi-optimal approximation error. 
    As in the proof of Theorem~\ref{theorem:fehlerschranke},
    there exists $C > 0$ depending only on $n$, $r$, and first-order shape parameters, such that 
    \begin{align}
        \left\| 
            \Pi^{k}_{T} u
            - 
            \Pi^{k} u
        \right\|_{\Lebesgue^{2}\Lambda^{k}(T,\metric)}
		\leq
		C h_{T}^{\frac n 2}
		\sup_{ \substack{ 
            S \in \Mesh, \; S \subseteq T,
            \\
            \eta \in \calP^{-}_{r+k-\dim(S)}\Lambda^{\dim(S) - k}(S)
            \\ 
            \| \eta \|_{\Lebesgue^{1}\Lambda^{\dim(S) - k}(S,\metric)} \leq 1 
        } }
        \left\lvert 
            \int_{S} \trace_{T,S}( \Pi^{k} u - \Pi^{k}_{T} u ) \wedge \eta 
        \right\rvert
        .  
	\end{align}
    Let $S$ and $\eta$ be as in the supremum. By construction, 
    \begin{align*}
        \int_{S} \trace_{T,S}( \Pi^{k} u - \Pi^{k}_{T} u ) \wedge \eta 
        =
        \int_{S} \left( 
            \trace_{S}( \Pi^{k} u )
            -
            \trace_{T,S}( \Pi^{k}_{T} u )
        \right) \wedge \eta 
        =
        \int_{S} 
        \left( 
            \trace_{T_{S},S}( \Pi^{k}_{T_{S}} u )
            -
            \trace_{T,S}( \Pi^{k}_{T} u )
        \right) \wedge \eta 
        .
    \end{align*}
    In the case $T = T_{S}$, we get $\int_{S} \trace_{T,S}( \Pi^{k} u - \Pi^{k}_{T} u ) \wedge \eta = 0$. 
    Consider the case $T \neq T_{S}$. 
    As in the proof of Theorem~\ref{theorem:weak-exterior-derivative-of-piecewise-forms}, 
    we use that the triangulation is face-connected:
    there exist pairwise distinct $K_{0}, K_{1}, \dots, K_{N} \in \Mesh^{n}$ such that $K_{0} = T$ and $K_N = T_{S}$ 
    and such that $S \subseteq F_{i} := K_{i} \cap K_{i-1} \in \Mesh^{n-1}$ for all $1 \leq i \leq N$.
    Note that $1 \leq N \leq \cardinalitybound{\Mesh}$. 
    We point out that $\trace_{K_{i},S} \Pi^{k}_{F_{i}} u = \trace_{K_{i-1},S} \Pi^{k}_{F_{i}} u$. 
    We use this in the telescoping expansion 
    \begin{align*}
        &
        \trace_{T_{S},S}( \Pi^{k}_{T_{S}} u )
        -
        \trace_{T,S}( \Pi^{k}_{T} u )
        =
        \trace_{K_{N},S}( \Pi^{k}_{K_{N}} u )
        -
        \trace_{K_{0},S}( \Pi^{k}_{K_{0}} u )
        \\&\quad 
        =
        \trace_{K_{N},S}( \Pi^{k}_{K_{N}} u - \Pi^{k}_{F_{N}} u ) 
        +
        {\textstyle\sum_{0 < i < N} }
        \trace_{K_{i},S}( \Pi^{k}_{F_{i+1}} u - \Pi^{k}_{F_{i}} u ) 
        +
        \trace_{K_{0},S}( \Pi^{k}_{F_{1}} u - \Pi^{k}_{K_{0}} u )
        .
    \end{align*}
    The traces $\trace_{K_{i},S} : \Cont^{0}\Lambda^{k}(K_{i}) \to \Cont^{0}\Lambda^{k}(S)$ are non-expansive. 
    By the triangle inequality and piecewise continuity, 
    \begin{align*}
        &
        \left\| \trace_{T_{S},S}( \Pi^{k}_{T_{S}} u ) - \trace_{T,S}( \Pi^{k}_{T} u ) \right\|_{\Lebesgue^{\infty}\Lambda^{k}(S,\metric)}
        \\&\quad 
        \leq 
        \left\| \Pi^{k}_{K_{N}} u - \Pi^{k}_{F_{N}} u \right\|_{\Lebesgue^{\infty}\Lambda^{k}(K_{N},\metric)}
        +
        {\textstyle\sum_{0 < i < N}}
        \left\| \Pi^{k}_{F_{i+1}} u - \Pi^{k}_{F_{i}} u \right\|_{\Lebesgue^{\infty}\Lambda^{k}(K_{i},\metric)}
        +
        \left\| \Pi^{k}_{F_{1}} u - \Pi^{k}_{K_{0}} u   \right\|_{\Lebesgue^{\infty}\Lambda^{k}(K_{0},\metric)}.
    \end{align*}
    By the inverse inequality~\eqref{math:inverse-inequality-lebesgue} and again the triangle inequality, 
    this amounts to a bound in terms of 
    \begin{align*}
        &
        \left\| u - \Pi^{k}_{K_{N}} u \right\|_{\Lebesgue^{2}\Lambda^{k}(K_{N},\metric)}
        +
        \left\| u - \Pi^{k}_{F_{N}} u \right\|_{\Lebesgue^{2}\Lambda^{k}(K_{N},\metric)}
        +
        \left\| u - \Pi^{k}_{F_{1}} u \right\|_{\Lebesgue^{2}\Lambda^{k}(K_{0},\metric)}
        +
        \left\| u - \Pi^{k}_{K_{0}} u \right\|_{\Lebesgue^{2}\Lambda^{k}(K_{0},\metric)}
        \\&\qquad \qquad 
        +     
        {\textstyle\sum_{0 < i < N}}
        \left( 
            \left\| u - \Pi^{k}_{F_{i+1}} u \right\|_{\Lebesgue^{2}\Lambda^{k}(K_{i  },\metric)}
            + 
            \left\| u - \Pi^{k}_{F_{i  }} u \right\|_{\Lebesgue^{2}\Lambda^{k}(K_{i},\metric)}
        \right)
        \\&
        \leq 
        \left\| u - \Pi^{k}_{K_{N}} u \right\|_{\Lebesgue^{2}\Lambda^{k}(K_{N},\metric)}
        +
        \left\| u - \Pi^{k}_{K_{0}} u \right\|_{\Lebesgue^{2}\Lambda^{k}(K_{0},\metric)}
        +
        \sqrt{2}
        \textstyle{ \sum_{1 \leq i \leq N} }
        \left\| u - \Pi^{k}_{F_{i}} u \right\|_{\Lebesgue^{2}\Lambda^{k}(\subsubdomain_{F_{i}},\metric)}
    \end{align*}
    up to a factor of $h_{T}^{-\frac n 2}$ and first-order shape parameters. 
    A scaling argument transforms these differences to reference settings:
    \begin{gather*}
        \left\| u - \Pi^{k}_{F_{i}} u \right\|_{\Lebesgue^{2}\Lambda^{k}(\subsubdomain_{F_{i}},\metric)}
        \leq 
        \widehat{C} 
        h_{F_{i}}^{ - k + \frac n 2 }
        \| \theta_{F_{i}}^{\pullbackast} u - \Pi^{k}_{\widehat\subsubdomain} \theta_{F_{i}}^{\pullbackast} u \|_{\Lebesgue^{2}\Lambda^{k}(\widehat\subsubdomain)}
        ,
        \qquad 
        1 \leq i \leq N
        ,
        \\
        \left\| u - \Pi^{k}_{K_{i}} u \right\|_{\Lebesgue^{2}\Lambda^{k}(K_{i},\metric)}
        \leq 
        \shapemeasure{\Mesh}^{k} \shapemeasuredeterminant{\Mesh}^{\frac 1 2}
        h_{K_{i}}^{ - k + \frac n 2 }
        \| \incl_{K_{i}}^{\pullbackast} u - \Pi^{k}_{\referencesimplex} \incl_{K_{i}}^{\pullbackast} u \|_{\Lebesgue^{2}\Lambda^{k}(\referencesimplex_{n})}
        ,
        \qquad 
        0 \leq i \leq N
        .
    \end{gather*}
    At this point, we cite the approximation result~\cite[Theorem~9.1]{licht2021local} and use quasi-optimality of the local projections:
    for some constant $C > 0$ that only depends on $n$ and $r$, 
    \begin{align*}
        \| \theta_{F_{i}}^{\pullbackast} u - \Pi^{k}_{\widehat\subsubdomain} \theta_{F_{i}}^{\pullbackast} u \|_{\Lebesgue^{2}\Lambda^{k}(\widehat\subsubdomain)}
        \leq 
        &
        C\Big(
        \| \incl_{K_{i}}^{\pullbackast} u - \Pi^{k}_{\referencesimplex} \incl_{K_{i}}^{\pullbackast} u \|_{\Lebesgue^{2}\Lambda^{k}(\referencesimplex_{n})}
        +
        \| \incl_{K_{i}}^{\pullbackast} \cartan u - \Pi^{k+1}_{\referencesimplex} \incl_{K_{i}}^{\pullbackast} \cartan u \|_{\Lebesgue^{2}\Lambda^{k+1}(\referencesimplex_{n})}
        \\&\qquad 
        +
        \| \incl_{K_{i-1}}^{\pullbackast} u - \Pi^{k}_{\referencesimplex} \incl_{K_{i-1}}^{\pullbackast} u \|_{\Lebesgue^{2}\Lambda^{k}(\referencesimplex_{n})}
        +
        \| \incl_{K_{i-1}}^{\pullbackast} \cartan u - \Pi^{k+1}_{\referencesimplex} \incl_{K_{i-1}}^{\pullbackast} \cartan u \|_{\Lebesgue^{2}\Lambda^{k+1}(\referencesimplex_{n})}
        \Big)
        .
    \end{align*}
    As in the proof of Theorem~\ref{theorem:bestapproximationestimate:sobolev}, the Bramble--Hilbert lemma gives 
    \begin{align*}
		&
		\| \incl_{K_{i}}^{\pullbackast} u - \Pi^{k}_{\referencesimplex} \incl_{K_{i}}^{\pullbackast} u \|_{\Lebesgue^{2}\Lambda^{k}(\referencesimplex_{n})}
		\leq
		C
        \lvert \incl_{K_{i}}^{\pullbackast} u \rvert_{\Sobolev^{m,2}\Lambda^{k}(\referencesimplex_{n})}
		,
        \qquad 
        0 \leq i \leq N
        ,
	\end{align*}
    and we apply Theorem~\ref{theorem:derivatives-of-pullbacks} to find 
	\begin{align*} 
        \lvert \incl_{K_{i}}^{\pullbackast} u \rvert_{\Sobolev^{m,2}\Lambda^{k}(\referencesimplex_{n})}
        &
        \leq
        C
        \sum_{j=0}^{m}
		P^{(k)}_{m,j} \left( 
            \|              \Jacobian\incl_{K_{i}} \|_{\Lebesgue^{\infty}}
            ,
            \| \nabla       \Jacobian\incl_{K_{i}} \|_{\Lebesgue^{\infty}}
            ,
            \dots
            ,
            \| \nabla^{m-j} \Jacobian\incl_{K_{i}} \|_{\Lebesgue^{\infty}}
        \right) 
        \shapemeasuredeterminant{\Mesh}^{1/2} h_{K_{i}}^{-\frac n 2}
		\| \nabla^{j} u \|_{\Lebesgue^{2}\Tensor^{k+j}(K_{i},\metric)}
		\\&\quad 
        \leq
        C
        h_{K_{i}}^{ k + m - \frac n 2 }
        \sum_{j = 0}^{m}
        \| \nabla^{j} u \|_{\Lebesgue^{2}\Tensor^{k+j}(K_{i},\metric)}
		.
	\end{align*}
    Here, the constants depend only on $m$, $n$, $r$, the first-order shape parameters, and the shape constants of $\Mesh$ up to order $m+1$.
    The terms involving $\cartan u$ are handled similarly, using that $\cartan u \in \Sobolev^{\ell,2}\Lambda^{k+1}(T',\metric)$ for every $T' \in \Mesh^{n}$. 
    Putting this all together, we use the local comparison of simplex diameters and the finite overlap property of the triangulation. The desired inequality follows.
\end{proof}

\begin{remark}
    The quasi-interpolants for non-smooth fields are important if the exact solution is too rough to allow the Sobolev embedding theorem (with acceptable constants), so that the canonical interpolant is not a good choice. Such situations arise, e.g., when the finite element problem involves rough coefficients. Non-smooth quasi-interpolants on Euclidean domains are well-known \cite[Chapter 4.8]{brennerscott} and have been researched extensively.
    They have been generalized to vector fields and differential forms~\cite{ern2017finite,licht2021local,licht2023averaging}. The generalized Scott--Zhang and Ern--Guermond--Oswald interpolants satisfy a piecewise Bramble--Hilbert lemma, which was first known for scalar functions~\cite{veeser2016approximating,camacho2015L2}. 
    Such lemmas are critical when handling non-affine meshes because transformations to affine reference settings typically break any higher interelement regularity conditions. 
    
    Theorem~\ref{theorem:ern-guermond-oswald} is a piecewise Bramble--Hilbert lemma for finite element spaces on manifolds. 
    The only global regularity needed is the existence of an exterior derivative, all other regularity exploited in the interpolant is piecewise. 
    The construction of the projection on differential forms on Euclidean domains~\cite{licht2021local,licht2023averaging} can be transferred to the reference setting without conceptual changes;
    however, in the proof above we take a shortcut and utilize the piecewise Bramble--Hilbert lemma already known for Euclidean settings within the reference face patch $\widehat\subsubdomain$, thus avoiding the technicalities of the full construction. 
\end{remark}

\begin{example}
    An example illustrates these estimates. 
    At this stage, the error analysis is completely analogous to finite element methods on polyhedral domains. 
    Suppose that we have a family of meshes $\Mesh_{h}$ with uniformly bounded shape parameters, indexed over the mesh size $h := \max_{ T \in \Mesh^{n}_{h} } h_{T}$, which is the largest diameter of any simplex in the triangulation. 
    Suppose that $f \in \Sobolev^{m,2}\Lambda^{k}(\Manifold,\metric)$ for some $m \in \bbN_{0}$. 
    Under the assumption of full elliptic regularity, 
    \begin{gather*}
        u \in \Sobolev^{m+2,2}\Lambda^{k}(\Manifold,\metric),
        \quad 
        P_{\frakB(\Manifold,\metric)} u \in \Sobolev^{m+2,2}\Lambda^{k}(\Manifold,\metric),
        \\
        \cartan u \in \Sobolev^{m+1,2}\Lambda^{k+1}(\Manifold,\metric),
        \quad 
        \sigma := \cartan^{\dualityast} u \in \Sobolev^{m+1,2}\Lambda^{k-1}(\Manifold,\metric),
        \quad 
        \cartan\sigma \in \Sobolev^{m,2}\Lambda^{k}(\Manifold,\metric),
        \quad 
        p \in \Sobolev^{m+2,2}\Lambda^{k}(\Manifold,\metric).
    \end{gather*}
    In addition, assume $m \leq r - 1$ and that the elliptic regularity and the finite element spaces ensure  
    \begin{align*}
        \coeff_{\delta} \in O(h^{2}), \quad \coeff_{\eta} \in O(h), \quad \coeff_{\mu} \in O(h^{2}).
	\end{align*}
	Then the combination of the estimates yields 
	\begin{gather*}
		\| \cartan( \sigma - \sigma_{h} ) \|_{\Lebesgue^{2}\Lambda^{k}(\Manifold,\metric)} \in O(h^{m}),
		\quad
		\| \sigma - \sigma_{h} \|_{\Lebesgue^{2}\Lambda^{k-1}(\Manifold,\metric)} \in O(h^{m+1}),
		\\
		\| \cartan( u - u_{h} ) \|_{\Lebesgue^{2}\Lambda^{k+1}(\Manifold,\metric)} \in O(h^{m+1}),
		\quad
		\| u - u_{h} \|_{\Lebesgue^{2}\Lambda^{k}(\Manifold,\metric)} \in O(h^{m+2}),
		\\
		\| p - p_{h} \|_{\Lebesgue^{2}\Lambda^{k}(\Manifold,\metric)} \in O(h^{m+2}).
	\end{gather*}
    Hence, all displayed error quantities converge at the optimal rates permitted by the chosen polynomial degrees.
    In this example, the big-$O$ notation includes constants that depend on the Poincar\'e--Friedrichs constant of $\Manifold$, the norm of the smoothed projection, the mesh regularity of $\Mesh$, and the polynomial degree. 
    In applications, we may assume that we work with families of intrinsic triangulations with uniformly bounded mesh constants, and then the smoothed projections are uniformly stable as well. 
\end{example}

\section{Computational triangulations}\label{section:computationaltriangulations}

Up to now, we have studied partial differential equations on a compact \emph{physical manifold} and intrinsic finite element methods that are defined with respect to an intrinsic triangulation.
However, this intrinsic finite element problem might not be computable in practice.
Practical computations such as in surface finite element methods can be formulated in terms of an approximate finite element problem on an extrinsic affine triangulation.
Although intrinsic triangulations are available in some special cases, practical computations generally must use extrinsic triangulations.
This section explains this transition to an implementable finite element problem.
Our discussion of the geometric error is inspired by the work of Holst and Stern but organizes the theory differently; see also Remark~\ref{remark:comparison-to-holst-and-stern} below.

\subsection{Extrinsic and intrinsic triangulations}

Let $\extMesh$ be an affine \emph{computational triangulation} of an $n$-dimensional \emph{computational manifold} $\extManifold$.
The latter manifold generally is a Lipschitz manifold in applications.
We equip every simplex $S \in \extMesh$ with the (intrinsic) Euclidean metric.
Let $\Manifold$ again be a connected compact Riemannian smooth manifold without boundary. 
Let $\homeo \colon \extManifold \to \Manifold$ be a bi-Lipschitz homeomorphism that is piecewise smooth, 
meaning that for every $S \in \extMesh$, the restriction $\homeo \restriction_{S} : S \to \homeo(S)$ is a smooth diffeomorphism of manifolds with corners.

This allows us to define the intrinsic triangulation $\Mesh$ on the physical manifold $\Manifold$:
if $\incl_{S} : \referencesimplex_{\Dim} \to \extManifold$ is any $\Dim$-dimensional affine simplex of the extrinsic triangulation, %
then the composition $\incl_{T} := \homeo \circ \incl_{S}$ is a smooth simplex with image in $\Manifold$. %
The resulting simplices form an intrinsic triangulation $\Mesh$ of the smooth manifold $\Manifold$.

In practice, we can control the shape parameters of the affine extrinsic triangulation and the derivatives of the homeomorphism up to a certain order. 
We now describe how this gives control of the shape parameters of the intrinsic triangulation. 
Consider any affine extrinsic simplex $S \in \extMesh$ with dimension $n$, and let $T := \homeo(S)$ be the corresponding intrinsic simplex. 
The composition $\incl_{T} := \homeo \circ \incl_{S}$ is bi-Lipschitz onto its image.
Recall that 
\begin{align*}
    L_{\homeo}^{-1} h_{T} \leq h_{S} \leq L_{\homeo} h_{T},
    \quad\text{ where }\quad 
    L_{\homeo} 
    := 
    \max\{ 
        \max_{ S \in \extMesh } \Lip( \homeo      \restriction_{S} ), 
        \max_{ T \in    \Mesh } \Lip( \homeo^{-1} \restriction_{T} ) 
    \}
    .
\end{align*}
Therefore, for all distinct $x, y \in \referencesimplex_{n}$,
\begin{align*}
    L_{\homeo}^{-1} \shapemeasurelipschitz{\extMesh}^{-1}
    h_{S} 
    \leq 
    \frac{
        \dist_{\metric}( \incl_{T}(x), \incl_{T}(y) )
    }{
        \dist( x, y )
    }
    \leq 
    L_{\homeo} \shapemeasurelipschitz{\extMesh} 
    h_{S}
    .
\end{align*}
The Jacobians and determinants are estimated by the chain rule:
\begin{gather*} 
    \| \Jacobian\incl_{T} \|_{\Lebesgue^{\infty}}
	\leq
    \| \Jacobian\homeo \|_{\Lebesgue^{\infty}}
    \| \Jacobian\incl_{S} \|_{\Lebesgue^{\infty}}
	\leq
    \| \Jacobian\homeo \|_{\Lebesgue^{\infty}}
    \shapemeasure{\incl_{S}}
	h_{S}
    ,
	\\
    \| \Jacobian(\incl_{T}^{-1}) \|_{\Lebesgue^{\infty}}
	\leq
	\| \Jacobian\homeo^{-1} \|_{\Lebesgue^{\infty}}
    \| \Jacobian\incl_{S}^{-1} \|_{\Lebesgue^{\infty}}
	\leq 
    \| \Jacobian\homeo^{-1} \|_{\Lebesgue^{\infty}}
    \shapemeasure{\incl_{S}}
	L_{\homeo}
    h_{T}^{-1}
    ,
    \\
    \| \determinant \Jacobian\incl_{T} \|_{\Lebesgue^{\infty}(\referencesimplex_{n})}
	\leq
	\| \determinant \Jacobian\homeo \|_{\Lebesgue^{\infty}(S)}
	\shapemeasuredeterminant{\incl_{S}}
	h_{S}^{n},
	\qquad 
    \| \determinant \Jacobian(\incl_{T}^{-1}) \|_{\Lebesgue^{\infty}(T)}
	\leq
	\| \determinant \Jacobian\homeo^{-1} \|_{\Lebesgue^{\infty}(T)}
	\shapemeasuredeterminant{\incl_{S}}
	h_{S}^{-n}.
\end{gather*}
In summary, 
\begin{align*} 
    \shapemeasurelipschitz{\Mesh} \leq \shapemeasurelipschitz{\extMesh} L_{\homeo}^{2}
    ,
    \qquad 
    \shapemeasure{\Mesh} \leq \shapemeasure{\extMesh} L_{\homeo}^{2}
    ,
    \qquad 
    \shapemeasuredeterminant{\Mesh} \leq \shapemeasuredeterminant{\extMesh} L_{\homeo}^{2n}
    .
\end{align*}
In addition, since the extrinsic and intrinsic meshes have the same combinatorial structure, 
\[
    \cardinalitybound{\Mesh} = \cardinalitybound{\extMesh}
    .
\]
Moreover, we estimate the comparison constant via 
\[
    \comparisonconstant{\Mesh}
    \leq 
    L_{\homeo}^{2} \comparisonconstant{\extMesh}
    .
\]
This shows that the first-order shape parameters of the intrinsic triangulation are controlled by the first-order parameters of the extrinsic triangulation and the bi-Lipschitz parameters of the homeomorphism. %

We now estimate the higher-order shape parameters of the intrinsic triangulation. 
This relies on a Fa\`a~di~Bruno formula for compositions of diffeomorphisms of manifolds, 
on the Jacobians of affine simplices being constant,
and on the piecewise Jacobians of the homeomorphism having bounded higher-order covariant derivatives. 

Let $m \in \bbN_{0}$ and let $\incl_{T} := \homeo \circ \incl_{S}$ be as above. 
Then 
\begin{align*}
    \| \nabla^{m} \Jacobian \incl_{T} \|_{\Lebesgue^{\infty}}
    &
    \leq 
    \| \nabla^{m} \Jacobian ( \homeo \restriction_{S} ) \|_{\Lebesgue^{\infty}}
    \| \Jacobian \incl_{S} \|_{\Lebesgue^{\infty}}^{m+1}
    \\&
    \leq 
    \| \nabla^{m} \Jacobian ( \homeo \restriction_{S} ) \|_{\Lebesgue^{\infty}}
    \shapemeasure{\incl_{S}}^{m+1}
	h_{S}^{m+1}
    \leq 
    \| \nabla^{m} \Jacobian ( \homeo \restriction_{S} ) \|_{\Lebesgue^{\infty}}
    \shapemeasure{\incl_{S}}^{m+1}
	L_{\homeo}^{m+1} 
    h_{T}^{m+1}
	.
\end{align*}
Therefore, the higher-order shape constants of the intrinsic triangulation are bounded by 
\[
    \highershapeconstant{m+1}{\Mesh}
	\leq 
        \max_{ S \in \extMesh } 
        \| \nabla^{m} \Jacobian ( \homeo \restriction_{S} ) \|_{\Lebesgue^{\infty}}
    \cdot 
    \shapemeasure{\extMesh}^{m+1}
	L_{\homeo}^{m+1} 
    .
\]

\begin{remark}
    The bi-Lipschitz map $\homeo$ goes from the computational manifold onto the physical manifold and preserves the triangulation structure.
    We assume bi-Lipschitz regularity because it ensures bounded pullback maps on the relevant Lebesgue spaces of tensor fields.
    Moreover, it has additional regularity on each simplex: 
    the $m$-th shape constant $\highershapeconstant{m}{\Mesh}$ is bounded in terms of the bi-Lipschitz constant and piecewise $\Cont^{m}$ regularity of $\homeo$ and the shape regularity of $\extMesh$. 
    This homeomorphism may be explicitly computable in some applications, but it may also be a purely theoretical device in the error analysis.
\end{remark}

\subsection{Extrinsic finite element de~Rham complexes}

This transformation, whether computationally available or not, plays the following role:
we transform the \emph{intrinsic finite element problem} on the original physical manifold to an equivalent \emph{extrinsic finite element problem} on the extrinsic manifold.
This is a simple change of variables that produces a mere rewriting of the intrinsic finite element problem.
However, a subsequent approximation of the pullback metric yields the computational problem that is actually implemented.

We formally define the Lebesgue spaces of differential forms conveniently as 
\[
    \Lebesgue^{p}\Lambda^{k}(\extManifold) := \bigoplus_{ S \in \extMesh^{n} } \Lebesgue^{p}\Lambda^{k}(S),
    \quad 
    p \in [1,\infty]
    .
\]
On each $n$-dimensional extrinsic simplex, these are spaces of sections of the $k$-th exterior power of the cotangent bundle of the simplex. 
Each local space is also a pullback $\Lebesgue^{p}\Lambda^{k}(S) = ( \homeo \restriction_{S} )^{\pullbackast} \Lebesgue^{p}\Lambda^{k}(T)$ from the physical manifold because the homeomorphism is piecewise bi-Lipschitz.

We transfer the Riemannian metric from the physical manifold to the computational manifold, piecewise by pullbacks on each simplex.
This defines the extrinsic metric on $\extManifold$ by 
\begin{gather*}
    \extrinsic{\metric} := \bigoplus_{ S \in \extMesh^{n} } \extrinsic{\metric} \restriction_{S}
    ,
    \qquad 
    \extrinsic{\metric} \restriction_{S} := ( \homeo \restriction_{S} )^{\pullbackast} \metric
    .
\end{gather*}
On any $S \in \extMesh^{n}$, the metric $\extrinsic{\metric} \restriction_{S}$ is a section of the second symmetric power of the cotangent bundle. 
The piecewise smooth Riemannian metric $\extrinsic{\metric}$ defines a norm on $\Lebesgue^{p}\Lambda^{k}(\extManifold)$. 
We define the inner product and norms 
\begin{gather}\label{math:piecewise-inner-product-and-norms}
	\langle {u}, {v} \rangle_{\extrinsic{\metric}}
	:=
	\sum_{ S \in \extMesh^{n} }
    \int_{S} \langle {u} \restriction_{S}, {v} \restriction_{S} \rangle_{\extrinsic{\metric} \restriction_{S}} \density^{S}_{\extrinsic{\metric} \restriction_{S}},
	\qquad 
    {u}, {v} \in \Lebesgue^{2}\Lambda^{k}(\extManifold)
    .
\end{gather}
We will use similar notation for any other piecewise metric.

\begin{remark}
    The manifold $\extManifold$ is embedded in Euclidean space and thus carries the structure of a Lipschitz manifold. This allows us to define the Lebesgue and Sobolev spaces of differential forms and the non-smooth Riemannian metric. 
    However, since $\extManifold$ only exists to enable an intermediate computational representation, the much simpler piecewise definition of Lebesgue spaces is sufficient for our purposes. 
    In particular, the finite element de~Rham complex is formally defined on any triangulation. 
\end{remark}

We have reviewed some analytical structures on the extrinsic triangulation by pullback from the physical manifold. 
However, the finite element de~Rham complex is defined genuinely on the extrinsic triangulation.
We recall the finite element de~Rham complex on the physical manifold,
\begin{gather}
	\begin{CD}
		0 \to
		\calP\Lambda^{0}(\Mesh)
		@>{\cartan}>>
		\calP\Lambda^{1}(\Mesh)
		@>{\cartan}>>
		\dots
		@>{\cartan}>>
		\calP\Lambda^{n}(\Mesh)
		\to
		0
		.
	\end{CD}
\end{gather}
We consider the corresponding finite element de~Rham complex on the extrinsic manifold:
\begin{gather} 
	\begin{CD}
		0 \to
		\calP\Lambda^{0}(\extMesh)
		@>{\cartan}>>
		\calP\Lambda^{1}(\extMesh)
		@>{\cartan}>>
		\dots
		@>{\cartan}>>
		\calP\Lambda^{n}(\extMesh)
		\to
		0
		.
	\end{CD}
\end{gather}
These are differential complexes of vector spaces.
The finite element spaces can be equipped with the inner product induced by the respective Riemannian metrics, making them finite-dimensional Hilbert spaces. 
We write $\calP\Lambda^{k}(\Mesh,\metric)$ for the space $\calP\Lambda^{k}(\Mesh)$ equipped with the inner product induced by $\metric$.
Likewise, $\calP\Lambda^{k}(\extMesh,\extrinsic{\metric})$ denotes the space $\calP\Lambda^{k}(\extMesh)$ equipped with the inner product induced by $\extrinsic{\metric}$.
Therefore, we may consider the intrinsic discrete complex with that metric 
\begin{gather}\label{math:intrinsic-discrete-complex-with-metric}
	\begin{CD}
		0 \to
		\calP\Lambda^{0}(\Mesh,\metric)
		@>{\cartan}>>
		\calP\Lambda^{1}(\Mesh,\metric)
		@>{\cartan}>>
		\dots
		@>{\cartan}>>
		\calP\Lambda^{n}(\Mesh,\metric)
		\to
		0
	\end{CD}
\end{gather}
and the extrinsic discrete complex with that metric 
\begin{gather}\label{math:extrinsic-discrete-complex-with-metric}
	\begin{CD}
		0 \to
		\calP\Lambda^{0}(\extMesh,\extrinsic{\metric})
		@>{\cartan}>>
		\calP\Lambda^{1}(\extMesh,\extrinsic{\metric})
		@>{\cartan}>>
		\dots
		@>{\cartan}>>
		\calP\Lambda^{n}(\extMesh,\extrinsic{\metric})
		\to
		0
        .
	\end{CD}
\end{gather}
The pullback maps are not only isomorphisms but also isometries because $\extrinsic{\metric} = \homeo^{\pullbackast}\metric$.
Consequently, any $\extrinsic u_{h}, \extrinsic v_{h} \in \calP\Lambda^{k}(\extMesh,\extrinsic{\metric})$ satisfy the identity of inner products 
\begin{align*}
	\langle \extrinsic{{u}}_{h}, \extrinsic{{v}}_{h} \rangle_{ \Lebesgue^{2}\Lambda^{k}(\extrinsic{\metric}) }
	=
    \left\langle 
        \homeo^{-\pullbackast} \extrinsic u_{h}, \homeo^{-\pullbackast} \extrinsic v_{h}
    \right\rangle
    _{\Lebesgue^{2}\Lambda^{k}(\Manifold,\metric)}
    .
\end{align*}
Since the pullback commutes with the exterior derivative, we have a commutative diagram
\begin{align*}
	\xymatrix@=3em{
	\cdots \ar[r]
	 &
	\calP\Lambda^{k}(\Mesh,\metric)
	\ar@<0.5em>[d]^{ \homeo_{   }^{\pullbackast} }
	\ar[r]^{\cartan}
	 &
	\calP\Lambda^{k+1}(\Mesh,\metric)
	\ar@<0.5em>[d]^{ \homeo_{   }^{\pullbackast} }
	\ar[r]
	 &
	\cdots \phantom{.}
	\\
	\cdots
	\ar[r]
	 &
	\calP\Lambda^{k}(\extMesh,\extrinsic{\metric})
	\ar@<0.5em>[u]^{ \homeo_{   }^{-\pullbackast} }
	\ar[r]^{\cartan}
	 &
	\calP\Lambda^{k+1}(\extMesh,\extrinsic{\metric})
	\ar@<0.5em>[u]^{ \homeo_{   }^{-\pullbackast} }
	\ar[r]
	 &
	\cdots
    ,
	}
\end{align*}
where the vertical arrows are isometries. 
A~fortiori, the Poincar\'e--Friedrichs inequalities in the Hilbert complexes~\eqref{math:intrinsic-discrete-complex-with-metric} and~\eqref{math:extrinsic-discrete-complex-with-metric} hold with the same constant. 
Moreover, the discrete harmonic spaces of the extrinsic and intrinsic finite element de~Rham complexes,
\begin{align*}
	\frakH^{k}(\extMesh,\extrinsic{\metric}) 
    &
    := 
    \left\{\; 
        \xi_{h} \in \calP\Lambda^{k}(\extMesh,\extrinsic{\metric}) 
        \suchthat* 
        \cartan \xi_{h} = 0 \text{ and } \xi_{h} \perp_{\extrinsic{\metric}} \cartan \calP\Lambda^{k-1}(\extMesh) 
    \;\right\}
	,
    \\ 
    \frakH^{k}(\Mesh,\metric) 
    &
    := %
    \left\{\; 
        \xi_{h} \in \calP\Lambda^{k}(\Mesh,\metric) 
        \suchthat* 
        \cartan \xi_{h} = 0 \text{ and } \xi_{h} \perp_{\metric} \cartan \calP\Lambda^{k-1}(\Mesh,\metric) 
    \;\right\}
    ,
\end{align*}
are isomorphically mapped onto each other by pullback along the homeomorphism.

We consider the following finite element Hodge--Laplace equation on the extrinsic triangulation.
Given $\extrinsic{f} \in \Lebesgue^{2}\Lambda^{k}(\extManifold,\extrinsic{\metric})$, 
we seek $\extrinsic{\sigma}_{h} \in \calP\Lambda^{k-1}(\extMesh,\extrinsic{\metric})$, $\extrinsic{{u}}_{h} \in \calP\Lambda^{k}(\extMesh,\extrinsic{\metric})$, and $\extrinsic p_{h} \in \frakH^{k}(\extMesh,\extrinsic{\metric})$ such that
\begin{subequations}
	\begin{alignat}{2}
		\langle \extrinsic{\sigma}_{h}, \extrinsic{\tau}_{h} \rangle_{\extrinsic{\metric}}
		-
		\langle \extrinsic{{u}}_{h}, \cartan \extrinsic{\tau}_{h} \rangle_{\extrinsic{\metric}}
		                                                                                      & = 0,
		                                                                                      & \quad \extrinsic{\tau}_{h}                            & \in \calP\Lambda^{k-1}(\extMesh,\extrinsic{\metric}),
		\\
		\langle \cartan \extrinsic{\sigma}_{h}, \extrinsic{{v}}_{h} \rangle_{\extrinsic{\metric}}
		+
		\langle \cartan \extrinsic{{u}}_{h}, \cartan \extrinsic{{v}}_{h} \rangle_{\extrinsic{\metric}}
		+
		\langle \extrinsic{p}_{h}, \extrinsic{{v}}_{h} \rangle_{\extrinsic{\metric}}
		                                                                                      & =
		\langle \extrinsic{f}, \extrinsic{{v}}_{h} \rangle_{\extrinsic{\metric}},
		                                                                                      & \quad
		\extrinsic{{v}}_{h}                                                                     & \in \calP\Lambda^{k}(\extMesh,\extrinsic{\metric}),
		\\
		\langle \extrinsic{{u}}_{h}, \extrinsic{q}_{h} \rangle_{\extrinsic{\metric}} & = 0,
		                                                                                      & \quad
		\extrinsic{q}_{h}                                                                       & \in \frakH^{k}(\extMesh,\extrinsic{\metric})
		.
	\end{alignat}
\end{subequations}
This problem has a unique solution, which is the pullback of the intrinsic finite element solution along $\homeo$.

\subsection{The computational problem and its geometric error}

Up to this point, we have been working on the extrinsic finite element de~Rham complex on $\extMesh$ with the exact pullback metric $\extrinsic{\metric}$, so that the extrinsic and intrinsic finite element problems are equivalent.
As a consequence, the extrinsic problem is generally not computable, being equivalent to the original problem:
the exact metric and the exact right-hand side might be unavailable. 

However, having formulated the finite element problem on the extrinsic triangulation, we can now transition to a computable finite element problem. 
We replace the \emph{exact} extrinsic finite element problem by the \emph{computational} extrinsic finite element problem that is actually implemented.
In our formalism, we approximate the exact Riemannian metric $\extrinsic{\metric}$ by a computational extrinsic Riemannian metric $\compute{\metric}$.
Implementations solve the computational problem posed on the extrinsic manifold $\extManifold$ using this approximate computational metric tensor $\compute{\metric}$.
This transition introduces additional error terms because the computational problem is an approximation of the exact extrinsic problem. 

There are several ways to define the computational metric and the computational problem, and examples will be provided later.
For now, we remain at an abstract level of discussion and analyze the \emph{geometric error} introduced by replacing the exact pullback metric and right-hand side by computational approximations.
We retain the same algebraic finite element de~Rham complex on $\extMesh$ but replace the exact pullback metric $\extrinsic{\metric}$ by the computational metric $\compute{\metric}$.
Formally, 
\[
    \compute{\metric} := \bigoplus_{ S \in \extMesh^{n} } ( \compute{\metric} \restriction_{S} )
    .
\]
In applications, $\compute{\metric}$ is a computationally feasible approximation of the true metric tensor $\extrinsic{\metric}$.
We can assume that $\compute{\metric} \restriction_{S}$ is a smooth section of the second symmetric power of the cotangent bundle of $S$ that is bounded and coercive on each simplex with respect to the extrinsic metric. 

Define the inner product 
\begin{align}\label{math:inner-product-computational-metric}
	\langle \extrinsic{{u}}, \extrinsic{{v}} \rangle_{\compute{\metric}}
	:=
    \sum_{ S \in \extMesh^{n} }
    \left\langle 
        \extrinsic u \restriction_{S}, \extrinsic v \restriction_{S}
    \right\rangle
    _{\Lebesgue^{2}\Lambda^{k}(S,\compute{\metric})}
    ,
    \qquad 
    \extrinsic{{u}}, \extrinsic{{v}} \in \Lebesgue^{2}\Lambda^{k}(\extManifold,\compute{\metric})
    .
\end{align}
We define the metric comparison constants 
\begin{align}\label{math:metric-comparison-constants}
    \CAri{k}
    :=
    \sup_{ \substack{ 
        {u} \in \calP\Lambda^{k  }(\extMesh)
        \\
        \| {u} \|_{\compute{\metric}} \leq 1
    } }
    \| {u} \|_{\extrinsic{\metric}}
    ,
    \qquad 
    \CinvAri{k}
    :=
    \sup_{ \substack{ 
        {u} \in \calP\Lambda^{k  }(\extMesh)
        \\
        \| {u} \|_{\extrinsic{\metric}} \leq 1
    } }
    \| {u} \|_{\compute{\metric}}
    .
\end{align}
There exists a linear operator
\begin{align}\label{math:gram-operator}
    \GramOperator_{k} : \calP\Lambda^{k}(\extMesh) \to \calP\Lambda^{k}(\extMesh)
\end{align}
uniquely defined by the Gram property 
\begin{align}\label{math:gram-property}
    \left\langle u_{h}, v_{h} \right\rangle_{\extrinsic{\metric}}
    = 
    \left\langle \GramOperator_{k} u_{h}, v_{h} \right\rangle_{\compute{\metric}},
    \qquad 
    u_{h}, v_{h} \in \calP\Lambda^{k}(\extMesh)
    .
\end{align}
We remark that its operator norm is bounded by $\CAri{k}^{2}$, as follows from 
\begin{align}\label{math:gram-operator-bound}
    \| \GramOperator_{k} u_{h} \|_{\compute{\metric}}^{2}
    =
    \left\langle \GramOperator_{k} u_{h}, u_{h} \right\rangle_{\extrinsic{\metric}}
    \leq 
    \CAri{k}^{2}
    \| \GramOperator_{k} u_{h} \|_{\compute{\metric}} 
    \| u_{h} \|_{\compute{\metric}},
    \quad 
    u_{h} \in \calP\Lambda^{k}(\extMesh)
    . 
\end{align}
Note that the operators $\GramOperator_{k}$ are self-adjoint with respect to the $\compute{\metric}$-induced inner products, while their inverses are self-adjoint with respect to the $\extrinsic{\metric}$-induced inner products.
Define the \emph{metric distortion errors}
\begin{align}\label{math:metric-distortion-errors}
    \CGramError{k}
    :=
    \sup_{
        \substack{
            v_h \in \calP\Lambda^{k}(\extMesh) \\ \| v_h \|_{\Lebesgue^{2}\Lambda^{k}(\extManifold,\compute{\metric})} \leq 1
        }
    }
    \| v_h - \GramOperator_{k} v_h \|_{\Lebesgue^{2}\Lambda^{k}(\extManifold,\compute{\metric})}
    ,
    \qquad 
    \CinvGramError{k}
    :=
    \sup_{
        \substack{
            v_h \in \calP\Lambda^{k}(\extMesh) \\ \| v_h \|_{\Lebesgue^{2}\Lambda^{k}(\extManifold,\extrinsic{\metric})} \leq 1
        }
    }
    \| v_h - \GramOperator_{k}^{-1} v_h \|_{\Lebesgue^{2}\Lambda^{k}(\extManifold,\extrinsic{\metric})}
    .
\end{align}
We therefore observe 
\begin{subequations}\label{math:metric-distortion-errors-estimate}
\begin{align}
    \CGramError{k}
    &
    = 
    \sup\left\{ 
        \Big\lvert
		\langle {v_h} - \GramOperator_{k} {v_h}, {v_h} \rangle_{\Lebesgue^{2}\Lambda^{k}(\extManifold,\compute{\metric})}
		\Big\rvert
        \suchthat* 
        {v_h} \in \calP\Lambda^{k}(\extMesh),
        \, 
        \| {v_h} \|_{\Lebesgue^{2}\Lambda^{k}(\extManifold,\compute{\metric})} \leq 1
    \right\}
    \\&
    =
    \sup\left\{ 
        \Big\lvert
            \langle {v_h}, {v_h} \rangle_{\Lebesgue^{2}\Lambda^{k}(\extManifold,\compute{\metric})}
            -
            \langle {v_h}, {v_h} \rangle_{\Lebesgue^{2}\Lambda^{k}(\extManifold,\extrinsic{\metric})}
		\Big\rvert
        \suchthat* 
        {v_h} \in \calP\Lambda^{k}(\extMesh),
        \, 
        \| {v_h} \|_{\Lebesgue^{2}\Lambda^{k}(\extManifold,\compute{\metric})} \leq 1
    \right\}
\end{align}
\end{subequations}
and analogously for $\CinvGramError{k}$. 
In applications, bounding these degree-wise metric distortion quantities controls the geometric contribution to the variational crime.

The finite element spaces and the finite element de~Rham complexes remain unchanged algebraically, but we change the Hilbert space structures. 
We write $\calP\Lambda^{k}(\extMesh,\compute{\metric})$ for the finite element space $\calP\Lambda^{k}(\extMesh)$ equipped with the inner product induced by $\compute{\metric}$.
Since the exterior derivative is independent of the metric, the finite element de~Rham complex becomes a Hilbert complex with the new inner products:
\begin{gather}\label{math:computational-discrete-complex-with-metric}
	\begin{CD}
		\dots
		@>\cartan>>
		\calP\Lambda^{k  }(\extMesh,\compute{\metric})
		@>\cartan>>
		\calP\Lambda^{k+1}(\extMesh,\compute{\metric})
		@>\cartan>>
		\dots
	\end{CD}
\end{gather}
The computational differential complex satisfies the following stability properties.
On the one hand, we have Poincar\'e--Friedrichs inequalities with Poincar\'e--Friedrichs constants
\begin{align}
    \Ccomputablepoincareweber{k}
	\leq 
    \CinvAri{k} \CAri{k+1}
    \Cdiscretepoincareweber{k}
    .
\end{align}
We also define the computational harmonic space by
\begin{align*}
	\frakH^{k}(\extMesh,\compute{\metric}) 
    := 
    \left\{\; 
        \xi_{h} \in \calP\Lambda^{k}(\extMesh) \suchthat* \cartan \xi_{h} 
        = 
        0 
        \text{ and } 
        \xi_{h} \perp_{\compute{\metric}} \cartan \calP\Lambda^{k-1}(\extMesh) 
    \;\right\}
	.
\end{align*}
In this finite-dimensional setting, $\dim \frakH^{k}(\extMesh,\compute{\metric}) = \dim \frakH^{k}(\extMesh,\extrinsic{\metric})$ is obvious. 
We are interested in a comparison of these two spaces of harmonic finite element differential forms. 
Let 
\[
    \extrinsic{P_{\frakH}} : \calP\Lambda^{k}(\extMesh,\extrinsic{\metric}) \to \frakH^{k}(\extMesh,\extrinsic{\metric}),
    \qquad 
    \compute{P_{\frakH}} : \calP\Lambda^{k}(\extMesh,\compute{\metric}) \to \frakH^{k}(\extMesh,\compute{\metric})
\]
be the $\extrinsic{\metric}$-orthogonal projection onto the extrinsic discrete harmonic forms
and the $\compute{\metric}$-orthogonal projection onto the computational discrete harmonic forms, respectively.
Define 
\begin{align*}
    \compute{\Charmoniccomparison}_{k} 
    :=
    \sup_{ \substack{ \compute{q}_{h} \in \frakH^{k}(\extMesh,\compute{\metric}) \\ \| \compute{q}_{h} \|_{\compute{\metric}} \leq 1 } }
    \| \compute{q}_{h} - \extrinsic{P_{\frakH}} \compute{q}_{h} \|_{\extrinsic{\metric}}
    ,
    \qquad 
    \extrinsic{\Charmoniccomparison}_{k} 
    :=
    \sup_{ \substack{ \extrinsic{q}_{h} \in \frakH^{k}(\extMesh,\extrinsic{\metric}) \\ \| \extrinsic{q}_{h} \|_{\extrinsic{\metric}} \leq 1 } }
    \| \extrinsic{q}_{h} - \compute{P_{\frakH}} \extrinsic{q}_{h} \|_{\compute{\metric}}
    .
\end{align*}
The harmonic error terms measure the deviation of the computational and extrinsic harmonic forms.

\begin{lemma}\label{lemma:harmoniccomparison}
    Let $\compute{\Charmoniccomparison}_{k}$ and $\extrinsic{\Charmoniccomparison}_{k}$ be defined as above. 
    Then $\extrinsic{\Charmoniccomparison}_k \leq \CAri{k} \CinvGramError{k}$ and $\compute{\Charmoniccomparison}_k \leq \CinvAri{k} \CGramError{k}$.
\end{lemma}
\begin{proof}
    First, consider any non-zero $\compute{q}_{h} \in \frakH^{k}(\extMesh,\compute{\metric})$. 
    Define $\extrinsic{b}_{h} := ( \Id - \extrinsic{P_{\frakH}} ) \compute{q}_{h}$. 
    Then $\extrinsic{b}_{h} \in \cartan \calP\Lambda^{k-1}(\extMesh)$. Thus, we have $\langle \extrinsic{b}_{h}, \compute{q}_{h} \rangle_{\compute{\metric}} = 0$.
    We compute the squared norm 
    \begin{align*}
        \| \extrinsic{b}_{h} \|_{\extrinsic{\metric}}^{2}
        &
        =
        \langle \extrinsic{b}_{h}, \extrinsic{b}_{h} \rangle_{\extrinsic{\metric}}
        =
        \langle \extrinsic{b}_{h}, \compute{q}_{h} \rangle_{\extrinsic{\metric}}
        =
        \langle \extrinsic{b}_{h}, \GramOperator_{k} \compute{q}_{h} \rangle_{\compute{\metric}}
        =
        \langle \extrinsic{b}_{h}, \GramOperator_{k} \compute{q}_{h} - \compute{q}_{h} \rangle_{\compute{\metric}}
        . 
    \end{align*}
    We apply the Cauchy--Schwarz inequality together with~\eqref{math:metric-comparison-constants} and~\eqref{math:metric-distortion-errors} to obtain
    \[
        \| \extrinsic{b}_{h} \|_{\extrinsic{\metric}}^{2}
        \leq
        \| \extrinsic{b}_{h} \|_{\compute{\metric}}
        \| \GramOperator_{k} \compute{q}_{h} - \compute{q}_{h} \|_{\compute{\metric}}
        \leq
        \CinvAri{k} 
        \| \extrinsic{b}_{h} \|_{\extrinsic{\metric}}
        \| \GramOperator_{k} \compute{q}_{h} - \compute{q}_{h} \|_{\compute{\metric}}
        \leq
        \CinvAri{k} 
        \CGramError{k} 
        \| \extrinsic{b}_{h} \|_{\extrinsic{\metric}}
        \| \compute{q}_{h} \|_{\compute{\metric}}
        .
    \]
    This yields the desired bound for $\compute{\Charmoniccomparison}_{k}$.
    The bound for $\extrinsic{\Charmoniccomparison}_{k}$ follows by symmetry when we switch the roles of the computational and extrinsic metrics.
\end{proof}

Suppose that we have discrete data $\compute{f}_{h} \in \Lebesgue^{2}\Lambda^{k}(\extManifold)$.
We seek $\compute{\sigma}_{h} \in \calP\Lambda^{k-1}(\extMesh,\compute{\metric})$, $\compute{u}_{h} \in \calP\Lambda^{k}(\extMesh,\compute{\metric})$, and $\compute{p}_{h} \in \frakH^{k}(\extMesh,\compute{\metric})$ 
satisfying the following computational mixed formulation, which is the one actually implemented:
\begin{subequations}
	\begin{alignat}{2}
		\langle \compute{\sigma}_{h}, \compute{\tau}_{h} \rangle_{\compute{\metric}}
		-
		\langle \compute{{u}}_{h}, \cartan \compute{\tau}_{h} \rangle_{\compute{\metric}}
		                                                                                & = 0,
		                                                                                & \quad \compute{\tau}_{h}                            & \in \calP\Lambda^{k-1}(\extMesh,\compute{\metric}),
		\\
		\langle \cartan \compute{\sigma}_{h}, \compute{{v}}_{h} \rangle_{\compute{\metric}}
		+
		\langle \cartan \compute{{u}}_{h}, \cartan \compute{{v}}_{h} \rangle_{\compute{\metric}}
		+
		\langle \compute{p}_{h}, \compute{{v}}_{h} \rangle_{\compute{\metric}}
		                                                                                & =
		\langle \compute{f}_{h}, \compute{{v}}_{h} \rangle_{\compute{\metric}},
		                                                                                & \quad
		\compute{{v}}_{h}                                                                 & \in \calP\Lambda^{k}(\extMesh,\compute{\metric}),
		\\
		\langle \compute{{u}}_{h}, \compute{q}_{h} \rangle_{\compute{\metric}} & = 0,
		                                                                                & \quad
		\compute{q}_{h}                                                                   & \in \frakH^{k}(\extMesh,\compute{\metric})
		.
	\end{alignat}
\end{subequations}
This problem has a unique solution satisfying
\begin{align}\label{math:stability-of-computational-problem}
    \begin{split} 
        &
        \| \compute{\sigma}_{h} \|_{\Lebesgue^{2}\Lambda^{k-1}(\extManifold,\compute{\metric})}
        +
        \| \cartan \compute{\sigma}_{h} \|_{\Lebesgue^{2}\Lambda^{k}(\extManifold,\compute{\metric})}
        \\&\qquad 
        +
        \| \compute{{u}}_{h} \|_{\Lebesgue^{2}\Lambda^{k}(\extManifold,\compute{\metric})}
        +
        \| \cartan \compute{{u}}_{h} \|_{\Lebesgue^{2}\Lambda^{k+1}(\extManifold,\compute{\metric})}
        +
        \| \compute{p}_{h} \|_{\Lebesgue^{2}\Lambda^{k}(\extManifold,\compute{\metric})}
        \leq
        C \| \compute{f}_{h} \|_{\Lebesgue^{2}\Lambda^{k}(\extManifold,\compute{\metric})}
        .
    \end{split}
\end{align}
Here, $C$ depends only on the computational Poincar\'e--Friedrichs constants $\Ccomputablepoincareweber{k-1}$ and $\Ccomputablepoincareweber{k}$.

Having established the computational finite element approximation and its solvability, we now address the effect of the variational crime.
Our analysis is guided in spirit by Theorem~3.10 of~\cite{HS1}, but we organize the analysis of the variational crime in a different fashion. 
Define the data consistency error by
\begin{align}\label{math:data-consistency-error}
    \eta_{f,h}
    :=
    \sup_{
        \substack{
            v_h \in \calP\Lambda^{k}(\extMesh) \\ \|v_h\|_{\Lebesgue^{2}\Lambda^{k}(\extManifold,\compute{\metric})} \leq 1
        }
    }
    \lvert
        \langle\compute f_h,v_h\rangle_{\compute{\metric}}
        -
        \langle\extrinsic f,v_h\rangle_{\extrinsic{\metric}}
    \rvert
    .
\end{align}
Estimates for the data consistency error depend on how $\compute{f}_{h}$ is constructed, and we refer to Theorem~3.12 of~\cite{HS1} for some theoretical results.
The difference between the exact extrinsic finite element solution and the approximate computational finite element solution is now addressed in the following theorem: 
in summary, the additional error terms depend on the metric distortion constants~\eqref{math:metric-comparison-constants} and metric distortion errors~\eqref{math:metric-distortion-errors}.

\begin{theorem}\label{theorem:variationalcrime}
    Let $(\extrinsic{\sigma}_{h},\extrinsic{u}_{h},\extrinsic{p}_{h})$ be the solution to the exact extrinsic finite element problem,
    and let $(\compute{\sigma}_{h},\compute{u}_{h},\compute{p}_{h})$ be the solution to the computational finite element problem. 
    
    Then there exists a constant $C > 0$, depending only on the computational Poincar\'e--Friedrichs constants in degrees $k-1$ and $k$, 
    such that
    \begin{align*}
        &
        \|\compute{\sigma}_{h}-\extrinsic{\sigma}_{h}\|_{\Lebesgue^{2}\Lambda^{k-1}(\extManifold,\compute{\metric})}
        +
        \|\cartan\compute{\sigma}_{h}-\cartan\extrinsic{\sigma}_{h}\|_{\Lebesgue^{2}\Lambda^{k}(\extManifold,\compute{\metric})}
        \\
        &\quad
        +
        \|\compute{u}_{h}-\extrinsic{u}_{h}\|_{\Lebesgue^{2}\Lambda^{k}(\extManifold,\compute{\metric})}
        +
        \|\cartan\compute{u}_{h}-\cartan\extrinsic{u}_{h}\|_{\Lebesgue^{2}\Lambda^{k+1}(\extManifold,\compute{\metric})}
        +
        \|\compute{p}_{h}-\extrinsic{p}_{h}\|_{\Lebesgue^{2}\Lambda^{k}(\extManifold,\compute{\metric})}
        \\
        &\leq
        C\Bigg(
            \CGramError{k-1}
            \|\extrinsic{\sigma}_{h}\|_{\Lebesgue^{2}\Lambda^{k-1}(\extManifold,\compute{\metric})}
            +
            \CGramError{k}
            \Big(
                \|\extrinsic{u}_{h}\|_{\Lebesgue^{2}\Lambda^{k}(\extManifold,\compute{\metric})}
                +
                \|\cartan\extrinsic{\sigma}_{h}\|_{\Lebesgue^{2}\Lambda^{k}(\extManifold,\compute{\metric})}
                +
                \|\extrinsic{p}_{h}\|_{\Lebesgue^{2}\Lambda^{k}(\extManifold,\compute{\metric})}
            \Big)
            \\
            &\qquad\qquad\qquad  
            +
            \CGramError{k+1}
            \|\cartan\extrinsic{u}_{h}\|_{\Lebesgue^{2}\Lambda^{k+1}(\extManifold,\compute{\metric})}
            +
            \compute{\Charmoniccomparison}_{k}
            \|\extrinsic{u}_{h}\|_{\Lebesgue^{2}\Lambda^{k}(\extManifold,\extrinsic{\metric})}
            +
            \extrinsic{\Charmoniccomparison}_{k}
            \|\extrinsic{p}_{h}\|_{\Lebesgue^{2}\Lambda^{k}(\extManifold,\extrinsic{\metric})}
            +
            \eta_{f,h}
        \Bigg).
    \end{align*}
\end{theorem}

\begin{proof} 
    We combine the extrinsic and computational formulations. 
    For all $\tau_{h} \in \calP\Lambda^{k-1}(\extMesh)$, 
    \begin{align*}
        \langle \compute{\sigma}_{h} - \extrinsic{\sigma}_{h}, \tau_{h} \rangle_{\compute{\metric}}
		-
		\langle \compute{{u}}_{h} - \extrinsic{{u}}_{h}, \cartan \tau_{h} \rangle_{\compute{\metric}}
        &
        =
        \langle \extrinsic{\sigma}_{h}, \tau_{h} \rangle_{\extrinsic{\metric}}
		-
		\langle \extrinsic{{u}}_{h}, \cartan \tau_{h} \rangle_{\extrinsic{\metric}}
        -
        \langle \extrinsic{\sigma}_{h}, \tau_{h} \rangle_{\compute{\metric}}
		+
		\langle \extrinsic{{u}}_{h}, \cartan \tau_{h} \rangle_{\compute{\metric}}
        \\&
        =
        \langle \GramOperator_{k-1} \extrinsic{\sigma}_{h} - \extrinsic{\sigma}_{h}, \tau_{h} \rangle_{\compute{\metric}}
		-
		\langle \GramOperator_{k} \extrinsic{{u}}_{h} - \extrinsic{{u}}_{h}, \cartan \tau_{h} \rangle_{\compute{\metric}}
        .
    \end{align*}
    For all ${{v}}_{h} \in \calP\Lambda^{k}(\extMesh)$, we find 
    \begin{align*}
        &
        \langle \cartan \compute{\sigma}_{h} - \cartan \extrinsic{\sigma}_{h}, {{v}}_{h} \rangle_{\compute{\metric}}
		+
        \langle \cartan \compute{{u}}_{h} - \cartan \extrinsic{{u}}_{h}, \cartan {{v}}_{h} \rangle_{\compute{\metric}}
		+
        \langle \compute{p}_{h} - \compute{P_{\frakH}} \extrinsic{p}_{h}, {{v}}_{h} \rangle_{\compute{\metric}}
        \\& 
        =
		\langle \compute{f}_{h}, {{v}}_{h} \rangle_{\compute{\metric}}
        -
        \langle \extrinsic{f}, {{v}}_{h} \rangle_{\extrinsic{\metric}}
        +
        \langle \cartan \extrinsic{\sigma}_{h}, {{v}}_{h} \rangle_{\extrinsic{\metric}}
		-
        \langle \cartan \extrinsic{\sigma}_{h}, {{v}}_{h} \rangle_{\compute{\metric}}
		+
		\langle \cartan \extrinsic{{u}}_{h}, \cartan {{v}}_{h} \rangle_{\extrinsic{\metric}}
		-
		\langle \cartan \extrinsic{{u}}_{h}, \cartan {{v}}_{h} \rangle_{\compute{\metric}}
		+
		\langle \extrinsic{p}_{h}, {{v}}_{h} \rangle_{\extrinsic{\metric}}
        -
		\langle \compute{P_{\frakH}} \extrinsic{p}_{h}, {{v}}_{h} \rangle_{\compute{\metric}}
        \\& 
        =
		\langle \compute{f}_{h}, {{v}}_{h} \rangle_{\compute{\metric}}
        -
        \langle \extrinsic{f}, {{v}}_{h} \rangle_{\extrinsic{\metric}}
        +
        \langle \GramOperator_{k} \cartan \extrinsic{\sigma}_{h} - \cartan \extrinsic{\sigma}_{h}, {{v}}_{h} \rangle_{\compute{\metric}}
		+
		\langle \GramOperator_{k+1} \cartan \extrinsic{{u}}_{h} - \cartan \extrinsic{{u}}_{h}, \cartan {{v}}_{h} \rangle_{\compute{\metric}}
		+
		\langle \GramOperator_{k} \extrinsic{p}_{h} - \compute{P_{\frakH}} \extrinsic{p}_{h}, {{v}}_{h} \rangle_{\compute{\metric}}
        .
    \end{align*}
    Note that 
    \begin{align*}
        \langle \GramOperator_{k} \extrinsic{p}_{h} - \compute{P_{\frakH}} \extrinsic{p}_{h}, {{v}}_{h} \rangle_{\compute{\metric}}
        =
        \langle \GramOperator_{k} \extrinsic{p}_{h} - \extrinsic{p}_{h}, {{v}}_{h} \rangle_{\compute{\metric}}
        +
		\langle \extrinsic{p}_{h} - \compute{P_{\frakH}} \extrinsic{p}_{h}, {{v}}_{h} \rangle_{\compute{\metric}}
        .
    \end{align*}
    Finally, for any $\compute{q}_{h} \in \frakH^{k}(\extMesh,\compute{\metric})$, we find  
    \begin{align*}
        \langle \compute{{u}}_{h} - \extrinsic{{u}}_{h}, \compute{q}_{h} \rangle_{\compute{\metric}}
        &
        = 
        \langle \extrinsic{{u}}_{h}, \extrinsic{P_{\frakH}} \compute{q}_{h} \rangle_{\extrinsic{\metric}} 
        -
        \langle \extrinsic{{u}}_{h}, \compute{q}_{h} \rangle_{\compute{\metric}}
        \\&
        = 
        \langle \GramOperator_{k} \extrinsic{{u}}_{h}, \extrinsic{P_{\frakH}} \compute{q}_{h} \rangle_{\compute{\metric}} 
        -
        \langle \extrinsic{{u}}_{h}, \compute{q}_{h} \rangle_{\compute{\metric}}
        \\&
        = 
        \langle \GramOperator_{k} \extrinsic{{u}}_{h}, \extrinsic{P_{\frakH}} \compute{q}_{h} - \compute{q}_{h} \rangle_{\compute{\metric}} 
        +
        \langle \GramOperator_{k} \extrinsic{{u}}_{h} - \extrinsic{{u}}_{h}, \compute{q}_{h} \rangle_{\compute{\metric}}
        = 
        \langle \extrinsic{{u}}_{h}, \extrinsic{P_{\frakH}} \compute{q}_{h} - \compute{q}_{h} \rangle_{\extrinsic{\metric}} 
        +
        \langle \GramOperator_{k} \extrinsic{{u}}_{h} - \extrinsic{{u}}_{h}, \compute{q}_{h} \rangle_{\compute{\metric}}
        .
    \end{align*}
    We use the stability~\eqref{math:stability-of-computational-problem} of the computational problem:
    \begin{align*}
        &
        \| \compute{{u}}_{h} - \extrinsic{{u}}_{h} \|_{\compute{\metric}}
        +
        \| \cartan \compute{{u}}_{h} - \cartan \extrinsic{{u}}_{h} \|_{\compute{\metric}}
        +
        \| \compute{\sigma}_{h} - \extrinsic{\sigma}_{h} \|_{\compute{\metric}}
        +
        \| \cartan \compute{\sigma}_{h} - \cartan \extrinsic{\sigma}_{h} \|_{\compute{\metric}}
        +
        \| \compute{p}_{h} - \compute{P_{\frakH}} \extrinsic{p}_{h} \|_{\compute{\metric}}
        \leq C \cdot A
        ,
    \end{align*}
    where 
    \begin{align*}
        &
        A := 
            \| \Id - \GramOperator_{k-1} \| \cdot \| \extrinsic{\sigma}_{h} \|_{\compute{\metric}}
            +
            \| \Id - \GramOperator_{k  } \| \cdot \| \extrinsic{{u}}_{h} \|_{\compute{\metric}}
            +
            \eta_{f,h}
        \\&\qquad \qquad 
            +
            \| \Id - \GramOperator_{k} \| \cdot \| \cartan \extrinsic{\sigma}_{h} \|_{\compute{\metric}}
            +
            \| \Id - \GramOperator_{k+1} \| \cdot \| \cartan \extrinsic{{u}}_{h} \|_{\compute{\metric}}
            +
            \| \Id - \GramOperator_{k  } \| \cdot \| \extrinsic{p}_{h} \|_{\compute{\metric}}
            +
            \| ( \Id - \compute{P_{\frakH}} ) \extrinsic{P_{\frakH}} \| \cdot \| \extrinsic{p}_{h} \|_{\compute{\metric}}
        \\&\qquad \qquad 
            +
            \| \GramOperator_{k  } \extrinsic{{u}}_{h} \|_{\compute{\metric}} \cdot \| ( \Id - \extrinsic{P_{\frakH}} ) \compute{P_{\frakH}} \| 
            + 
            \| \Id - \GramOperator_{k  } \| \cdot \| \extrinsic{{u}}_{h} \|_{\compute{\metric}} 
            .
    \end{align*}
    Here, the operator norms of the operators $\Id - \GramOperator_{k}$ are defined with respect to the $\compute{\metric}$-induced norm. 
    In addition, %
    \begin{align*}
        \| \compute{p}_{h} - \extrinsic{p}_{h} \|_{\compute{\metric}}
        \leq 
        \| \compute{p}_{h} - \compute{P_{\frakH}} \extrinsic{p}_{h} \|_{\compute{\metric}}
        + 
        \| \extrinsic{p}_{h} - \compute{P_{\frakH}} \extrinsic{p}_{h} \|_{\compute{\metric}}
        .
    \end{align*}
    Together with~\eqref{math:metric-comparison-constants} and~\eqref{math:gram-operator-bound}, the desired result follows. 
\end{proof}

\begin{remark}\label{remark:comparison-to-holst-and-stern}
    Holst and Stern~\cite{HS1} discuss geometric variational crimes in an abstract Hilbert complex framework. 
    Conceptually, they embed the computational de~Rham complex into the $L^{2}$ de~Rham complex and assume the existence of a projection. 
    We establish such a projection (Theorem~\ref{theorem:smoothed-projection}), and analyze the geometric error in a conceptually different manner, only comparing the computational discrete complex~\eqref{math:computational-discrete-complex-with-metric} with the exact discrete subcomplex~\eqref{math:extrinsic-discrete-complex-with-metric}.
    This simplifies the analysis. 
    Our operator~\eqref{math:gram-operator} and the metric distortion error~\eqref{math:metric-distortion-errors} correspond to the operator $J_{h}$ and the quantity $\| I - J_{h} \|$ in~\cite{HS1}. 
\end{remark}

\subsection{Examples of computational geometries}

In the remainder of this section, we address the existence and construction of the computational manifold $\extManifold$, the extrinsic triangulation $\extMesh$, the exact transformation $\homeo \colon \extManifold \to \Manifold$, and the computational approximation of the pullback metric.
In general, the computational manifold $\extManifold$ is a known piecewise linear manifold equipped with a known triangulation $\extMesh$.
There are different possibilities, depending on the application setting.
We take the literature on surface finite element methods as our conceptual point of reference. 
By no means is the following a comprehensive survey; instead, we outline two representative examples that illustrate how our framework applies in practice. 

The error analysis of the variational crime in Theorem~\ref{theorem:variationalcrime} is abstract and does not depend on the specific construction of the computational problem.
Bounding the geometric error reduces to bounding the quantities $\CAri{k}$, $\CinvAri{k}$, $\CGramError{k}$ and $\CinvGramError{k}$ for the relevant values of $k$.

\subsubsection{Fixed coordinate charts}

 	Our first example concerns applications that come with a canonical choice of geometric transformation $\homeo \colon \extManifold \to \Manifold$ from a fixed computational manifold $\extManifold$.
    Any triangulation $\extMesh$ of $\extManifold$ is transferred onto an intrinsic triangulation $\Mesh$ of the physical manifold $\Manifold$.
    We emphasize that $\extMesh$ is equipped, on each cell, with the Riemannian metric $\euclideanmetric$ that is the pullback of the Euclidean metric tensor on the ambient space, 
    and which equips $\extMesh$ with the intrinsic path metric. 
    
    The triangulation parameters of the intrinsic triangulation are controlled in terms of the parameters of the extrinsic triangulation and the properties of the homeomorphism.
    In practice, we start with a fixed initial triangulation $\extMesh_{0}$ of $\extManifold$ and a bi-Lipschitz homeomorphism $\homeo \colon \extManifold \to \Manifold$ that is a piecewise diffeomorphism. 
    A family of triangulations is obtained from $\extMesh_{0}$ via some geometrically stable mesh refinement method. 
    Clearly, all mesh constants of the affine extrinsic triangulations are uniformly controlled, and the fixed homeomorphism $\homeo$ is a piecewise diffeomorphism,
    thus allowing uniform control of the intrinsic mesh constants. 
    
    For instance, the \emph{cubed sphere} in computational physics implements these ideas~\cite{sadourny1972conservative,ronchi1996cubed}.
    The boundary of the unit cube $[-1,1]^3$ is mapped onto the unit sphere $\mathbb{S}^{2} \subset \bbR^3$ along a bi-Lipschitz mapping with the following piecewise regularity property:
    each restriction of the transformation to the relative interior of any face of the cube is a diffeomorphism.
    Any triangulation of the cube boundary leads to an intrinsic triangulation of the sphere. 
    In the cubed-sphere setting, the homeomorphism and the extrinsic manifold are fixed. 
    
    The induced extrinsic Riemannian metric $\extrinsic{\metric}$ along any simplex $S \in \extMesh$ is
	\begin{align*}
		\extrinsic{\metric}_{|x}( \vec a, \vec b )
		:=
		\metric_{|\homeo(x)}\left( \Jacobian\homeo_{|x} \vec a, \Jacobian\homeo_{|x} \vec b \right)
        ,
	\end{align*}
	where $x$ lies in the interior of $S$ and $\vec a, \vec b$ are tangential vectors at $x$.
	We can thus formulate the intrinsic finite element problem and its equivalent counterpart on the extrinsic triangulation.
    The stability and convergence of the intrinsic finite element method are now fully controlled. 
    
    To analyze the computational finite element method, which is the one practically implemented, we examine the approximation of the Riemannian metric. 
    The extrinsic Riemannian metric $\extrinsic{\metric}$ is defined via pullback.
    The computational Riemannian metric $\compute{\metric}$ is a piecewise polynomial degree-$\ell$ interpolant of $\extrinsic{\metric}$.
    On each $n$-dimensional simplex, these metrics are smooth sections of the second symmetric power of the cotangent bundle.
    Since $\homeo$ is bi-Lipschitz and piecewise smooth, the extrinsic metric is equivalent to the Euclidean metric:
    there exists a constant $C > 0$ such that for almost all $x \in \extManifold$ and all tangential vectors $\vec a$ at $x$, we have
    \begin{align*}
        C^{-1} \euclideanmetric_{|x}( \vec a, \vec a )
        \leq
        \extrinsic{\metric}_{|x}( \vec a, \vec a )  
        \leq
        C \euclideanmetric_{|x}( \vec a, \vec a )
        .
    \end{align*}
    Using the smoothness of these metrics on each simplex, we can apply standard interpolation theory coefficientwise to obtain the following bounds with an expedient choice of degree-$\ell$ polynomial interpolant $\compute{\metric}$ of the extrinsic Riemannian metric $\extrinsic{\metric}$:
    for each $n$-dimensional simplex $S$,
    \begin{align*}
        \| \extrinsic{\metric} - \compute{\metric} \|_{\Lebesgue^{\infty}\Tensor^{2}(S,\euclideanmetric)}
        &\leq 
        C h_{S}^{\ell+1} \| \extrinsic{\metric} \|_{\Sobolev^{\ell+1,\infty}\Tensor^{2}(S,\euclideanmetric)}
        .
    \end{align*}
    Here, $C > 0$ depends on the dimension, the polynomial degree $\ell$, and on the shape regularity of the extrinsic triangulation $\extMesh$;
    see~\cite[4.3.8]{brennerscott},~\cite[Lemma~11.9]{ern2021finite}. 
    Provided that the mesh size is small enough, the interpolant is pointwise positive definite and thus a Riemannian metric. 
    
    Abbreviate $h := \max_{ S \in \extMesh } \diam(S)$ for the extrinsic triangulation's Euclidean mesh size.
    For a piecewise polynomial degree-$\ell$ interpolant of the Riemannian metric and $u \in \Lebesgue^{2}\Lambda^{k}(\extManifold,\extrinsic{\metric})$ we have  
	\begin{align*}
		\left\lvert
            \langle {u}, {u} \rangle_{\Lebesgue^{2}\Lambda^{k}(\extManifold,\extrinsic{\metric})}
            -
            \langle {u}, {u} \rangle_{\Lebesgue^{2}\Lambda^{k}(\extManifold,\compute{\metric})}
		\right\rvert
		\leq 
        C 
        \max_{ S \in \extMesh } h_{S}^{\ell+1} 
        \| \extrinsic{\metric} \|_{\Sobolev^{\ell+1,\infty}\Tensor^{2}(S,\euclideanmetric)}
        \cdot 
        \langle {u}, {u} \rangle_{\Lebesgue^{2}\Lambda^{k}(\extManifold,\extrinsic{\metric})}
		.
	\end{align*}
    Here, $C > 0$ depends on the dimension, the polynomial degree $\ell$, the shape regularity of the extrinsic triangulation $\extMesh$, 
    and the size of the extrinsic Riemannian metric in terms of the Euclidean Riemannian metric.

    In a typical scenario, refinement of the triangulation lets the mesh size $h$ go to zero, 
    so that the difference between $\extrinsic{\metric}$ and $\compute{\metric}$ goes to zero at the rate $h^{\ell+1}$. 
    This implies upper bounds for the constants $\CAri{k}$ and $\CinvAri{k}$, and forces $\CGramError{k}$ and $\CinvGramError{k}$
    to converge to zero at the rate $h^{\ell+1}$ because of~\eqref{math:metric-distortion-errors-estimate}.
    Consequently, in this scenario the geometric error converges at rate $\ell+1$.

\subsubsection{Closest-point projection}

    We discuss an example in the spirit of surface finite element methods.
    The literature on this topic is extensive, and we merely give a rough outline of how our formalism applies without giving detailed computations;
    see Figure~\ref{figure:surface-fem-setting} for an illustration.
    The physical manifold $\Manifold$ is a compact connected smooth $n$-dimensional hypersurface without boundary embedded in Euclidean space $\bbR^{n+1}$.
    Importantly, $\Manifold$ is equipped with the Riemannian metric $\metric$ induced by the Euclidean metric $\euclideanmetric$ of the ambient space.

\begin{figure}
    \begin{tikzpicture}[x=1.5cm,y=1.5cm]

        \pgfmathdeclarefunction{hat}{2}{%
          \pgfmathparse{max(0,1-abs(#1-#2))}%
        }

        \pgfmathdeclarefunction{pa}{1}{%
          \pgfmathparse{
            0.15
            + 1.60*(#1)
            + 0.40*(#1)^2
            + 0.05*(exp(#1) - 1)
          }%
        }

        \pgfmathdeclarefunction{pb}{1}{%
          \pgfmathparse{
            1.45
            + 0.20*(#1-1)
            - 1.10*(#1-1)^2
          }%
        }

        \pgfmathdeclarefunction{pc}{1}{%
          \pgfmathparse{
            0.25
            - 0.00*(#1-2)
            + 0.25*(#1-2)^2
            + 0.01*(exp(#1) - exp(2))
          }%
        }

        \pgfmathdeclarefunction{pd}{1}{%
          \pgfmathparse{
            1.50
            - 0.15*(#1-3)
            - 0.10*(#1-3)^2
          }%
        }

        \pgfmathdeclarefunction{pe}{1}{%
          \pgfmathparse{
            0.10
            - 1.60*(#1-4)
            + 0.30*(#1-4)^2
          }%
        }

        \pgfmathdeclarefunction{den}{1}{%
          \pgfmathparse{
              hat(#1,0)^2
            + hat(#1,1)^2
            + hat(#1,2)^2
            + hat(#1,3)^2
            + hat(#1,4)^2
          }%
        }

        \pgfmathdeclarefunction{Gamma}{1}{%
          \pgfmathparse{
            (
                hat(#1,0)^2*pa(#1)
              + hat(#1,1)^2*pb(#1)
              + hat(#1,2)^2*pc(#1)
              + hat(#1,3)^2*pd(#1)
              + hat(#1,4)^2*pe(#1)
            ) / den(#1)
          }%
        }

        \draw[line width=1.5pt] plot[ domain=0:4, samples=200, variable=\x ] ({\x},{Gamma(\x)}); 

        \node[left] at (0.37,0.82) {$\Gamma$};
        
    \end{tikzpicture}
    \qquad 
    \begin{tikzpicture}[x=1.5cm,y=1.5cm]

        \coordinate (P0) at (0,0.15);
        \coordinate (P1) at (1,1.45);
        \coordinate (P2) at (2,0.25);
        \coordinate (P3) at (3,1.50);
        \coordinate (P4) at (4,0.10);

        \draw[line width=0.45pt] (P0) -- (P1) -- (P2) -- (P3) -- (P4);

        \draw[dashed, line width=0.9pt]
          (P0)
          .. controls (0.15,0.52) and (0.47,1.32) .. (P1)
          .. controls (1.44,1.30) and (1.56,0.38) .. (P2)
          .. controls (2.34,0.18) and (2.66,1.67) .. (P3)
          .. controls (3.54,1.62) and (3.67,0.45) .. (P4);

        \foreach \P in {P0,P1,P2,P3,P4} \fill (\P) circle (1.5pt);

        \node[above] at (3.43,0.50) {$\Gamma_h$};

        \node[above right] at (3.40,1.17) {$\Gamma_h^{(\ell)}$};

    \end{tikzpicture}
    \caption{%
        Illustration of the geometric setting in surface finite element methods. 
        An embedded manifold $\Manifold$ (left) is approximated by a piecewise affine manifold $\Manifold_{h}$ (right solid line) whose vertices interpolate the exact surface. More generally, the approximate surface $\Manifold_{h}^{(\ell)}$ is a piecewise polynomial image of the affine surface (right dashed line).
        Equip each surface with the Riemannian metric induced by the ambient Euclidean metric. 
        $\Manifold_{h}^{(\ell)}$, being a piecewise polynomial image of $\Manifold_{h}$, carries a triangulation. 
        The closest-point projection maps this triangulation from $\Manifold_{h}^{(\ell)}$ to a smooth triangulation of $\Manifold$.  
        Finite element spaces are defined on $\Manifold_{h}$ as usual, then transported to $\Manifold_{h}^{(\ell)}$,
        and then to $\Manifold$, along the aforementioned closest-point projection.
        The exact finite element problem is defined on $\Manifold$ and transported back to an equivalent exact finite element problem on $\Manifold_{h}$. 
        Critically, the exact finite element problem on $\Manifold_{h}$ carries coefficients that are induced by the pullback from $\Manifold$. 
        If instead we pull back the Riemannian metric from $\Manifold_{h}^{(\ell)}$ onto $\Manifold_{h}$, the resulting computational finite element problem has coefficients that are computable and approximate the exact coefficients. 
        Conceptually, we may think of the exact finite element problem posed on $\Manifold$ and the computational problem posed on $\Manifold_{h}^{(\ell)}$. 
    }
    \label{figure:surface-fem-setting}
\end{figure}

    Since $\Manifold$ is a connected compact codimension-one surface, it encloses a bounded open set. %
	We define the \emph{signed distance function} $\distancefunc \colon \bbR^{n+1} \to \mathbb{R}$ by two conditions:
    first, $\absolutevalue{\distancefunc(x)} = \dist\left( x, \Manifold \right)$; second, $\distancefunc$ is negative in the bounded open set enclosed by the surface and positive in the unbounded open set surrounding the surface.
	There exists an open neighborhood $\calU \subseteq \bbR^{n+1}$ of $\Manifold$
    such that for every $x \in \calU$ there exists a unique $\cpp(x) \in \Manifold$ with $\| x - \cpp(x) \| = \absolutevalue{\distancefunc(x)}$, and we call $\cpp(x)$ the \emph{closest point projection} of $x$ onto $\Manifold$.
	Write $\nu(x)$ for the outward-pointing unit normal at $x \in \Manifold$.
    One can show that $x - \cpp(x) = \distancefunc(x) \nu(\cpp(x))$ for all $x \in \calU$. 
	We extend the unit normal to all of $\calU$ by setting $\nu(x) := \nu( \cpp(x) )$ when $x \in \calU$.
    Consequently, the decomposition $x = \cpp(x) + \distancefunc(x) \nu(x)$ holds for each $x \in \calU$.
	There exists $\delta_0 > 0$ such that the set $\calU$ contains the tubular neighborhood
	\begin{align*}
		\left\{\;
            x + \delta \nu(x)
            \suchthat
            x \in \Manifold, \, -\delta_0 < \delta < \delta_0
		\;\right\}.
	\end{align*}
    The closest point projection thus satisfies $\cpp(x) = x - \distancefunc(x) \nu(x)$ for any $x \in \calU$. 
 	After shrinking $\calU$ if necessary, the signed distance function $\distancefunc$, the closest-point projection $\cpp$, and the extended unit normal $\nu$ are smooth on $\calU$.

    The computational manifold $\extManifold$ is another embedded piecewise affine hypersurface.
    We assume that $\extManifold$ is already equipped with a triangulation $\extMesh$.
    As an initial coupling between the extrinsic surface and the physical surface, we assume that the triangulation's vertices lie on the physical manifold.
    If the mesh size of $\extMesh$ is small enough, 
    We make the standard assumption that the closest point projection gives rise to a bijection $\cpp \colon \extManifold \to \Manifold$.
    Globally, the mapping $\cpp \colon \extManifold \to \Manifold$ is generally only bi-Lipschitz.
    The restriction $\cpp \restriction_{T} \colon T \to \cpp(T)$ is a diffeomorphism of manifolds with corners for every simplex $T \in \extMesh$.
    
    We approximate $\cpp \colon \extManifold \to \Manifold$ by a continuous interpolant $\cpp_{h} \colon \extManifold \to \bbR^{n+1}$ piecewise polynomial of degree-$\ell$ with respect to $\extMesh$.
    If the mesh size is sufficiently small, then this interpolated map is a bi-Lipschitz homeomorphism onto its image,
    sending $\extManifold$ onto a Lipschitz surface $\compute{\Manifold}_{\ell}$.
    We interpret $\compute{\Manifold}_{\ell}$ as a triangulated piecewise polynomial surface;
    its purpose is to approximate the physical surface $\Manifold$ more accurately than $\extManifold$ does. 
    Note that the special case $\ell=1$ recovers the original (affine) surface, $\compute{\Manifold}_{1} = \extManifold$.
    We have a homeomorphism $\cpp_{h} : \extManifold \to \compute{\Manifold}_{\ell}$. 
    Transporting the triangulation $\extMesh$ along that homeomorphism produces the triangulation $\compute{\Mesh}_{\ell}$. 
    
    We study the closest-point projection restricted to the new surface, $\cpp \colon \compute{\Manifold}_{\ell} \to \Manifold$.
    Its (tangential) Jacobian $\Jacobian \cpp_{|x} \colon \Tangential_{x}\compute{\Manifold}_{\ell} \to \Tangential_{\cpp(x)}\Manifold$ on any smooth simplex $\compute{T} \in \compute{\Mesh}_{\ell}$ satisfies
    \begin{align}\label{math:sfem-estimates}
        \sup_{ \substack{ x \in \compute{T}_{\ell} } }
        \sup_{ \substack{ v \in \Tangential_{x}\compute{T}_{\ell} \\ \|v\| \leq 1 } }
        \| \Jacobian \cpp_{|x} v - v \|
        \leq 
        C h_{\compute{T}}^{\ell},
        \quad
        \| \det \Jacobian \cpp - \bboldone_{\compute{T}} \|_{\Lebesgue^{\infty}(\compute{T})} 
        \leq 
        C h_{\compute{T}}^{\ell+1}.
    \end{align}
    Again, the constant $C > 0$ depends only on the derivatives of $\cpp$ and the shape regularity of the mesh.
	This follows from Proposition~1.4.2 of Christiansen~\cite{christiansen2002resolution}. %

    We define the homeomorphism $\homeo \colon \extManifold \to \Manifold$ by setting $\homeo := \cpp \circ \cpp_{h}$.
    Applying~\eqref{math:sfem-estimates} allows us to estimate the piecewise bi-Lipschitz constant of $\homeo \colon \extManifold \to \Manifold$.
    We can therefore bound the shape measures of the intrinsic triangulation and the other parameters introduced in Section~\ref{section:triangulations}.
    The higher-order shape constants are bounded in terms of the higher derivatives of $\cpp$ and the triangulation shape parameters. 
    
    We define the extrinsic metric tensor $\extrinsic{\metric} := \cpp_{h}^{\pullbackast} \cpp^{\pullbackast} \metric$ by pulling back the Riemannian metric from $\Manifold$ to the intermediate surface $\compute{\Manifold}_{\ell}$ and then to $\extManifold$. 
    Within our abstract formalism, this defines the intrinsic and (exact) extrinsic finite element methods, and it establishes their stability and convergence theory. 
    
    It remains to define the computational finite element problem that is actually implemented and assess the geometric error.
    The physical problem is pulled back from $\Manifold$ onto $\extManifold$ along $\cpp \circ \cpp_{h}$. 
    The computational problem is pulled back from $\compute{\Manifold}_{\ell}$ onto $\extManifold$ along $\cpp_{h}$. 
    To complete our formalism, 
    we define the computational metric tensor $\compute{\metric} := \cpp_{h}^{\pullbackast} \euclideanmetric$, where $\euclideanmetric$ denotes the Euclidean metric induced on $\compute{\Manifold}_{\ell}$.
	After both Riemannian metrics are pulled back to $\extManifold$, 
    the geometric error that compares the exact metric $\extrinsic{\metric}$ with the computational metric $\compute{\metric}$ is controlled by the corresponding metric distortion estimate.
    A rewriting of results by Holst and Stern~\cite[Theorem~4.4, Corollary~4.5]{HS1} leads to a geometric error estimate of the form 
    $\CGramError{k} + \CinvGramError{k} \leq  C h^{\ell + 1}$
    in the mesh size $h$. This recovers the optimal error estimates for the finite element method.

\end{document}